\documentclass[10pt]{article}
\newdimen\proofrulebreadth \proofrulebreadth=.05em
\newdimen\proofdotseparation \proofdotseparation=1.25ex
\newdimen\proofrulebaseline \proofrulebaseline=2ex
\newcount\proofdotnumber \proofdotnumber=3
\let\then\relax
\def\hfi{\hskip0pt plus.0001fil}
\mathchardef\squigto="3A3B
\newif\ifinsideprooftree\insideprooftreefalse
\newif\ifonleftofproofrule\onleftofproofrulefalse
\newif\ifproofdots\proofdotsfalse
\newif\ifdoubleproof\doubleprooffalse
\let\wereinproofbit\relax
\newdimen\shortenproofleft
\newdimen\shortenproofright
\newdimen\proofbelowshift
\newbox\proofabove
\newbox\proofbelow
\newbox\proofrulename
\def\shiftproofbelow{\let\next\relax\afterassignment\setshiftproofbelow\dimen0 }
\def\shiftproofbelowneg{\def\next{\multiply\dimen0 by-1 }%
\afterassignment\setshiftproofbelow\dimen0 }
\def\setshiftproofbelow{\next\proofbelowshift=\dimen0 }
\def\setproofrulebreadth{\proofrulebreadth}

\def\prooftree{
\ifnum  \lastpenalty=1
\then   \unpenalty
\else   \onleftofproofrulefalse
\fi
\ifonleftofproofrule
\else   \ifinsideprooftree
        \then   \hskip.5em plus1fil
        \fi
\fi
\bgroup
\setbox\proofbelow=\hbox{}\setbox\proofrulename=\hbox{}%
\let\justifies\proofover\let\leadsto\proofoverdots\let\Justifies\proofoverdbl
\let\using\proofusing\let\[\prooftree
\ifinsideprooftree\let\]\endprooftree\fi
\proofdotsfalse\doubleprooffalse
\let\thickness\setproofrulebreadth
\let\shiftright\shiftproofbelow \let\shift\shiftproofbelow
\let\shiftleft\shiftproofbelowneg
\let\ifwasinsideprooftree\ifinsideprooftree
\insideprooftreetrue
\setbox\proofabove=\hbox\bgroup$\displaystyle 
\let\wereinproofbit\prooftree
\shortenproofleft=0pt \shortenproofright=0pt \proofbelowshift=0pt
\onleftofproofruletrue\penalty1
}

\def\eproofbit{
\ifx    \wereinproofbit\prooftree
\then   \ifcase \lastpenalty
        \then   \shortenproofright=0pt  
        \or     \unpenalty\hfil         
        \or     \unpenalty\unskip       
        \else   \shortenproofright=0pt  
        \fi
\fi
\global\dimen0=\shortenproofleft
\global\dimen1=\shortenproofright
\global\dimen2=\proofrulebreadth
\global\dimen3=\proofbelowshift
\global\dimen4=\proofdotseparation
\global\count255=\proofdotnumber
$\egroup  
\shortenproofleft=\dimen0
\shortenproofright=\dimen1
\proofrulebreadth=\dimen2
\proofbelowshift=\dimen3
\proofdotseparation=\dimen4
\proofdotnumber=\count255
}

\def\proofover{
\eproofbit 
\setbox\proofbelow=\hbox\bgroup 
\let\wereinproofbit\proofover
$\displaystyle
}%
\def\proofoverdbl{
\eproofbit 
\doubleprooftrue
\setbox\proofbelow=\hbox\bgroup 
\let\wereinproofbit\proofoverdbl
$\displaystyle
}%
\def\proofoverdots{
\eproofbit 
\proofdotstrue
\setbox\proofbelow=\hbox\bgroup 
\let\wereinproofbit\proofoverdots
$\displaystyle
}%
\def\proofusing{
\eproofbit 
\setbox\proofrulename=\hbox\bgroup 
\let\wereinproofbit\proofusing
\kern0.3em$
}

\def\endprooftree{
\eproofbit 
  \dimen5 =0pt
\dimen0=\wd\proofabove \advance\dimen0-\shortenproofleft
\advance\dimen0-\shortenproofright
\dimen1=.5\dimen0 \advance\dimen1-.5\wd\proofbelow
\dimen4=\dimen1
\advance\dimen1\proofbelowshift \advance\dimen4-\proofbelowshift
\ifdim  \dimen1<0pt
\then   \advance\shortenproofleft\dimen1
        \advance\dimen0-\dimen1
        \dimen1=0pt
        \ifdim  \shortenproofleft<0pt
        \then   \setbox\proofabove=\hbox{%
                        \kern-\shortenproofleft\unhbox\proofabove}%
                \shortenproofleft=0pt
        \fi
\fi
\ifdim  \dimen4<0pt
\then   \advance\shortenproofright\dimen4
        \advance\dimen0-\dimen4
        \dimen4=0pt
\fi
\ifdim  \shortenproofright<\wd\proofrulename
\then   \shortenproofright=\wd\proofrulename
\fi
\dimen2=\shortenproofleft \advance\dimen2 by\dimen1
\dimen3=\shortenproofright\advance\dimen3 by\dimen4
\ifproofdots
\then
        \dimen6=\shortenproofleft \advance\dimen6 .5\dimen0
        \setbox1=\vbox to\proofdotseparation{\vss\hbox{$\cdot$}\vss}%
        \setbox0=\hbox{%
                \advance\dimen6-.5\wd1
                \kern\dimen6
                $\vcenter to\proofdotnumber\proofdotseparation
                        {\leaders\box1\vfill}$%
                \unhbox\proofrulename}%
\else   \dimen6=\fontdimen22\the\textfont2 
        \dimen7=\dimen6
        \advance\dimen6by.5\proofrulebreadth
        \advance\dimen7by-.5\proofrulebreadth
        \setbox0=\hbox{%
                \kern\shortenproofleft
                \ifdoubleproof
                \then   \hbox to\dimen0{%
                        $\mathsurround0pt\mathord=\mkern-6mu%
                        \cleaders\hbox{$\mkern-2mu=\mkern-2mu$}\hfill
                        \mkern-6mu\mathord=$}%
                \else   \vrule height\dimen6 depth-\dimen7 width\dimen0
                \fi
                \unhbox\proofrulename}%
        \ht0=\dimen6 \dp0=-\dimen7
\fi
\let\doll\relax
\ifwasinsideprooftree
\then   \let\VBOX\vbox
\else   \ifmmode\else$\let\doll=$\fi
        \let\VBOX\vcenter
\fi
\VBOX   {\baselineskip\proofrulebaseline \lineskip.2ex
        \expandafter\lineskiplimit\ifproofdots0ex\else-0.6ex\fi
        \hbox   spread\dimen5   {\hfi\unhbox\proofabove\hfi}%
        \hbox{\box0}%
        \hbox   {\kern\dimen2 \box\proofbelow}}\doll%
\global\dimen2=\dimen2
\global\dimen3=\dimen3
\egroup 
\ifonleftofproofrule
\then   \shortenproofleft=\dimen2
\fi
\shortenproofright=\dimen3
\onleftofproofrulefalse
\ifinsideprooftree
\then   \hskip.5em plus 1fil \penalty2
\fi
}

\usepackage{latexsym}
\usepackage[all]{xy}
\usepackage{amsmath, amsthm, amssymb, amsfonts}
\usepackage{enumerate}
\usepackage[latin1]{inputenc}
\usepackage{multirow}
\usepackage{hyperref}
\usepackage{abstract, indentfirst}
\usepackage{amscd, stmaryrd, setspace,color}
\usepackage{graphicx}
\usepackage{turnstile}
\usepackage{wasysym}

\newcommand{\Set}{\mathbf{Set}}
\newcommand{\sSet}{\mathbf{sSet}}
\newcommand{\sCat}{\mathbf{sCat}}

\newcommand{\defeq}{=_{\mathrm{def}}} 
\newcommand{\Trib}{\mathbf{Trib}}

\newcommand{\Clan}{\mathbf{Clan}}

\newcommand{\Cat}{\mathbf{Cat}}

\definecolor{mat}{rgb}{0.57,0.75,1}

\newtheorem{thm}{Theorem}[subsection]
\newtheorem{prop}[thm]{Proposition}
\newtheorem{cor}[thm]{Corollary}
\newtheorem{lemma}[thm]{Lemma}

\theoremstyle{definition}
\newtheorem{defi}[thm]{Definition}

\newtheorem{rem}[thm]{Remark}

\begin{document}

\title{{\bf Notes on clans and tribes}
\vspace{1cm}}

\author{
\sc
A. Joyal \footnote{
\href{http://www.cirget.uqam.ca/}{\sc cirget, UQ\`AM}, 
\href{mailto:joyal.andre@uqam.ca}{joyal.andre@uqam.ca} 
}\\ 
}

\maketitle

\vspace{1cm}

\begin{abstract}

The purpose of these notes is to give a categorical presentation/analysis of homotopy type theory.
The notes are incomplete as they stand (October 2017). The chapter on univalent tribes is missing.
The references are not always connected to the text. A better version is in preparation.

\end{abstract}

\newpage
\setcounter{tocdepth}{3}
\tableofcontents

\newpage

\section*{Introduction}
\addcontentsline{toc}{section}{\bfseries Introduction}

Few things can better illustrate the unity of mathematics than
the homotopy interpretation of Martin-L\" of type theory discovered
by Awodey-Warren [AW] and Voevodsky [Vo].
{\it Homotopy type theory} is the new field of mathematics 
arising form these discoveries [HoTTb].
A long term goal is to develop a user-friendly computerised proof-assistant 
for homotopy theorists (and for all mathematicians).
There are many evidences  that homotopy type theory can effectively contribute to homotopy theory.
For examples, a new proof of the Balkers-Massey theorem was found by using type theoretic methods [FFLL],
the theorem was generalised to an arbitrary modality  in a higher topos [ABFJ1]
and the generalisation applied to Goodwillie's Calculus [ABFJ2].
We believe that category theory can be used as a common ground between the two fields.
The goal of the present notes is to contribute to this common ground.
The notion of tribe presented here has emerged in discussions with Steve Awodey and 
Michael Shulman during the Univalent Foundation Program at the IAS 
in 2012-2013; it is a categorical approximation of Martin-L\"of type theory;
it is closely related to the notion of fibration category introduced by Ken Brown.
We hope that the theory of tribes will be useful both in
homotopy theory and in type theory.

\medskip

The theory of tribes presented here is modular and progressive: we begin with
the theory of clans, followed by the theory of $\pi$-clans, of tribes and $\pi$-tribes.
A {\it clan} is defined to be category equipped
with a class of {\it fibrations} closed under 
composition and base changes.
The theory of clans can be regarded as a categorical version of the theory of dependant types,
without product and propositional equality.
There is a notion of anodyne map in every clan and
a clan is a {\it tribe} if every base change of an
anodyne map is anodyne and
every map can be factored as an anodyne map followed 
by a fibration. Every tribe has the structure of  a Brown fibration category.
We also introduce a notion of simplicial tribe (and of simplicial clan);
it is playing an important role in the {\it homotopy theory of tribes}
that will be developed in a subsequent paper.
A notion of semi-simplicial tribes was introduced by Kapulkin and Szimilo [KS].

\medskip

The present notes are incomplete and much remains to be done.
The notion of {\it univalent tribe} (a $\pi$-tribe with a universe satisfying Voevodsky univalence axiom)
remains to be introduced. 

\medskip

\newpage
\section{Theory of Clans}

\subsection{Basic aspects}

Recall that an object $C$ in a category $ \mathcal{E}$  is said to be {\bf carrable} if its cartesian
product $A\times C$ with any other object $A\in  \mathcal{E}$ exists.
A map $p:C\to B$ is said to be {\bf carrable} if the object $(C,p)$ of the slice category $ \mathcal{E}/B$
is carrable: this means that the fiber product $A\times_B C$
with any other map $f:A\to B$ exists (see Definition \ref{defcarrablemap}). 
$$\xymatrix{
A\times_B C  \ar[d]_{p_1}\ar[rr]^{p_2}& & C \ar[d]^{p} \\
A \ar[rr]^{f}&& B
}$$ 
The projection $p_1:A\times_B C\to A$
is called the {\bf base change} of $p$ along $f$.
We say that a class $\mathcal{F}$ of maps in $\mathcal{E}$ 
is {\bf closed under base changes} if every map in $\mathcal{F}$
is carrable, and the base change of a map in $\mathcal{F}$ along any map in  $ \mathcal{E}$ belongs to $\mathcal{F}$.

\medskip

\begin{defi} \label{defclan} If $\mathcal{E}$ is a category with terminal
object $1$, then a {\bf clan structure} on  $\mathcal{E}$  is a class  of maps $\mathcal{F}\subseteq \mathcal{E}$  satisfying the
following conditions:
  \begin{itemize}
 \item{} Every isomorphism belongs to $\mathcal{F}$;
  \item{} $\mathcal{F}$ is closed under composition and base changes;
       \item{} the unique map $X\to 1$ belongs to $\mathcal{F}$
       for every object $X\in \mathcal{E}$.
  \end{itemize}
    A map in $\mathcal{F}$ is called a {\bf fibration}.
  A {\bf clan} is a category with terminal object equipped with a clan structure.
  \end{defi}

\smallskip

We shall often picture a fibration in a clan with a two headed arrow $A\twoheadrightarrow B$.

\bigskip

Examples of clans:
 \begin{itemize}
    \item{} A category with finite limits has the structure of a clan, where every map is a fibration;  
     \item{} The category of small categories is a clan, where a fibration is an iso-fibration       
          \item{} The category of Kan complexes is a clan, where a fibration is a Kan fibration between Kan complexes.
               \item{} The category of fibrant objects of a Quillen model category is a clan, where a fibration is a fibration between fibrant objects.

\end{itemize}

\medskip

\begin{defi} \label{defcartcat} We shall say that a category 
with finite products is {\it cartesian}.
  \end{defi}

\medskip

\begin{prop}\label{tribeofgenuineprojection0}
Every clan is a cartesian category.
The cartesian product of two fibrations in a clan is a fibration.
\end{prop}

\begin{proof} 
A clan $\mathcal{E}$ has a terminal object $1$ by definition.
 If $A$ and $B$ are two objects of $\mathcal{E}$, 
 then the base change of the map $B\to 1$
 along the map $A\to 1$ exists, since the map $B\to 1$ is a fibration.
 This shows that $\mathcal{E}$ has finite cartesian products.
 Let us show that the cartesian product
of two fibrations $f:A'\to A$ and $g:B'\to B$ is a fibration. 
The map $f\times B'$ is a fibration by base change,
since the following square is cartesian.
$$\xymatrix{
A'\times B'  \ar[d]_{f\times B'}\ar[rr]^{p_1}& & A' \ar[d]^f \\
A \times B' \ar[rr]^{p_1} && A
}$$ 
Similarly, the map $A\times g$ is a fibration by base change,
since the 
following square
is cartesian.
$$\xymatrix{
A\times B' \ar[d]_{A\times g}\ar[rr]^{p_1}& & B' \ar[d]^g \\
A \times B \ar[rr]^{p_1} && B
}$$ 
It follows that the map
$f\times g=(A\times g)(f\times B')$ is 
a fibration,
since the composite of two fibrations is a fibration.
\end{proof}

\begin{defi}\label{defcartproj} If $\mathcal{E}$
is a cartesian category, 
we say that a map $p:E\to B$ in $\mathcal{E}$
is a {\bf cartesian projection} if there exists a map $p':E\to E'$ such that 
the following square is cartesian.
\begin{equation} \label{sqaredefcartproj}
\xymatrix{
E  \ar[d]_{p}\ar[rr]^{p'}& & E' \ar[d] \\
B \ar[rr] && 1
}\end{equation}
\end{defi}

In a clan, a cartesian projection $p:E\to B$ is a fibration,
since the map $E'\to 1$
in the square (\ref{sqaredefcartproj})
is a fibration.

\begin{prop}\label{tribeofgenuineprojection}
A cartesian category $\mathcal{E}$ has the structure
of a clan, where a fibration is a cartesian projection; it is the smallest clan
structure on $\mathcal{E}$.
\end{prop}

\begin{proof} An isomorphism $u:A\to B$ is a cartesian projection, since 
the following square is cartesian when the map $u$ is invertible.
$$\xymatrix{
A  \ar[d]_{u}\ar[r]^{}& \ar[d] 1 \\
B \ar[r] & 1
}$$
Let us show that the composite of two 
cartesian projections $g:E\to B$ and $f:B\to A$
is a cartesian projection. By definition, there
exists two cartesian squares:
$$
\xymatrix{
E  \ar[d]_{g}\ar[rr]^{g'} \ar@{}[drr]|{(1)} & & E' \ar[d] \\
B \ar[rr] && 1
}\quad \quad \quad
\xymatrix{
B  \ar[d]_{f}\ar[rr]^{f'} \ar@{}[drr]|{(2)} & & B' \ar[d] \\
A \ar[rr] && 1
}$$
The bottom square $(c)$ of the following commutative diagram 
is cartesian, since the square $(2)$ above is cartesian.
The square $(b)$ is cartesian by construction.
The composite $(a)+(b)$ of the top squares of the diagram 
is cartesian, since the square $(1)$ above is cartesian.
$$\xymatrix{
E  \ar[d]_{g}\ar[rr]^{(f'g,g')}\ar@{}[drr]|{(a)}& &B'\times  E' \ar[d]_{p_1} \ar[rr]^{p_2} \ar@{}[drr]|{(b)} && E' \ar[d]  \\
B  \ar[d]_f  \ar[rr]^{f'} \ar@{}[drr]|{(c)} && B'\ar[d] \ar[rr] && 1 \\
A\ar[rr] && 1
}$$
It then follows by Lemma \ref{lemmacartesiansq} that the square $(a)$ is cartesian.
Hence the composite square $(a)+(c)$ is cartesian  by the same lemma, since square $(c)$
is cartesian. This shows that the map $fg:E\to A$ is a cartesian projection.
Let us show that every
cartesian projection $p$ is carrable.
By definition, a cartesian projection $p:E\to B$
is the base change of a map $E'\to 1$.
But the map $E'\to 1$ is carrable, since every object is carrable 
in a cartesian category. It follows that the map $p:E\to B$
is carrable, since a base change of a carrable map is carrable.
\end{proof}

\medskip

If $B$ is an object of a clan $\mathcal{E}$, we shall denote by $\mathcal{E}(B)$
the full subcategory of $\mathcal{E}/B$ whose objects are the fibrations $X\twoheadrightarrow B$.
Let us say that a morphism $f:(X,p)\to (Y,q)$ in $\mathcal{E}(B)$  is a {\it fibration}
if the map $f:X\to Y$ is a fibration in $\mathcal{E}$.
$$\xymatrix{
X \ar[dr]_p \ar[rr]^f && Y \ar[dl]^q\\
& B &
}$$
Observe that $\mathcal{E}(1)=\mathcal{E}/1=\mathcal{E}$.

\begin{prop}\label{clanslice}
The category $\mathcal{E}(B)$ has
the structure of a clan with the fibrations defined above.
\end{prop}

\begin{proof} Left to the reader
\end{proof}

\begin{defi} 
We shall say that $\mathcal{E}(B)$ is the {\bf local clan} of $\mathcal{E}$
at $B$. 
 \end{defi}

Recall that a functor $F:\mathcal{E}\to \mathcal{E}'$ is said to
{\it preserve} the base change of a map $p:X\to B$ along a
map $f:A\to B$ if it takes the cartesian square 
$$\xymatrix{
A\times_B X  \ar[d]_{p_1}\ar[rr]^{p_2}& & X \ar[d]^{p} \\
A \ar[rr]^{f}&& B
}$$ 
to a cartesian square.

 \begin{defi} \label{defmorphismclans} We say that a functor between clans $F:\mathcal{E}\to \mathcal{E}'$
 is a {\it morphism of clans} 
if he following conditions hold:
 \begin{itemize}
 \item{} $F$ takes fibrations to fibrations;
  \item{} $F$ preserves base changes of fibrations;
 \item{} $F$ takes terminal objects to terminal objects.
 \end{itemize}
  \end{defi}

\medskip

 We shall denote by  by $\Clan$ the category whose objects are small clans and whose morphisms
 are morphisms of clans.
 The category $\Clan$ has the structure of a 2-category,
 where a  2-cell is a natural transformation. We say that a morphism of clans
 is an {\it equivalence} if it is an equivalence in this 2-category.

\begin{rem} 
If a functor  isomorphic to a morphism of clans, then it is a morphism of clans.
 \end{rem}

 A clan is a cartesian category by Proposition \ref{tribeofgenuineprojection0}.

\begin{prop}\label{morphismtribecart}
A morphism of clans is a cartesian functor.
\end{prop}

\begin{proof}  The functor $F:\mathcal{E} \to \mathcal{E}'$ 
preserves base changes of fibrations. Hence it takes a cartesian square
$$\xymatrix{
A\times B  \ar[d]_{p_1}\ar[rr]^{p_2}& & B \ar[d] \\
A \ar[rr]^{}&& 1
}$$ 
to a cartesian square,
$$\xymatrix{
F(A\times B)   \ar[d]_{F(p_1)}\ar[rr]^{F(p_2)}& & F(B) \ar[d]\\
F(A) \ar[rr]^{}&& F(1)
}$$
 since the map $B\twoheadrightarrow 1$ is a fibration. 
The result follows, since $F(1)$ is a terminal object.
\end{proof}

If $F:\mathcal{E} \to \mathcal{E}'$ is a morphism of clans,
we shall denote by $F_{(A)} :\mathcal{E}(A) \to \mathcal{E}'(FA)$
the functor defined by putting $F_{(A)}(E,p)=(FE,F(p))$
for an object $(E,p)\in \mathcal{E}(A)$.

\begin{prop}\label{inducedhomoclan}
If $F:\mathcal{E} \to \mathcal{E}'$ is a morphism of clans,
then so is the functor
$$F_{(A)} :\mathcal{E}(A) \to \mathcal{E}'(FA)$$
induced by $F$ for every object $A\in \mathcal{E} $.
\end{prop}

\begin{proof} Left to the reader.
\end{proof}

Let $f:A\to B$ be a map in a clan $\mathcal{E}$.
Then the base change of a fibration $p:X\twoheadrightarrow B$ 
along $f$ is a fibration $p_1:A\times_B X\twoheadrightarrow A$.
$$\xymatrix{
A\times_B X   \ar[d]_{p_1}\ar[rr]^{p_2}& & X \ar[d]^{p} \\
A \ar[rr]^{f}&& B.
}$$
Let us put $f^\star(X,p)= (A\times_B X,p_1)$.
This defines 
the  {\it base change functor} 
$$f^\star:\mathcal{E}(B)\to \mathcal{E}(A)$$
We shall often denote the projection $p_2:A\times_BX \to X$ by $f_X:f^\star(X)\to X$.
 If $u:(X,p_X)\to (Y,p_Y)$ is a map in $\mathcal{E}(B)$,
then $f^\star(u)\defeq A\times_Bu$ and the following diagram commutes
  \begin{equation}\label{squareforbasechange6}
  \xymatrix{
  f^\star(X)  \ar[d]_{ f^\star(u)} \ar[rr]^{f_X}& & X \ar[d]_{u}  \ar@/^1,5pc/[dd]^-{p_X}   \\
  f^\star(Y)  \ar[d]^{}\ar[rr]^{f_Y}& & Y \ar[d]_{p_Y}   \\
A \ar[rr]^{f}& & B.  }
\end{equation}

\begin{lemma}\label{squareforbasechange5}
The top square in the diagram (\ref{squareforbasechange6}) is cartesian for any map  $u:(X,p_X)\to (Y,p_Y)$ in $\mathcal{E}(B)$.
 \end{lemma}

\begin{proof}
 The bottom square of the diagram \ref{squareforbasechange6} is cartesian by construction:
The composite square is also cartesian by construction.
Hence the top square is cartesian by Lemma \ref{lemmacartesiansq}.
\end{proof}

Recall from Proposition \ref{clanslice} that if $\mathcal{E}$ is a clan, then so is the category $\mathcal{E}(A)$ 
for every object $A\in \mathcal{E}$.

\begin{prop}\label{basechangetribe} If  $\mathcal{E}$ is a clan,
then the base change functor $f^\star:\mathcal{E}(B)\to \mathcal{E}(A)$
is a morphism of clans for every map $ f:A\to B$ in $ \mathcal{E}$.
\end{prop}

\begin{proof} If $u:(X,p_X)\to (Y,p_Y)$ is a fibration in $\mathcal{E}(B)$,
let us show that the map $f^\star(u):f^\star(X)\to f^\star(Y)$
is a fibration  in $\mathcal{E}(A)$. But the map $f^\star(u)$
is the base change of the map $u:X\to Y$ along the map
$f_Y:f^\star(Y)\to Y$, since the square (\ref{squareforbasechange6})
is cartesian by Lemma \ref{squareforbasechange5}.
Thus, $f^\star(u)$ is a fibration, since the base change of a fibration
is a fibration.
The functor $f^\star$ preserves terminal objects, since
the following square is cartesian
\[
 \xymatrix{
A  \ar[d]_{1_A} \ar[rr]^{f}& & B \ar[d]^{1_B}   \\
A  \ar[rr]^{f}& & B
}  
\]
Let us show that the functor $f^\star$ preserves base changes of fibrations.
Let $g:X\to Y$ a fibration in $\mathcal{E}(B)$ and 
$p_1:W\to Z$ be the base change of $g$
along a map $u:Z\to Y$ in $\mathcal{E}(B)$.
\begin{equation}\label{anothercartsquare7}
\xymatrix{
W \ar[d]_{p_1}  \ar[rr]^{p_2} && X \ar[d]^{g} \\
Z\ar[rr]^{u} && Y
}
\end{equation}
We wish to show that the left hand face of following commutative cube is cartesian:
\[
\xymatrix{
f^\star(W) \ar[dr]  \ar[rrr]^{f_W}  \ar[ddd] & &  & W  \ar '[d] [ddd] \ar[dr] &  \\
  & f^\star(X)\ar[ddd]_(0.35){f^\star(g)} \ar[rrr]^(0.4){f_X} &  & & X\ar[ddd]^(0.35){g} \\ 
  &&&&\\
 f^\star(Z) \ar[dr]_{f^\star(u)}    \ar '[r] [rrr]^(0.3){f_Z} & &  &Z  \ar[dr]^u  &   \\
  & f^\star(Y)\ar[rrr]^{f_Y} &  & & Y \ }
\]
The right hand face of the cube is cartesian by hypothesis.
The front and back faces are cartesian by Lemma \ref{squareforbasechange5}.
It then follows from the cube lemma \ref{cubelemmacartesiansq} that 
the left hand face is cartesian.
This shows that the functor $f^\star$ preserves base changes of fibrations.
We have proved that it is a morphism of clans.
\end{proof}

\begin{prop}\label{compositebasechangetribe} If $f:A\to B$ and $g:B\to C$
are two maps in a clan $\mathcal{E}$, then the composite of the base change functors 
$$\xymatrix{
\mathcal{E}(C)\ar[r]^{g^\star} &
 \mathcal{E}(B)  \ar[r]^{f^\star} & \mathcal{E}(A)
}$$
is isomorphic to the base change functor 
$(gf)^\star:\mathcal{E}(C)\to \mathcal{E}(A)$
\end{prop}

\begin{proof} If $X=(X,p)\in \mathcal{E}(C)$, 
then we have the following diagram
of cartesian squares
$$\xymatrix{
f^\star(g^{\star}(X))  \ar[d] \ar[rr]^{}& &g^{\star}(X) \ \ar[d] \ar[rr] && X \ar[d]^p\\
A \ar[rr]^{f} && B \ar[rr]^{g}  && C
}  $$ 
The composite square is cartesian by Lemma \ref{lemmacartesiansq}.
Thus $f^\star(g^{\star}(X)) $ is canonically isomorphic to $(gf)^\star(X)$.
\end{proof}

If $A$ is an object of a clan,  we shall denote $e_A:\mathcal{E}\to \mathcal{E}(A)$ the base change functor 
along the unique map $t_A:A\to 1$.
By definition, we have  $e_A(X)=(A\times X,p_1)$ for every object $X\in \mathcal{E}$.
The functor $e_A$ is a morphism of clans by Proposition \ref{basechangetribe}.

\begin{cor}\label{compositebasechangetribecor} If $f:A\to B$ is a map
 in a clan $\mathcal{E}$, then the following triangle of functors
 commutes up to a canonical isomorphism
$$\xymatrix{
\mathcal{E}\ar[rr]^{e_B}  \ar[drr]_{e_A} && \mathcal{E}(B)  \ar[d]^{f^\star}\\
& &\mathcal{E}(A)  }$$
\end{cor}

\begin{proof}  If $t_A:A\to 1$ and $t_B:B\to 1$ are the canonical maps,
then we have $t_Bf=t_A$. 
The result then follows form Proposition \ref{compositebasechangetribe}, since $e_A=t_A^\star$
and $e_B=t_B^\star$.
\end{proof} 

\medskip
Let $f:A\twoheadrightarrow B$ be a fibration in a clan $\mathcal{E}$, 
then $(A,f)\in \mathcal{E}(B)$. 

\begin{lemma}\label{compositebasechangetribecor2} Let $f:A\twoheadrightarrow B$ be a fibration
 in a clan $\mathcal{E}$. Then $\mathcal{E}(A)=\mathcal{E}(B)(A,f)$.
 \end{lemma}

\begin{proof} The functor  $\Phi:\mathcal{E}(A) \to \mathcal{E}(B)(A,f)$
 defined by putting $\Phi(E,p)=((E,fp),p)$ is an isomorphism of clans.
  $$
\xymatrix{ 
E\ar[r]^{p}  \ar[dr]_{fp} & A\ar[d]^f \\
 & B 
}$$
The inverse isomorphism takes an object $((E,g),p)\in \mathcal{E}(B)(A,f)$ to 
the object $(E,p)\in \mathcal{E}(A)$.
\end{proof}

\begin{lemma} \label{compositionelementary1} 
Let $f:A\twoheadrightarrow B$ be a fibration in a clan $\mathcal{E}$.
Then the elementary morphism of clans
$$e_{(A,f)}: \mathcal{E}(B)\to  \mathcal{E}(B)(A,f)$$
coincides with
the base change functor
$f^\star:   \mathcal{E}(B) \to  \mathcal{E}(A)$.
\end{lemma}

\begin{proof} 
The functor $e_{(A,f)}$
takes an object $(X,p)\in \mathcal{E}(B)$
to the object $((A,f)\times_B (X,p), p_1)=((A\times_B X, fp_1), p_1)$
which is identified with the object $(A\times_B X, p_1)$
of $\mathcal{E}(A)$.
$$
\xymatrix{ 
A\times_B X \ar[d]_{p_1} \ar[rr] && X \ar[d]^{p}  \\
A\ar[rr]^(0.4)f && B 
}
$$
\end{proof}

\subsection{Types, judgments, elements and contexts}
 
\bigskip

We may adopt the language of type theory, by saying that an object $E$ in a clan $\mathcal{E}$ is a {\bf type}
and write
 $$\mathcal{E} \vdash E:Type$$
 When the category $\mathcal{E}$ is clear from the context, we may write more simply 
 \begin{equation}\label{firstjudgment}
 \vdash E:Type
  \end{equation}
The expression (\ref{firstjudgment}) is an instance of what is called a {\bf judgment} in type theory.
If $1$ is the terminal object of $\mathcal{E}$,
then a map $a:1 \to E$ is 
an {\bf element} or a {\bf point} of type $E$ and we write 
$$\mathcal{E} \vdash a:E \quad {\rm or\ more\ simply } \quad \vdash a:E$$

\begin{rem}
An element $a:E$ is often called a {\it term} by type theorists. But this terminology is incorrect, since the notion of term is syntaxical
while the notion of element is semantical. 
\end{rem}

\begin{rem}
The language of type theory may include maps between types $f:E\to F$ in addition to elements. 
If $f\in \mathcal{E}(E,F)$, we may write $ \vdash f:E\to F.$
\end{rem}

In type theory, the assertion that two elements $a:E$ and $b:E$ are equal is written formally as a judgment:
$\vdash a=b:E$.
Similarly, the assertion that two types $E$ and $F$ are equally 
is written formally as a judgment $\vdash E=F$.

\begin{rem}
The equality relation $a=b$ introduced here is said to be {\bf intentional}.  
It should not be confused with the {\bf propositional} equality relation $a\simeq b$
 introduced later with the notion of tribes.
\end{rem}

\medskip

If $p:E\to A$ is a fibration in a clan $ \mathcal{E}$,
then the object $(E,p)\in \mathcal{E}(A)$
is defining a {\bf dependant type in context} $A$.
Recall that the {\bf fiber} of $p:E\twoheadrightarrow A$
at an element $x:A$ is the object $E\langle x\rangle=x^\star(E) $ defined by the pullback square
$$\xymatrix{
E\langle x\rangle  \ar[d] \ar[rr]& & E \ar[d]^{p} \\
1 \ar[rr]^{x}&& A.
}$$
This defines a {\bf family} of objects $E\langle x\rangle$
indexed by a variable element $x:A$. 
The family $(E\langle x\rangle|\ x\!:\!A)$ 
is said to be {\bf internal} to the clan, since it is indexed by an object $A$ of the clan.
In type theory, the object $(E,p)\in \mathcal{E}(A)$ is regarded
as a type $E\langle x\rangle$ which {\bf depends} on a variable element  $x:A$.
The assertion that
$E\langle x\rangle $ is a {\bf dependant type}, depending on $x:A$,
is written as a judgement:
$$x:A\vdash E\langle x\rangle :Type$$
The expression $x:A$ on the left hand side of the symbol $\vdash$
is called the {\bf context} of the judgment.
The context describes the parameter space of the dependant type.
The notions of fibration $p:E\twoheadrightarrow A$ and of dependant type $(E,p)\in \mathcal{E}(A)$ are equivalent.
From a section $s:A\to E$ of a fibration $p:E\to A$ ($ps=1_A$)
we obtain a family of elements $s(x):E(x)$, one for each $x:A$.
$$\xymatrix{
E\langle x\rangle  \ar[rr] \ar[d] && E\ar[d]_{p}\\
  \ar@/^1.5pc/[u]^-{s(x)}  1 \ar[rr]^{x}  &&A.   \ar@/_1.5pc/[u]_-{s}  
}$$
The assertion that the element $s(x):E\langle x\rangle$ depends on $x:A$
can be written as a judgement:
$$x:A\vdash s(x):E\langle x\rangle $$

 \bigskip
 
  A map $f:(E,p)\to (F,q)$ in $\mathcal{E}(A)$ induces
 a map $f\langle x\rangle=x^\star(f) :E\langle x\rangle \to F\langle x\rangle$ for each element $x:A$.
   \[
\xymatrix{
E\langle x\rangle \ar[d]_{f\langle x\rangle}    \ar[rr]   &  & E  \ar[d]^f   \\
  F\langle x\rangle \ar[rr]  \ar[d] & & F  \ar[d]^q  \\
 1\ar[rr]^x && A 
   }
\]
Hence a map $f:(E,p)\to (F,q)$ in $\mathcal{E}(A)$ can be regarded as a {\bf family of maps} $f\langle x\rangle:E\langle x\rangle\to F\langle x\rangle$
indexed by a variable element $x:A$. 
The assertion that $f\langle x\rangle:E\langle x\rangle\to F\langle x\rangle$
 is a map depending on $x:A$ can be written as a judgement:
$$x:A\vdash f\langle x\rangle:E\langle x\rangle\to F\langle x\rangle$$

\medskip

In type theory, most calculations involve changes of context, moving back and forth between
different contexts.  Recall that that for every map $f:A\to B$ in a clan $\mathcal{E}$,
we have a base change functor $f^\star:\mathcal{E}(B)\to \mathcal{E}(A)$.

\begin{prop}\label{basechangetribe2}  
If $f:A\to B$ is a map in a clan $\mathcal{E}$, then
 we have 
 $$E\langle f(x)\rangle=f^\star(E)\langle x \rangle $$
  for every object $E=(E,p)\in \mathcal{E}(B)$ and every element $x:A$.
 \end{prop}

\begin{proof}
The two squares of the following diagram are cartesian by construction.
$$\xymatrix{
 f^\star(E)\langle x \rangle  \ar[d]_{}\ar[rr]^{} && f^\star(E)  \ar[d]_{}\ar[rr]^{}& & E \ar[d]^{p} \\
1 \ar[rr]^x&& A \ar[rr]^{f}&& B
}$$ 
Hence the composite square is cartesian by lemma \ref{lemmacartesiansq}.
$$\xymatrix{
 f^\star(E)\langle x \rangle  \ar[d]_{} \ar[rr]^{}& & E \ar[d]^{p} \\
1 \ar[rr]^{f(x)} && B
}$$ 
Thus $E\langle f(x)\rangle=f^\star(E)\langle x \rangle$. 
\end{proof}

In type theory, the base change functor $f^\star:\mathcal{E}(B)\to \mathcal{E}(A)$
is expressed by the following {\it substitution rules} for dependant types and elements:
 \[
\begin{prooftree} 
 y:B \vdash E\langle y \rangle:Type
 \justifies
x:A\vdash E\langle f(x)\rangle:Type
\end{prooftree}
\quad \quad  \quad
\begin{prooftree} 
 y:B \vdash s(y):E\langle y \rangle
 \justifies
x:A\vdash s(f(x)):E\langle f(x)\rangle
\end{prooftree}
\]
It can include a substitution rule for dependant maps:
\[
\begin{prooftree} 
 y:B \vdash u\langle y\rangle:E\langle y \rangle \to F\langle y \rangle
 \justifies
x:A\vdash u\langle f(x)\rangle :E\langle f(x)\rangle  \to F\langle f(x) \rangle
\end{prooftree}
\]

\medskip

Recall that if $A$ is an object of a clan $\mathcal{E}$, then the functor 
$e_A:\mathcal{E}\to \mathcal{E}(A)$ defined by putting $e_A(E)=(A\times E,p_1)$
is a morphism of clans (called {\it elementary} morphism).
The fiber of the projection $p_1:A\times E\to A$ at $x:A$ is canonically isomorphic to $E$,
since the following square is cartesian,
 $$\xymatrix{
E  \ar[d]\ar[rr]^(0.4){x\times E}&& A\times E \ar[d]^{p_1} \\
1 \ar[rr]^{x}&&A.
}$$
Thus, $e_A(E)\langle x\rangle =E$ for every $x:A$.
In type theory, the functor $e_A:\mathcal{E}\to \mathcal{E}(A)$
 is expressed by the following {\it weakening rules}: 
 \medskip
 \[
\begin{prooftree} 
\vdash E:Type
 \justifies
x:A\vdash E:Type.
\end{prooftree}
\quad \quad \quad
\begin{prooftree} 
\vdash t:E 
 \justifies
x:A\vdash t: E.
\end{prooftree}
\quad \quad \quad
\begin{prooftree} 
\vdash u:E\to F 
 \justifies
x:A\vdash u: E\to F
\end{prooftree}
\]

\bigskip

Notice that $e_A(A)=(A\times A,p_1)\in  \mathcal{E}(A)$.
The diagonal
$\delta_A:A\to A\times A$ is a section of the projection $p_1:A\times A\to A$
and its value at $x:A$ is the element $x$ of $e_A(A)\langle x\rangle =A$.
$$\xymatrix{
A  \ar[d]\ar[rr]^(0.4){x\times A}&& A\times A \ar[d]_{p_1} \\
 \ar@/^1pc/[u]^-{x}  1  \ar[rr]^{x}  &&A   \ar@/_1pc/[u]_-{\delta_A}  
}$$
Hence the diagonal $\delta_A:A\to A\times A$ defines the element $x:A$ in context $x:A$.
The judgment 
$$x:A\vdash x:A$$
is a basic {\it axiom} of type theory.

\medskip

\subsection{Sums}

If $A$ is an object of a clan $ \mathcal{E}$,
then the forgetful functor $\mathcal{E}(A)\to \mathcal{E}$
takes a fibration $p:E\to A$
to its domain $E$. 
 Intuitively, we have
$$E=\sum_{x:A} E\langle x\rangle$$
since the domain of a fibration $p:E\to A$ is the disjoint
union of its fibers.  The forgetful functor $\mathcal{E}(A)\to \mathcal{E}$
can be denoted as a summation operation
$$\Sigma_A:\mathcal{E}(A)\to \mathcal{E}$$
In type theory, the functor $\Sigma_A$ is created 
by the following $\Sigma$-{\it formation rule}:
\medskip
 \[
\begin{prooftree} 
 x:A\vdash E\langle x\rangle:Type
 \justifies
\vdash \sum_{x:A} E\langle x\rangle :Type
\end{prooftree}
\]
\medskip
For typographical reasons, we may write $\Sigma(x:A) E\langle x\rangle$ instead of $ \sum_{x:A} E\langle x\rangle$.

\bigskip

If $E=(E,p)\in   \mathcal{E}(A)$,
 then the  unit of the adjunction 
$$\Sigma_A:\mathcal{E}(A)\longleftrightarrow \mathcal{E}:e_A$$
is given by the map $\eta_E\defeq (p,1_E):E\to A\times E$ in $ \mathcal{E}(A)$.
The map $\eta_E$ defines the family of {\it inclusions} 
$$x:A \vdash \eta\langle x\rangle :E\langle x\rangle \rightarrow E$$
If $a:A$ and $b:E\langle a\rangle$, then the element $\eta\langle a\rangle(b)$ is denoted $(a,b)$.
In type theory, the maps $\eta_E$ are created by the
$\Sigma$-{\it introduction rule}:
\medskip
 \[
\begin{prooftree} 
\vdash   a:A \quad \quad \vdash b:E\langle a\rangle
 \justifies
\vdash (a,b): \sum_{x:A} E\langle x\rangle
\end{prooftree}
\]
The fibration $p:E\to A$ is the first projection $p_1: \sum_{x:A} E\langle x\rangle\to A$ 
and it is called the {\it display map} of the sum.

\medskip

 It follows from the adjunction $\Sigma_A\dashv e_A$
 that for every object $B\in \mathcal{E}$ and every dependant  map
  $$x:A\vdash h\langle x\rangle:E\langle x\rangle \to B$$
 there exits a unique map 
 $$\vdash h:\sum_{x:A} E\langle x\rangle \to B$$
such $h\circ \eta\langle x\rangle=h\langle x\rangle$ for every element $x:A$.
By construction, $h(x,y):=h\langle x\rangle(y)$ for $x:A$ and  $y:E\langle x\rangle$.

\medskip

Notices that we have $\Sigma_A (e_A(X))=A\times X$
for every object $X\in \mathcal{E}$. Thus,
$$\sum_{x:A} X =A\times X.$$ 
In particular, $\sum_{x:A} 1 =A$.

\bigskip

If $f:A\to B$ is a fibration in a clan $\mathcal{E}$,
then the {\bf summation functor}  
\begin{equation}\label{sumfunctor}
\Sigma_f :\mathcal{E}(A)\to \mathcal{E}(B) \, .
\end{equation}
is defined by putting $\Sigma_f(E,p)=(E,fp)$ for a fibration $p:E\twoheadrightarrow A$.

\begin{prop} \label{sumalong} 
Let $f:A\twoheadrightarrow B$ be a fibration in a clan $\mathcal{E}$.
Then the summation functor $\Sigma_f :\mathcal{E}(A)\to \mathcal{E}(B) $
is left adjoint to the base change functor $f^\star: \mathcal{E}(B)\to \mathcal{E}(A)$.
The functor $\Sigma_f$ preserves fibrations and base change of fibrations.
 \end{prop}

 \begin{proof}  Left to the reader.  \end{proof}

  \begin{prop}\label{bergerprinciple} 
If $f:A\to B$ and $g:B\to C$ are fibrations in a clan $\mathcal{E}$, then
$\Sigma_{gf}=\Sigma_g \Sigma_f$.
In particular, $\Sigma_{A}=\Sigma_B \Sigma_f$.
\end{prop}
 
   \begin{proof}  For every $(E,p)\in  \mathcal{E}(A)$ we have
$$\Sigma_g( \Sigma_f(E,p))=  \Sigma_g(E,fp)= (E, gfp)=\Sigma_{gf}(E,p).$$
  \end{proof}

It follows from Proposition \ref{bergerprinciple} that if $A$ and $B$
are two objects of clan $\mathcal{E}$ and if
$p_1:A\times B\to A$ and $p_2:A\times B\to B$ are
the projections.
then we have 
\begin{equation} \label{0Fubini}
\Sigma_B\Sigma_{p_2}=\Sigma_{A\times B}=\Sigma_A \Sigma_{p_1}
\end{equation}
We may denote the functor $\Sigma_{p_1}$ by
 $\Sigma_B:\mathcal{E}(A\times B)\to \mathcal{E}(A)$
and the functor $\Sigma_{p_2}$ by 
 $\Sigma_A:\mathcal{E}(A\times B)\to \mathcal{E}(B)$.
With this notation, the identity (\ref{0Fubini}) takes the following form:
$$\Sigma_B\Sigma_{A}=\Sigma_{A\times B}=\Sigma_A \Sigma_{B}$$
This is {\it Fubini theorem}:
\begin{equation} \label{1Fubini}
\sum_{y:B} \sum_{x:A}E\langle x, y\rangle =\sum_{(x,y):A\times B}E\langle x, y\rangle= \sum_{x:A}  \sum_{y:B} E\langle x, y\rangle .
\end{equation}

\begin{prop}\label{bergerprinciple2} 
  If $f:A\to B$ is a fibration in a clan $ \mathcal{E}$,  
then for every object $(E,p)\in \mathcal{E}(A)$ and every element $y:B$ we have
$$
\Sigma_f(E)\langle y \rangle=\sum_{x:A(y)} E\langle x \rangle=\sum_{f(x)=y} E\langle x \rangle.
$$
where $A\langle y \rangle$ is the fiber of $f$ at $y:B$.
\end{prop}

  \begin{proof} 
  The map $p:E\to A$ induces a map $p:(E,fp)\to (A,f)$
  in $ \mathcal{E}(B)$ and we have $(E,fp)=\Sigma_f(E,p)$.
 If $y:B$ and $x:A\langle y \rangle$, then we have a commutative diagram
  $$\xymatrix{
(\Sigma_fE)\langle y\rangle\langle x\rangle  \ar[rr] \ar[dd]&& (\Sigma_fE)\langle y\rangle \ar[dd]^{p\langle y\rangle}\ar[rr] & &E \ar[dd]^p \\
& (3)&&(2)&\\
1 \ar[rr]^{x}&& A\langle y \rangle \ar[dd] \ar[rr]^{}&& A \ar[dd]^f  \\
&&&(1)&\\
&& 1 \ar[rr]^{y} && B.
}$$
The squares (1) and (3) in this diagram are cartesian by construction. 
The square (2) is cartesian by Lemma \ref{squareforbasechange5}.
Hence the composite square (1+2) is cartesian by Lemma \ref{lemmacartesiansq}.
Thus, $(\Sigma_fE)\langle y\rangle\langle x\rangle=E\langle x\rangle$.
Hence we have 
$$(\Sigma_fE)\langle y\rangle=\sum_{x:A(y)} (\Sigma_fE)\langle y\rangle\langle x\rangle=\sum_{x:A(y)} E\langle x\rangle$$
 \end{proof}

\subsection{The category of clans}

 \begin{defi} \label{reflectsfibrations} 
We say that a morphism of clans $F:\mathcal{E}\to \mathcal{E}'$ 
{\it reflects} fibrations, if the implication 
$$F(f)\enspace  {\rm is\ a\  fibration} \Rightarrow f  \enspace  {\rm is \ a\   fibration} $$
is true for every map $f:A\to B$ in $\mathcal{E}$.
 \end{defi}

Recall that the category of clans has the structure of a 2-category
where a 2-cell is a natural transformation. There is a notion
of equivalence in any 2-category: a morphism of clans $F:\mathcal{E}\to \mathcal{E}'$ 
is an {\it equivalence of clans} if and only if there exists a morphism  of clans $G:\mathcal{E}'\to \mathcal{E}$ 
together with two natural isomorphism $GF\simeq Id$ and $FG\simeq Id$

\begin{prop}\label{equivalenceofclans}
A morphism of clans $F:\mathcal{E}\to \mathcal{E}'$ is an equivalence of clans
if and only if it is an equivalence of categories and it reflects fibrations.
\end{prop}

\begin{proof} Left to the reader.
\end{proof}

\begin{rem} 
If $F:\mathcal{E}\to \mathcal{E}'$ is an equivalence of categories and 
one of the categories $\mathcal{E}$ or $ \mathcal{E}'$ has the structure of a clan
then the other category has a unique clan structure for which the functor $F$ is an equivalence of clans.
\end{rem}

\medskip

\begin{defi} \label{defembedding} 
 We say that a morphism of clans $F:\mathcal{E}\to \mathcal{E}'$
 is an {\bf embedding} if it is fully faithful and it reflects fibrations.
  \end{defi}

 \begin{defi} \label{defsubclan} 
 If $\mathcal{E}$ is a clan, we say that a full sub-category $\mathcal{L}\subseteq  \mathcal{E}$
is a {\bf sub-clan} if the following two conditions hold:
\begin{enumerate}
\item{} if $f:A\to B$ and $p:E\to B$ are maps in $\mathcal{L}$ and $p$ is a fibration in $\mathcal{E}$,
then there exists a cartesian square in $\mathcal{E}$ 
$$\xymatrix{
C  \ar[d]_{}\ar[rr]^{}& & E \ar[d]^{p} \\
A \ar[rr]^{f}&& B
}$$
with $C\in \mathcal{L}$.
\item{} the subcategory $\mathcal{L}$ contains an object which is terminal in $\mathcal{E}$;
\end{enumerate}
 \end{defi} 
 
   For example, the full subcategory $\{\top\}$ spanned by a single terminal
object $\top$ of a clan $\mathcal{E}$ is a sub-clan of $\mathcal{E}$.

 \begin{prop} \label{subclan=embedding} 
A sub-clan $\mathcal{L}\subseteq  \mathcal{E}$ has the structure of a clan, where a map
 in $\mathcal{L}$ is a fibration if it is a fibration in $\mathcal{E}$.
 Moreover, the inclusion functor $\mathcal{L}\to \mathcal{E}$ is an 
 embedding.  
 \end{prop}
 
 \medskip

Recall that a full sub-category $\mathcal{L}\subseteq \mathcal{E}$ 
is said to be {\it replete} if every object of $\mathcal{E}$ which is isomorphic
to an object in $\mathcal{L}$ belongs to $\mathcal{L}$.
Every full sub-category $\mathcal{L}\subseteq \mathcal{E}$ 
is contained in a smallest replete full sub-category its {\it replete closure}
$\mathcal{L}^{rep} \subseteq \mathcal{E}$.
An object $X\in \mathcal{E}$ belongs to $\mathcal{L}^r$
if and only if $X$ is isomorphic to an object of $\mathcal{L}$.
Moreover, the inclusion functor $\mathcal{L} \to \mathcal{L}^{rep}$
is an equivalence of categories.

 \begin{prop} \label{subclanvsreplete} 
A full sub-category $\mathcal{L}$ of a clan $\mathcal{E}$
is a sub-clan if and only its replete closure $\mathcal{L}^{rep} \subseteq \mathcal{E}$
is a sub-clan. 
 \end{prop}

\begin{proof} Left to the reader.
\end{proof} 

\medskip

The intersection of a family of replete sub-clans of a clan $\mathcal{E}$ is a replete sub-clan.
Every sub-category $\mathcal{C} \subseteq \mathcal{E}$ is contained 
in a smallest replete sub-clan $\overline{\mathcal{C}}$, the (replete) sub-clan 
{\it generated} by $\mathcal{C} $.

\medskip

The image of a fully faithful functor  $F:\mathcal{E}\to \mathcal{E}'$ is a full subcategory $F(\mathcal{E})\subseteq  \mathcal{E}'$
and the induced functor $\mathcal{E} \to F(\mathcal{E})$ is an equivalence of categories.
If $F:\mathcal{E}\to \mathcal{E}'$ is an embedding of clans, then $F(\mathcal{E})\subseteq  \mathcal{E}'$
is a sub-clan and the induced functor $\mathcal{E} \to F(\mathcal{E})$ is an equivalence of clans.

\medskip

Recall that the category of small categories $\Cat$ has the structure of clan in which
a fibration is an iso-fibration.
We denote the category of small clans and morphisms of clans by $\Clan$.

\begin{lemma}\label{pullbackofclans} Suppose that we have a pullback square of categories
$$\xymatrix{ 
\mathcal{A}\times_\mathcal{B}\mathcal{E} \ar[rr]^{}   \ar[d]_{}   && \mathcal{E} \ar[d]^F \\
\mathcal{A}\ar[rr]^U &&  \mathcal{B}
}
$$
in which the functors $U$ and $F$ are morphisms of clans.
If the functor $F$ is an isofibration,
then the category $\mathcal{A}\times_\mathcal{B}\mathcal{E}$ has the structure of a clan
in which a map $(f_1,f_2):(X_1,X_2) \to (Y_1,Y_2)$ is a fibration if the maps $f_1$ and $f_2$
are fibrations. Moreover, the square is a pullback in the category $\Clan$.
\end{lemma}

\begin{proof}  Left to the reader.
\end{proof}

\begin{prop}\label{clanofclans} The category of small clans
$\Clan$ has the structure of a clan, 
where a morphism of clans is a fibration if it is an iso-fibration.
The forgetful functor $\Clan \to \Cat$ is a morphism of clans.
\end{prop}

\begin{proof} This follows from Lemma \ref{pullbackofclans}.
\end{proof}

\begin{cor}\label{0productofclans} The cartesian 
product of two clans $\mathcal{E}_1$ and $ \mathcal{E}_2$ has the structure of a clan,
where a map $(f_1,f_2):(X_1,X_2)\to (Y_1,Y_2)$ in $ \mathcal{E}_1\times  \mathcal{E}_2$ 
is a fibration if $f_1$ is a fibration in  $ \mathcal{E}_1$
and $f_2$ is a fibration in  $ \mathcal{E}_2$. 
The terminal category has the structure of a clan.
 If $ \mathcal{E}$
is a clan, then the cartesian product functor $\times:\mathcal{E}\times \mathcal{E}\to  \mathcal{E}$
is a morphism of clans.
\end{cor}

\begin{proof} Left to the reader.
\end{proof}

\begin{prop}\label{localisofibration} If a morphism of clans $F:\mathcal{E} \to \mathcal{E}'$ is 
an isofibration, so is the functor $F_{(A)}:\mathcal{E}(A) \to \mathcal{E}'(FA)$
induced by $F$ for every object $A\in \mathcal{E}$.
\end{prop}

 \begin{proof} Let us show that the functor $F_{(A)}$ is an isofibration.
 If $(X,p)\in \mathcal{E}(A)$ and $u:(Y,q)\to (FX,F(p))$ is an isomorphism 
 in $ \mathcal{E}'(FA)$, then 
there exists an isomorphism $v:Z\to X$ in $\mathcal{E}$
such that $F(v)=u$, since the functor $F$  is an isofibration.
The map $v:(Z,pv)\to (X,p)$ is an isomorphism
and we have $F(v)=u:(Y,q)\to (FX,F(p))$, since
$F(Z,pv)=(FZ, F(p)F(v))=(Y, F(p)u)=(Y,q)$.
\end{proof}

\subsection{Generic elements}

Recall that if $K$ is a commutative ring,
then the polynomial ring $K[x]$ is obtained
by adding freely a new element $x$ to $K$.
The freeness of the extension $i:K\to K[x]$ means
that if $f:K\to R$ is a homomorphism of commutative rings, then for every element $r\in R$,
then  there exists a unique
homomorphism $h:K[x]\to R$
such that $h i=f$ and $h(x)=r$,
$$\xymatrix{
K \ar[rr]^{i}\ar[drr]_{f}&&  K[x]  \ar[d]^{h}\\
&& R}
$$
In other words, the element $x\in K[x]$  can
take a value $r\in R$ which is chosen freely.
We shall say that $x$ is {\bf generic}.
Let us denote by $Hom'(K, R)$
the set of pair $(f,r)$, where $f:K\to R$ is a homomorphism of commutative
rings and $r\in R$. Then  the map 
$$i^\star: Hom(K[x], R)\simeq Hom'(K,R)$$
defined by putting $i^\star(h)=(hi,h(x))$ 
is bijective for every commutative ring $R$.

\medskip

If $A$ is an object of a clan $\mathcal{E}$, then the morphism of clans
$e_A:\mathcal{E}\to  \mathcal{E}(A)$  takes an object $E$ to
the object  $e_A(E)=(A\times A,p_1)$.
In particular, $e_A(A)=(A\times A,p_1)$ and 
$e_A(1)=(A,1_A)$ is the terminal object $\top_A$ of the clan $\mathcal{E}(A)$.
The diagonal $\delta_A:A\to A\times A$ is a section of
the projection $p_1:A\times A\to A$.
It thus defines an element $\delta_A: e_A(A)$ in $ \mathcal{E}(A)$.

 \begin{lemma} \label{genericfiber}
 If $p:E\to A$ is a fibration in  $\mathcal{E}$, 
 then the object $(E,p)$ of $ \mathcal{E}(A)$
 is the fiber  of the fibration $e_A(p):e_A(E)\to e_A(A)$ at $\delta_A:e_A(A)$.
 \end{lemma}

\begin{proof}
The right hand square of the following diagram is cartesian by construction:
 $$\xymatrix{
E \ar[d]_p\ar[rr]^{(p,1_E)} && A\times E \ar[d]^{A\times p} \ar[rr]^{p_2} && E\ar[d]^p \\
A \ar[rr]^{\delta_A}&& A\times A \ar[rr]^{p_2} && A
}$$
The composite of the two squares is trivially cartesian.
Hence the left hand square is cartesian by 
Lemma  \ref{lemmacartesiansq}.
This shows that  following square is cartesian in $\mathcal{E}(A)$,
\begin{equation}\label{tautologicalpullback2}
   \xymatrix{
(E,p) \ar[d]_p\ar[rr]^{(p,1_E)} & & e_A(E) \ar[d]^{e_A(p)} \\
(A,1_A) \ar[rr]^{\delta_A}&& e_A(A).
}\end{equation}
This proves that $(E,p)$ 
 is the fiber at $\delta_A:e_A(A)$ of the fibration $e_A(p):e_A(E)\to e_A(A)$
\end{proof}

\medskip

Let $A$ be an object in a clans $\mathcal{E}$.
If $\mathcal{R}$ is a clan,  
let us denote by $Hom(\mathcal{E}, \mathcal{R})_A$
the category whose objects are the pairs $(F,a)$, where $F:\mathcal{E} \to \mathcal{R}$
is a morphism of clans and $a:F(A)$.
A map $(F,a)\to (G,b)$ in this category is a natural transformation $\phi:F\to G$
such that $\phi_A(a)=b$. 
We shall say that a morphism of clans $e:\mathcal{E} \to \mathcal{E}'$ 
is {\it freely generated} by an element $x_A:e(A)$ if the 
the functor
\begin{equation}\label{restrictionalongfree}
e^\star:Hom(\mathcal{E}',  \mathcal{R})\to Hom(\mathcal{E}, \mathcal{R})_A  \, 
\end{equation}
defined by putting  $e^\star(F)=(F\circ e,F(x_A))$ 
is an equivalence of categories for every clan $ \mathcal{R}$.
When this condition is satisfied, we say 
that the element $x_A:e(A)$ is {\it generic}
and write that $\mathcal{E}'=\mathcal{E}[x_A]$.

\medskip

 \begin{thm} \label{genericelement} 
The morphism of clans $e_A:\mathcal{E} \to \mathcal{E}(A)$ 
is freely generated by the element $\delta_A:e_A(A)$.
 \end{thm}

 \begin{proof} If $ \mathcal{R}$ is a clan, let us show that the functor 
\begin{equation} \label{genericelemequiv}
e_A^\star:Hom(\mathcal{E}(A),  \mathcal{R})\to Hom(\mathcal{E}, \mathcal{R})_A  \, 
\end{equation}
defined by putting  $e_A^\star(F)=(F\circ e_A,F(\delta_A))$ 
 is an equivalence of categories. We shall first verify that the functor $e_A^\star$ is essentially surjective.
 Let $F:\mathcal{E}\to \mathcal{R}$ be a morphism of clans
 and let $a: F(A)$. The functor
$F_{(A)}:   \mathcal{E}(A) \to   \mathcal{R}(FA) $
 induced by $F$ is a morphism of clans by 
\ref{inducedhomoclan}, hence also the functor
$G\defeq a^\star \circ F_{(A)}$, where $a^\star:  \mathcal{R}(F(A)) \to   \mathcal{R}$
is the base change functor along the map $a:1 \to F(A)$.
By definition of the functor $G$,  
we have a cartesian square
 \begin{equation}\label{defH}
 \xymatrix{
G(E,p) \ar[d]_{p_1}\ar[rr]^{p_2} & & F(E) \ar[d]^{F(p)} \\
1 \ar[rr]^{a}&& F(A).
} \end{equation}
for every object $(E,p)\in  \mathcal{E}(A)$.
Let us put $ \mathsf{in}{(E,p)}:=p_2$.
If $\top_A=(A,1_A)$ denotes the terminal
object of $\mathcal{E}(A) $, then we have $G(\top_A)=1$ and $\mathsf{in}(\top_A)=a$,
since the following square is cartesian.
$$\xymatrix{
1 \ar@{=}[d] \ar[rr]^{a}& & F(A) \ar[d]^{F(1_A)} \\
1 \ar[rr]^{a}&& F(A)
}$$
Let us construct a natural isomorphism  $\phi: G\circ e_A \simeq F$.
 If $X\in \mathcal{E}$, we have $e_A(X)=(A\times X, p_1)$
and the left hand square of  the following diagram
is cartesian, since the square (\ref{defH}) is cartesian,
 $$\xymatrix{
G(e_A(X)) \ar[d]_{}\ar[rr]^{\mathsf{in}(e_A(X))}& & F(A\times X) \ar[d]^{F(p_1)} \ar[rr]^{F(p_2)} && F(X) \ar[d]\\
1 \ar[rr]^{a}&& F(A) \ar[rr] && F(1).
}$$
The right hand square is also cartesian, since the functor $F$ preserves products.
Hence the composite square is cartesian by lemma \ref{lemmacartesiansq}.
But the map $1\to F(A) \to  F(1)$ is 
invertible, since the object $F(1)$ is terminal.
It follows that
the map 
$$\phi_X\defeq F(p_2)\circ  \mathsf{in}(e_A(X)):G(e_A(X)) \to FX$$ 
 is invertible.
This defines a natural isomorphism $\phi:G\circ e_A \simeq F$.
Let us show that  $\phi_A(G(\delta_A))=a$. 
Observe that the map $ \mathsf{in}(E,p):G(E,p)\to F(E)$ is a natural transformation between two functors of $(E,p)\in  \mathcal{E}(A)$.
Hence the following square commutes by the naturally, 
$$\xymatrix{
G(\top_A)   \ar[d]_{G(\delta_A)}\ar[rr]^{ \mathsf{in}(\top_A)}& & F( A) \ar[d]^{F(\delta_A)}  \\
G(e_A(A)) \ar[rr]^{ \mathsf{in}(e_A(A))}&& F(A\times A).
}$$
Thus, $ \mathsf{in}(e_A(A))G(\delta_A)=F(\delta_A) \mathsf{in}(\top_A)$.
We saw above that  $ \mathsf{in}(\top_A)=a$. Thus,
\begin{eqnarray*}
 \phi_AG(\delta_A)&=&F(p_2) \mathsf{in}(e_A(A))G(\delta_A)\\
 &=& F(p_2) F(\delta_A) \mathsf{in}(\top_A)\\
 &=& F(p_2\delta_A) a \\
 &=& F(1_A)a \\
  &=& 1_{F(A)}a \\
  &=& a.
\end{eqnarray*}
We have proved that the functor $e_A^\star$ is essentially surjective.
It remains to show that it is fully faithful.
Let $G_1, G_2:\mathcal{E}(A)\to  \mathcal{R}$ be two morphisms of clans
and let $\alpha:G_1\circ e_A\to G_2\circ e_A$ be a natural transformation
such that $\alpha_A G_1(\delta_A)=G_2(\delta_A)$,
$$\xymatrix{
1  \ar[d]_{G_1(\delta_A)}\ar@{=}[rr]& & 1 \ar[d]^{G_2(\delta_A)}  \\
G_1(e_A(A)) \ar[rr]^{\alpha_A}&& G_2(e_A(A)).
}$$
We shall prove that there is a unique natural transformation $\beta:G_1\to G_2$
such that $\beta\circ e_A=\alpha$.
Consider the following diagram of solid arrows:
\begin{equation}\label{cubegeneric}
\xymatrix{
G_1(E,p) \ar@{..>}[dr]^(0.5){\beta(E,p)}  \ar[rrr] \ar[ddd]  & &  & G_1(e_A(E))\ar '[d] [ddd]^(0,3){G_2(e_A(p))}   \ar[dr]^{\alpha_E}  &  \\
  &  G_2(E,p)  \ar[ddd]_(0.45){} \ar[rrr] &  & &  G_2(e_A(E))\ar[ddd]^{G_2(e_A(p))}  \\ 
  &&&&\\
 1 \ar@{=}[dr] \ar '[r] [rrr]^(0.3){G_1(\delta_A)} & &  &G_1(e_A(A))   \ar[dr]^{\alpha_A}  &   \\
  & 1 \ar[rrr]^{G_2(\delta_A)} &  & &  G_2(e_A(A))  }.
\end{equation}
The front face of the box is the image of the square (\ref{tautologicalpullback2}) by the functor $G_2$
and the back face is the image of the same square by $G_1$.
Thus, both faces are cartesian, since 
the functors $G_1$ and $G_2$ preserve base changes of fibrations
and the square (\ref{tautologicalpullback2}) is the base change square of a fibration by Lemma \ref{genericfiber}.
The bottom face of the box commutes by the hypothesis on $\alpha$.
The right hand face of the box commutes by naturality of $\alpha$. 
It follows that there is a unique map $\beta(E,p):G_1(E,p) \to G_2(E,p)$
such that the resulting cube (\ref{cubegeneric}) commutes. 
Let us show that $\beta(E,p)$ is the unique map $G_1(E,p) \to G_2(E,p)$
such that  the following square commutes,
\begin{equation}\label{magicsquare}
\xymatrix{
 G_1(E,p)   \ar@{..>}[d]_{\beta(E,p)}\ar[rr]^{G_1(p, 1_E)} & & G_1(e_A(E)) \ar[d]^{\alpha_E}  \\
 G_2(E,p) \ar[rr]^{G_2(p, 1_E)}&& G_2(e_A(E)).
}
\end{equation}
The map $G_2(\delta_A)$ in the front face of
cube (\ref{cubegeneric}) is monic, since its domain is a terminal object. 
It follows that $G_2(p,1_E)$ is monic, since the front face of the cube is cartesian.
Thus, $\beta(E,p)$ is the unique map
for which the square (\ref{magicsquare}) commutes.
Let us now prove that the map $\beta(E,p):G_1(E,p) \to G_2(E,p)$
is natural in $(E,p)\in \mathcal{E}(A)$. For this we have to
show that the following square commutes for any map $f:(E,p)\to (L,q)$ in $ \mathcal{E}(A)$,
\begin{equation}\label{natursquareh1h2}
\xymatrix{
  G_1(E,p)   \ar[rr]^{\beta(E,p)}\ar[d]_{G_1(f)} & &   G_2(E,p) \ar[d]^{G_2(f)}  \\
G_1(L,q) \ar[rr]^{\beta(L,q)} && G_2(L,q).
}
\end{equation}
We shall use the following cube:
\[
\xymatrix{
G_1(E,p)  \ar[dr]^{\beta(E,p)}  \ar[rrr]^{G_1(p, 1_E)} \ar[ddd]_{G_1(f)}  & &  & G_1(e_A(E))\ar '[d] [ddd]^(0,3){G_1(e_A(f))}   \ar[dr]^{\alpha_E}   &  \\
  &  G_2(E,p)  \ar[ddd]_(0.40){G_2(f)} \ar[rrr]^(0.4){G_2(p, 1_E)} &&& G_2(e_A(E))\ar[ddd]^{G_2(e_A(f))} \\ 
  &&&&\\
  G_1(L,q)  \ar[dr]^{\beta(L,q)}   \ar '[r] [rrr]^(0.3){^{G_1(q, 1_L)}} & &  &G_1(e_A(L))  \ar[dr]^{\alpha_L}  &   \\
  &  G_2(L,q) \ar[rrr]^{^{G_2(q, 1_L)}} &  & &  G_2(e_A(L))  }.
\]
The front face commutes, since the following 
 square commutes
$$   \xymatrix{
(E,p) \ar[d]_f\ar[rr]^{(p,1_E)} & & e_A(E) \ar[d]^{e_A(f)} \\
(L,q) \ar[rr]^{(q,1_L)} & & e_A(L) .
}$$
Similarly, the back face of the cube commutes.
The top face commutes by definition of $\beta(E,p)$ in the square (\ref{magicsquare}).
Similarly, the bottom face commutes by definition of $\beta(L,q)$.
The right hand face of the cube commutes by naturally of $\alpha$.
It follows that the left hand face commutes, since the map  $G_2(q,1_L)$ 
is monic. The naturality of $\beta$ is proved. 
Let us show that $\beta\circ e_A=\alpha$.
For this, we have to show that $\beta(A\times X,p_1)=\alpha_X$
for every $X\in  \mathcal{E}$.
If $(E,p)=(A\times X,p_1)$, then $e_A(E)=(A\times A\times X,p_1)$
and $(p,1_E)=\delta_A\times X$.
Hence the left hand square of the following diagram commutes, since the square (\ref{magicsquare}) commutes.
$$
\xymatrix{
 G_1(A\times X,p_1)   \ar[d]_{\beta(A\times X,p_1)}\ar[rr]^{G_1(\delta_A\times X)} & & G_1(A\times A\times X,p_1) \ar[d]^{\alpha_{A\times X}} \ar[rr]^{G_1(A\times p_2)} & &  G_1( A\times X,p_1)\ar[d]^{\alpha_X} \\
 G_2(A\times X,p_1) \ar[rr]^{G_2(\delta_A\times X)}&& G_2(A\times A\times X,p_1)
  \ar[rr]^{G_2(A \times p_2)} & &  G_2( A\times X,p_1)
}
$$
The right hand square commutes by the naturally of $\alpha$ applied to the map $p_2: A\times X \to  X$.
But the composite of the horizontal arrows of this diagram are identities, since
$(A\times p_2)(\delta_A\times X)=1_{A\times X}$.
This shows that $\beta(A\times X,p_1)=\alpha_X$.
The existence of $\beta$ is proved. Let us prove the uniqueness of $\beta$.
Let $\beta':G_1\to G_2$ a natural transformation such that $\beta'\circ e_A=\alpha$.
If $(E,p)\in  \mathcal{E}(A)$, then the  following square commutes by naturality of $\beta'$
applied to the map
$(p, 1_E):(E,p)\to (A\times E,p_1)=e_A(E)$,
$$
\xymatrix{
 G_1(E,p)   \ar[d]_{\beta'(E,p)}\ar[rr]^{G_1(p, 1_E)} & & G_1(e_A(E)) \ar[d]^{(\beta'\circ e\!_A)(E)}  \\
 G_2(E,p) \ar[rr]^{G_2(p, 1_E)}&& G_2(e\!_A(E)).
}$$
But we have $(\beta'\circ e_A)(E)=\alpha_E$ by the hypothesis on $\beta'$.
Hence the following square commutes 
$$
\xymatrix{
 G_1(E,p)   \ar[d]_{\beta'(E,p)}\ar[rr]^{G_1(p, 1_E)} & & G_1(e_A(E)) \ar[d]^{\alpha_E}  \\
 G_2(E,p) \ar[rr]^{G_2(p, 1_E)}&& G_2(e_A(E)).
}$$
We saw above that the map $G_2(p, 1_E)$ is monic.
It follows that $\beta'(E,p)=\beta(E,p)$, since the square (\ref{magicsquare})
commutes.
 \end{proof}

 The element $\delta_A:e_A(A)$ of the extension $e_A:\mathcal{E} \to \mathcal{E}(A)$ 
 is generic by theorem \ref{genericelement}. 
We shall write that  $\mathcal{E}(A)= \mathcal{E}[\delta_A]$.

\subsection{The Yoneda embedding}

Let $\Hat{\mathcal{E}}:=[\mathcal{E}^{op},Set]$
be the category of presheaves on a small category $\mathcal{E}$
and let $y:\mathcal{E}\to \Hat{\mathcal{E}}:$
be the Yoneda functor. 
For every object $A\in \mathcal{E}$ and every 
$X\in \Hat{\mathcal{E}}$  the map 
$$\theta:Hom(yA,X)\to F(A)$$
defined by putting $\theta(\alpha)=\alpha(1_A)$
is bijective (Yoneda lemma).
We shall often regard the functor $y$ as an inclusion and write $A$ instead of $yA$.
We shall often identify a natural transformation $\alpha:yA \to X$
with the element $a=\theta(\alpha)$ and write $a:A\to X$.
In this notation, we have
$$F(A)=Hom(yA,X)=Hom(A,X).$$
Moreover, 
if $u:A\to B$ is a map in $\mathcal{E}$
and $b\in X(B)$, then we shall denote the element
$F(u)(b)\in X(A)$ as a composite $bu:A\to X$.
$$
\xymatrix{ 
A \ar@/^1.5pc/[rr]^{bu} \ar[r]^u & B \ar[r]^(0.4){b} & X 
}
$$
Similarly,  a natural transformation  $f:X\to Y$ in $\Hat{\mathcal{E}}$
defines a map $f_A:X(A)\to Y(A)$ for every object $A\in \mathcal{E}$;
if $a\in F(A)$, we shall denote the element $f_A(a)\in F(A)$
as a composite $fa:A\to Y$
$$
\xymatrix{ 
A \ar@/^1.5pc/[rr]^{fa} \ar[r]^a & X \ar[r]^(0.4){f} & Y
}
$$

\begin{defi} 
A map $f:X\to Y$ in $\Hat{\mathcal{E}}$
is said to be {\it representable} if the presheaf 
$b^\star(X)=B\times_Y X$ is representable
$$
\xymatrix{ 
b^\star(X) \ar[d]_{}\ar[r]  & X \ar[d]^f  \\
B\ar[r]^b& Y
}$$
for every object $B\in \mathcal{E}$
and every map $b:B\to Y$, 
\end{defi}

\begin{prop} \label{prop:reppseudoclans} Let $\mathcal{E}$ be a small category.
Then the class of representable maps in the category $\Hat{\mathcal{E}}:=[\mathcal{E}^{op},Set]$
is closed under composition and base changes. 
\end{prop} 

\begin{proof} Left to the reader.\end{proof}

\begin{defi} \label{def:extendedfibrations} If $ \mathcal{E} $ is a clan, 
we say that a map $f:X\to Y$ in $\Hat{\mathcal{E}}$
is an {\it extended fibration} if it is representable
and the map $b^\star(X) \to B$
is a fibration for every object $B\in \mathcal{E}$
and every map $b:B\to Y$.
\end{defi}

\begin{prop} \label{prop:extpresheaf} Let $\mathcal{E}$ be a clan.
Then the class of extended
fibrations in the category $\Hat{\mathcal{E}}:=[\mathcal{E}^{op},Set]$
 is closed under composition and base changes. 
\end{prop} 

\begin{proof} Left to the reader.\end{proof}

Let $\mathcal{E}$ be a clan.  If $F\in \Hat{\mathcal{E}}$,
we shall denote by $ \Hat{\mathcal{E}}(F)$ the full subcategory
of $ \Hat{\mathcal{E}}/F$ whose objects are extended
fibrations $X\to F$. If $g:F\to G$ is a map in $\Hat{\mathcal{E}}$,
then the base-change functor $g^\star:\Hat{\mathcal{E}}/G\to \Hat{\mathcal{E}}/F$
induces a functor
$$g^\star:\Hat{\mathcal{E}}(G)\to \Hat{\mathcal{E}}(F)$$

\begin{prop} \label{prop:extlocalclans} Let $\mathcal{E}$ be a clan.
If $F\in \Hat{\mathcal{E}}$, then the category $ \Hat{\mathcal{E}}(F)$
has the structure of a clan where a map $f:(X,p)\to (Y,q)$
is a fibration if the map $f:X\to Y$ is an extended fibration.
Moreover, if $g:F\to G$ is a map in $\Hat{\mathcal{E}}$,
then the functor $g^\star:\Hat{\mathcal{E}}(G)\to \Hat{\mathcal{E}}(F)$
is a morphism of clans.
\end{prop} 

\begin{proof} Left to the reader.\end{proof}

Notice that if $A\in \mathcal{E}$, then the 
Yoneda functor $y:\mathcal{E}\to \Hat{\mathcal{E}}$
induces an equivalence of clans
$\mathcal{E}(A)\simeq \Hat{\mathcal{E}}(A)$.

\subsection{Reedy fibrant squares and cubes}

If $\mathcal{E}$ is a clan and $\mathcal{K}$ is a small category,
we shall say that a natural transformation $\alpha:F\to G$
between two functors $F,G:\mathcal{K}\to \mathcal{E}$
is a {\bf pointwise fibration} if the map $\alpha(k):F(k)\to G(k)$ is a fibration for every object $k\in \mathcal{K}$.
The category $\mathcal{E}^\mathcal{K}$ of functors $\mathcal{K}\to \mathcal{E}$
has a the structure of a clan in which a fibration is a pointwise fibration.

\medskip

If $\mathcal{E}$ is a category and $[1]$ is the poset $\{0,1\}$, then
an object of the category $\mathcal{E}^{[1]}$ is a map $a:A_0\to A_1$
in the category $\mathcal{E}$, and a {\it morphism} $f:a\to b$ in $\mathcal{E}^{[1]}$
is a pair of maps $(f_0;f_1)$ in a commutative square
\begin{equation}\label{thefirstsquarereally}
\xymatrix{ 
A_0  \ar[rr]^{f_0}   \ar[d]_{a}   &&B_0 \ar[d]^{b} \\
A_1 \ar[rr]^{f_1} && B_1
}
\end{equation}
If $\mathcal{E}$ is a clan, then the  pointwise fibrations give 
 the category $\mathcal{E}^{[1]}$ the structure of a clan.

\medskip

A commutative square in a category $\mathcal{E}$
is a functor $X:[1]\times [1]\to \mathcal{E}$.
The four sides of the square can be denoted as follows
\begin{equation}\label{fibrantsquaredef}
\xymatrix{ 
 X_{00}   \ar[rr]^{X_{\star 0}}  \ar[dd]_{X_{0 \star}}  &&X_{10}\ar[dd]^{X_{1 \star}} \\
& X_{\star\star} & \\
 X_{01} \ar[rr]^{X_{\star 1}} && X_{11}
}
\end{equation}

\begin{defi}
If the fiber product $X_{01} \times_{ X_{11} }{X_{10}} $ exists, we shall say that
the map 
$$(X_{0\star}, X_{\star 0}):X_{00}\to X_{01}\times_{X_{11}} X_{10} $$
is the (cartesian) {\bf gap map} of the square (\ref{fibrantsquaredef}). 
\end{defi}

\begin{defi} 
If $\mathcal{E}$ is a clan, we say that a commutative square $X:[1]\times [1]\to \mathcal{E}$ is {\bf Reedy fibrant}
 if the maps  
 $$X_{\star 1}:X_{01}\to X_{11}, \quad  X_{1\star}:X_{10}\to X_{11}  \quad {\rm and}\quad  (X_{0\star}, X_{\star 0}):X_{00}\to X_{01}\times_{X_{11}} X_{10} $$ 
are fibrations.
\end{defi}

\medskip

In other words, a square $X:[1]\times [1]\to \mathcal{E}$ is Reedy fibrant if and only $X_{\star 1}$, 
$ X_{1\star}$ and its gap map are fibrations,

The four sides of a Reedy fibrant square $X$ are fibrations.
To see this, observe that the projection $p_1$ in the following diagram
 is a fibration by base change, since the map $X_{1\star}$ is a fibration:
$$\xymatrix{ 
 X_{00}    \ar@/_1.5pc/[ddr]_-{X_{0\star}}    \ar@/^1.5pc/[drrr]^-{X_{\star 0}}      \ar@{-->}[dr]^{ (X_{0\star}, X_{\star 0})} & && \\
 &X_{01}\times_{X_{11}} X_{10}  \ar[rr]^{p_2}   \ar[d]_{p_1}   &&X_{10}\ar[d]^{X_{1\star}}  \\
 & X_{01}\ar[rr]^{X_{\star 1}} &&X_{11}
}
$$
It follows that $X_{0\star}:X_{00}\to X_{01} $ is a fibration, since 
$X_{0\star}=p_1(X_{0\star}, X_{\star 0})$
and $(X_{0\star}, X_{\star 0})$ is a fibration.
Similarly, $X_{\star 0}:X_{00}\to X_{10} $ is a fibration.

\medskip

The following example of Reedy fibrant squares will be useful later.

\begin{prop}\label{projectionsReedyfibrantquares}
If $f_1:A_1\to B_1$ and $f_2:A_2 \to B_2$
are fibrations in a clan,
then the following two squares are Reedy fibrant:
$$
\xymatrix{ 
  A_1 \times A_2\ar[d]_{p_1} \ar[rr]^{f_1\times f_2} && B_1\times B_2\ar[d]^{p_1} \\
 A_1 \ar[rr]^{f_1} && B_1
}\quad \quad
\xymatrix{ 
  A_1 \times A_2\ar[d]_{p_2} \ar[rr]^{f_1\times f_2} && B_1\times B_2\ar[d]^{p_2}\\
 A_2 \ar[rr]^{f_2} && B_2
}
$$
\end{prop}

\begin{proof}  Let us show that the first square is Reedy fibrant.
Consider the following diagram with a pullback square:
$$\xymatrix{ 
 A_1\times A_2  \ar@{-->}[rrd]^{A_1\times f_2}  \ar@/^1.5pc/[drrrr]^-{f_1\times f_2}   \ar@/_1.5pc/[ddrr]_-{}   &&&&\\
& & A_1\times B_1  \ar[d]_{p_1} \ar[rr]^{f_1\times B_1}   && B_1\times B_2 \ar[d]^{p_1} \\
& & A_1\ar[rr]^{f_1} && B_1
}
$$
The map $f_1$ is a fibration by hypothesis.
The projection $p_1:B_1\times B_2\to B_1$ is a fibration by Proposition \ref{tribeofgenuineprojection0}.
Also the map $A_1\times f_2$, since $f_2$ is a fibration.
This completes the proof that  the first square is Reedy fibrant.
 \end{proof}

\medskip

\begin{lemma}\label{compostionfibsquare}
The class of Reedy fibrant squares is closed under horizontal (and vertical) composition.
\end{lemma}

Ê
\begin{proof} Suppose that the squares of the following diagram are  Reedy fibrant:
\begin{equation}\label{compositeoffibrantsquares}
\xymatrix{ 
A_0  \ar[rr]^{f_0}   \ar[d]_{a}   \ar@{}[rrd]|{(1)} &&B_0\ar[d]^b \ar[rr]^{g_0}     \ar@{}[rrd]|{(2)} &&C_0\ar[d]^c  \\
A_1 \ar[rr]^{f_1} && B_1  \ar[rr]^{g_1} && C_1
}
\end{equation}
Let us show that the composite square (1)+(2) is Reedy fibrant.
The maps $f_1$, $g_1$ and $c$ are fibrations, since the squares (1) and (2) are Reedy fibrant.
Hence the maps $g_1f_1$ and $c$ are fibrations.
Let us show that the gap map of the composite square (1)+(2) is a fibration.
By taking fiber products and using Lemma \ref{lemmacartesiansq}, 
we can construct the following diagram with three pullback squares:
$$\xymatrix{ 
A_0  \ar[drr]^{f_0}   \ar[d]_{(a,f_0)}   &&&&  \\
A_1\times_{B_1} B_0  \ar[rr]^{}   \ar[d]_{\gamma}  \ar@{}[rrd]|{(3)} &&B_0\ar[d]_{(b,g_0)} \ar[drr]^{g_0}  &&  \\
A_1\times_{C_1} C_0  \ar[rr]^{}   \ar[d]_{}   \ar@{}[rrd]|{(4)} &&B_1\times_{C_1} C_0\ar[d] \ar[rr]^{}    \ar@{}[rrd]|{(5)} &&C_0\ar[d]^c  \\
A_1 \ar[rr]^{f_1} && B_1  \ar[rr]^{g_1} && C_1
}
$$
The gap map of square (1) is the map $(a,f_0)$, the gap map of square (2) is the map $(b,g_0)$ 
and the gap map of the composite square (1)+(2) is the $\gamma (a,f_0)$.   
The maps $(a,f_0)$ and $(b,g_0)$ are fibrations, since squares (1) and (2) are Reedy fibrant.
The map $\gamma$ is also a fibration by base change, since square (3) is cartesian.
It follows that the map $\gamma (a,f_0)$ is a fibration.
This completes the proof that the composite square (1)+(2) is Reedy fibrant.
\end{proof} 

\medskip

If $\mathcal{E}$ is a clan, we shall denote by $\mathcal{E}^{(1)}$
the full subcategory of $\mathcal{E}^{[1]}$
whose objects are the fibrations $a:A_0\to A_1$ in $\mathcal{E}$.
The codomain functor $\partial_1:\mathcal{E}^{(1)}\to \mathcal{E}$
is a Grothendieck fibration. Notice that $\partial_1^{-1}(A)=\mathcal{E}(A)$ for every object $A\in \mathcal{E}$.
A morphism $f:a\to b$ in $ \mathcal{E}^{(1)}$
is $\partial_1$-cartesian if and only if the associated square is cartesian:
\begin{equation}\label{anothersquarereally}
\xymatrix{ 
A_0  \ar[rr]^{f_0}   \ar[d]_{a}   &&B_0 \ar[d]^{b} \\
A_1 \ar[rr]^{f_1} && B_1
}
\end{equation}
By definition, a morphism $f:a\to b$ is $\partial_1$-unit if the map $f_1=\partial_1(f)$ is a unit.
Every morphism $f:a\to b$ in $\mathcal{E}^{(1)}$ admits a decomposition $f=f^\sharp f_\sharp$
with $f_\sharp$ a $\partial_1$-unit and $f^\sharp$ a $\partial_1$-cartesian morphism:
\begin{equation}\label{decompositionpartialunitcart}
\xymatrix{ 
A_0  \ar[d]_{a}   \ar[rr]^{(a,f_0)}   && A_1\times_{B_1} B_0  \ar[rr]^{p_2}   \ar[d]_{p_1}   &&B_0\ar[d]^b \\
A_1  \ar@{=}[rr]^{}    && A_1 \ar[rr]^{f_1} && B_1
}
\end{equation}

Let us say that a morphism
$f:a\to b$ in $\mathcal{E}^{(1)}$ is a {\bf Reedy fibration}
if the associated square (\ref{anothersquarereally})
is Reedy fibrant.
More concretely, a morphism $f:a\to b$ is a Reedy fibration if and only
if the maps $f_1:A_1\to B_1$ and $(a,f_0):A_0\to A_1\times_{B_1} B_0$ are fibrations (the map $b:B_0\to B_1$
is a fibration since it is an object of $\mathcal{E}^{(1)}$).
A Reedy fibration $f:a\to b$ is
a pointwise fibration, since the maps $f_0$ and $f_1$ are fibrations in this case.

\begin{prop}\label{arrowscat} If $\mathcal{E}$ is a clan, then
the category $\mathcal{E}^{(1)}$ has the structure of a clan,
where a fibration is a Reedy fibration. 
Moreover, the inclusion functor $\mathcal{E}^{(1)} \subseteq \mathcal{E}^{[1]}$
is a morphisms of clans.
\end{prop}

\begin{proof} Let is proves the first statement.
If $u:a\to b$ is an isomorphism in $\mathcal{E}^{(1)}$
then the square 
$$\xymatrix{ 
A_0  \ar[rr]^{u_0}   \ar[d]_{a}   &&B_0\ar[d]^b \\
A_1 \ar[rr]^{u_1} && B_1
}
$$
is Reedy fibrant, since it is cartesian and the map $u_1$ and $b$ are fibrations. 
Thus, every isomorphism in $\mathcal{E}^{(1)}$ is a Reedy fibration.
The composite of two Reedy fibrations is a Reedy fibration
by Lemma \ref{compostionfibsquare}.
Let us now show that the base change of a Reedy fibration $g:e\to b$
along any morphism $f:a\to b$ in $\mathcal{E}^{(1)}$ exists and is a Reedy fibration.
The morphisn $g:e\to b$ is a fibration in $\mathcal{E}^{[1]}$,
since a Reedy fibration is a pointwise fibration.
Hence the base change of $g:e\to b$ 
along $f:a\to b$ exists in the clan $\mathcal{E}^{[1]}$.
Let us show that $a\times_b e$ is an object of
$\mathcal{E}^{(1)}$ and that the projection
$a\times_b e\to a$ is a Reedy fibration.
We shall use the following cube in which the top and bottom
faces are pullback squares:
\[
\xymatrix{
A_0\times_{B_0}E_0\ar[dr]^{} \ar[rrr]^{} \ar[ddd]_{a\times_b c} & &  & E_0 \ar '[d] [ddd]_(0,4){e}   \ar[dr]^{g_0}   &  \\
  & A_0 \ar[ddd]_(0.4){a} \ar[rrr]^(0.4){f_0} &  & &  B_0\ar[ddd]^{b} \\ 
  &&&&\\
A_1\times_{B_1}E_1 \ar '[r] [rrr]^(0.4){} \ar[dr]_{} & &  &E_1 \ar[dr]^{ g_1} &   \\
  & A_1  \ar[rrr]^{f_1} &  & & B_1   \ }
\]
Let us first consider the case where $g$ is cartesian.
The right hand face of the cube is cartesian in this case.
Hence the left hand face of the cube is cartesian by Lemma \ref{cubelemmacartesiansq}, since
the top, bottom and right hand faces are cartesian.
Hence the map $a\times_b c:A_0\times_{B_0}\times E_0\to A_1\times_{B_1}\times E_1$ is a fibration by
base change, since the map $a:A_0\to A_1$ is a fibration.
This shows that the morphism $a\times_b e$ is a fibration. 
It is thus an object of $\mathcal{E}^{(1)}$.
The projection $A_1\times_{B_1} E_1 \to A_1$
is a fibration by base change, since the map $g_1:E_1\to B_1$ is a fibration.
It follows that the left hand face of the cube is a Reedy fibrant square, since it is cartesian.
We have proved that the projection
$a\times_b e\to a$ is a Reedy fibration in the case where $g$ is cartesian.
Let us now consider the case where $g$ is $\partial_1$-unit.
The map $g_1:E_1\to B_1$ is a unit in this case and
we may suppose that the projection $A_1\times_{B_1}E_1\to A_1$ is also a unit,
since the bottom face of the cube is cartesian by construction.
Hence the projection $a\times_b e\to a$ is a $\partial_1$-unit in this case.
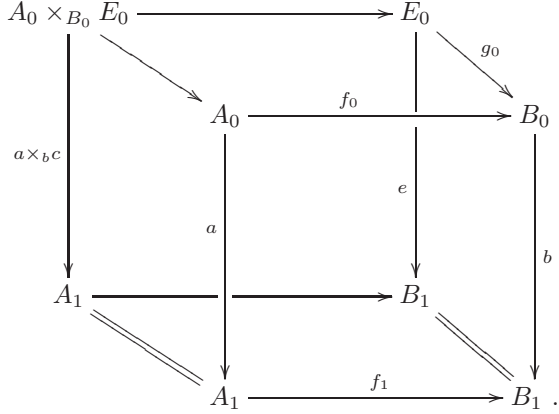
\[
\xymatrix{
A_0\times_{B_0}E_0\ar[dr]^{} \ar[rrr]^{} \ar[ddd]_{a\times_b c} & &  & E_0 \ar '[d] [ddd]_(0,4){e}   \ar[dr]^{g_0}   &  \\
  & A_0 \ar[ddd]_(0.4){a} \ar[rrr]^(0.4){f_0} &  & &  B_0\ar[ddd]^{b} \\ 
  &&&&\\
A_1 \ar '[r] [rrr]^(0.4){} \ar@{=}[dr] & &  &B_1 \ar@{=}[dr] &   \\
  & A_1  \ar[rrr]^{f_1} &  & & B_1   \ .}.
\]
The projection $A_0\times_{B_0}\times E_0\to A_0$
is a fibration by base change, since the map $g_0:E_0\to B_0$ is a fibration and the top
square is cartesian.
Hence the map $a\times_b c$ is a fibration, since $a$ is a fibration.
This shows that $a\times_b e$ is an object of $\mathcal{E}^{(1)}$.
Moreover, the projection $a\times_b e\to a$ is a Reedy fibration, since it is a $\partial_1$-unit
and the map $A_0\times_{B_0}\times E_0\to A_0$ is a fibration.
This completes the proof that the projection
$a\times_b e\to a$ is a Reedy fibration in the case where $g$ is a $\partial_1$-unit.
In the general case, the morphism $g:e\to b$ is the
 the composite of a $\partial_1$-unit $g_\sharp:e\to k$ followed by a cartesian morphism $g^\sharp:k\to b$.
The following diagram of pullback squares
exists by what we have proved so far.
$$
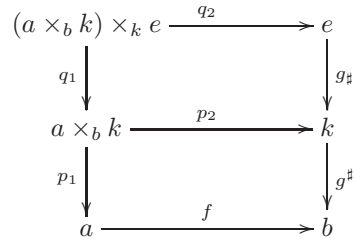
\xymatrix{ 
(a\times_b k)\times_k e  \ar[rr]^{q_2}   \ar[d]_{q_1}   &&e \ar[d]^{g_\sharp} \\
a\times_b k  \ar[rr]^{p_2}   \ar[d]_{p_1}   &&k \ar[d]^{g^\sharp} \\
a \ar[rr]^{f} && b
}
$$
Moreover, the morphism $p_1$ and $q_1$ are Reedy fibrations,
since the morphisms $g^\sharp$ and $g_\sharp$ are Reedy fibrations.
But we have $a\times_k e=(a\times_b k)\times_k e$ and $p_1q_1$
is the projection $a\times_k e\to a$,
since the composite of the two squares is cartesian by Lemma \ref{lemmacartesiansq}.
This proves that the projection $a\times_k e\to a$
 is a Reedy fibration, since $p_1$ and $q_1$ are Reedy fibrations.
 It remains to verify that the map $a\to \top$ is a Reedy fibration for every object $a:A_0\to A_1$
  in $\mathcal{E}^{(1)}$. But this is clear, since the following square is fibrant
  for any fibration $a$,
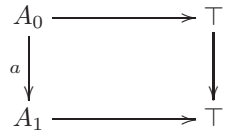
  $$\xymatrix{ 
A_0  \ar[rr]^{}   \ar[d]_{a}   && \top \ar[d] \\
A_1 \ar[rr]^{} && \top
}
$$
This completes the proof that the Reedy fibrations give the category $\mathcal{E}^{(1)}$ the structure of a clan.
 \end{proof}

\bigskip

\subsection{Spans in a clan}

Recall that a {\it span} $A\rightharpoonup B$ between two objects $A$ and $B$ of category $\mathcal{E}$
is a triple $X=(X,u,v)$ where $u:X\to A$ and $v:X\to B$
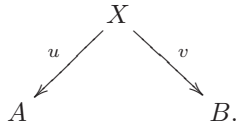
$$\xymatrix{ 
 & \ar[dl]_{u} X\ar[dr]^{v}   & \\
A && B.
}
$$
The {\it left leg} $\lambda_X$ of the span $(X,u,v)$ is the map $u$
and its {\it right leg} $\rho_X$  is the map $v$.
A span in $ \mathcal{E}$ is a contravariant functor $X:Sd[1]\to \mathcal{E}$,
where $Sd[1]$ denotes the poset of non-empty subsets of $[1]=\{0,1\}$ (it is the barycentric subdivision
of the simplicial interval $[1]$).
$$
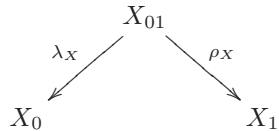
\xymatrix{ 
 & \ar[dl]_{\lambda_X} X_{01}\ar[dr]^{\rho_X}   & \\
X_0 && X_1 
}$$
A {\it map} of spans $f:X\to Y$ 
is a natural transformation:
$$\xymatrix{ 
X_0  \ar[d]_{f_0}   && \ar[ll]_{\lambda_X}   \ar[d]_{f_{01}} X_{01}  \ar[rr]^{\rho_X}  && X_1  \ar[d]^{f_1}   \\
Y_0   && \ar[ll]_{\lambda_Y}  Y_{01}  \ar[rr]^{\rho_Y}  && Y_1  
}$$
We shall denote the category
of spans in $\mathcal{E}$ by $\mathcal{E}^{\wedge}$. 
We shall denote by $\partial:\mathcal{E}^{\wedge} \to \mathcal{E}\times  \mathcal{E}$
 the functor defined by putting $\partial(X)=(X_0,X_1)$.
If $\mathcal{E}$ is a cartesian category, we shall denote by
$\Lambda:\mathcal{E}^{\wedge}\to  \mathcal{E}^{[1]}$ the functor
defined by putting $\Lambda(X)=(\lambda_X,\rho_X):X_{01}\to X_0\times X_1$.
 If $C:\mathcal{E}\times \mathcal{E}\to  \mathcal{E}$ denotes the cartesian product functor,
then we have a pullback square of categories: 
\begin{equation}\label{fibrantspanssquare}
\xymatrix{ 
\mathcal{E}^{\wedge} \ar[rr]^{\Lambda}   \ar[d]_{\partial}   && \mathcal{E}^{[1]} \ar[d]^{\partial_1} \\
\mathcal{E}\times \mathcal{E} \ar[rr]^C &&  \mathcal{E}
}
\end{equation}

If $\mathcal{E}$ is a clan, then cartesian product functor $C:\mathcal{E}\times \mathcal{E}\to  \mathcal{E}$
is a morphism of clans by \ref{0productofclans}.  
Moreover, the codomain functor $\partial_1: \mathcal{E}^{[1]}\to \mathcal{E}$ is a morphism of clans.

\begin{lemma}\label{spanscatconst} If $\mathcal{E}$ is a clan, then
the square (\ref{fibrantspanssquare}) is cartesian in the category $\Clan$.
\end{lemma}

\begin{proof} 
A map of spans $f:X\to Y$ is a pointwise fibration if and only if 
the map $\partial(f)=(f_0,f_1)$ is a fibration in $\mathcal{E}\times \mathcal{E}$ and the map
$\Lambda(f)=(f_{01}, f_0\times f_1)$ is a (pointwise) fibration in $ \mathcal{E}^{[1]} $. 
Hence the result follows from Lemma \ref{pullbackofclans}, since
 the functor $\partial_1$ is an iso-fibration.
 \end{proof}

\bigskip

Let us  say that a span $X=(X,\lambda_X,\rho_X)$ in clan $\mathcal{E}$ is {\bf fibrant} 
 if the map $(\lambda_X,\rho_X):X_{01}\to X_0\times X_1$ is a fibration.
 We shall denote by $\mathcal{E}^{(\wedge)}$ the full subcategory
of $\mathcal{E}^{\wedge}$ whose objects are fibrant spans.
The functor $\Lambda:\mathcal{E}^{\wedge}\to  \mathcal{E}^{[1]}$
induces a functor $\Lambda:\mathcal{E}^{(\wedge)}\to  \mathcal{E}^{(1)}$
and we have a pullback square of categories
\begin{equation}\label{fibrantspanssquare2}
\xymatrix{ 
\mathcal{E}^{(\wedge)} \ar[rr]^{\Lambda}   \ar[d]_{inc}   && \mathcal{E}^{(1)} \ar[d]^{inc} \\
\mathcal{E}^{\wedge} \ar[rr]^{\Lambda}  &&  \mathcal{E}^{[1]}
}
\end{equation}
We say that a map between fibrant spans $f:X\to Y$ is a {\bf Reedy fibration}
if $\partial(f)=(f_0,f_1) $ is a fibration in $  \mathcal{E}\times   \mathcal{E}$ and
 the map $\Lambda(f):\Lambda(X)\to \Lambda(Y)$ is s Reedy fibration in  $\mathcal{E}^{(1)}$
In other words, $f:X\to Y$ is a Reedy fibration if $f_0$ and $f_1$ are fibrations
and the square
 \begin{equation}\label{squareforspan}
\xymatrix{ 
 X_{01}   \ar[rr]^{f_{01}}   \ar[d]_{(\lambda_X , \ \rho_X)}   &&Y_{01}\ar[d]^{(\lambda_Y,\ \rho_Y)} \\
  X_0 \times X_1\ar[rr]^{f_0\times f_1} && Y_0\times Y_1
}
 \end{equation}
is Reedy fibrant.  In which case the map $f_{01}$ is also a fibration.
The squares  (\ref{fibrantspanssquare2}) and (\ref{fibrantspanssquare})
can be composed vertically, Their composite is  a pullback square of categories: 
\begin{equation}\label{fibrantspanssquare3}
\xymatrix{ 
\mathcal{E}^{(\wedge)} \ar[rr]^{\Lambda}   \ar[d]_{\partial}   && \mathcal{E}^{(1)} \ar[d]^{\partial_1} \\
\mathcal{E}\times \mathcal{E} \ar[rr]^C &&  \mathcal{E}
}
\end{equation}
where the functor $\partial:\mathcal{E}^{(\wedge)} \to \mathcal{E}\times  \mathcal{E}$
is induced by the functor
$\partial:\mathcal{E}^{\wedge} \to \mathcal{E}\times  \mathcal{E}$.

\medskip

\begin{prop}\label{spanscat}
The Reedy fibrations gives the category $\mathcal{E}^{(\wedge)}$ the structure of a clan.
The square (\ref{fibrantspanssquare2}) is a pullback in the category of clans and
the inclusion functor $\mathcal{E}^{(\wedge)} \to \mathcal{E}^{\wedge}$ is a morphism of clans.
\end{prop}

\begin{proof} 
By definition, a map between fibrant spans $f:X\to Y$
is a Reedy fibration if and only if $\partial(f)=(f_0,f_1)$
is a fibration in $\mathcal{E}\times \mathcal{E}$  and $\Lambda(f)$ is a (Reedy) fibration  in $\mathcal{E}^{(1)}$.
Hence the square (\ref{fibrantspanssquare2}) is a pullback in the category $\Clan$
by Lemma \ref{pullbackofclans}, since the inclusion $\mathcal{E}^{(1)}\to \mathcal{E}^{[1]}$
is an isofibration and a morphism of clans.
The square (\ref{fibrantspanssquare2}) is cartesian in the category $\Clan$
since it is a composite of two cartesian squares in this category. 
\end{proof}

If $(\lambda_X,\rho_X):X_{01}\to X_0\times X_1$ is a fibrant span, then
the map $\lambda_X=p_1(\lambda_X,\rho_X)$ is a fibration, since the
projection $p_1: X_0\times X_1\to X_0$ is a fibration.
This defines a functor $\lambda:\mathcal{E}^{(\wedge)}\to  \mathcal{E}^{(1)}$.
Similarly, we have a functor $\rho:\mathcal{E}^{(\wedge)}\to  \mathcal{E}^{(1)}$.

\begin{prop}\label{spanscatproj}
The leg functors $\lambda, \rho: \mathcal{E}^{(\wedge)} \to \mathcal{E}^{(1)}$ 
are morphisms of clans. 
\end{prop}

\begin{proof} 
 If $f:X\to Y$ is fibration in $ \mathcal{E}^{(\wedge)}$,
let us show that morphism $\lambda(f):\lambda(X)\to \lambda(Y)$ is a fibration in $ \mathcal{E}^{(1)}$.
The top square of the following diagram is Reedy fibrant, since $f$ is a (Reedy) fibration.
$$
\xymatrix{ 
 &X_{01}   \ar[rr]^{f_{01}}   \ar[d]_{(\lambda_X,\rho_X)}   &&Y_{01}\ar[d]^{(\lambda_Y,\rho_Y)} \\
 & X_0 \times X_1\ar[d] \ar[rr]^{f_0\times f_1} && Y_0\times Y_1\ar[d]\\
  & X_0 \ar[rr]^{f_0} && Y_0
}
$$
The maps $f_0:X_0\to Y_0$ and $f_1:X_1\to Y_1$ are also fibrations since  $f$ is a (Reedy) fibration.
Hence the bottom square of the diagram is Reedy fibrant by Lemma \ref{projectionsReedyfibrantquares}.
It follows by Lemma \ref{compostionfibsquare} that the composite square  is Reedy fibrant:
$$
\xymatrix{ 
 X_{01}   \ar[rr]^{f_{01}}   \ar[d]_{\lambda_X}   &&Y_{01}\ar[d]^{\lambda_Y} \\
 X_0 \ar[rr]^{f_0} && Y_0
}
$$
 This shows that the morphism  $\lambda(f):\lambda(X)\to \lambda(Y)$
is a fibration in $ \mathcal{E}^{(1)}$.
 It remains to show that the functor $\lambda:  \mathcal{E}^{(\wedge)} \to \mathcal{E}^{(1)}$
 preserves base changes of fibrations and terminal objects. 
  We shall use the following the following commutative square of functors,
  where the functor $\lambda':  \mathcal{E}^{\wedge} \to \mathcal{E}^{[1]}$ is defined like $\lambda$,
        $$
\xymatrix{ 
\mathcal{E}^{(\wedge)} \ar[rr]^{\lambda}   \ar[d]_{i_1}  && \mathcal{E}^{(1)} \ar[d]^{i_2} \\
\mathcal{E}^{\wedge} \ar[rr]^{\lambda'}  &&  \mathcal{E}^{[1]} 
}
$$
 It is easy to verify that the functor  $L:\mathcal{E}^{[1]}\to \mathcal{E}^{\wedge}$ defined by putting $L(a)=(a,1_{A_0}):A_0\to A_1\times A_0 $
 for an object $a:A_0\to A_1$ in $\mathcal{E}^{[1]}$ is left
 adjoint to the functor $\lambda'$.
 Hence the functor $\lambda'$ preserves base changes of fibrations and terminal objects,
 since it preserves all limits. 
  The inclusion functors $i_2:  \mathcal{E}^{(1)} \to   \mathcal{E}^{[1]}$ preserves base changes of fibrations and terminal objects
 by Proposition \ref{arrowscat}. Moreover,  $i_2$ is conservative, since it is fully faithful.
It follows that $\lambda$ preserves base changes of fibrations and terminal objects,
since this is true of the functor $\lambda'$.
 \end{proof}

The {\bf composite} $Y\circ_B X$ of two (fibrant) spans $X:A\rightharpoonup B$ and $Y:B\rightharpoonup C$
is the span $X\times_B Y:A\rightharpoonup C$
defined by the following diagram
 with a pullback square:
$$ \xymatrix{
& &  \ar@/_2pc/[ddll]_-{\lambda}   X\times_B Y  \ar@/^2pc/[ddrr]^-{\rho}    \ar[dl]^{p_1} \ar[dr]_{p_2}& & \\
& X\ar[dl]^{\lambda_X} \ar[dr]_{\rho_X} &  & Y\ar[dl]^{\lambda_Y} \ar[dr]_{\rho_Y}   & \\
A && B && C 
 }$$

\begin{prop}\label{spanscomposition}
The composite of $Y\circ_B X:A\rightharpoonup C$ of two fibrant spans $X:A\rightharpoonup B$ and $Y:B\rightharpoonup C$
is fibrant.
\end{prop}

\begin{proof}
 It is easy to verify that the following square is cartesian,
\begin{equation}\label{twofoldspan}
 \xymatrix{
X\times_B Y \ar[d]_{(p_1\!,\rho)} \ar[rr]^{(\lambda,p_2)}& &A \times Y \ar[d]^{A\times (\lambda_Y\!, \rho_Y)}  \\
X\times C\ar[rr]^{(\lambda_X\!,\rho_X)\times C} &  & A\times B\times C 
 }
 \end{equation}
The map $(\lambda_X,\rho_X)\times C$ is a fibration, since $(\lambda_X,\rho_X)$ is a fibration.
Similarly, the map $A\times (\lambda_Y,\rho_Y)$ is a  fibration.
Thus, $X\times_B Y$ is the cartesian product of two objects of the category $\mathcal{E}(A\times B\times C)$.
Hence the diagonal $d:X\times_B Y\to A\times B\times C$ of the square (\ref{twofoldspan}) is a fibration 
It follows that the map $(\lambda,\rho):X\times_B Y\to X\times Y$ is a fibration,
since the projection $(pr_1,pr_3):A\times B\times C \to A\times C$
is a fibration and we have $(\lambda,\rho)=(pr_1,pr_3)d$.
\end{proof}

 \newpage

\section{Internal products in a clan}

\subsection{Cartesian closed categories}

Recall that if $A$ is a carrable object in a category $\mathcal{E}$,
then the product functor $(-)\times A:\mathcal{E}\to \mathcal{E}$
can be defined.

\begin{defi} Let $A$ be a carrable object in a category $\mathcal{E}$.
The {\it exponential} of an object $B\in \mathcal{E}$ by the object $A$ is 
an object $B^A \in \mathcal{E}$ equipped with a map
$\epsilon:B^A\times A\to B$ which is cofree with respect to the functor $(-)\times A:\mathcal{E}\to \mathcal{E}$.
The map  $\epsilon: B^A \times A \to B$ is called the the {\it evaluation}.
\end{defi}

 The  exponential $B^A$ is also called the {\it internal hom} and denoted $[A,B]$.
 The map $\epsilon: [A,B] \times A \to B$ is also called the {\it application}.
Classically, $\epsilon(f,x)=f(x)$ for $f:[A,B]$ and $x:A$.

\smallskip

By definition,  for every object $C\in \mathcal{E}$ and every map $f:C\times A \to B$
 there exists a unique map $g:C\to [A,B]$ 
 such that $\epsilon(g\times A)=f$.
 $$\xymatrix{
C\times A\ar[rr]^{g\times A} \ar[drr]_f  & & [A,B] \times A\ar[d]^{\epsilon}\\
&& B.
}$$
We shall denote the map $g$ by $\lambda^A(f)$. Then
the following identities hold 
$$\epsilon(\lambda^A(f)\times A)=f\quad \quad {\rm and}\quad \quad \lambda^A(\epsilon(g\times A))=g$$
for $f:C\times A\to B$ and $g:C\to [A,B]$.
The first equality  is called  the $\beta$-{\it conversion rule} 
and the second equality the $\eta$-{\it conversion rule}.
If $f:A\to B$, then $\lceil f\rceil:=\lambda^A(f):1\to [A,B]$ is an element of type $[A,B]$.

\medskip

\begin{defi} We say that a carrable object $A$ in a cartesian category $\mathcal{E}$
is {\it exponentiable} if the functor $(-)\times A:\mathcal{E}\to \mathcal{E}$
has a right adjoint $[A,-]:\mathcal{E}\to \mathcal{E}$.
We say that a cartesian category $\mathcal{E}$ is {\it cartesian closed}
if every object $A\in \mathcal{E}$ is exponentiable.
\end{defi}

\medskip
For examples, the category of sets $\Set$
and the category of small categories $\Cat$ are cartesian closed.

 \begin{prop} \label{prop:preshefloccartclosed}
 The category of presheaves $\Hat{\mathcal{C}}=[\mathcal{C}^{op},Set]$ 
on a small category  $\mathcal{C}$ is cartesian closed. 
 \end{prop}

\begin{proof} If $F,G\in \Hat{\mathcal{C}}$, let $[F,G]:\mathcal{C}^{op} \to Set$
be the presheaf defined by putting 
$$[F,G](C)=Hom(F\times yC, G)$$
for every object $C \in \mathcal{C}$,
where $y:\mathcal{C}\to \Hat{\mathcal{C}}$ is the Yoneda functor.
Let $\epsilon:[F,G]\times F\to G$
be the natural transformation defined by putting
$\epsilon_C(\alpha,x)=\alpha_C(x,1_C)$
for every $C\in \mathcal{C}$, where $\alpha:F\times yC \to G$ is a natural transformation 
and $x\in F(C)$. The proof that $\epsilon$
is cofree with respect to the functor $(-)\times F$
is left to the reader.
\end{proof}

\medskip

Let $F:\mathcal{E} \to \mathcal{E}'$ be a cartesian functor
between cartesian closed categories.
If $A, B\in \mathcal{E}$, then the functor $F$ takes
the evaluation $\epsilon:[A,B]\times A\to B$ 
to a map $F(\epsilon):F[A,B]\times FA \to FB$.
The map  $\psi_{A,B}=\lambda^{F\! A}(F(\epsilon)):F[A,B]\to [F\!A,F\!B]$ 
is called the {\it comparaison map}.

\begin{defi} Let $\mathcal{E}$ and $\mathcal{E}'$ be cartesian closed categories.
We say that a functor $F:\mathcal{E} \to \mathcal{E}'$ is {\it cartesian closed} 
if it preserves finite products and the comparaison map $\psi_{A,B}:F[A,B]\to [FA,FB]$ 
is an isomorphism for every pair of objects $A,B\in \mathcal{E}$.
\end{defi}

\medskip

 The composite of two  cartesian closed functors is cartesian closed.
 There is then a 2-category whose objects are  cartesian closed categories, whose morphisms are cartesian closed
 functors and whose 2-cells are natural isomorphisms.

\medskip

\begin{prop}\label{yonedacarclosed} Let $\mathcal{C}$
be a cartesian closed category. Then
the Yoneda functor $y:\mathcal{C}\to  [\mathcal{C}^{op},\mathbf{Set}]$
is cartesian closed. 
\end{prop}

\begin{proof} Left to the reader.
\end{proof}

Recall that if $A$ is an object in a cartesian category $\mathcal{E}$,
then the base change functor $e_A: \mathcal{E}\to  \mathcal{E}/A$
is defined by putting $e_A(X)=(X\times A,p_2)$.

\begin{defi} Let $A$ be a carrable object  in a category $\mathcal{E}$.
The {\it internal  product} of a map $p:E\to A$
is defined to be an object $P\in \mathcal{E}$ equipped with a map
a map $\epsilon:e_A(P)\to E$ in $\mathcal{E}/A$ which is cofree with respect to the base change functor $e_A$.
We shall denote the object $P$ by $\Pi_A(E,p)$, or $\Pi_A(E)$ and
say that the map $\epsilon:\Pi_A(E)\times A\to E$ the {\it evaluation}.
\end{defi}

By definition,  for every object $C\in \mathcal{E}$ and every map $f:C\times A \to E$  
 in $\mathcal{E}/A$, there exists a unique map $g:C\to \Pi_A(E)$ 
 such that $\epsilon(g\times A)=f$.
 $$\xymatrix{
C\times A\ar[rr]^{g\times A} \ar[drr]_f  & &  \Pi_A(E) \times A\ar[d]^{\epsilon}\\
&& E.
}$$
We shall denote the map $g$ by $\lambda^A(f)$.

\begin{prop}\label{expisfullyprop}  Let $A$ be an exponentiable object  in a category with finite limits $\mathcal{E} $.
Then the base change functor $e_A:\mathcal{E} \to \mathcal{E}/A$ has a 
right adjoint $\Pi_A: \mathcal{E}/A\to  \mathcal{E}$.
More precisely, if $E=(E,p)\in \mathcal{E}/A$, 
then $\Pi_A(E)$ is the fiber of the map $[A,p]:[A,E]\to [A,A]$
at $\lceil 1_A\rceil:1\to [A,A]$,
 \begin{equation}\label{fiberforinternalproduct}
 \xymatrix{
\Pi_A(E) \ar[d] \ar[d]\ar[rr]^{\iota}& &[A,E] \ar[d]^{[A,p]} \\
1 \ar[rr]^{\lceil 1_A\rceil}&& [A,A]
}
 \end{equation}
The evaluation $\epsilon':\Pi_A(E) \times A\to E$
is induced by the evaluation $\epsilon_E:[A,E] \times A\to E$.
\end{prop}

\begin{proof} We have to show that every map $p:E\to A$ has an internal product $\Pi_A(E,p)$.
Consider the following pullback square
 \begin{equation}\label{fiberforinternalproduct5}
 \xymatrix{
P \ar[d] \ar[d]_{} \ar[rr]^{\iota}& &[A,E] \ar[d]^{[A,p]} \\
1 \ar[rr]^{\lceil 1_A\rceil}&& [A,A]
}
 \end{equation}
 where $P$ be the fiber of the map $[A,p]$
at $\lceil 1_A\rceil:1\to  [A,A]$.
Let us put $\epsilon':=\epsilon_E (\iota\times A):P\times A\to  E$
 \begin{equation*}
 \xymatrix{
P\times A \ar[rr]^{\iota \times A}  \ar@/^2.5pc/[rrrr]^-{\epsilon'}  
& &[A,E]\times A  \ar[rr]^{\epsilon_E}&& E
}
 \end{equation*}
 Let us first verify that $\epsilon'$ is a map in $ \mathcal{E}/A$.
 For this we have to show that $p\epsilon'=p_2: P\times A\to A$.
The left hand square of following diagram commutes, since the square (\ref{fiberforinternalproduct5})
commutes.
 \begin{equation*}
 \xymatrix{
P\times A \ar[d]_{!\times A}\ar[rr]^{\iota \times A}  \ar@/^2.5pc/[rrrr]^-{\epsilon'}  
& &[A,E]\times A \ar[d]^{[A,p]\times A} \ar[rr]^{\epsilon_E }&& E\ar[d]^p \\
1\times A \ar[rr]^{\lceil 1_A\rceil \times A}  \ar@/_2.5pc/[rrrr]^-{1_A}   && [A,A]\times A \ar[rr]^{\epsilon_A } && A
}
 \end{equation*}
 The right hand square commutes by definition of the map $[A,p]$.
The top triangle commutes by defintion of $\epsilon'$.
The bottom triangle commutes by defintion of $\lceil 1_A\rceil $.
Thus, $p\epsilon'=\ !\! \times A=p_2$ and this shows that $\epsilon'$ is a map in $ \mathcal{E}/A$.
 Let us show that $\epsilon':e_A(P)\to E$ is cofree with respect to the functor $e_A$.
  For this we have to show that for every object $C\in \mathcal{E}$ and 
  every map $f:e_A(C)\to E$ in $\mathcal{E}/A$,
 there exists a unique map $g:C\to P$ such that $\epsilon' e_A(g)=f$.
  The following square commutes, since $f$ is a map in $ \mathcal{E}/A$.
  \begin{equation*}
 \xymatrix{
C\times A \ar[d]_{!\times A}\ar[rr]^{f}  && E\ar[d]^p \\
1\times A \ar[rr]^{1_A}    &&  A
} \end{equation*}
Hence the following square commutes by the adjunction $(-)\times A \dashv [A,-]$,
   \begin{equation*}
 \xymatrix{
C \ar[d]_{!}\ar[rr]^{\lambda^A(f)}  && [A,E]\ar[d]^{[A,p]} \\
1 \ar[rr]^{\lambda^A(1_A)}    &&  [A,A]
}
 \end{equation*}
Notice that $\lambda^A(1_A)=\lceil 1_A\rceil$.
There is then a unique map $g:C\to P$
such that $\iota g=\lambda^A(f)$, since the square (\ref{fiberforinternalproduct5}) is cartesian.
We then have 
$$\epsilon'(g\times A)= \epsilon_E(\iota\times A )(g\times A)= \epsilon_E(\iota g\times A )=\epsilon_E(\lambda^A(f) \times A )=f.$$
The existence of $g$ is proved; its uniqueness is left to the reader. 
    \end{proof}

\begin{rem} \label{rem:exponentialareprod} With the hypothesis of Proposition \ref{expisfullyprop}, we have $[A,B]=\Pi_A(B\times A,p_2)$
for every object $B\in \mathcal{E} $.
To see this, observe that the product functor $(-)\times A:\mathcal{E}\to \mathcal{E}$ is the composite 
 of the base change functor $e_A:\mathcal{E}\to \mathcal{E}/A$ followed by the
 forgetful functor $\Sigma_A:\mathcal{E}/A\to \mathcal{E}$.
By composing the adjunctions $ \Sigma_A\dashv e_A$ and
  $e_A\dashv \Pi_A$ we obtain an adjunction $\Sigma_A\circ e_A\dashv \Pi_A\circ e_A$.
  Hence the functor $\Pi_A\circ e_A:\mathcal{E}\to \mathcal{E}$
  is right adjoint to the functor $(-)\times A:\mathcal{E}\to \mathcal{E}$.
 \end{rem}

Let $f:A\to B$ be a carrable map in a category $\mathcal{E}$.
Recall that the base change functor 
 $f^\star: \mathcal{E}/B\to  \mathcal{E}/A$
is defined by putting $f^\star(X)=(X\times_B A,p_2)$
for every object $X=(X,p)\in \mathcal{E}/B$.

\begin{defi} \label{def:internalproductalon} Let $f:A\to B$ be a carrable map in a category $\mathcal{E}$.
The {\it internal product } of an object $E=(E,p)\in \mathcal{E}/A$ 
along the map $f$
is defined to be an object $P=(P,q)\in \mathcal{E}/B$
equipped with a map $\epsilon:f^\star(P)\to E$ in $\mathcal{E}/A$ 
which is cofree with respect to the functor $f^\star:\mathcal{E}/B \to \mathcal{E}/A$.
We shall denote the object $P$ by $\Pi_f(E,p)$, or by $\Pi_f(E)$, and
say that the map $\epsilon:\Pi_f(E)\times_B A\to E$
is the {\it evaluation}.
\end{defi}

By definition, the map 
$$
\xymatrix{
(\mathcal{E}/B)(C,P)\ar[rr]^{f^\star} && (\mathcal{E}/A)(f^\star Q,f^\star P) \ar[rr]^{\epsilon\circ (-)} &&  (\mathcal{E}/A)(f^\star Q,E)  
}
$$
is bijective for every object $C=(C,g)\in \mathcal{E}/B$,
In other words, for every map $u:C\times_B A \to E$ in $\mathcal{E}/A$,
there exists a unique map $v:C\to \Pi_f (E)$ in $\mathcal{E}/B$
such that $\epsilon (v\times_B A)=u$.
$$\xymatrix{
C\times_B A\ar[rr]^{v\times_B A} \ar[drr]_u  & &  \Pi_f(E) \times_B A\ar[d]^{\epsilon}\\
&& E.
}$$
We shall denote the map $v$ by $\lambda^A(u)$.

\medskip

\begin{rem} \label{rem:intprod}
It follows from Definition  \ref{def:internalproductalon} that
$$\Pi_f(E,p)=\Pi_{(A,f)}(E,p)$$
where $\Pi_{(A,f)}(E,p)$  denotes the internal product of the
map $p:(E,fp)\to (A,f)$ in $ \mathcal{E}/B$.
.\end{rem}

\subsection{Locally cartesian closed categories}

\begin{defi} A category with finite limits $\mathcal{E}$ is said to be {\it locally cartesian closed}, 
if the category $\mathcal{E}/A$ is cartesian closed for every object $A\in \mathcal{E}$.
\end{defi}

\smallskip

It follows directly from this definition that if $\mathcal{E}$ is locally cartesian closed,
 then so is the category $\mathcal{E}/A$ for every object $A\in \mathcal{E}$.

\medskip

For examples, the category  of sets $\Set$ is locally cartesian closed.
We shall see below that the category of presheaves 
on a small category is locally cartesian closed. 
More generally, a topos is locally cartesian closed.

\medskip

 \begin{prop} \label{preshefloccartclosed}
 The category of presheaves $\Hat{\mathcal{C}}=[\mathcal{C}^{op},Set]$ 
on a small category  $\mathcal{C}$ is locally cartesian closed. 
 \end{prop}

\begin{proof} Let us show that the category $\Hat{\mathcal{C}}/F$
is cartesian closed for every $F\in \Hat{\mathcal{C}}$.
But the category $\Hat{\mathcal{C}}/F$ is equivalent to the presheaf category $[(\mathcal{C}/F)^{op},Set]$,
where $\mathcal{C}/F$ denotes the category of elements of $F$.
Hence the category $\Hat{\mathcal{C}}/F$ is cartesian 
closed by Proposition \ref{prop:preshefloccartclosed}.
\end{proof}

\begin{prop} \label{prop:cartesianclosedvsinternalprod}
A category with finite limits $\mathcal{E}$ is locally cartesian closed
if and only if the base change functor
$f^\star:\mathcal{E}/B\to \mathcal{E}/A$
has a right adjoint 
$f_\star=\Pi_f:\mathcal{E}/A\to \mathcal{E}/B$
for every map $f:A\to B$ in $\mathcal{E}$.
\end{prop}

\begin{proof} If $\mathcal{E}$ is locally cartesian closed and $f:A\to B$
is a map in $\mathcal{E}$ let us show that every map $p:E\to A$
has an  internal product $\Pi_f(E,p)$ along $f$.
The category $\mathcal{E}/B$ is cartesian closed by hypothesis.
Moreover, $\mathcal{E}/B$ has finite limits, since 
the category $\mathcal{E}$ has finite limits. It then follows by Lemma \ref{expisfullyprop}
that the map $p:(E,fp)\to (A,f)$ of $\mathcal{E}/B$ has an internal product $\Pi_{(A,f)}(E,p)$.
But we have $\Pi_f(E,p)=\Pi_{(A,f)}(E,p)$ by Remark \ref{rem:intprod}.
Conversely, if the functor $f^\star$ has a right adjoint $\Pi_f$
for every map $f:A\to B$, let us show that the category $\mathcal{E}$ is locally cartesian closed.
Let us first check that $\mathcal{E}$ is cartesian closed.
For this we have to show that every object $A\in \mathcal{E}$ is exponentiable.
 The base change  functor $e_A:\mathcal{E}\to \mathcal{E}/A$ has a right adjoint $\Pi_A:\mathcal{E}/A\to \mathcal{E}$
 by hypothesis.  Thus, $A$ is exponentiable,
 since we have $[A,B]=\Pi_A(B\times A,p_2)$
 for every object $B\in \mathcal{E}$ (see Remark \ref{rem:exponentialareprod}).
We have proved that $\mathcal{E}$ is cartesian closed.
It remains to show that the category $\mathcal{E}/C$ is cartesian
closed for every object $C\in  \mathcal{E}$.
But the argument above, applied to  $\mathcal{E}/C$ instead of $\mathcal{E}$, shows that $\mathcal{E}/C$ is cartesian closed. 
\end{proof}

\begin{defi} Let $\mathcal{E}$ and $\mathcal{E}'$ be locally cartesian closed categories.
We say that a functor $F:\mathcal{E} \to \mathcal{E}'$ 
is  {\it  locally cartesian closed} if it preserves finite limits
and the induced functor $F/A:\mathcal{E}/A \to \mathcal{E}'/F\!A$
is cartesian closed for every object $A\in \mathcal{E}$. 
 \end{defi}

 \medskip
 
 The composite of two  locally cartesian closed functors is 
 locally cartesian closed.
  There is then a 2-category whose objects are  locally cartesian closed categories, whose morphisms are locally cartesian closed
 functors and whose 2-cells are natural isomorphisms.

 \medskip

\begin{thm} \label{thm:BC} {\rm (Beck-Chevalley law)} Let $\mathcal{E} $
be a locally cartesian closed category. If the following square in $\mathcal{E} $ is cartesian
\begin{equation} \label{BeckChevalysqr}
 \xymatrix{
C \ar[d]_{u}\ar[rr]^{g}& & D \ar[d]^v \\
A \ar[rr]^{f}&& B
}
\end{equation}
then the following squares of functors commute up
to a natural isomorphism:
\begin{equation} \label{BeckChevally3}
 \xymatrix{
\mathcal{E}/C \ar[d]_{u_!}& &\ar[ll]_{g^\star} \mathcal{E}/D \ar[d]^{v_!} \\
\mathcal{E}/A && \ar[ll]_{f^\star} \mathcal{E}/B
}\quad  \quad \quad
 \xymatrix{
\mathcal{E}/C \ar[rr]^{g_\star} & & \mathcal{E}/D \\
\mathcal{E}/A  \ar[u]^{u^\star} \ar[rr]^{f_\star}  && \mathcal{E}/B \ar[u]_{v^\star} \
}
\end{equation}
\end{thm}

\begin{proof} 
If $E=(E,p)\in  \mathcal{E}/B$, then the composite of the 
following two squares is cartesian by Lemma \ref{lemmacartesiansq}
\begin{equation}  \label{BeckChevally2}
\xymatrix{
&C\times_D E \ar[d]_{p_1}\ar[rr]^{p_2}& & E \ar[d]^p \\
&C \ar[d]_{u}\ar[rr]^{g}& & D \ar[d]^v \\
& A \ar[rr]^{f}&& B
}
\end{equation}
Thus, $C\times_D E =A\times_B E$. This shows that
$$
 u_!g^\star(E,p)=u_!(C\times_D E,p_1) = (C\times_D E,up_1)\simeq (A\times_B E, up_1)= f^\star(E,vp)=f^\star v_!(E,p)$$
 Hence the functor $ u_!g^\star$ is isomorphic to the functor $f^\star v_!$.
 We have $f^\star v_! \dashv v^\star f_\star $, since we have $v_!\dashv v^\star$ and $f^\star \dashv f_\star$.
 Similarly, we have  $u_!g^\star\dashv g_\star u^\star$, since we have $g^\star \dashv g_\star$ and  $u_!\dashv u^\star$.
 Hence the functor $g_\star u^\star$ is isomorphic to the functor $v^\star f_\star $,
 since the functor $ u_!g^\star$ is isomorphic to the functor $f^\star v_!$ 
 and isomorphic functors have isomorphic right adjoints.
 \end{proof}

Let $\mathcal{E}$ and $\mathcal{E}'$ be  locally cartesian closed categories
and let $F:\mathcal{E} \to \mathcal{E}'$ be a functor which preserves finite limits.
If $f:A\to B$ and $p:E\to A$ are maps in $\mathcal{E}$ 
then the functor $F$ takes the evaluation map $\epsilon:\Pi_f(E,p)\times_B A\to E$
in $\mathcal{E}/A$ to a map $F(\epsilon):F(\Pi_f(E,p))\times_{F\!B} F\!A\to F\!E$ in
the category $ \mathcal{E}'/F\!A$.
We shall say that $F$ preserves {\it internal products} if the map 
 $$\lambda^{F\!A}F(\epsilon):F\Pi_f(E)\to \Pi_{F(f)}(F\!E)$$
 is invertible for every pair of maps $f:A\to B$ and $p:E\to A$ 
in the category $\mathcal{E}$.

\begin{prop} \label{locallyclosedandproduct} Let $F:\mathcal{E}\to \mathcal{E}'$ 
be a functor between locally cartesian closed categories.
If $F$ preserves finite limits, then $F$ is locally cartesian closed functor if and only if it
preserve internal  products. 
 \end{prop}

\begin{proof} $(\Rightarrow)$ It follows from Proposition \ref{expisfullyprop}
that if a cartesian closed functor preserves
finite limits, then it preserves the internal product of every map.
Thus, the functor $F$ preserves the product
$\Pi_A(E,p)$ for every map $p:E\to A$ in $\mathcal{E}$.
Let us show more generally that 
$F$ preserves the internal product 
$\Pi_f(E,p)$ of every map $p:E\to A$ along every map $f:A\to B$.
The functor $F/B:\mathcal{E}/B\to \mathcal{E}'/F\! B$ induced by $F$
is locally cartesian closed, since the functor $F$ is locally cartesian closed.
Hence the functor $F/B$ preserves the internal product  $\Pi_{(A,f)}(E)$ by the first 
part of the proof. This proves the result, since $\Pi_f(E)=\Pi_{(A,f)}(E)$
by Remark \ref{rem:intprod}. The converse $(\Leftarrow)$ is left to the reader.
\end{proof}

 \medskip

\begin{prop} \label{basechangecartclosed} Let $\mathcal{E}$ be a locally cartesian closed category.
Then the base change functor 
$$f^\star:\mathcal{E}/B \to \mathcal{E}/A$$
is locally cartesian closed for every map $f:A\to B$. 
\end{prop}

\begin{proof} That the functor $f^\star$ preserves internal
products follows from Theorem \ref{thm:BC}.
Hence the functor $f^\star$ is  locally cartesian closed 
by Proposition \ref{locallyclosedandproduct}.
\end{proof}

\begin{lemma} \label{yonedapreservesintproduct} Let $f:A\to B$
be a carrable map in a small category  $\mathcal{C}$.
Then the Yoneda functor $y:\mathcal{C}\to  \Hat{\mathcal{C}}$
preserves the internal product $\Pi_f(E,p)$ of a map $p:E\to A$ along $f$ when the
internal product exists. 
\end{lemma}

\begin{proof} Left to the reader.
\end{proof}

\medskip

\begin{prop}\label{yonedacarclosed2} If $\mathcal{C}$ is a small locally cartesian closed category,
then the Yoneda functor $y:\mathcal{C}\to  \Hat{\mathcal{C}}$
is locally cartesian closed.
\end{prop}

\begin{proof} This follows from Lemma \ref{yonedapreservesintproduct}.
\end{proof}

 \subsection{Polynomial functors}

 Let $ \mathcal{E}$ be a locally cartesian closed category.
We shall say that a
 diagram in $ \mathcal{E}$ of the form
  $$\xymatrix{
 & E \ar[dl]_u  \ar[r]^p & B \ar[dr]^v & \\
 I && & J
 }$$
 is  {\it polynomial span} $P=(u,p,v):I\to J$.
  The {\it polynomial functor} $P(-):  \mathcal{E}/I \to \mathcal{E}/J$
  associated to $P$ is defined to be the composite:
 $$\xymatrix{
 &  \mathcal{E}/E  \ar[r]^{p_\star} &  \mathcal{E}/B \ar[dr]^{v_!} & \\
  \mathcal{E}/I \ar[ur]^{u^\star} && &  \mathcal{E}/J
 }$$
 We shall see that the composite of two polynomial functors $P(-):  \mathcal{E}/I \to \mathcal{E}/J$
 and $Q(-):  \mathcal{E}/J\to \mathcal{E}/K$ is a  polynomial functor $(Q\circ P)(-): \mathcal{E}/I \to \mathcal{E}/K$.

 \medskip

Recall that a functor $U : \mathcal{E}\to  \mathcal{F}$ 
induces a functor  $U/A: \mathcal{E}/A\to  \mathcal{F}/U(A)$
 for every object $A\in  \mathcal{E}$, where we put $(U/A)(E,p)=(U(E),U(p))$
 for $(E,p)\in \mathcal{E}/A$.

  \begin{lemma} \label{localleftadjoint} 
  If a functor $U:\mathcal{E}\to \mathcal{F}$
  has a left adjoint $F:\mathcal{F}\to \mathcal{E}$,
  then the functor $U/A: \mathcal{E}/A\to  \mathcal{F}/U(A)$
  has a left adjoint $F_A: \mathcal{F}/U(A) \to \mathcal{E}/A$
  for every object $A\in \mathcal{E}$.
 By construction,  $F_A(E,p)=(F(E),\epsilon_A F(p))$ for a map $p:E\to U(A)$, where
  $\epsilon_A:FU(A)\to A$ is the counit of the adjunction $F\vdash U$.
   \end{lemma}

\begin{proof} Left to the reader.
\end{proof}

Let $f:A\to B$ be a map in a locally cartesian closed category $ \mathcal{E}$.
The composite of the functors
$$\xymatrix{
\mathcal{E}/A\ar[rr]^{\Sigma_f} && \mathcal{E}/B \ar[rr]^{\Pi_B} &&  \mathcal{E}
}$$
takes the terminal object $(A,1_A)\in \mathcal{E}/A$
to the object $\Pi_B(A):=\Pi_B(A,f)$. It thus induces a functor
$$\Pi_{B/A}:\mathcal{E}/A\to \mathcal{E}/\Pi_B(A)$$
 if we put $\Pi_{B/A}(E,p)=(\Pi_B(E,fp),\Pi_B(p))$
 for an object $(E,p)\in \mathcal{E}/A$.

 \begin{lemma}\label{choice1}
Let $f:A\to B$ be map in a locally cartesian closed category $ \mathcal{E}$.
Then the functor $\Pi_{B/A}:\mathcal{E}/A \to \mathcal{E}/\Pi_B(A)$
defined above
is isomorphic to the composite
$$\xymatrix{
  \mathcal{E}/A \ar[rr]^(0.4){\epsilon^\star} && \mathcal{E}/ \Pi_B(A)\! \times \! B \ar[rr]^{\Pi_p} && \mathcal{E}/ \Pi_B(A)
 }$$
where $\epsilon :\Pi_B(A)\times B\to A$ is the evaluation and $p: \Pi_B(A) \times B\to  \Pi_B(A)$
is the projection.
\end{lemma}

\begin{proof} 
Let us apply Lemma \ref{localleftadjoint} to the functors  $\Pi_B: \mathcal{E}/B\to \mathcal{E}$
and to the object $(A,f) \in \mathcal{E}/B$.
The functor $\Pi_B$ is right adjoint to the base change
functor  $F:=e_B:  \mathcal{E} \to \mathcal{E}/B$.
Hence the induced functor $\Pi_{B/A}: \mathcal{E}/A\to \mathcal{E}/ \Pi_B(A)$
 has a left adjoint $F_{A}$ by Lemma \ref{localleftadjoint}.
  By construction, 
 $F_A(C,g)=(e_B(C), \epsilon e_B(g))$ for a map $g:C\to \Pi_B(A,f)$,
 where $\epsilon:\Pi_B(A)\times B\to A$ is the evaluation.
  But we have $(e_B(C), e_B(g))=(C\times B, g\times B)=p^\star(C,g)$,
where $p:\Pi_B(A)\times B\to \Pi_B(A)$ is the projection.
$$\xymatrix{
 C \ar[d]_{g} & &C\times B \ar@{=}[rr]  \ar[d]^{g\times B}\ar[ll]&&  F_A(C,q) \ar[d]_{}\\
 \Pi_B(A) && \ar[rr]^(0.6){\epsilon} \Pi_B(A)\times B \ar[ll]_{p} &&  A 
}$$
Thus, $F_A=\epsilon_! p^\star$. It follows by
adjointness that $\Pi_{B/A}=p_\star \epsilon^\star$.
\end{proof}

 \begin{lemma}\label{choice2}
 Let $f:A\to B$ be a map in locally cartesian closed category $ \mathcal{E}$.
 Then the functor $\Pi_B\circ  \Sigma_f: \mathcal{E}/A\to  \mathcal{E}$
 is isomorphic to the composite 
 $$\xymatrix{
  \mathcal{E}/A \ar[rr]^(0.4){\epsilon^\star} && \mathcal{E}/ \Pi_B(A)\! \times \! B \ar[rr]^{p_\star} && \mathcal{E}/ \Pi_B(A,f) 
   \ar[rr]^{\Sigma_{\Pi_B(A)}} && \mathcal{E}
 }$$
\end{lemma}

\begin{proof} The following square of functors commutes by definition of the functor $\Pi_{B/A}$,
$$\xymatrix{
 \mathcal{E}/A  \ar[d]_{\Sigma_f}  \ar[rr]^(0.5){\Pi_{B/A}}& &    \mathcal{E}/\Pi_B(A,f) \ar[d]^{\Sigma_{\Pi_B(A,f)}} \\
 \mathcal{E}/B  \ar[rr]^(0.5){\Pi_B} &&\mathcal{E}
}$$
The result then follows from Lemma \ref{choice1}.
\end{proof}

 \begin{prop}\label{prop:comppoly}
  The composite of two polynomial functors is a polynomial functor.
  \end{prop}

  \begin{proof} 
  Let $P=(s,p,t):I\to J$ and $Q=(u,q,v):J\to K$ be two polynomial spans.
  \begin{equation} \label{pospans}
  \xymatrix{
 & E \ar[dl]_s  \ar[r]^p & A \ar[dr]^t & & F \ar[dl]_{u} \ar[r]^{q} & B \ar[dr]^{v} & \\
 I && & J &&&  K
 }
  \end{equation}
We shall prove that the composite of the following sequence of functors is a polynomial functor:
$$\xymatrix{
 & \mathcal{E}/E  \ar[r]^{\Pi_p} & \mathcal{E}/A \ar[dr]^{\Sigma_t} & & \mathcal{E}/F \ar[r]^{\Pi_q} & \mathcal{E}/B \ar[dr]^{\Sigma_v} & \\
 \mathcal{E}/I \ar[ur]^{s^\star}  && & \mathcal{E}/J  \ar[ru]^{u^\star} &&&  \mathcal{E}/K
 }$$
For this, we shall construct the following diagram of functors, with three squares and one pentagon.
  \begin{equation} \label{pospans2}
 \xymatrix{
 &&& \mathcal{E}/R \ar[r]^\Pi &\mathcal{E}/C   \ar[r]^\Pi & \mathcal{E}/D  \ar[dd]^\Sigma  \\
 &&\mathcal{E}/E' \ar[r]^{\Pi}  \ar[ur]^{\star} &\mathcal{E}/A' \ar[dr]^{\Sigma} \ar[ur]^{\star}&&& \\
 & \mathcal{E}/E \ar[ur]^\star \ar[r]^{\Pi_p} & \mathcal{E}/A \ar[dr]^{\Sigma_t} \ar[ur]^\star & & \mathcal{E}/F \ar[r]^{\Pi_q} & \mathcal{E}/B \ar[dr]^{\Sigma_v} & \\
 \mathcal{E}/I \ar[ur]^{s^\star}  && & \mathcal{E}/J  \ar[ru]^{u^\star} &&&  \mathcal{E}/K
 }  \end{equation}
In this diagram, the arrows marked $\star$ are pullback functors, the arrows marked $\Sigma$ are 
sommation functors and the arrows marked $\Pi$ are  internal product functors.
The first step in constructing the diagram (\ref{pospans2}) is to complete the diagram \ref{pospans}
with two pullback squares, $A'=A\times_J F$ and $E'=E\times_A  A'$,
and to use Beck-Chevalley law \ref{thm:BC}:
$$\xymatrix{
  & & E' \ar[dl] \ar[r]^{p'} & A' \ar[dr]^{t'}  \ar[dl]&  && \\
 & E \ar[dl]_s  \ar[r]^p & A \ar[dr]^t & & F \ar[dl]_{u} \ar[r]^{q} & B \ar[dr]^{v} & \\
 I && & J &&&  K
 }$$
The second step is to add the pentagon by using Lemma \ref{choice2} applied to the map $t':A'\to F$ in the category $ \mathcal{E}/B$.
$$\xymatrix{
  & & &  & \ar[dl] \ar[r] C & \ar[dd] D\\
  & &E'\ar[dl] \ar[r]^{p'} & A'\ar[dr]^{t'}  \ar[dl]&  && \\
 & E \ar[dl]_s  \ar[r]^p & A \ar[dr]^t & & F \ar[dl]_{u} \ar[r]^{q} & B \ar[dr]^{v} & \\
 I && & J &&&  K
 }$$
By construction, $C=\Pi_{q}(A',t')\times_B F $ and $D=\Pi_{q}(A',t')$.
The third step is to complete the last diagram by adding a third pullback square with $R=E'\times_{A'}C$
and to use Beck-Chevalley law \ref{thm:BC}:
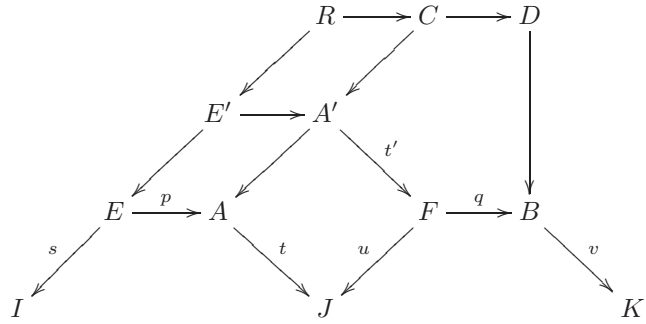
$$\xymatrix{
  & & &\ar[r] R  \ar[ld] & \ar[dl] \ar[r] C & \ar[dd] D\\
  & &E'\ar[dl] \ar[r] & A'\ar[dr]^{t'}  \ar[dl]&  && \\
 & E \ar[dl]_s  \ar[r]^p & A \ar[dr]^t & & F \ar[dl]_{u} \ar[r]^{q} & B \ar[dr]^{v} & \\
 I && & J &&&  K
 }$$
 The final step consists in shrinking the diagram (\ref{pospans2}) into  the simpler diagram 
  \begin{equation} \xymatrix{
 &  \mathcal{E}/R  \ar[r]^{\Pi} &  \mathcal{E}/D \ar[dr]^{\Sigma} & \\
  \mathcal{E}/I \ar[ur]^{\star} && &  \mathcal{E}/K
 }
  \end{equation}
 by using the fact that a composite of base change functors
 is a base change functor, and similarly for a composite of summation functors
 and a composite of internal product functors.
   \end{proof}

\newpage

\subsection{$\pi$-clans}

Recall that every fibration in a clan is carrable.

\bigskip

\begin{defi}\label{definitiontribe}
We say that a clan $\mathcal{E}$ is a $\pi$-{\bf clan}, if every fibration $p:E\to A$
has an internal product $\Pi_f(E,p)$ along every fibration $f:A\to B$ and the structure map $\Pi_f(E)\to B$ is a fibration.
  \end{defi}

 By definition,  the evaluation $\epsilon:\Pi_f(E,p)\times_B A \to E$
is cofree with respect to the functor $f^\star: \mathcal{E}/B \to  \mathcal{E}/A$.
More explicitly, for every object $C=(C,g)\in \mathcal{E}/B$
and every map $u:C\times_B A \to E$ in $\mathcal{E}/A$,
there exists a unique map $v:C\to \Pi_f (E)$ in $\mathcal{E}/B$
such that $\epsilon (v\times_B A)=u$.
 $$\xymatrix{
C\times_B A\ar[rr]^{v\times_B A} \ar[drr]_u  & &  \Pi_f(E) \times_B A\ar[d]^{\epsilon}\\
&& E.
}$$
We shall denote the map $v$ by $\lambda^A(u)$. 
It follows from Definition \ref{definitiontribe} applied to the map $A\to 1$
that every fibration $p:E\to A$ 
has an internal product $\Pi_A(E,p)\in \mathcal{E}$.

\medskip

\begin{rem} It follows from the definition of the internal product   $\Pi_f(E,p)$
that the evaluation $\epsilon:\Pi_f(E,p)\times_B A \to E$
is cofree with respect to the functor $f^\star: \mathcal{E}(B) \to  \mathcal{E}(A)$,
since it is cofree with respect to the functor $f^\star: \mathcal{E}/B \to  \mathcal{E}/A$.
The former condition is weaker.
\end{rem}

Examples of $\pi$-clans:
 
\begin{itemize}
    \item{} A locally cartesian closed category, where every map is a fibration;  
   \item{} The category of small groupoids $\mathbf{Grpd}$, where
  a fibration is an iso-fibration;  
\item{} The category of Kan complexes {\bf Kan}, where a fibration is a Kan fibration.
\end{itemize}

 \begin{prop} \label{pitribecartclosed} If $f:A\to B$ is a fibration in a $\pi$-clan $\mathcal{E}$, then
 the base change functor $f^\star: \mathcal{E}(B)  \to  \mathcal{E}(A)$
has a right adjoint $\Pi_f: \mathcal{E}(A) \to  \mathcal{E}(B)$.
\end{prop}

  \begin{proof} If $(E,p)\in  \mathcal{E}(A)$, then $\Pi_f(E,p)\in   \mathcal{E}(B)$,
  since the structure map  $\Pi_f(E,p)\to B$ is a fibration.
   The evaluation $\epsilon:\Pi_f(E,p)\times_B A \to E$
is cofree with respect to the functor  $f^\star: \mathcal{E}(B) \to  \mathcal{E}(A)$,
  since it is cofree with respect to the functor $f^\star: \mathcal{E}/B \to  \mathcal{E}/A$.
  Hence the functor  $\Pi_f: \mathcal{E}(A) \to  \mathcal{E}(B)$
  is right adjoint to the functor $f^\star: \mathcal{E}(B) \to  \mathcal{E}(A)$.
   \end{proof}

\begin{prop}\label{localisationpitribe}
If $\mathcal{E}$  is a $\pi$-clan, then so is the clan $\mathcal{E}(B)$ for every object $B\in \mathcal{E}$.
If $p:E\to A$ and $f:A\to B$ are fibrations in $\mathcal{E}$, then  
$$\Pi_f(E,p)=\Pi_{(A,f)}(E,p)$$
where $\Pi_{(A,f)}(E,p)$  denotes the internal product of the
fibration $p:(E,fp)\to (A,f)$ in  $ \mathcal{E}(B)$.
\end{prop}

\begin{proof} Left to the reader. \end{proof}

 \begin{prop} \label{pitribecartclosed2} A $\pi$-clan $ \mathcal{E}$ is a cartesian closed category.
 More generally, the category $ \mathcal{E}(A)$
 is cartesian closed for every object $A\in  \mathcal{E}$.
  \end{prop}

\begin{proof}   By Remark \ref{rem:exponentialareprod}, we have
we have $[A,B]=\Pi_A(B\times A,p_2)$ for every  $A,B\in  \mathcal{E}$.
 This show that $ \mathcal{E}$ is cartesian closed.  
 It follows that $ \mathcal{E}(A)$  is cartesian closed  for every object $A\in  \mathcal{E}$,
 since $ \mathcal{E}(A)$ is a $\pi$-clan by Proposition \ref{localisationpitribe}.
  \end{proof}

 \begin{prop} \label{cartclosedispitribe} 
 A cartesian closed category has the structure of $\pi$-clan,
 where a fibration is a cartesian projection.
 \end{prop}

\begin{proof} Let $ \mathcal{E}$ be a cartesian closed category.
 It follows from Proposition \ref{tribeofgenuineprojection} that $ \mathcal{E}$ has the structure of a clan
 where a fibration is a cartesian projection. 
Let us show that  every cartesian projection $p:E\to A$ admits a product along every cartesian projection  $f:A \to B$
and that the structure map $\Pi_f(E)\to B$ is a cartesian projection.
We may suppose that $p$ is the projection $F\times A \to A$ for some object $F$
and that $f$ is the projection $C\times B\to B$ for some object $C$;
in which case, the map $p$ is the projection $F\times C\times B \to C\times B.$ 
We let the reader verify that  the object
$([C,F]\times B,p_2)$ of $\mathcal{E}/B$  is the product
 of the map $F\times C\times B \to C\times B$ 
along the map $C\times B\to B$. 
  \end{proof}

Recall from Proposition \ref{prop:extpresheaf}
that if $ \mathcal{E} $ is a clan, then the class 
of extended fibrations in $\Hat{\mathcal{E}}=[\mathcal{E}^{op},Set]$
is closed under composition and base changes.
If $F\in \Hat{\mathcal{E}}$, then the full subcategory 
$ \Hat{\mathcal{E}}(F)$ of $ \Hat{\mathcal{E}}/F$ 
whose objects  are extended
fibrations $X\to F$ has the structure of a clan by Proposition \ref{prop:extlocalclans}.

\medskip

\begin{defi}\label{definitionhomopitribe}
Let $\mathcal{E}$ and $ \mathcal{E}'$ be $\pi$-clans.
We shall say that a morphism of clans $F:\mathcal{E} \to \mathcal{E}'$
is $\pi$-{\it closed}, or that it is a {\it morphism of $\pi$-clans},
if the induced functor $F_{(A)}:\mathcal{E}(A) \to \mathcal{E}'(FA)$
is cartesian closed for every object $A\in \mathcal{E}$.
\end{defi}

 \medskip

The composite of 
$\pi$-closed morphisms $F:\mathcal{E}\to  \mathcal{E}'$ and $G:\mathcal{E}'\to  \mathcal{E}"$
is $\pi$-closed.  The category of $\pi$-clans and $\pi$-closed
morphisms has the structure of a 2-category
where a 2-cell is a natural isomorphism.

\begin{prop}\label{homopitribe2} A  morphism of $\pi$-clans $F:\mathcal{E} \to \mathcal{E}'$
preserves products
of fibrations along fibrations.
\end{prop}

  \begin{proof}  Left to the reader.
    \end{proof}

\medskip

\begin{lemma} \label{lem:extpresheafpiclan} Let $\mathcal{E}$ be a $\pi$-clan,
then so is the clan $ \Hat{\mathcal{E}}(G)$ for every object $G\in  \Hat{\mathcal{E}}$.
Moreover, the inclusion functor $ \Hat{\mathcal{E}}(G)\to  \Hat{\mathcal{E}}/G$
preserves internal products.
\end{lemma}

\begin{proof} Let us show that if $p:E\to F$ and $f:F\to G$ are extended fibrations in $\Hat{\mathcal{E}}$,
then the map $q:\Pi_f(E,p)\to G$ is an extended fibration.
For this we have to show that for every object 
$B\in \mathcal{E}$ and every map $b:B\to G$
the object $b^\star( \Pi_f(E,p))$ is repersentable and that
the map $b^\star(q):b^\star( \Pi_f(E,p))\to B$
is a fibration. 
Consider the following diagram of pullback squares
$$\xymatrix{
b^\star(E) \ar[rr]^b \ar[d]_{b^\star(p)}  && E \ar[d]^p \\
b^\star(G) \ar[rr] \ar[d]_{b^\star(f)}  && F  \ar[d]^f \\
B\ar[rr]^b && G 
 }$$
 The base change functor $b^\star :\Hat{\mathcal{E}}/G \to \Hat{\mathcal{E}}/B$
preserves internal product by Proposition \ref{basechangecartclosed}
Thus, $b^\star( \Pi_f(E,p))$ is the internal
product of the map $b^\star(p):b^\star(E)\to b^\star(F)$ 
along the map $b^\star (f):b^\star(F)\to b^\star(G)$.
The presheaf $b^\star(F)$ is representable and the map $b^\star(f):b^\star(F)\to B$
is a fibration, since the map $f:F\to G$ is an extended fibration by hypothesis.
Similarly,  the presheaf $b^\star(E)$ is representable and the map $b^\star(p):b^\star(E)\to b^\star(F)$
is a fibration.
It follows that the presheaf $\Pi_{b^\star(f)}(b^\star(E),b^\star(p)))$
is representable, since $\mathcal{E}$ be a $\pi$-clan
and the Yoneda functor $\mathcal{E}\to \Hat{\mathcal{E}}$
preserves internal products by Lemma \ref{yonedapreservesintproduct}.
Moreover, the map $b^\star(q):\Pi_{b^\star(f)}(b^\star(E),b^\star(p)))\to B$ is a fibration,
since in a $\pi$-clan, the structure map of an internal product of a fibration 
is a fibration.
\end{proof}

Recall from Proposition \ref{basechangetribe} that if $f:A\to B$ is a map in a clan $\mathcal{E}$,
then the base change functor
$f^\star:\mathcal{E}(B)\to \mathcal{E}(A)$
is a morphism of clans,

\begin{prop}\label{charhomopitribe3} If $\mathcal{E}$ is a $\pi$-clan,
then the morphism of clans
$f^\star:\mathcal{E}(B)\to \mathcal{E}(A)$
is $\pi$-closed for any map $f:A\to B$.
\end{prop}

\begin{proof}  This should follows from Proposition \ref{basechangecartclosed},
except that the category $\mathcal{E}$ may not be locally cartesian closed.
The trick is to replace $\mathcal{E}$ by the presheaf category  $ \Hat{\mathcal{E}}$ 
and use Lemma \ref{lem:extpresheafpiclan}.
The base change functor 
$f^\star:\Hat{\mathcal{E}}/B\to\Hat{\mathcal{E}}/A$
is $\pi$-closed, since the caltegory $\mathcal{E}$
is locally cartesian closed. 
Hence, the base change functor 
$f^\star:\Hat{\mathcal{E}}(B)\to\Hat{\mathcal{E}}(A)$
is $\pi$-closed 
since the inclusions $ \Hat{\mathcal{E}}(A)\to  \Hat{\mathcal{E}}/A$
and $ \Hat{\mathcal{E}}(B)\to  \Hat{\mathcal{E}}/B$
preserves products.
This shows that the functor $f^\star:\mathcal{E}(B)\to \mathcal{E}(A)$
is $\pi$-closed, since $\mathcal{E}(A)\simeq \Hat{\mathcal{E}}(A)$
and $\mathcal{E}(B)\simeq \Hat{\mathcal{E}}(B)$.
\end{proof}

\begin{lemma}\label{localisationhomopitribe}
If $F:\mathcal{E} \to \mathcal{E}'$ is a morphism of $\pi$-clans, then so is the functor
$F_{(A)}: \mathcal{E}(A)\to \mathcal{E}'(F\!A)$ for every object $A\in  \mathcal{E}$. 
\end{lemma}

\begin{proof} Left to the reader. \end{proof} 

Let $\mathcal{E} $ be a $\pi$-clan and
consider  a pullback square in $\mathcal{E} $ 
$$ \xymatrix{
C \ar[d]_{u}\ar[rr]^{g}& & D \ar[d]^v \\
A \ar[rr]^{f}&& B
}$$
where $v$ is fibration; the map $u$ is also a fibration by base change. 
The base change functor $u^\star:\mathcal{E}(C)\to \mathcal{E}(A)$
a right adjoint $\Pi_u$ and the base change functor 
$v^\star:\mathcal{E}(D)\to \mathcal{E}(B)$ has a right 
adjoint $\Pi_v$.

\begin{cor} The following square of functors commutes up
to a canonical isomorphism,
\begin{equation} \label{BeckChevaly2}
 \xymatrix{
\mathcal{E}(C) \ar[d]_{\Pi_u}& & \mathcal{E}(D)\ar[d]^{\Pi_v}  \ar[ll]_{g^\star} \\
\mathcal{E}(A) && \ar[ll]_{f^\star}  \mathcal{E}(B)
}
\end{equation}
\end{cor}

\begin{proof} 
This essentially follows from Proposition \ref{charhomopitribe3}, except
that the category $\mathcal{E}$ is not locally cartesian closed.
The trick is to replace the category  $\mathcal{E}$ by the presheaf category  $ \Hat{\mathcal{E}}$ 
and use Lemma \ref{lem:extpresheafpiclan}.
\end{proof}

Let $f:A\to B$ be a fibration in a $\pi$-clan $ \mathcal{E}$.
The composite of the functors
$$\xymatrix{
\mathcal{E}(A)\ar[rr]^{\Sigma_f} && \mathcal{E}(B) \ar[rr]^{\Pi_B} &&  \mathcal{E}
}$$
takes the terminal object $(A,1_A)\in \mathcal{E}(A)$
to the object $\Pi_B(A):=\Pi_B(A,f)$. It thus induces a functor
$$\Pi_{B/A}:\mathcal{E}(A)\to \mathcal{E}/\Pi_B(A)$$
 if we put $\Pi_{B/A}(E,p)=(\Pi_B(E,fp),\Pi_B(p))$
 for an object $E=(E,p)\in \mathcal{E}(A)$.

 \begin{lemma}\label{prodfib} Let $f:A\to B$ be a fibration in a $\pi$-clan $ \mathcal{E}$.
 If $p:E\to A$ is a fibration, then the map $\Pi_B(p):\Pi_B(E,fp)\to \Pi_B(A)$
 is a fibration. Moreover, the functor $\Pi_{B/A}$
 is isomorphic to the composite
  $$\xymatrix{
  \mathcal{E}(A) \ar[rr]^(0.4){\epsilon^\star} &&\mathcal{E}(\Pi_B(A)\! \times \! B) \ar[rr]^{\Pi_p} &&\mathcal{E}(\Pi_B(A))
 }$$
  where $\Pi_B(A):=\Pi_B(A,f)$.
 \end{lemma}

\begin{proof} This essentially follows from Lemma \ref{choice1}, except
that the category $\mathcal{E}$ is not locally cartesian closed.
The trick is to replace the category  $\mathcal{E}$ by the presheaf category  $ \Hat{\mathcal{E}}$ 
and use Lemma \ref{lem:extpresheafpiclan}.
More precisely, it follows from Lemma \ref{choice1} that
the functor 
$$\Pi_{B/A}:\ \Hat{\mathcal{E}}/A\to  \Hat{\mathcal{E}}/\Pi_B(A)$$
defined by putting $\Pi_{B/A}(E,p)=(\Pi_B(E,fp),\Pi_B(p))$
 for an object $(E,p)\in  \Hat{\mathcal{E}}/A$
  is isomorphic to the composite
 $$\xymatrix{
  \Hat{\mathcal{E}}/A \ar[rr]^(0.4){\epsilon^\star} &&  \Hat{\mathcal{E}}/ \Pi_B(A)\! \times \! B \ar[rr]^{\Pi_p} && \Hat{\mathcal{E}}/ \Pi_B(A)
 }$$
 The first functor induces a functor $\epsilon^\star: \Hat{\mathcal{E}}(A)\to  \Hat{\mathcal{E}}(\Pi_B(A)\! \times \! B  )$
 since the base change of an extended fibration is an extended fibration.
By Lemma \ref{lem:extpresheafpiclan}, the  second functor induces a functor $\Pi_p: \Hat{\mathcal{E}}(\Pi_B(A)\! \times \! B  )\to \Hat{\mathcal{E}}(\Pi_B(A))$.
Thus, if $p:E\to A$ is an extended fibration, then
so is the map $\Pi_B(p):\Pi_B(E,fp)\to \Pi_B(A,f)$.
This shows that the map $\Pi_B(p):\Pi_B(E,fp)\to \Pi_B(A,f)$
 is a fibration when $p:E\to A$ is a fibration.
 \end{proof} 

\medskip

 \begin{lemma}\label{choicepitribe}
 Let $f:A\to B$ be a fibration in a $\pi$-clan $ \mathcal{E}$.
 Then the functor $\Pi_B\circ  \Sigma_f: \mathcal{E}(A)\to  \mathcal{E}$
 is isomorphic to the composite 
 $$\xymatrix{
  \mathcal{E}(A) \ar[rr]^(0.4){\epsilon^\star} && \mathcal{E}(\Pi_B(A)\! \times \! B) \ar[rr]^{p_\star} && \mathcal{E}(\Pi_B(A)) 
   \ar[rr]^{\Sigma_{\Pi_B(A)}} && \mathcal{E}
 }$$
 where $\Pi_B(A):=\Pi_B(A,f)$.
\end{lemma}

\begin{proof} The following square of functors commutes by definition of the functor $\Pi_{B/A}$,
$$\xymatrix{
 \mathcal{E}/A  \ar[d]_{\Sigma_f}  \ar[rr]^(0.5){\Pi_{B/A}}& &    \mathcal{E}/\Pi_B(A) \ar[d]^{\Sigma_{\Pi_B(A)}} \\
 \mathcal{E}/B  \ar[rr]^(0.5){\Pi_B} &&\mathcal{E}
}$$
The result then follows from Lemma \ref{prodfib}.
\end{proof}

\begin{prop}\label{distributivitylaw2} If $f:A\to B$ and $g:B\to C$
are fibrations in a $\pi$-clan $\mathcal{E}$,
then following diagram of functors commutes up to a canonical isomorphism,
$$\xymatrix{
 \mathcal{E}(A) \ar[rr]^{\Sigma_f}\ar[rd]_{\epsilon_A^\star} &&\mathcal{E}(B)  \ar[rr]^{\Pi_g} && \mathcal{E}(C)\\
&  \mathcal{E}(g^\star \Pi_g(A))  \ar[rr]^{\Pi_p}& &   \mathcal{E}(\Pi_g(A)) \ar[ru]_{\Sigma_q} & \\
}$$
where $ \Pi_g(A)= \Pi_g(A,f)$, where $\epsilon_A: g^\star \Pi_g(A) \to A$ is the counit of the adjunction $g^\star \dashv \Pi_g$,
where $p$ is the projection $g^\star \Pi_g(A) \to \Pi_g(A)$ and where  $q$
is the structure map $\Pi_g(A)\to C$.
\end{prop}

\begin{proof} This follows from Lemma \ref{choicepitribe} applied to  the map $f:(A,gf)\to (B,g)$ in the
$\pi$-clan $\mathcal{E}(C)$.
\end{proof}

\begin{prop}\label{productmorphismofclan} If $f:A\to B$ be a fibration in a $\pi$-clan,
 then the functor $\Pi_f:\mathcal{E}(A)\to \mathcal{E}(B)$
 is a morphism of clans.
 \end{prop}

 \begin{proof} We may suppose that $B=1$.
 By Lemma \ref{prodfib}, the functor $\Pi_A: \mathcal{E}(A)\to \mathcal{E}$ takes a fibration to a fibration.
 The functor preserves base changes and terminal
 objects, since ot is a right adjoint.
  \end{proof} 

\bigskip

Recall from \ref{defembedding}  a morphism of clans  $F:\mathcal{E}\to \mathcal{E}'$
 is said to be an {\it embedding} if it is fully faithful and if it reflects fibrations.

 \begin{defi} \label{embeddingpiclan} We shall say that a morphism of $\pi$-clans  $F:\mathcal{E}\to \mathcal{E}'$
 is an {\it embedding} if it is it is fully faithful and if it reflects fibrations.
  \end{defi}

 \begin{defi} \label{subpiclan=embedding} If  $\mathcal{E}$ is a $\pi$-clan,
 we shall that a sub-clan $\mathcal{L}\subseteq \mathcal{E}$
 is {\it $\pi$-closed} if $\mathcal{L}$
is a $\pi$-clan and the inclusion functor 
$\mathcal{L}\to \mathcal{E}$ is $\pi$-closed.
\end{defi}

\subsection{Products in type theory}

If $A$ is an object of a $\pi$-clan $ \mathcal{E}$
and  $E=(E,p)\in \mathcal{E}(A)$, let us put
$$\prod_{x:A} E(x)\defeq \Pi_A(E).$$
For typographic reasons, we may write $ \prod(x:A) E(x)$ instead of  $\prod_{x:A} E(x)$.
In type theory, the internal product  is created by the {\it $\Pi$-formation rule}
 \[
\begin{prooftree} 
 x:A \vdash E(x):Type
 \justifies
 \vdash \prod_{x:A} E(x):Type
\end{prooftree}
\]
The counit of the adjunction $e_A\dashv \Pi_A$
is the evaluation $\epsilon:\Pi_A(E)\times A\to E$. 
It is a map in $\mathcal{E}(A)$ and 
we have $\epsilon(s,x)=s(x)$ by definition.
Thus, the evaluation $\epsilon$ is 
created by the $\Pi$-{\it elimination rule}:
\[
\begin{prooftree} 
\vdash s:\prod(x:A)E(x) \quad \quad \vdash a:A
 \justifies
\vdash s(a):E(a)
\end{prooftree}
\]
The map $\epsilon:\Pi_A(E)\times A\to E$
defines the  internal family of {\it projections}
$$a:A\vdash \epsilon(a): \prod_{x:A} E(x) \to E(a).$$
 It follows from the adjunction $e_A\dashv \Pi_A$
 that for every object $B\in \mathcal{E}$ and every internal family of maps
  $$x:A\vdash f(x): B\to E(x)$$
 there exits  a unique map 
 $$f: B\to  \prod_{x:A} E(x)$$
 such $\epsilon(x)\circ f=f(x)$ for every $x:A$.
 Let us put $\lambda(x:A)f(x)\defeq f$.
 This defines an operation which takes a term $f(x,y):E(x)$
 in context $(x:A, y:B)$ to a
 term $\lambda(x:A)f(x,y): \Pi_A(E,p) $
 in context $y:B$. In type theory, the operation
 is provided by the $\Pi$-{\it introduction rule} 
 \[
\begin{prooftree} 
 x:A, y:B  \vdash f(x,y):E(x)
 \justifies
y:B  \vdash \lambda(x:A)f(x,y):\prod(x:A) E(x)
\end{prooftree}
\]

\bigskip

It follows from Lemma \ref{choicepitribe} that the 
 following general {\it distributivity law} holds:
$$\prod_{x:A} \sum_{y:B(x)} E(x,y)=\sum_{s:\prod_{x:A} B(x) } \ \  \prod_{x:A} E(x, s(x))$$
for dependant types
$$x:A \vdash B(y):Type \quad {\rm and} \quad x:A,  y:B(x)\vdash E(x,y):Type $$

\bigskip

If $A$ is an object of a $\pi$-clan $\mathcal{E}$, 
then the base change functor $e_A:\mathcal{E}\to  \mathcal{E}(A)$
along the map $A\to 1$
 is $\pi$-closed by \ref{charhomopitribe3}. 
 Moreover,  the diagonal $\delta_A:A\to A\times A$ is an element  of type $e_A(A)$.

\medskip

 \begin{thm} \label{genericelementlambdatribe} If $A$ is an object in a $\pi$-clan $\mathcal{E}$,
 then the morphism of  $\pi$-clans  $e_A:\mathcal{E}\to  \mathcal{E}(A)$ 
 is freely generated by the element $\delta_A:e_A(A)$.
 Thus, $\mathcal{E}(A)= \mathcal{E}[\delta_A]$.
 \end{thm}

 \begin{proof} 
  If $ \mathcal{R}$ is a $\pi$-clan, let us denote 
  the category of morphisms of $\pi$-clans $\mathcal{E}\to \mathcal{R}$ by
  by $Hom_\pi(\mathcal{E}, \mathcal{R})$ (it is a groupoid, since a 2-cell
  is an isomorphism). Let us denote by $Hom_\pi(\mathcal{E}, \mathcal{R})_A$
  the groupoid whose objects are the pairs $(F,a)$, where $F:\mathcal{E} \to \mathcal{R}$
is a morphism of $\pi$-clans and $a$ is an element of type $F(A)$.
A map $(F,a)\to (G,b)$ in this category is a natural isomorphism $\phi:F\to G$
such that $\phi_A(a)=b$. 
We shall prove that the functor
\begin{equation}\label{restrictionalongfreepi}
e'_A:Hom_\pi(\mathcal{E}(A),  \mathcal{R})\to Hom_\pi(\mathcal{E}, \mathcal{R})_A  \, 
\end{equation}
defined by putting  $e'_A(H)=(H\circ e_A,H(\delta_A))$ 
is an equivalence of groupoids. We shall use Theorem \ref{genericelement}.
Let $Hom(\mathcal{E},  \mathcal{R})$ be the category whose objects
are the morphism of clans $\mathcal{E}\to  \mathcal{R}$
and whose morphism are the natural transformation.
Then the functor 
$$
e^\star_A:Hom(\mathcal{E}(A),  \mathcal{R})\to Hom(\mathcal{E}, \mathcal{R})_A  \, 
$$
defined by putting  $e^\star_A(H)=(H\circ e_A,H(\delta_A))$ 
is an equivalence of catefories by Theorem \ref{genericelement}.
Let us denote by $iHom(\mathcal{E},  \mathcal{R})$
the groupoid of isomorphisms of the category $Hom(\mathcal{E},  \mathcal{R})$.
It follows from Theorem \ref{genericelement} that the functor 
$$
ie^\star_A:iHom(\mathcal{E}(A),  \mathcal{R})\to iHom(\mathcal{E}, \mathcal{R})_A  \, 
$$
induced by $e_A^\star$ is an equivalence of groupoids.
We then have the following commutative square  of functors,
 $$
 \xymatrix{
 Hom_\pi(\mathcal{E}(A),  \mathcal{R}) \ar[rr]^{e'_A} \ar[d] && Hom_\pi(\mathcal{E}, \mathcal{R})_A \ar[d]\\
  iHom(\mathcal{E}(A),  \mathcal{R}) \ar[rr]^{ie_A^\star} && iHom(\mathcal{E}, \mathcal{R})_A
  }
 $$
 where the vertical functors are inclusions of full sub-groupoids.
  The functor $e'_A$ functor is fully faithful, since 
 the functor $ie_A^\star$ is an equivalence of  groupoids.
It remains to show that the functor $e'_A$ is essentially surjective.
 But if $(F,a)\in Hom_\pi (\mathcal{E},  \mathcal{R})_A $, then there exists
 a functor $H\in Hom(\mathcal{E}(A),  \mathcal{R})$ such that
 $e_A^\star(H)\simeq (F,a)$, since the  functor $e_A^\star$ is essentially surjective.
  Let us show that $H\in  Hom_\pi(\mathcal{E}(A),  \mathcal{R})$.
  The following square is cartesian by \ref{genericfiber},
 since the functor $H$ preserves base changes of fibrations,
 $$  \xymatrix{
H(E,p) \ar[d]\ar[rr]^{H(p,1_E)} & & He_A(E) \ar[d]^{He_A(p)} \\
1 \ar[rr]^{H(\delta_A)}&& He_A(A)\ .
}$$
It follows that the functor $H$ is isomorphic to the 
composite
$$\xymatrix{
 \mathcal{E}(A)\ar[rr]^(0.4){(He_A)_{(A)}} && \mathcal{R}(He_A(A))\ar[rr]^(0.6){H(\delta_A)^\star} &&  \mathcal{R}
 }
 $$
where the functor $(He_A)_{(A)}$ is induced by the functor $He_A: \mathcal{E}\to \mathcal{R}$.
The morphism of clans $He_A$ is $\pi$-closed since $He_A$ is isomorphic to $F$
and $F$ is $\pi$-closed. It then follows by Lemma \ref{localisationhomopitribe}
that the functor $(He_A)_{(A)}$ is $\pi$-closed.
The base change functor  $H(\delta_A)^\star$ is also $\pi$-closed
by Proposition \ref{charhomopitribe3}.
This shows that the functor $H$ is $\pi$-closed and hence  that $H\in  Hom_\pi(\mathcal{E}(A),  \mathcal{R})$.
We have proved that the functor $e'_Ai$ is an equivalence of categories.
\end{proof}

 If $p:E\to A$ is a fibration in a $\pi$-clan  $ \mathcal{E}$,
 let us denote by 
 $$e_{p}:\mathcal{E}\to  \mathcal{E}(\Pi_A(E))$$ the
base change functor along the map $\Pi_A(E)\to 1$.
If $\epsilon:\Pi_A(E)\times A\to E$ is the evaluation, then we have $p\epsilon =p_1$.
It follows that the map $s_p:=(p_1,\epsilon):\Pi_A(E)\times A\to \Pi_A(E)\times E$
 is a section of the fibration $e_p(p):\Pi_A(E)\times p: \Pi_A(E)\times E \to \Pi_A(E)\times A$.

 \begin{prop} \label{genericelementlambdatribe2} If $p:E\to A$ is a fibration in a
 $\pi$-clan  $ \mathcal{E}$, then the morphism of $\pi$-clans $e_{p}:\mathcal{E}\to  \mathcal{E}(\Pi_A(E))$
 is freely generated by the section $s_p:e_p(E)\to e_p(A)$ of the fibration $e_p(p)$.
 Thus, $\mathcal{E}(\Pi_A(E) )= \mathcal{E}[s_p]$.
 \end{prop}

\begin{proof} This follows from theorem \ref{genericelementlambdatribe},
since there is canonical bijection between the sections of the fibration $p:E\to A$
and the elements of type $\Pi_A(E)$. In this bijection, a generic section
will correspond to a generic element.
\end{proof}

If $A$ and $B$ are two objects of a $\pi$-clan, let us denote 
by 
$$e_{[A,B]}:\mathcal{E}\to  \mathcal{E}([A,B])$$
the
base change functor along the map $[A,B]\to 1$.
If $\epsilon:[A,B]\times A\to B$ is the evaluation, then 
$f_{AB}:=(p_1,\epsilon):[A,B]\times A\to [A,B]\times B$
is a map $e_{[A,B]}(A)\to e_{[A,B]}(B)$.

 \begin{prop} \label{genericelementlambdatribe3} If $A$ and $B$ are two objects in a $\pi$-clan
$ \mathcal{E}$, then the morphism of $\pi$-clans $e_{[A,B]}:\mathcal{E}\to  \mathcal{E}([A,B])$
 is freely generated by the map $f_{AB}:e_{[A,B]}(A)\to e_{[A,B]}(B)$. 
 Thus, $\mathcal{E}([A,B])= \mathcal{E}[f_{AB}]$.
 \end{prop}

\begin{proof} Similar to the proof of Proposition \ref{genericelementlambdatribe2}
\end{proof}

\newpage

\section{Theory of Tribes}

\subsection{Anodyne maps}

Recall that a map $u:A\to B$ in a category $\mathcal{E}$  is said to have the
{\it left lifting property} with respect to a map $f:X\to Y$,
and that $f$ is said to have the {\it right lifting property} 
with respect to $u$, if every commutative square
\begin{equation} \label{squareanofib}
\xymatrix{A\ar[d]_u \ar[r]^a& X\ar[d]^f \\
B \ar[r]^b &Y
}\end{equation}
has a diagonal filler $d:B\to X$ ($fd=b$ and $du =a$),
$$
\xymatrix{A\ar[d]_u \ar[r]^a& X\ar[d]^f \\
B \ar@{-->}[ru]^{d}  \ar[r]^b &Y
}
$$
We shall denote this relation by $u\pitchfork f$.

\medskip

If $\mathcal{K}$ is a class of maps in a category $ \mathcal{E}$,
we shall denote by ${}^\pitchfork \mathcal{K}$ (resp. $\mathcal{K}^\pitchfork$) the class of maps $f\in \mathcal{E}$
having the left (resp. right)  lifting property 
with respect to every map in $ \mathcal{K}$.

\medskip

\begin{defi} 
We say that a map in a clan 
is {\bf anodyne} if it has the left lifting property with respect
to every fibration.
 \end{defi}

We shall often picture an anodyne map $A\to B$ by an arrow with a tail $A\rightarrowtail B$.

\begin{prop}\label{anodyne closurecompositionret}  
The class of anodyne maps is closed under composition
and retracts. 
\end{prop}

\begin{proof} If $\mathcal{F}$ denotes the class of fibrations, then the class ${}^\pitchfork \mathcal{F}$
is closed under composition and retracts by Proposition \ref{triviftingprop}.
\end{proof}

\medskip

Recall that a map $r:B\to A$ is called a {\it retraction}
of a map $i:A\to B$ if $ri=1_A$.
A map $i:A\to B$ is called a {\it split monomorphism}
if it admits a retraction $r:B\to A$.

\begin{lemma}\label{anodyne closure}  Every anodyne map is a split monomorphism.
If the composite $vu:A\to C$ of two maps $u:A\to B$ and $v:B\to C$
is anodyne and $v$ is anodyne, then $u$ is anodyne.
\end{lemma}

\begin{proof} If $u:A\to B$ is anodyne, then the following square
has a diagonal filler $r:B\to A$, since the map $A\to 1$
is a fibration. 
$$
\xymatrix{ 
A\ar[d]_{u} \ar@{=}[r]& A \ar[d]\\
B \ar[r] & 1
}
$$
Thus, $u$ is a split monomorphism.
Let us prove the last statement of the proposition.
The map $v:B\to C$ is a split monomorphism, since it is anodyne.
Let us choose a retraction $r:C\to B$ ($rv=1_B$).
The map $u:A\to B$ is then a (codomain) retract of the map $vu:A\to C$,
since $r(vu)=u$,
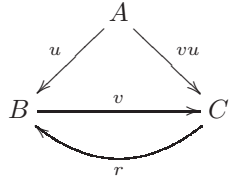
$$
 \xymatrix{ 
&A  \ar[dl]_u \ar[dr]^{vu}  & \\
B \ar[rr]^v  &  &  C   \ar@/^1.5pc/[ll]^-{r}
}
$$
Thus, $u$ is anodyne by \ref{anodyne closurecompositionret},
since $vu$ is anodyne.
\end{proof}

\begin{lemma} \label{charanodyne} A map $u:A\to B$ in a tribe
is anodyne if and only if  for every fibration $f:E\to B$ and every map $a:A\to E$
such that $fa=u$, there exists a map $s:B\to E$ such that $fs=1_B$
and $su=a$.
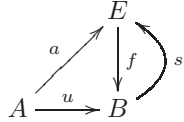
$$
\xymatrix{ & E\ar[d]^f\\
A\ar[ur]^a \ar[r]^u &B \ar@/_1.5pc/[u]_-{s}
}$$
\end{lemma} 

\begin{proof}  ($\Rightarrow$) The following commutative square has a diagonal filler,
since $u$ is anodyne and $f$ is a fibration.
$$
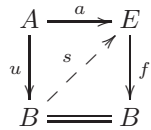
\xymatrix{ A \ar[d]_u \ar[r]^{a}& E\ar[d]^f\\
B\ar@{-->}[ur]^s \ar@{=}[r] &B
}$$
($\Leftarrow$) Let us shows that $u$ is anodyne.
For this we have to show that if $f:X\twoheadrightarrow Y$ is a fibration, 
then every commutative square
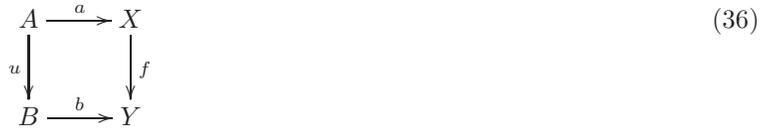
\begin{equation} \label{squareforlifting}
\xymatrix{A\ar[d]_u \ar[r]^a& X\ar[d]^f \\
B \ar[r]^b &Y
}\end{equation}
has a diagonal filler. Consider the 
 following commutative diagram, in which the right hand square is a pullback,
$$
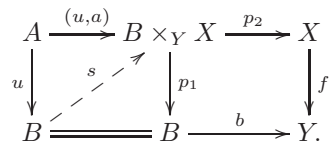
\xymatrix{
 A \ar[r]^(0.4){(u,a)} \ar[d]_u& B\times_Y X\ar[d]^{p_1}\ar[r]^(0.6){p_2} & X\ar[d]^f\\
B \ar@{-->}[ur]^s \ar@{=}[r]  & B \ar[r]^b & Y.     
}
$$
The projection $p_1$ is a fibration by base change, since $f$ is a fibration.
Hence there exists a map $s:B\to B\times_Y X$ such that $p_1s=1_B$
and $su=(u,a)$ by the hypothesis on $u$.
The composite $p_2 s:B\to X$ is then a diagonal filler
of square (\ref{squareforlifting}).
\end{proof}

\begin{defi} \label{defi:AF-fact}
We say that a factorisation $f=pu:A\rightarrowtail  E\twoheadrightarrow B$
of a map  $f:A\to B$  in a clan is an  $AF$-{\bf factorization} 
if $u$ is anodyne and $p$ is a fibration.
\end{defi}

\begin{defi}  \label{defpretypos} We say that
a clan $\mathcal{E}$
is a {\bf tribe}, if the following
two conditions are satisfied
\begin{itemize}
\item{}  every map $f:A\to B$ in $ \mathcal{E}$
admits an $AF$-factorization  $f=pu:A\rightarrowtail  E\twoheadrightarrow B$;
\item{} the base change of an anodyne map along a fibration is anodyne.
\end{itemize}
 \end{defi}

\bigskip

Examples of tribes:
 
\begin{itemize}
    \item{}  The category of Kan complexes {\bf Kan}, where a fibration is a Kan fibration; 
    a map is anodyne if and only if it is a monic homotopy equivalence,
    \item{} The category of small categories $\mathbf{Cat}$, where a fibration is an iso-fibration;
      a functor is anodyne if and only if it is an equivalence of categories monic on objects.
  \item{} The category of small groupoids $\mathbf{Grpd}$, where
  a fibration is an iso-fibration;  a functor is anodyne if and only if it is an equivalence of groupoids monic on objects.
\item{} The category of fibrant object of a right proper model category, if the cofibrations
are the monomorphisms. 
\item{} The category of quasi-categories, where a fibration is a quasi-fibration;
\item{} The category of quasi-categories with finite limits (and maps preserving finite limits),
 where a fibration is a quasi-fibration.
\end{itemize}

\begin{rem} The notion of tribe is invariant under equivalences of clans:
if $F:\mathcal{E}\to \mathcal{E}'$ is an equivalence of clans and 
one of them is a tribe, then so is the other.
\end{rem}

\begin{prop}\label{cartanodyne}  In a tribe, the cartesian product of two
anodyne maps is anodyne.
\end{prop}

\begin{proof} If $u:A\rightarrowtail B$ is anodyne, then
$u\times C$ is anodyne by base change, 
since the following square is cartesian and the projection $B\times C\to B$
is a fibration.
$$
\xymatrix{ 
A\times C\ar[d]_{u\times C} \ar[r]^{} &A  \ar[d]^{u}\\
B\times C \ar[r]^{} & B.
}
$$
Similarly, if $v:C\rightarrowtail  D$ is anodyne, then $B\times v$ is anodyne.
Thus, $u\times v=(B\times v)(u\times C)$ is anodyne,
since the composite of two anodyne maps is anodyne by \ref{anodyne closurecompositionret}.
\end{proof}

We saw in \ref{clanslice} that if $\mathcal{E}$ is a clan,
then so is the category $\mathcal{E}(A)$ for every object $A\in \mathcal{E}$.
A morphism $f:(E,p)\to (F,q)$ in $\mathcal{E}(A)$ is a fibration
if and only the map $f:E\to F$ is a fibration in $\mathcal{E}$.

\begin{prop} \label{tribeslice} If $\mathcal{E}$ is a tribe, 
then so is the clan $\mathcal{E}(A)$ for every object $A\in \mathcal{E}$.
A morphism $u:(E,p)\to (F,q)$ in $\mathcal{E}(A)$ is anodyne 
if and only the map $u:E\to F$ is anodyne  in $\mathcal{E}$.
\end{prop}

\begin{proof} It follows from Lemma \ref{charanodyne}
that a map $u:(E,p)\to (F,q)$ 
in  $\mathcal{E}(A)$ is anodyne if and only if 
the map $u:E\to F$
is anodyne  in $\mathcal{E}$ .
The base change of an anodyne map in $\mathcal{E}(A)$
along a fibration in $\mathcal{E}(A)$ is anodyne.
t remains to verify that every map
$f:(E,p)\to (F,q)$ in $\mathcal{E}(A)$
can be factored as an anodyne map
followed by a fibration.
Let us choose a factorization $f=gu:E\to C\to F$ 
of the map $f:E\to F$ in $\mathcal{E}$, with $u:E\to C$
an anodyne map and  $g:C\to F$ a fibration.
The map $qg:C\to A$ is a fibration, since $q$ and $g$ are fibration.
Thus, $(C,qg)\in \mathcal{E}(A)$ and we obtain a factorisation $f=pu:(E,p)\to (C,qg)\to (F,q)$ 
 in $ \mathcal{E}(A)$.
 The map $g:(C,qg)\to (F,q)$ is a fibration  in $ \mathcal{E}(A)$,
since  $g:C\to F$  is a fibration in $\mathcal{E}$.
 Similarly, the map $u:(E,p)\to (C,qg)$ is anodyne in $ \mathcal{E}(A)$ by the first part
of the proof, since $u:E\to C$ is anodyne in $\mathcal{E}$.
The proposition is proved.
\end{proof}

\begin{defi} 
We shall say that $\mathcal{E}(A)$ is the {\bf local tribe} of $\mathcal{E}$
at $A$. 
 \end{defi}

\medskip

 \begin{defi} \label{defmorphismhotribes}  
 If $\mathcal{E}$ and  $\mathcal{E}'$ are tribes,  
 we say that a morphism of clans $F:\mathcal{E} \to \mathcal{E}'$ is a  {\bf morphism of tribes}  
 if it takes anodyne maps to anodyne maps. 
   \end{defi} 
 
\medskip

We shall denote  by $\mathbf{Trib}$ the category of tribes and morphism of tribes.
The terminal category $\mathbf{1}$ has the structure a tribe
 and the resulting tribe $\mathbf{1}$ is terminal in the category $\mathbf{Trib}$.
 The cartesian product of two tribes $\mathcal{E}_1$ and  $ \mathcal{E}_2$ has the structure of a tribe.
 A map $(u_1,u_2):(X_1,X_2)\to (Y_1,Y_2)$
 in $\mathcal{E}_1\times \mathcal{E}_2$ is anodyne
 if and only if $u_1$ is anodyne in $\mathcal{E}_1$ and $u_2$ is anodyne in $ \mathcal{E}_2$.
The category $\Trib$ is cartesian and the forgetful functor $\Trib \to \Clan$
is cartesian.

\medskip

The category $\mathbf{Trib}$ has also the structure of a 2-category in which a 
2-cell is a natural transformation.
An {\bf equivalence of tribes} is an equivalence in this 2-category.
A morphism of tribes $F:\mathcal{E} \to \mathcal{E}'$
is an equivalence of tribes if and only if it is an
equivalence of clans.

\medskip

We say that a morphism of clans $F:\mathcal{E}\to \mathcal{E}'$ 
{\it reflects} anodyne maps if the implication 
$$F(u) \quad {\rm anodyne} \quad \Rightarrow \quad u \quad {\rm anodyne}$$
is true for every map $u:A\to B$ in $\mathcal{E}$.

 \begin{lemma} \label{reflectionofanodynemorphclan} 
 If  a morphism of clans $F:\mathcal{E}\to \mathcal{E}'$ 
 is fully faithful, then it reflects anodyne maps.
 \end{lemma}

\begin{proof} Left to the reader.
\end{proof}

\medskip

\begin{thm}\label{basechangehomotopical}
 If $ f:A\to B$ is a map in a tribe $\mathcal{E}$,
 then the base change functor
$$f^\star:\mathcal{E}(B)\to \mathcal{E}(A)$$
is a morphism of tribes.
\end{thm}

\begin{proof} 
We say in \ref{basechangetribe} that the functor $f^\star$ is a morphism of clans.
It remains to show that $f^\star$ preserves anodyne maps.
Let us choose a factorization
$f=gu:A\to C\to B$ with $u:A\to C$ an anodyne map and $g:C\to B$ a fibration.
The functor $g^\star$ preserves anodyne maps, since the base change of an anodyne
map along a fibration is anodyne, by definition of a tribe  \ref{defpretypos}.
Let us show that the functor $u^\star$ preserves anodyne maps.
If $v:(X,p)\to (Y,q)$ is an anodyne map  in $ \mathcal{E}(C)$,
then we have the following diagram of cartesian squares:
$$ \xymatrix{
A\times_C X \ar[d]_{A\times_C v}\ar[rr]^(0.6){u_X }  && X \ar[d]_{v}    \ar@/^1.5pc/[dd]^-{p}   \\
A\times_C Y\ar[d]^{}\ar[rr]^(0.6){u_Y }  && Y\ar[d]_{q}  \\
A \ar[rr]^{u} &&  C \ .
 }$$
The map 
$u_X$ is anodyne by base change, since  $u$ is anodyne and $p:X\to C$ is a fibration.  
Similarly, the map $u_Y$ is also anodyne by base change, since since  $u$ is anodyne and $q:Y\to C$ is a fibration.  
The composite $v u_X$ is anodyne by Lemma \ref{anodyne closurecompositionret}, since $u_X$ and $v$ are anodyne.
It follows that $A\times_C v$ is anodyne by the same lemma,
since $u_Y$ is anodyne and $u_Y(A\times_C v)=v u_X$ is anodyne.
Hence the functor $u^\star$ preserves anodyne maps.
It follows that the functor  $f^\star=u^\star g^\star$
preserves anodyne maps.
\end{proof}

\begin{prop}\label{inducedhfunctor}
If $F:\mathcal{E} \to \mathcal{E}'$ is a morphism of tribes,
then so is the functor
$$F_{(A)} :\mathcal{E}(A) \to \mathcal{E}'(FA)$$
induced by $F$ for every object $A\in \mathcal{E} $.
\end{prop}

\begin{proof} We saw in Proposition \ref{inducedhomoclan} that 
the functor $F_{(A)}$ is a morphism of clans.
It follows  from  \ref{tribeslice}
The functor $F_{(A)}$ preserves anodyne maps by 
Proposition \ref{tribeslice}, since the functor
$F$ preserves anodyne maps.
\end{proof}

\begin{defi} \label{pathobjecetdef} A {\it path object} for an object $A$ in a tribe is
a quadruple $(PA,\partial_0,\partial_1, \sigma)$ obtained by factoring the diagonal $A\to A\times A$ 
as an anodyne map $\sigma:A\to PA$ followed by a fibration $(\partial_0,\partial_1):PA \to A\times A$.
 $$ \xymatrix{
A \ar[rr]^-{(1_A,1_A)} \ar[dr]_\sigma && A\times A \\
& PA  \ar[ur]_{(\partial_0,\partial_1)}& 
 }$$
\end{defi}

\medskip

The quadruple $(PA,\partial_0,\partial_1, \sigma)$ yields a commutative diagram
$$ \xymatrix{
&  \ar@/_1pc/[ld]_-{1_A} A\ar[d]_\sigma   \ar@/^1pc/[rd]^-{1_A} &  \\
A &\ar[l]_{\partial_0}  PA \ar[r]^{\partial_1} & A
 }$$
The maps $\partial_0$ and $\partial_1$ 
are fibrations, since  the projections $p_1,p_2:A\times A\to A$ are fibrations
and we have $\partial_0=p_1( \partial_0,\partial_1)$ and
$\partial_1=p_2(\partial_0,\partial_1)$.

\medskip

We shall denote by $P\!A(x,y)$ the fiber of the fibration $(\partial_0,\partial_1):PA \to A\times A$
at an element $(x,y):A\times A$. 
An element of type $PA(x,y)$ is a {\it path} $p:x\leadsto y$ (or a homotopy) from $x$ to $y$.
The element $\sigma(x):PA(x,x)$ is the {\it identity path} $\sigma(x):x\leadsto x$.

\subsection{Paths in type theory}

\bigskip

In Martin-L\"of type theory, there is a type constructor which associate to a type $A$ 
another type $Equ_A$ in context $A\times A$, called the {\bf equality type} of $A$,
  \[
\begin{prooftree} 
  \enspace A:Type
 \justifies
 x\!:\!A.\enspace y\!:\!A \enspace \vdash Equ_A(x,y):Type  
\end{prooftree}
\]
An element of type $Equ_A(x,y)$ is regarded as a {\it proof} that the elements $x$ and $y$ are 
{\it equivalent} $x\simeq y$ (propositional equality). 
There is also an operation which associate to each type $A$ and each element $x:A$ 
an element $r_A(x):Equ_A(x,x)$, called the {\bf reflexivity term},
  \[
\begin{prooftree} 
  \enspace A:Type
 \justifies
x\!:\!A\enspace  \vdash  r_A(x):Equ_A(x,x)
\end{prooftree}
\]
The element $r_A(x)$ is the tautological proof that $x\simeq x$ for element $x:A$.

\begin{rem} In type theory, the dependant type  $Equ_A(x,y)$ is often called
the {\it identity type} and denoted $Id_A(x,y)$.
\end{rem}

It was discovered by Awodey and Warren \cite{} that the dependant type $Equ_A(x,y)$ 
can represents the space of paths between two points $x$ and $y$  
of a space $A$.  Equivalently, the type
 $$Equ(A)=\sum_{x:A,y:A} Equ_A(x,y)$$
is the path space of $A$ and the  fibration
$(\partial_0,\partial_1):Equ(A)\to A\times A$
is the display map of the sum.
The element $r(x):Equ_A(x,x)$ is defining a map $r:A\to Equ(A)$ in a commutative
diagram
$$\xymatrix{
 && Equ(A)\ar[dd]^{(\partial_0,\partial_1)}\\
 &&\\
A \ar[rr]^{( 1_A,1_A)} \ar[rruu]^{r} && A\times A .  
}
$$
The proof that $r$ is anodyne depends on the operation $J$ 
of Martin-L\"of type theory. This operation associates to 
every fibration $f:C\to  Equ(A)$ and
every commutative square 
$$\xymatrix{
  A\ar[d]_r \ar[r]^t & C \ar[d]^f   \\
 Equ(A)  \ar@{=}[r]  &  Equ(A)    
} \quad \quad \xymatrix{
  A\ar[d]_r \ar[r]^t & C \ar[d]^f   \\
 Equ(A)\ar@{-->}[ru]^J  \ar@{=}[r]  &  Equ(A)    
}
$$
a diagonal filler $J=J(t,f): Equ(A) \to C$.
The fibration $f:C\to   Equ(A)$ is the total space
of a dependant type $C(x,y,p)$ in context $(x\!:\!A.\  y\!:\!A. \ p\!:\!Equ_A(x,y))$,
$$ x\!:\!A.\enspace  y\!:\! A.\enspace p\!:\! Equ_A(x,y) \enspace  \vdash  C(x,y,p)\!:\!Type $$
The map $t:A\to C$ is a section of the fibration $r^\star(C)\to A$.
It is thus defined by an element $ t(x): C(x,x,r(x))$ in context $x:A$.
$$ x\!:\!A  \enspace \vdash  t(x)\!:\! C(x,x,r(x))$$
The diagonal filler $J(t,f):Equ(A) \to C$
is a section of the fibration $f:C\to  Equ(A)$.
It is thus defined by an element 
$J(t,C)(x,y,p):C(x,y,p)$
in context $(x\!:\!A.\ y\!:\!A.\ p\!:\!Equ_A(x,y))$,
$$x\!:\!A.  \enspace   y\!:\!A.  \enspace   p\!:\! Equ_A(x,y)  \vdash  J(t,C)(x,y,p)\!:\! C(x,y,p)$$
We have $J(t,f)r=t$, since the square commutes.
Thus, $J(t,C)(x,x,r(x))=t(x)$ for every $x:A$.

\medskip

\subsection{The homotopy category}

\begin{defi} Let $(PB,\sigma, \partial_0,\partial_1)$ be a path object
 for an object $B$ in a tribe $\mathcal{E}$.  A {\bf homotopy} between two maps $f,g:A\to B$ in $\mathcal{E}$
is a map $H:A\to PB$ such that $\partial_0H=f$ and $\partial_1H=g$. 
$$ \xymatrix{
&& B \\
A  \ar@/_1pc/[rrd]_-{g}  \ar@/^1pc/[rru]^-{f}  \ar[rr]^H  && PB \ar[u]_{\partial_0}\ar[d]^{\partial_1}& \\
 && B
 }$$
We shall write that $H:f\leadsto g$.
Two maps $f,g:A\to B$ are  {\it homotopic}, $f\sim g$,  if there exists a homotopy $H:f\leadsto g$.
\end{defi}

The notion of homotopy defined above depends on the choice of a path object 
for $B$. The following lemma implies that the homotopy
relation between the maps $A\to B$ is independent of the choice
of a path object for $B$.

\medskip

\begin{lemma} \label{independancepathob} Let $(PB,\partial_0,\partial_1,\sigma)$
and $(P'B,\partial'_0,\partial'_1, \sigma')$ be two path objects for $B$.
Then there exists a map $d:PB\to P'B$ such that the following diagram commmutes,
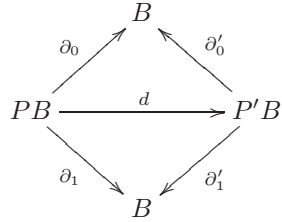
$$ \xymatrix{
& B & \\
 PB \ar[ur]^{\partial_0}\ar[dr]_{\partial_1}\ar[rr]^d& & P'B   \ar[ul]_{\partial'_0}\ar[dl]^{\partial'_1}\\
 & B & 
 }$$
 If $H:A\to PB$ is a homotopy between two maps $f,g:A\to B$,
 then $H'\defeq dH:A\to P'B$  is a homotopy between the same maps.
\end{lemma}
 
 \begin{proof} The following square has a diagonal filler $d:PB\to P'B$, 
 since $\sigma$ is anodyne and
$(\partial'_0,\partial'_1)$ is a fibration.
 \begin{equation}
 \xymatrix{
B \ar[d]_{\sigma} \ar[rr]^{\sigma'} &&P'B \ar[d]^{(\partial'_0,\partial'_1)}  \\
PB \ar[rr]^{(\partial_0,\partial_1)} && B\times B
 }
 \end{equation}
  If $(\partial_0,\partial_1)H=(f,g)$ then 
$$(\partial'_0,\partial'_1)H'=(\partial'_0,\partial'_1)dH=(\partial_0,\partial_1)H=(f,g).$$
 \end{proof}

\begin{lemma} \label{hoisequiv} Let $A$ and $B$ be two objects of  a tribe $\mathcal{E}$.
Then the homotopy relation $f\sim g $ on the set of maps $A\to B$ 
 is an equivalence.
\end{lemma}

\begin{proof} Let $(PB,\partial_0,\partial_1,\sigma)$ a path object for $B$.
The homotopy relation on the set $\mathcal{E}(A, B)$
 is reflexive, since the map $\sigma f:A\to PB$ is a homotopy $f\leadsto f$ for any map $f:A\to B$.
The homotopy relation is symmetric, since 
$(PB,\partial_1,\partial_0,\sigma)$ is a path object for $B$.
Let us show that the homotopy relation is transitive. 
The triple $PB=(PB,  \partial_0,  \partial_1)$ is a span $PB:B\multimap B$.
If we composite this  span 
with itself, we obtain a span $(P_2B, \partial_0q_0, \partial_1q_1)$.
$$ \xymatrix{
& &P_2B \ar[dl]_{q_0} \ar[dr]^{q_1}& & \\
& PB\ar[dl]_{\partial_0} \ar[dr]^{\partial_1} &  & PB\ar[dl]_{\partial_0} \ar[dr]^{\partial_1}   & \\
B && B && B 
 }$$
 It follows from Proposition \ref{spanscomposition} that the map
$(\partial_0q_0,\partial_1q_1):P_2B\to B\times B$ is a fibration.
 There is a unique map $\sigma_2:B\to P_2B$ such
 that $q_0\sigma_2=\sigma$ and $q_1\sigma_2 =\sigma$,
 since $\partial_1\sigma=1_B=\partial_0 \sigma$.
 We have $\partial_0q_0\sigma_2=1_B=\partial_1q_1\sigma_2$.
 $$ 
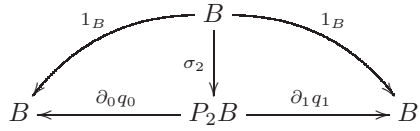
\xymatrix{
&&  \ar@/_1pc/[lld]_-{1_B} B\ar[d]_{\sigma_2}   \ar@/^1pc/[rrd]^-{1_B} &&  \\
B &&\ar[ll]_{\partial_0q_0}  P_2B \ar[rr]^{\partial_1q_1} && B
 }$$
Let us show that the quadruple $(P_2B, \partial_0q_0, \partial_1q_1,\sigma_2)$  is a path object for $B$.
For this we have to show  that the map $\sigma_2:B\to P_2B$ is anodyne.
We have $\partial_1^\star(PB,\partial_0)= (P_2B,q_0)$,
since the following square is cartesian
$$ \xymatrix{
  P_2B\ar[d]_{q_0} \ar[rr]^{q_1}&& PB \ar[d]^{\partial_0}\\
 PB \ar[rr]^{\partial_1}  && B 
}$$
The map $s\defeq \partial_1^\star(\sigma)$
is a section of the fibration $q_0:P_2B\to PB$, since
the map $\sigma:B\to PB$ is a section of the fibration $\partial_0:PB\to B$.
$$ \xymatrix{
  P_2B\ar[d]^{q_0} \ar[rr]^{q_1}&& PB \ar[d]_{\partial_0}\\
 \ar@/^1pc/[u]^-{s}    PB \ar[rr]^{\partial_1}  && B  \ar@/_1pc/[u]_-{\sigma} 
}$$
Thus, $s$ is anodyne by Theorem \ref{basechangehomotopical}, since $\sigma$ is anodyne. 
By definition, $q_0 s=1_{PB}$ and $q_1s= \sigma\partial_1$.
Hence we have $\sigma_2=s\sigma$, since 
$q_0\sigma_2=\sigma=q_0s \sigma$ and
$q_1\sigma_2 =\sigma= \sigma\partial_1\sigma= q_1s \sigma$.
Thus,  $\sigma_2$
 is anodyne, since $s$ and $\sigma$ are anodyne.
 If $H_0:f\leadsto g$ 
is a homotopy between two maps $f,g:A\to B$ 
and $H_1:g\leadsto h$ a homotopy 
between two maps $g,h:A\to B$,
then we have $\partial_1 H_0=g=\partial_0 H_1$.
Hence there is a unique map $H:A\to P_2B$ such that
  $q_0H=H_0$ and $q_1H=H_1$.
  This defines a homotopy $H:f\leadsto h$,
  since $\partial_0q_0 H=\partial_0 H_0=f$ and $\partial_1q_1 H=\partial_1 H_1=h$.
  \end{proof}

Recall that an equivalence relation $\equiv$ between the arrows of a category $\mathcal{C}$
is said to be a {\it congruence} if it is compatible with the operation of composition.

\begin{prop} \label{homotopycongruence} The homotopy relation $f\sim g$
between the maps of a tribe $\mathcal{E}$ is a congruence. 
\end{prop}

 \begin{proof} We saw in Lemma \ref{hoisequiv}  that the relation $f\sim g$ is an equivalence on 
 each hom set $\mathcal{E}(A,B)$.
Let us prove that if $f,g:A\to B$ and $f\sim g$,
then $fu\sim gu$ for every map $u:A'\to A$
and $vf\sim vg$ for every map $v:B\to B'$.
If $(PB,\partial_0,\partial_1,\sigma)$ is a path object for $B$ 
and $H:A\to PB$ is a homotopy $f\leadsto g$, then $Hu:A'\to PB$
is a homotopy $fu\leadsto gu$. 
If  $(PB', \partial'_0,\partial'_1,\sigma')$ is a path object
 for $B'$, then
 the following square commutes, 
  \begin{equation}\label{homotopycongruencesq}
 \xymatrix{
B \ar[d]_{\sigma} \ar[rr]^{\sigma'v} &&PB' \ar[d]^{(\partial'_0,\partial'_1)}  \\
PB \ar[rr]^{(v\partial_0,v\partial_1)} && B'\times B'
 }
 \end{equation}
since the following diagram commutes
$$ \xymatrix{
&PB  \ar[dr]^{(\partial_0,\partial_1)} & \\
B\ar[d]_v\ar[ur]^{\sigma }  \ar[rr]^-{(1_B,1_B)}  &&  B \times B \ar[d]^{v\times v} \\
B'\ar[rd]_{\sigma'}   \ar[rr]^-{(1_{B'},1_{B'})} && B'\times B' \\
&PB'  \ar[ur]_{(\partial'_0,\partial'_1)} &
 }$$
The square (\ref{homotopycongruencesq}) has a diagonal filler $w:PB\to PB'$, 
since $\sigma$ is anodyne and $(\partial'_0, \partial'_1)$ is a fibration.
If $H:A\to PB$ is a homotopy $f\leadsto g$, then $wH:A\to PB'$ is a homotopy $vf\leadsto vg$.
  \end{proof}

\medskip

We shall denote by $[f]$ the homotopy class of a map $f:X\to Y$ and 
by  $\pi_0\mathcal{E}(X,Y)$ the quotient of the set $\mathcal{E}(X,Y)$
by the homotopy relation.
The composite of two homotopy classes $[f]:X\to Y$ and $[g]:Y\to Z$
is defined by putting $[g][f]\defeq [gf]$.
This is the composition law of a category $Ho(\mathcal{E})$
having the same objects as $\mathcal{E}$
if we put $$Ho(\mathcal{E})(X,Y)\defeq \pi_0\mathcal{E}(X,Y)$$
We shall say that  $Ho(\mathcal{E})$
is the {\it homotopy category} of $\mathcal{E}$.

\medskip

\begin{def} \label{homotopyequiv} 
If $\mathcal{E}$ is a tribe, 
we say that a map 
$f:X\to Y$ in $\mathcal{E}$  is a {\bf homotopy equivalence},  
if the morphism $[f]:X\to Y$ is invertible in the homotopy category $Ho(\mathcal{E})$.
We say that a fibration $f:X\to Y$
 is {\bf trivial} if it is a homotopy equivalence.
We say that an object $A\in \mathcal{E}$ is 
{\bf contractible} if the map $A\to 1$ is a homotopy equivalence.
\end{def}

By definition, a map $f:X\to Y$   is a homotopy equivalence if and only if
there is a map $g:Y\to X$ such that $gf\sim 1_X$ and $fg\sim 1_Y$;
in which case $f$ and $g$ are mutually {\it homotopy inverse}.
The homotopy inverse of a map is unique up to homotopy.

\medskip

Recall that a class $\mathcal{W}$ of maps 
in a category is said to have the {\it 3-for-2 property}  if 
the following condition is satisfied: if two of three maps in a commutative triangle 
$$\xymatrix{
A \ar[r]^u \ar[dr]_{vu} & B\ar[d]^v \\
& C
}$$
belongs to $\mathcal{W}$, then so does the third.

\begin{prop} \label{6four2homotopyequiv} 
The class 
of homotopy equivalences in a tribe 
has the 3-for-2 property and it is closed under retracts.
\end{prop}

\begin{proof}  The class of isomorphism in a category
has the 3-for-2 property and is closed under retracts.
\end{proof}

\begin{prop} \label{anohomotopyequiv} An anodyne map in a tribe is a homotopy equivalence.
\end{prop}

\begin{proof} If $u:A\to B$ is anodyne, let us show that the morphism $[u]:A\to B$
is invertible in the homotopy category. 
There exists a map $r:B\to A$ such that $ru=1_A$,
since an anodyne map is a split monomorphism by \ref{charanodyne}.
Thus, $[r][u]=1_A$. Let us show that $[u][r]=1_B$.
If $(PB,\partial_0,\partial_1,\sigma)$ is a path object for $B$, then the following square commutes
$$ \xymatrix{
A\ar[d]_{u}\ar[rr]^{\sigma u}  && PB
\ar[d]^{(\partial_0,\partial_1)} \\
B\ar[rr]^(0.4){(ur,1_B) }&& B\times B
 }$$
since $(ur,1_B)u=(uru,u)=(u,u)=(1_B,1_B)u=(\partial_0,\partial_1)\sigma u$.
 Hence the square has a diagonal filler $H:B\to PB$, since $u$
 is anodyne and $(\partial_0,\partial_1)$ is a fibration.
The map $H:B\to PB$, is a homotopy $ur\leadsto 1_B$, since $(\partial_0,\partial_1) H=(ur,1_B)$.
Thus, $[u][r]=[ur]=[1_B]=1_B$.
  \end{proof}

Let us say that a functor $F:\mathcal{E} \to \mathcal{C}$  {\it inverts}
a map $f:A\to B$ in $\mathcal{E}$ if the map $F(f):FA\to FB$ is invertible.
If $\mathcal{W}$ is a class of maps in a category $\mathcal{E}$, 
then there is a category $ \mathcal{W}^{-1}\mathcal{E}$ equipped with a functor 
 $Q:\mathcal{E}\to  \mathcal{W}^{-1}\mathcal{E}$ which inverts universally the maps in $\mathcal{W}$.
The universality of $Q$ means that for every functor $F:\mathcal{E} \to \mathcal{C}$ which inverts the maps  in $\mathcal{W}$
there exists a unique functor $F': \mathcal{W}^{-1}\mathcal{E} \to \mathcal{C} $ such that $F=F'Q$,
$$\xymatrix{ 
\mathcal{E}\ar[rr]^Q \ar[drr]_F && \mathcal{W}^{-1}\mathcal{E} \ar[d]^{F'} \\
 &&\mathcal{E}.  \\
}$$

\begin{lemma} \label{inverthoequiv7} If $\mathcal{E}$ is a tribe
and $\mathcal{C}$ is a category, then the following conditions on a functor $F:\mathcal{E}\to \mathcal{C}$
are equivalent:
\begin{enumerate} 
\item{}  $F$ respects the homotopy relation: $f\sim g \Rightarrow F(f)=F(g)$;
\item{} $F$ inverts homotopy equivalences;
\item{} $F$ inverts anodyne maps.
\end{enumerate} 
\end{lemma} 

\begin{proof} 
($1\Rightarrow 2$)
The functor $F$ induces a functor $\tilde{F}:Ho(\mathcal{E}) \to \mathcal{C}$,
since it respects the homotopy relation. If a map $f:A\to B$ in $\mathcal{E}$ is a homotopy equivalence, then the
map $F(f)=\tilde{F}([f])$ is invertible,
since $[f]$ is invertible. 
$(2\Rightarrow 3)$ 
An anodyne map is a homotopy equivalence by \ref{anohomotopyequiv}.
($3\Rightarrow 1$)
If two maps $f,g:A\to B$ are homotopic, let us show that $F(f)=F(g)$.
Let  $(PB, \partial_0,\partial_1,\sigma)$ a path object for $B$ and let $H:A\to PB$
 a homotopy $f\leadsto g$.  
The image of the path object $PB$ by the functor $F$  is a commutative
diagram
$$ \xymatrix{
&&  \ar@/_1pc/[lld]_-{1_{FB}} B\ar[d]_{F(\sigma)}   \ar@/^1pc/[rrd]^-{1_{FB}} &&  \\
FB &&\ar[ll]_{F(\partial_0)}  F(PB) \ar[rr]^{F(\partial_1)} && FB
 }$$
 The map $F(\sigma)$ is invertible by the hypothesis on $F$, since $\sigma$ is anodyne.
Hence we have $F(\partial_0)=F(\partial_1)$, since we have $F(\partial_0)F(\sigma)=F(\partial_1)F(\sigma)$.
But we have $f=\partial_0 H$ and $g=\partial_1H$, since $H$ is a  homotopy $f\leadsto g$.  
Thus,
$$F(f)=F(\partial_0 H)=F(\partial_0)F(H)=F(\partial_1)F(H)=F(\partial_1H)=F(g).$$
\end{proof} 

\begin{prop} \label{inverthoequiv2} If $\mathcal{W}$ (resp. $\mathcal{A}$) is
 the class of homotopy equivalences (resp. anodyne maps) 
in a tribe $\mathcal{E}$, then we have 
 $$Ho(\mathcal{E})=\mathcal{W}^{-1}\mathcal{E}=\mathcal{A}^{-1}\mathcal{E}.$$
\end{prop} 

\begin{proof} 
The canonical functor $\mathcal{E}\to Ho(\mathcal{E})$ inverts universally
the maps in $ \mathcal{W}$ by Lemma \ref{inverthoequiv7}.
Thus, $Ho(\mathcal{E})=\mathcal{W}^{-1}\mathcal{E}$. 
Similarly, $Ho(\mathcal{E})=\mathcal{A}^{-1}\mathcal{E}$. 
\end{proof}

\begin{lemma} \label{homotakepathobject} 
If $F: \mathcal{E} \to  \mathcal{E}'$ is a morphism of tribes
and $(PB, \partial_0,\partial_1,\sigma)$ is a path object
for an object $B\in \mathcal{E}$, then $(F(PB), F(\partial_0), F(\partial_1),F(\sigma))$ 
is a path object for $FB$. 
\end{lemma}

\begin{proof} Left to the reader.
\end{proof}

\begin{prop} \label{h-homotohomotop} A morphism of tribes  $F: \mathcal{E} \to  \mathcal{E}'$
preserves the  homotopy
relation. It thus induces a functor
$$Ho(F):Ho(\mathcal{E})\to Ho(\mathcal{E}')$$
if we put $Ho(F)([f])=[F(f)]$ for a map $f:X\to Y$ in $\mathcal{E}$.
\end{prop}

\begin{proof}
The image by the functor $F$ of a path object  $PB=(PB, \partial_0,\partial_1,\sigma)$
for an object $B\in \mathcal{E}$ is a path object
$(F(PB), F(\partial_0), F(\partial_1),F(\sigma))$ for $FB$ by Lemma \ref{homotakepathobject}.
If $H:A\to PB$ is a homotopy between two maps $f,g:A\to B$,
then $F(H):FA\to FPB$  is a homotopy between the maps $F(f),F(g):FA\to FB$.
\end{proof}

\begin{cor} \label{inverthoequiv3} If $f:A\to B$ is a map in a tribe $ \mathcal{E}$, then
the base change functor $f^\star:\mathcal{E}(B)\to \mathcal{E}(A)$
induces a functor
$$Ho(f^\star):Ho(\mathcal{E}(B))\to Ho(\mathcal{E}(A)).$$
\end{cor}

\begin{proof} The functor $f^\star$ is a morphism of tribes by Theorem \ref{basechangehomotopical}. 
\end{proof}

If $F:  \mathcal{E}\to \mathcal{E}' $ and $G:  \mathcal{E}'\to \mathcal{E}'' $
are morphisms of tribes, then $Ho(GF)=Ho(G)Ho(F)$.
This defines a functor $Ho:\Trib \to \Cat$.
Moreover, a natural transformation $\alpha:F_1\to F_2$
between two morphisms of tribes $F_1,F_2: \mathcal{E}\to \mathcal{E}'$ 
induces a natural transformation $Ho(\alpha):Ho(F_1)\to Ho(F_2)$.

\begin{prop} \label{the2-functorHo} The functor $Ho:\Trib \to \Cat$
has the structure of a 2-functor.
\end{prop}

\begin{proof} Left to the reader.
\end{proof}

The {\it cartesian product} of a path object $PA=(PA,\partial_0^A,\partial_1^A, \sigma^A)$ for $A$
with a path object $PB=(PB,\partial_0^B,\partial_1^B, \sigma^B)$ for $B$ 
is defined by putting 
$$PA\times PB=(PA\times PB,\partial_0^A\times \partial_0^B ,\partial_1^A \times \partial_1^B, \sigma^A\times \sigma^B)$$

\begin{lemma}\label{productpathobject} 
The cartesian product of a path object $PA$ for $A$
with a path object $PB$ for $B$ is a path object for $A\times B$.
\end{lemma}

\begin{proof} 
The map $\sigma^A\times \sigma^B:A\times B \to PA\times PB$ is anodyne, since the cartesian product of two anodyne maps
is anodyne by lemma \ref{cartanodyne}.
The map
$$\langle \partial_0^A\times \partial_0^B, \partial_1^A\times \partial_1^B\rangle:PA\times PB \to A\times B \times A\times B$$
is isomorphic to the map
$$\langle \partial^0_A,\partial^1_A\rangle \times \langle \partial_0^B, \partial_1^B\rangle: PA\times PB \to A\times A \times B\times B.$$
It is thus a fibration, since the cartesian product of two fibrations is a fibration by 
 Proposition \ref{tribeofgenuineprojection0}.
 \end{proof}

\begin{lemma} \label{homotopyrelprod} 
If $(f,g):C\to A\times B$ and $(f',g'):C\to A\times B$, then
$$(f,g)\sim (f',g')\quad \Longleftrightarrow \quad f\sim f' \quad {\rm and }\quad g\sim g'$$
\end{lemma} 

\begin{proof} This follows from Lemma \ref{productpathobject}.
\end{proof}

\begin{prop}\label{hocatiscartesian} The homotopy category of a tribe $\mathcal{E}$ is cartesian and
the canonical functor $\mathcal{E}\to Ho(\mathcal{E})$ is cartesian.
\end{prop}

\begin{proof} This follows from Lemma \ref{homotopyrelprod}.
\end{proof}

\begin{cor}\label{carthomotopy} 
The cartesian product of two 
homotopy equivalences is a homotopy equivalences.
\end{cor}

\begin{proof} This follows from Proposition \ref{hocatiscartesian}.
\end{proof}

The cartesian product of two tribes $\mathcal{E}_1$ 
and $\mathcal{E}_2$ is a tribe $\mathcal{E}_1\times \mathcal{E}_2$.
By functoriality,  we obtain a canonical functor
\begin{equation}\label{homotopyofproduct}
Ho(\mathcal{E}_1\times \mathcal{E}_2)\to Ho(\mathcal{E}_1)\times Ho(\mathcal{E}_2).
\end{equation}
By definition, the functor takes a morphism 
$[(u_1,u_2)]:(A_1,A_2)\to (B_1,B_2)$ in $Ho(\mathcal{E}_1\times \mathcal{E}_2)$
to the morphism $([u_1],[u_2]):(A_1,A_2)\to (B_1,B_2)$ in  $Ho(\mathcal{E}_1)\times Ho(\mathcal{E}_2)$ .

\begin{prop} \label{Hoproduct} 
The functor (\ref{homotopyofproduct}) is an isomorphism of categories.
Hence the functor $Ho:\Trib \to \Cat$ preserves finite products.
\end{prop}

\begin{proof} Left to the reader.
\end{proof}

\subsection{Mapping path objects}

\begin{defi} \label{mappingpathobjectdef} 
The {\it mapping path object} of a map  $f:A\to B$ in a tribe $\mathcal{E}$ is a quadruple $(M(f),\delta_0,\delta_1,u)$
obtained by factoring the map $(1_A,f) :A\to A\times B$ as 
an anodyne map  $u:A \to M(f)$
followed by a fibration $(\delta_0,\delta_1) :M(f)\to A\times B$.
\end{defi}

The quadruple $(M(f),\delta_0,\delta_1,u)$ yields a commutative diagram
$$ \xymatrix{
&&  \ar@/_1pc/[lld]_-{1_{A}} A \ar[d]_{u}   \ar@/^1pc/[rrd]^-{f} &&  \\
A &&\ar[ll]_{\delta_0}  M(f) \ar[rr]^{\delta_1} && B
 }$$
  The maps $\delta_0$ and $\delta_1$ 
are fibrations, since  the projections $\mathsf{pr}_1:A\times B\to A$ and $\mathsf{pr}_2:A\times B\to B$ are fibrations
and we have $\delta_0=\mathsf{pr}_1( \delta_0,\delta_1)$ and
$\delta_1=\mathsf{pr}_2(\delta_0,\delta_1)$.
 The map $u$ is a homotopy equivalence by \ref{anohomotopyequiv},
since $u$ is anodyne.
It follows that  $\delta_0$ is a homotopy equivalence by 3-for-2, since $\delta_0u=1_A$.
Thus, $u$ is a section of a trivial fibration.

\medskip

The following lemma is essentially due to Ken Brown:

\begin{lemma} \label{KBrownlemma0forhtribes} 
Every map $f:A\to B$ in a tribe admits a factorization
$f=pu:A\to M\to B$, where $p$ is a fibration and where $u$ is anodyne and the section of a trivial fibration $M\to A$.
\end{lemma}

\begin{proof}  It suffices to take $p=\delta_1$ in
the mapping path factorization $(1_A,f)=(\delta_0,\delta_1)u:A\to M\to A\times B$.
\end{proof}

\medskip

Mapping path objects can be constructed from path objects.
If $PB=(PB, \partial_0,\partial_1, \sigma)$ is a path object for $B$
and $f:A\to B$, consider the following diagram with a pullback square
\begin{equation}\label{constmappingpath}
\xymatrix{
&& B \\
M(f)  \ar@/^1pc/[rru]^-{\partial_1 p_2}   \ar[d]_{p_1} \ar[rr]^{p_2}  && PB  \ar[d]^{\partial_0}  \ar[u]_{\partial_1} \\
A \ar[rr]^f && B 
}
\end{equation}
where $M(f)\defeq A\times_B PB$.
There is a unique map $u:A \to M(f)$, such that $p_1u=1_A$ and $p_2u=\sigma f$,
since $f1_A=f=\partial_0 \sigma f$.
 Let us put $\delta_0\defeq p_1$ and $\delta_1\defeq \partial_1p_2$.

\begin{prop} \label{mappingspaceconstruction} 
The quadruple $(M(f),\delta_0,\delta_1,u)$ constructed above is a mapping path object for $f:A\to B$.
\end{prop}

\begin{proof} We have $\delta_0u=p_1u=1_A$
and $\delta_1u=\partial_1p_2u=\partial_1 \sigma f=f$.
It follows that $\langle \delta_0, \delta_1\rangle u=\langle 1_A,f\rangle$.
The following diagram commutes, since $\partial_0 p_2=fp_1=f\delta_0$ and $\partial_1 p_2=\delta_1$.
\begin{equation}\label{cartsquareformappingpath}
\xymatrix{ M(f) \ar[d]_{\langle \delta_0,\delta_1\rangle} \ar[rr]^{p_2} && PB  \ar[d]^{\langle \partial_0,\partial_1\rangle} \\
 A\times B \ar[d]_{\mathsf{pr}_1} \ar[rr]^{f\times B} && B\times B \ar[d]^{\mathsf{pr}_1} \\
  A \ar[rr]^f && B
}
\end{equation}
The top square of the diagram is cartesian by Lemma \ref{lemmacartesiansq}, since  the composite square 
is cartesian by construction and the bottom square is cartesian,. Hence the map $\langle \delta_0,\delta_1\rangle$ is a fibration by base change, 
since the map $\langle \partial_0,\partial_1\rangle$ is  a fibration.
Let us show that the map $u:A\to M(f)$ is anodyne.
 If $f^\star:\mathcal{E}(B)\to \mathcal{E}(A)$
is the base change functor, then we have $f^\star(PB,\partial_0)= (M(f), \delta_0)$,
since the square in the diagram (\ref{constmappingpath}) is cartesian.
The map $\sigma:B\to PB$ defines a morphism $\sigma:(B,1_B)\to (PB,\partial_0)$ in $\mathcal{E}(B)$,
since $\partial_0 \sigma=1_B$.  Similarly, the map $u:A\to M(f)$ defines a morphism $u:(A,1_A)\to (M(f), \delta_0)$ in $\mathcal{E}(A)$,
since $\delta_0 u=1_A$. Moreover, we have $f^\star(\sigma)=u$,
since the following diagram commutes
$$ \xymatrix{
  M(f) \ar[rr]^{p_2}&& PB \\
   A \ar[rr]^{f} \ar[u]^-{u}    && B   \ar[u]_-{\sigma} 
}$$
Thus, $u$ is anodyne by  Theorem \ref{basechangehomotopical},  since $\sigma$ is anodyne. 
\end{proof}

The map $H:=p_2:M(f)\to PB$ in the diagram (\ref{constmappingpath}) above 
is a homotopy $H:f\delta_0\leadsto \delta_1$, since $\partial_0 H=\partial_0 p_2=p_1f=\delta_0 f$
and $\partial_1 H=\partial_1 p_2=\delta_1$.
The homotopy $H:f\delta^0\leadsto \delta^1$  is universal
 in the following sense:

\begin{lemma}\label{verygoodmappingpath}
For any object $C$, any pair of maps
$a:C\to A$, $b:C\to B$ and any homotopy $h:fa\leadsto b$
with values in $PB$, there exists a unique map $w:C\to M(f)$
such that $\delta_0w=a$, $\delta_1w=b$ and $Hw=h$,
\end{lemma}

\begin{proof} The square in the following diagram
is cartesian and we have 
 $fa=\partial_0h$.
 $$ \xymatrix{
C  \ar@/_1.5pc/[rdd]_-{a}    \ar@/^1.5pc/[drrr]^-{h}   \ar@{-->}[dr]^{w}&&& \\
&M(f) \ar[rr]^{H} \ar[d]^{\delta_0}&& \ar[d]^{\partial_0}   PB \\
& A\ar[rr]^{f}  && B
 }$$
Hence there is a unique map $w:C\to M(f)$ such that
$\delta_0w=a$ and $Hw=h$. Then we have
$\delta_1 w=\partial_1p_2 w=\partial_1H w=\partial_1 h=b$.
\end{proof}

For example if $C=A$ and $h:f1_A \leadsto f$ is the unit homotopy $\sigma f:f\leadsto f$,
then $w=u:A\to M(f)$.

\begin{lemma} {\rm (Straigthtening lemma)} \label{homotopyliffting} 
Let $p:E\to B$ a fibration in a tribe and suppose that the
following triangle commutes up to homotopy:
\begin{equation}\label{uphomotoptrian}
\xymatrix{
 && E  \ar[d]^{p}\\
A \ar[rr]^(0.6)f \ar[urr]^g && B.
}
\end{equation}
Then there exists a map $g':A\to E$
homotopic to $g$ such that $pg'=f$.
\end{lemma}

\begin{proof} If $PB=(PB, \partial_0,\partial_1,\sigma)$ is a path object for $B$,
then there exists a map $h:A\to PB$ such that $(\partial_0,\partial_1)h=(pg,f)$,
since $pg$ is homotopic to $f$.
Let $M(p)=(M(p),\delta_0, \delta_1,u)$ be the mapping path object for the map
$p:E\to B$ constructed from $PB$
in Proposition \ref{mappingspaceconstruction}.
The following square commutes, since $\delta_1 u=p$.
$$ \xymatrix{
E \ar[d]_u \ar@{=}[rr]&& E \ar[d]^{p} \\
M(p) \ar[rr]^(0.6){\delta_1}  && B
}$$
The square has a diagonal filler $d:M(p)\to E$, since $u$
is anodyne and $p$ is a fibration.
$$ \xymatrix{
E \ar[d]_u \ar@{=}[rr]&& E \ar[d]^{p} \\
M(p)\ar@{-->}[urr]^(0.45)d \ar[rr]^(0.6){\delta_1}  && B
}$$
By Lemma \ref{verygoodmappingpath}, 
there is a unique map $w:C\to M(p)$ such that
$(\delta_0,\delta_1)w=(g,f)$ and $Hw=h$.
Let us put $g'=d w:C\to E$. Then we have
$pg'=pd w=\delta_1 w=f$. Let us show that $g'\sim g$.
We have $d \sim \delta_0$, since $u$ is a homotopy equivalence by \ref{anohomotopyequiv}
and $d u=1_E=\delta_0u$.
Thus, $g'=d w\sim  \delta_0w=g $.
\end{proof}

\begin{prop} \label{homotopyrightinverse} A fibration in a tribe has a section
if and only if it has a section in the homotopy category.
\end{prop}

\begin{proof} This follows from Lemma \ref{homotopyliffting}.
\end{proof}

\begin{cor} \label{trivialfibrationhavesection} In a tribe, a trivial fibration 
has a section.
\end{cor}

\begin{proof} This follows from Proposition \ref{homotopyrightinverse},
since a trivial fibration is invertible in the homotopy category.
\end{proof}

\medskip

A map $i:A\to B$ is called a {\it split monomorphism}
if it admits a retraction $r:B\to A$.

\begin{defi} We say that a map $i:A\to B$ in a tribe
is a {\it deformation retract} if it is a split monomorphism and 
a homotopy equivalence.
\end{defi}

For example, an anodyne map is a deformation retract
by \ref{charanodyne}  and \ref{anohomotopyequiv}.

\begin{lemma} \label{defoormationretract} 
If $r:B\to A$ is any retraction of a deformation retract $i:A\to B$, then $ri=1_A$ and $ir\sim 1_B$.
\end{lemma} 

\begin{proof} We have $iri=i1_A=1_Bi$, since $ri=1_A$ . Thus, $ir\sim 1_B$, since the map $i$ is invertible in the homotopy category.
\end{proof}

\medskip

\begin{defi}\label{defstrongdefretract}
We shall say that a split monomorphism $i:A\to B$
is a {\it strong deformation retract} if it admits a retraction $r:B\to A$ together
with a homotopy $h:ir\leadsto 1_B$ 
such that  $hi=\sigma i$, where $\sigma$ is the unit 
of a path object  $(PB,\partial_0, \partial_1, \sigma)$ for $B$,
\end{defi}

\begin{rem}
It follows from proposition \ref{independancepathob}  that
the existence of a homotopy $h:ir\leadsto 1_B$
satisfying the condition $hi=\sigma i$ is independent of the 
choice of a  path object for $B$.
\end{rem}

\begin{prop} \label{strongdefoormation} 
A map in a tribe is anodyne
if and only if it is a strong deformation retract.
\end{prop}

\begin{proof} ($\Rightarrow$) Let us show that an anodyne map
$i:A\to B$ is a strong deformation retract. 
We shall use  the  mapping path object 
 $M(i)=(M(i),\delta_0, \delta_1,u)$ constructed from the  path object 
 $PB=(PB, \partial_0, \partial_1, \sigma)$ 
 in Lemma \ref{verygoodmappingpath}.
By construction, the following square is cartesian
 \begin{equation}\label{anothercartesiansqr}
\xymatrix{
M(i)\ar[d]_{\delta^0} \ar[rr]^{H} && PB  \ar[d]^{\partial_0}\\
A \ar[rr]^i && B,
}
\end{equation}
and $u\defeq (1_A, \sigma i)$ and $\delta_1\defeq \partial_1H$.
The map $H:M(i)\to PB$ is a homotopy $i\delta_0\leadsto \delta_1$
and we have $Hu=\sigma i$.
The following square has a diagonal filler $s:B\to M(i)$, since $i$ is anodyne and $\delta_1$ is a fibration.
$$ \xymatrix{
A \ar[rr]^(0.4)u\ar[d]_i  && M(i) \ar[d]^{\delta_1}  \\
B\ar@{-->}[urr]^s \ar[rr]_{1_B}  && B
}$$
The map $r\defeq \delta_0s:B\to A$ is then a retraction of the map  $i:A\to B$,
since $r i=\delta_0si=\delta_0u=1_A$.
The map $h\defeq Hs:B\to PB$ is a homotopy $ir \leadsto 1_B$,
since the map $H:M(i)\to PB$ is a homotopy $i\delta_0\leadsto \delta_1$
and we have $i\delta_0s=ir$ and  $\delta_1s= 1_B$.
Moreover,  $hi=Hsi=Hu= \sigma i$.
This shows that the map $i:A\to B$ is a strong deformation retract.
 ($\Leftarrow$)
Conversely, let us show that a strong deformation retract $i:A\to B$ is anodyne.
We shall first prove that $i$ is a retract of the map  $u:A\to M(i)$.
By hypothesis, there exists a retraction 
$r:B\to A$ together with a homotopy $h:ir \leadsto 1_B$ (with codomain $PB$)
such that $hi=\sigma i$.  
It then follows from lemma \ref{verygoodmappingpath} that there is a unique map $s:B\to M(i)$
such that $\delta_0s=r$, $\delta_1s=1_B$ and $Hs=h$,
$$ \xymatrix{
&&A \ar[rr]^i && B \\
B \ar@/^1pc/[rru]^-{r}  \ar@/_1pc/[rrrrd]_-{1_B} \ar@{-->}[rr]^s &&M(i)  \ar[rrd]_-{\delta_1}  \ar[u]^-{\delta_0}  \ar[rr]^H  && PB \ar[u]_{\partial_0}\ar[d]^{\partial_1}& \\
&& && B
 }$$
It follows that $si=u$, since the square (\ref{anothercartesiansqr}) is cartesian
and $\delta_0si=ri=1_A=\delta_0u$ and $Hsi=hi=\sigma i =Hu$.
It then follows from the relations $\delta_1s=1_B$ and $si=u$ that the map $i:A\to B$ is a codomain  retract of 
 the map $u:A\to M(i)$, 
  $$ \xymatrix{
&\ar[dl]_i  A  \ar[dr]^{u} & \\
B \ar[rr]^(0.45){s}  && M(i)  \ar@/^1pc/[ll]^-{\delta_1}.
}$$
Thus, $i$ is anodyne by Lemma \ref{anodyne closurecompositionret}, since $u$ is anodyne. 
\end{proof}

\subsection{Fibrewise homotopy}

\medskip

The {\it fibrewise diagonal} of a fibration $p:E\to B$  in a tribe  $\mathcal{E}$ is the diagonal
$E\to E\times_B E$ of the object $(E,p)$ of $ \mathcal{E}(B)$.
A {\bf fibrewise path object} for $p:E\to B$ is a path object for $(E,p)$.
By definition, it is a quadruple $(P(E,p), \partial_0, \partial_1, \sigma)$
obtained by choosing an $AF$-factorisation $(\partial_0,\partial_1)\sigma:E\to P(E,p) \to E\times_B E$
 of the fibrewise diagonal $E\to E\times_B E$.
 $$ \xymatrix{
E \ar[rr]^-{(1_E,1_E)} \ar[dr]_\sigma && E\times_B E \\
& P(E,p)  \ar[ur]_{(\partial_0,\partial_1)}& 
 }$$
The quadruple $(P(E,p), \partial_0, \partial_1, \sigma)$ defines a commutative diagram in the category $ \mathcal{E}$.
$$ \xymatrix{
&  \ar@/_1pc/[ld]_-{1_E} E\ar[d]_\sigma   \ar@/^1pc/[rd]^-{1_E} &  \\
E\ar[dr]_p &\ar[l]_{\partial_0}  P(E,p) \ar[r]^{\partial_1} & E\ar[dl]^p \\
&B&
 }$$
A {\bf fibrewise homotopy} $h:f\leadsto_B g$ between two maps 
$f,g:(A,u)\to (E,p)$ in $\mathcal{E}(B)$ is a map $h:A\to P(E,p)$
such that $\partial_0h=f$ and $\partial_1h=g$. 
$$ \xymatrix{
&& E \\
A  \ar@/_1pc/[rrd]_-{g}  \ar@/^1pc/[rru]^-{f}  \ar[rr]^h  && P(E,p) \ar[u]^{\partial_0}\ar[d]_{\partial_1}\\
 &&  E 
 }$$
 The maps $f,g:(A,u)\to (E,p)$ are {\bf fibrewise homotopic}, $f\sim_B g$, if there exists a fibrewise homotopy $h:f\leadsto_B g$.

\begin{lemma} \label{fibrewisehomotopicishomotopic} 
If two maps $f,g:(A,u)\to (E,p)$ in $\mathcal{E}(B)$ are fibrewise homotopic
then the underlying maps $f,g:A\to E$ in $\mathcal{E}$ are homotopic.
\end{lemma}

\begin{proof} Let $(PE, \partial_0, \partial_1, \sigma\rangle$
be a path object for $E\in \mathcal{E}$ and $(P(E,p), \partial'_0, \partial'_1, \sigma'\rangle$
be a path object of $(E,p)\in \mathcal{E}(B)$.
Then the following square has a diagonal filler $d:PB\to P'B$, 
 since $\sigma'$ is anodyne and $\langle\partial_0,\partial_1\rangle$ is a fibration.
 \begin{equation}
 \xymatrix{
B \ar[d]_{\sigma'} \ar[rr]^{\sigma} &&PE \ar[d]^{\langle\partial_0,\partial_1\rangle}  \\
P'(E,p) \ar[rr]^{\langle\partial'_0,\partial'_1\rangle } && E\times E
 }
 \end{equation}
If $\langle\partial'_0,\partial'_1\rangle H=\langle f,g\rangle$ then 
$\langle\partial_0,\partial_1\rangle dH=\langle\partial'_0,\partial'_1\rangle H=\langle f,g\rangle.$
 \end{proof}

\begin{prop} \label{inverthoequiv4} 
If $f:A\to B$ is a fibration in a tribe $\mathcal{E}$,
then the functor $\Sigma_f:\mathcal{E}(A) \to \mathcal{E}(B)$
induces a functor
$$Ho(\Sigma_f):Ho(\mathcal{E}(A))\to Ho(\mathcal{E}(B))$$
left adjoint to the functor $Ho(f^\star)$.
\end{prop}

\begin{proof} We may suppose that $B=1$ by Lemma \ref{compositionelementary1},
in which case $f^\star=e_A$ and $\Sigma_f$
is the forgetful functor $\Sigma_A: \mathcal{E}(A) \to \mathcal{E}$.
But the forgetful functor $\Sigma_A$ preserves the homotopy relation
by Lemma \ref{fibrewisehomotopicishomotopic}.
 It thus induces a functor
$Ho(\Sigma_A): Ho(\mathcal{E}(A)) \to Ho( \mathcal{E})$.
It is then easy to verify that the adjunction $\Sigma_A \dashv e_A$
induces an adjunction $Ho(\Sigma_A)\dashv Ho(e_A)$.
\end{proof}

 \medskip

If $B$ is an object of a tribe $\mathcal{E}$, then
a {\bf fibrewise mapping path object} of a map $f:(X,p)\to (Y,q)$ in 
$\mathcal{E}(B)$ is obtained by factoring the map $(1_X,f) :X\to X\times_B Y$ as 
an anodyne map  $u:X \to P_B(f)$
followed by a fibration $\langle \delta_0,\delta_1\rangle :P_B(f)\to X\times_B Y$.
From the quadruple $(P_B(f),\delta_0,\delta_1,u)$ we obtain the following commutative diagram in the category $\mathcal{E}$,
$$ \xymatrix{
&&  \ar@/_1pc/[lld]_-{1_{X}} X \ar[d]_{u}   \ar@/^1pc/[rrd]^-{f} &&  \\
X \ar[drr]_p &&\ar[ll]_{\delta_0}  P_B(f) \ar[rr]^{\delta_1} && Y \ar[lld]^{q}\\
&& B &&
 }$$

\begin{prop} \label{homotopyequiv}  
A fibration $f:X\to B$ in a tribe $\mathcal{E}$ is trivial if and only if the object $(X,f)\in  \mathcal{E}(B)$ is contractible.
 \end{prop}

\begin{proof} ($\Rightarrow$) 
Let $(PX,\partial_0,\partial_1,\sigma)$ be a path object for $X$.
The map $f$ is invertible in the homotopy category $Ho(\mathcal{E})$,
since it is a homotopy equivalence.
Hence the map $f:X\to B$ has a section $s:B\to X$
by \ref{trivialfibrationhavesection}, since it is a fibration.
The map $s$ is then the inverse of $f$
in the category $Ho(\mathcal{E})$, since $f$ is invertible in $Ho(\mathcal{E})$
and $fs=1_B$. Thus, $sf \sim 1_X$.
It follows that there is a map
$h:X\to PX$ such that $\partial_0 h=sf $ and $\partial_1 h=1_X$.
Let $(M(s),\delta_0,\delta_1,u)$ be mapping path object of the map $s:B\to X$ constructed 
from the path object $(PX,\partial_0,\partial_1,\sigma)$ and
 the pullback square
 $$ \xymatrix{
M(s) \ar[rr]^H \ar[d]_{\delta_0}&& PX \ar[d]^{\partial_0} \\
B\ar[rr]^{s}  && X
}$$
By construction,  $\delta_1=\partial_1 H $ and $u=\langle 1_B, \sigma s \rangle$.
There is then a unique map 
$g:X\to M(s) $ such that the following diagram commutes, since $sf=\partial_0 h$.
$$ \xymatrix{
X  \ar@/^1.5pc/[drrr]^-{h}    \ar@/_1.5pc/[ddr]_-{f}  \ar[dr]^{g}&&& \\
&M(s) \ar[rr]^H \ar[d]_{\delta_0}&& PX \ar[d]^{\partial_0}  \\
& B \ar[rr]^{s}  &&X.
}$$
We have $\delta_1g=\partial_1 Hg= \partial_1 h =1_X$.
Thus, $g$ is a section of $\delta_1$.
The map $s:(B,1_B)\to (X,f)$
 in $\mathcal{E}(B)$ has a  fibrewise mapping path object 
 $(P_B(s), \delta'_0,\delta'_1, u')$ constructed from a fibrewise path
 object $(P_B(X,f),  \partial'_0,\partial'_1,\sigma')$ for $(X,f)$.
 By construction, we have the following commutative diagram,
    $$ \xymatrix{
P_B(s) \ar[rr]^{H'} \ar[d]_{\delta_0'}  \ar@/^2.5pc/[rrrr]^-{\delta'_1}   && P_B(X, f) \ar[d]^{\partial'_0}   \ar[rr]^{\partial'_1} && X\ar[d]^f  \\
B \ar[rr]^{s}  && X\ar[rr]^f && B
}$$
The following  square commutes, since $\delta_1u=s=\delta'_1u'$.
$$ \xymatrix{
B \ar[d]_{u} \ar[r]^{u'}& \ar[d]^{\delta_1'} P_B(s)  \\
M(s)\ar[r]^{\delta_1}  & X
}$$
Hence the square has a diagonal filler $d:M(s)\to P_B(s)$,
since $u$ is anodyne and $\delta'_1$ is a fibration.
Let us show that the composite $h'\defeq H'dg:X\to P_B(X)$
is a fibrewise homotopy $sf\leadsto_B 1_X$.
Notice that 
$$f\delta'_1=f\partial'_1H'=f\partial'_0H'=fs\delta'_0=\delta'_0.$$
Thus 
$$\partial'_0h'=\partial'_0 H' dg=s\delta'_0 d g=sf\delta'_1 dg=sf \delta_1 g=sf 1_X=sf.$$
Moreover, 
$$\partial'_1h'=\partial'_1 H'dg=\delta'_1 dg=\delta_1g=1_X.$$
Thus, $sf\sim_B 1_X$ and this shows that the object $(X,f)$ is contractible in the category $\mathcal{E}(B)$.
($\Leftarrow$) The map $f:(X,f)\to (B,1_B)$ is a homotopy equivalence in $\mathcal{E}(B)$,
since the object $(X,f)$ is contractible. But the forgetful functor $\Sigma_B: \mathcal{E}(B)\to \mathcal{E}$
preserves homotopy equivalences by Proposition \ref{inverthoequiv4}.
This shows that the map $f:X\to B$ is a homotopy equivalence in $\mathcal{E}$.
\end{proof}

\begin{thm} \label{homotopyequirefl} If $f:A\to B$
is a fibration in a tribe $ \mathcal{E}$, then the functor 
$\Sigma_f:\mathcal{E}(A)\to \mathcal{E}(B)$ preserves and reflects homotopy
equivalences. 
\end{thm}

\begin{proof} 
The functor $\Sigma_f$ preserves homotopy
equivalences by Proposition \ref{inverthoequiv4}.
Let us show that it reflects homotopy equivalences.
We may suppose that $B=1$ by Lemma \ref{compositionelementary1},
in which case $f^\star=e_A$ and $\Sigma_f$
is the forgetful functor $\Sigma_A: \mathcal{E}(A) \to \mathcal{E}$.
Let $w:(X,p)\to (Y,q)$ be a morphism in $ \mathcal{E}(A)$;
if the map $w:X\to Y$ is a homotopy equivalence,
let us show that the morphism $w:(X,p)\to (Y,q)$ is a homotopy equivalence in $ \mathcal{E}(A)$.
For this, let us choose an $AF$-factorization $w=gu:X\to E\to  Y$ in $ \mathcal{E}$.
The map $u:X\to E$ is a homotopy equivalence  by \ref{anohomotopyequiv},
since $u$ is anodyne. Hence the map $g:E\to Y$ is a homotopy equivalence 
by 3-for-2, since $w=gu$ is a homotopy equivalence by hypothesis.
Hence the object $(E,g)\in \mathcal{E}(Y)$
is contractible by Proposition \ref{homotopyequiv}.
The map $g:(E,qg)\to (Y,q)$ is a fibration in $\mathcal{E}(A)$ by Proposition \ref{clanslice}
since $g$ is a fibration. 
$$ \xymatrix{
E \ar[dr]_{qg} \ar[r]^{g}& \ar[d]^{q} Y \\
 & A
}$$
But we have $\mathcal{E}(A)(Y,q)=\mathcal{E}(Y)$
and $((E,qg),g)=(E,g)$ by Lemma \ref{compositebasechangetribecor2}.
Hence the object $((E,qg),g)\in \mathcal{E}(A)(Y,q)$
is contractible, since the object $(E,g)\in \mathcal{E}(Y)$ is contractible.
It then follows by Proposition \ref{homotopyequiv}  
that the morphism  $g:(E,qg)\to (Y,q)$ is a homotopy equivalence in $\mathcal{E}(A)$.
But the morphism $u:(X,p)\to (E,qg)$ is anodyne in $ \mathcal{E}(A)$ by
Proposition \ref{tribeslice}, since $u$ is anodyne. Hence
the morphism $u:(X,p)\to (E,qg)$ is a  homotopy equivalence in $ \mathcal{E}(A)$
by \ref{anohomotopyequiv}.
It follows that that the composite $w=gu:(X,p)\to (E,qg)\to (Y,q)$
is a homotopy equivalence in $\mathcal{E}(A)$.
\end{proof}

\begin{prop} \label{trivialfibrationbasechange} In a tribe, the base change of a trivial fibration
along any map is a trivial fibration.  
\end{prop}

 \begin{proof} Let  $p:E\to B$ a trivial fibration in a tribe $ \mathcal{E}$;
 if $f:A\to B$ is a map in $ \mathcal{E}$, let us show that the projection $p_1$
 in the following pullback square is a trivial fibration
 $$ \xymatrix{
A\times_BC  \ar[d]_{p_1} \ar[r]& \ar[d]^{p} C  \\
A\ar[r]^{f}  & B.
}$$
The object $(E,p)$ is contractible in $\mathcal{E}(B)$ by Proposition \ref{homotopyequiv}.
Hence the object $(A\times_BC,p_1)=f^\star(C,p)$ is contractible in $\mathcal{E}(A)$,
 since the base change functor $f^\star: \mathcal{E}(B) \to \mathcal{E}(A)$
preserves homotopy equivalences by \ref{inverthoequiv3}.
It follows  by Proposition \ref{homotopyequiv} that the map $p_1:A\times_BX \to A$
is a homotopy equivalence in $\mathcal{E}$.
 \end{proof}

The following result uses the notion of  {\it fibration category} defined
in appendix \ref{Appendix on fibration categories}.

\begin{thm} \label{atribeisafibrationcategory} A tribe has the structure of a Brown fibration category in which a weak equivalence
is a homotopy equivalence.
\end{thm} 

 \begin{proof} This follows directly from proposition \ref{trivialfibrationbasechange}.
 \end{proof}

\begin{cor} \label{basechangehomotopyequiv} In a tribe, the base change of a homotopy equivalence along a fibration 
is a homotopy equivalence. 
\end{cor}

 \begin{proof}
This follows from Theorem \ref{atribeisafibrationcategory}, since a tribe is a Brown fibration category by Proposition \ref{atribeisafibrationcategory}.
 \end{proof}

\subsection{Weak equivalences of tribes}

We shall denote  the category of tribes and morphism of tribes by $\mathbf{Trib}$.
The category $\mathbf{Trib}$ has also the structure of a 2-category in which a 
2-cell is a natural transformation.
An {\bf equivalence of tribes} is an equivalence in this 2-category.
A morphism of tribes $F:\mathcal{E} \to \mathcal{E}'$
is an equivalence of tribes if and only if it is an
equivalence of clans.

\begin{defi} \label{weakequivaltribedef}  We say that a morphism of tribes
$F:\mathcal{E} \to \mathcal{E}'$ is a {\bf  weak equivalence}
if the induced functor $Ho(F): Ho(\mathcal{E}) \to Ho(\mathcal{E}')$ is 
an equivalence of categories.
\end{defi}

\begin{defi}
We shall say that a tribe $\mathcal{E}$ is {\bf contractible} if 
if every object of $\mathcal{E}$ is contractible.
\end{defi} 

\begin{prop} \label{caraterisationcontracttribe} A tribe $\mathcal{E}$ is contractible
if and only if the canonical functor $\mathcal{E}\to \mathbf{1}$ is a weak equivalence.
\end{prop} 

\begin{proof} Left to the reader.
\end{proof}

\begin{prop} \label{fullyfaithfulhomo} If a morphism of tribes $F: \mathcal{E} \to  \mathcal{E}'$ 
is fully faithful, then so is the functor $Ho(F): Ho(\mathcal{E}) \to Ho(\mathcal{E}')$.
\end{prop} 

\begin{proof} The map $\pi_0\mathcal{E}(A,B)\to \pi_0\mathcal{E}(FA,FB)$
induced by $F$ is surjective for every objects $A,B\in  \mathcal{E} $,
since the functor $F$ is full. Let us show that it is injective.
Let $f,g:A\to B$ and suppose that the maps $F(f),F(g):FA\to FB$
are homotopic. If $(PB,\partial_0,\partial_1,\sigma)$
is a path object for $B$, then $(FPB,F(\partial_0), F(\partial_1),F(\sigma))$
is a path object for $FB$ by Lemma \ref{homotakepathobject}.
Hence there is a map $H:A\to FPB$ such that $F(\partial_0)H=F(f)$
and $F(\partial_1)H=F(g)$, since $F(f)$ is homotopic to $F(f)$.
But we have $H=F(h)$ for a map $h:A\to FPB$, since
$F$ is full. Moreover, we have $\partial_0 h=f$, since the functor $F$ is faithful
and we have $F(\partial_0 h)=F(\partial_0) F(h)=F(\partial_0)H=F(f)$,
Similarly, $\partial_1 h=g$.
 Thus, $f$ is homotopic to $g$. This shows that the functor $Ho(F)$
 is faithful. \end{proof}

 \begin{defi} \label{defembeddingtribe} 
 We say that a morphism of tribes $F:\mathcal{E}\to \mathcal{E}'$
 is an {\bf embedding} if it is fully faithful
 and it reflects fibrations and anodyne maps.
 \end{defi}

\begin{lemma} \label{embeddingtribes} If a morphism of tribes  $F:\mathcal{E}\to \mathcal{E}'$ is
 fully faithful and reflects fibrations, then it is an embedding.
\end{lemma}

\begin{proof} This follows from Proposition \ref{reflectionofanodynemorphclan}.
\end{proof}

\medskip

 \begin{prop} \label{embeddingtribeclan} 
 Let  $F:\mathcal{E}\to \mathcal{E}'$ be an embedding of clans.
 Suppose that $ \mathcal{E}'$ is a tribe
 and that every map $f:X\to Y$ in $\mathcal{E}$  admits a factorization
 $f=pu:X\to E\to Y$ where $F(u)$ is anodyne in $\mathcal{E}'$
 and  $F(p)$ is a fibration  $\mathcal{E}'$.
 Then $\mathcal{E}$ is a tribe and $F$ is an embedding of tribes.
  \end{prop}

\begin{proof} The map $p:E\to Y$ is a fibration in $\mathcal{E}$ since the functor $F$ is an embedding of clans. 
Moreover, the map $u:X\to E$ is anodyne in $\mathcal{E}$, since the functor $F$
reflects anodyne maps by Proposition \ref{reflectionofanodynemorphclan}.
The class of anodyne maps in $\mathcal{E}$ is closed under base changes
along fibrations in  $\mathcal{E}$ by a similar argument,
\end{proof}

Recall from Proposition \ref{subclan=embedding} 
that a full sub-category $\mathcal{L}$  of a clan $\mathcal{E}$
is a {\it sub-clan} if is a clan 
and the inclusion functor $\mathcal{L}\to \mathcal{E}$
is an embedding of clans.

 \begin{defi} \label{defsubtribe} 
 If $\mathcal{E}$ is a tribe, we say that a sub-clan $\mathcal{L}$ 
is a {\bf sub-tribe} if every map $f:A\to B$ in $\mathcal{L}$ 
admits an $AF$-factorisation $f=pu:A\to E\to B$ in $\mathcal{E}$, with $E\in \mathcal{L}$.
 \end{defi} 

A sub-tribe $\mathcal{L}\subseteq \mathcal{E}$ has the structure of a tribe
and the inclusion functor $\mathcal{L}\to \mathcal{E}$ 
is an embedding if tribes. Conversely, if $F:\mathcal{E}\to \mathcal{E}'$ is an embedding of tribes,
then the (essential) image of the functor $F$ is a sub-tribe $F(\mathcal{E})\subseteq \mathcal{E}'$ 
and the functor $\mathcal{E} \to F(\mathcal{E})$ induced by $F$
is an equivalence of tribes.

\medskip

If $\top$ is a terminal object of a tribe $\mathcal{E}$, 
then the (full) sub-category of $\mathcal{E}$ spanned by the singleton $\{\top\}$ is a sub-tribe.
A full sub-category $\mathcal{L}$ of a tribe $\mathcal{E}$
is a sub-tribe if and only its replete closure $\mathcal{L}^{rep} \subseteq \mathcal{E}$
is a sub-tribe (see Proposition \ref{subclanvsreplete}).

\begin{defi}\label{homotopicallyrepletedef} If $\mathcal{E}$ is a tribe,
we say that a full subcategory $\mathcal{L}\subseteq \mathcal{E}$
is {\bf  homotopically replete} (in short, {\bf h-replete}) if every object of $\mathcal{E}$ which is homotopically equivalent 
to an object of $\mathcal{L}$ belongs to $\mathcal{L}$.
\end{defi}

Every full subcategory $\mathcal{L}\subseteq \mathcal{E}$ is contained
in a smallest h-replete (full) sub-category  $\mathcal{L}^{hrep}$,
the {\it homotopically replete closure} of $\mathcal{L}$. 
By construction, an object $X\in \mathcal{E}$ belongs
to  $\mathcal{L}^{hrep}$ if and only if $X$ is homotopically equivalent to an object in $\mathcal{L}$.

\begin{prop}\label{repleteclosureofsubtribe} 
If $\mathcal{L}$ is a sub-tribe of a tribe $\mathcal{E}$, then so is the subcategory $\mathcal{L}^{hrep}$ 
and the inclusion functor $\mathcal{L}\subseteq \mathcal{L}^{hrep}$
is a weak equivalence.
\end{prop} 

 \begin{proof} Without loss of generality, we may suppose that the sub-tribe $\mathcal{L}$ is replete.
 Let us first show that for any map $f:X\to B$ in $\mathcal{L}^{hrep}$ 
 there exists a homotopy cartesian square
\begin{equation} \label{ahopulback1}
\xymatrix{
X \ar[d]_f  \ar[rr]^k&&  Y \ar[d]^{p}  \\
B \ar[rr]^j&&  C 
}  \end{equation} 
with $p$ a fibration in $\mathcal{L}$.  
There exists an object $E\in  \mathcal{L}$ together
with a homotopy equivalence $i:E\to X$, since $X\in \mathcal{L}^{hrep}$.
There is also an object $C\in  \mathcal{L}$ together
with a homotopy equivalence $j:B\to C$, since $B\in \mathcal{L}^{hrep}$.
\begin{equation} \label{ahopulback0}
\xymatrix{
X \ar[d]_f  &&    \ar[ll]_i E  \\
B \ar[rr]^j&&  C 
}  \end{equation} 
The map $jfi:E\to C$ belongs to $ \mathcal{L}$, since $\mathcal{L}$ is a full subcategory.
 It thus admits an $AF$-factorisation $pu:E\to Y\to C$ with $Y\in \mathcal{L}$,
 since $\mathcal{L}$ is a sub-tribe. Let us choose an $AF$-factorisation $i=qv:E\to D\to X$.
  $$\xymatrix{
X \ar[dd]_f  && \ar[ll]_{i} E \ar[dr]^{u}  \ar[dl]_v  & \\
 &D \ar[ul]_q   && Y \ar[ld]^{p}  \\
B \ar[rr]^{j} && C & 
}  $$ 
There is a map $d:D\to Y$ such that $pd=jfq$ and $dv=u$, 
 since $v$ is anodyne and $p$ is a fibration.
  $$\xymatrix{
X \ar[dd]_f  && \ar[ll]_{i} E \ar[dr]^{u}  \ar[dl]_v  & \\
 & D \ar[ul]_q \ar@{-->}[rr]^d     \ar[dl]_(0.4){fq}   \ar[dr]_(0.4){jfq}  && Y \ar[ld]^{p}  \\
B \ar[rr]^(0.4){j} && C  & 
}  $$ 
 The maps $u$ and $v$ are homotopy equivalences, since an anodyne map
 is a homotopy equivalence by Proposition \ref{anohomotopyequiv}.
Thus, $d$ is a homotopy equivalence by 3-for-2, since $dv=u$.
The map $q$ is also a homotopy equivalence by 3-for-2, since $qv=i$
is a homotopy equivalence. Hence the fibration $q$ admits a section $s:X\to C$
by Proposition \ref{trivialfibrationhavesection}.
The following square commutes, since $pds=jfqs=jf1_X=jf$.
\begin{equation} \label{ahopulback}
\xymatrix{
X \ar[d]_f  \ar[rr]^{ds} &&  Y \ar[d]^{p}  \\
B \ar[rr]^{j}&&  C 
} 
\end{equation} 
The map $s$ is a homotopy equivalence by 3-for-2, since $q$ is a homotopy equivalence and $qs=1_X$.
Thus, $ds:X\to Y$ is a homotopy equivalence, since $d$ is a homotopy equivalence.
It then follows from Lemma \ref{homotopycartesiancriterion1} that the square (\ref{ahopulback}) is homotopy cartesian.
We can now prove that the base change of a fibration  $f:X\to B$ in $\mathcal{L}^{hrep}$
along any map $g:A\to B$ in $\mathcal{L}^{hrep}$ belongs to $\mathcal{L}^{hrep}$.
There exists a homotopy cartesian square (\ref{ahopulback}),
with $p$ a fibration in $\mathcal{L}$ by what we just proved.
There exists a homotopy equivalence $w:A'\to A$ with $A'\in \mathcal{L}$, since $A\in \mathcal{L}^{hrep}$.
The first two squares of the following diagram are homotopy cartesian,
since $f:X\to Y$ is a fibration and the two squares
are cartesian. 
$$\xymatrix{
A'\times_B X  \ar[d]  \ar[rr]^{w\times_B X}&& A\times_B X\ar[d]_{p_1}  \ar[rr] && X \ar[d]_f  \ar[rr]&&  Y \ar[d]^{p}  \\
A '\ar[rr]^w && A \ar[rr]^g && B \ar[rr]^j &&  C 
}  $$
Hence the composite square is homotopy cartesian by Lemma \ref{lemmahocartesiansqcriterion2}.
It follows that the induced map $A'\times_B X\to (jgw)^\star(Y)$ is a homotopy
equivalence, since $p$ is a fibration.
But the object $(jgw)^\star(Y)$ belongs to $\mathcal{L}$, since the objects $A'$, $C$ and $Y$ belongs to $ \mathcal{L}$,
and $\mathcal{L}$ is a replete sub-tribe of $ \mathcal{E}$.
Thus, $A'\times_B X\in \mathcal{L}^{hrep}$, since the map $A'\times_B X\to (jgw)^\star(Y)$ is a homotopy
equivalence.
The map $p_1:A\times_B X\to A$ is a fibration by base change, since $f$ is a fibration
and the middle square is cartesian.
The map $w\times_B X$ is the base change of the map $w:A'\to A$
along $p_1$, since the first square is cartesian.
Hence the map $w\times_B X$ is a homotopy equivalence
by Corollary \ref{basechangehomotopyequiv}, since $w$ is a homotopy equivalence
and $p_1$ is a fibration.
Thus, $A\times_B X\in \mathcal{L}^{hrep}$, since $A'\times_B X\in \mathcal{L}^{hrep}$.
Hence the projection $p_1:A\times_B X\to A$ belongs to $ \mathcal{L}^{hrep}$, since the subcategory
$ \mathcal{L}^{hrep}$ is full.
We have proved that the base change of $f:X\to B$
along $g:A\to B$ belongs to $\mathcal{L}^{hrep}$.
The subcategory $\mathcal{L}^{hr}$ contains a terminal object $1 \in \mathcal{E}$,
since $\mathcal{L}\subseteq \mathcal{L}^{hrep}$ and $1 \in \mathcal{L}$.
We have proved that $\mathcal{L}^{hrep}$ is a sub-clan of $\mathcal{L}^{hrep}$.
Let us now show that every map $f:A\to B$ in  $\mathcal{L}^{hrep}$ admits an $AF$-factorisation in  $\mathcal{L}^{hrep}$.
The map $f:A\to B$ admits an $AF$-factorisation $f=pu:A\to E\to B$ in  $ \mathcal{E}$, since  $ \mathcal{E}$ is a tribe.
But the map $u:A\to E$ is a homotopy equivalence, 
since an anodyne map
 is a homotopy equivalence by Proposition \ref{anohomotopyequiv}.
 Thus, $E\in \mathcal{L}^{hrep}$, since $A\in \mathcal{L}^{hrep}$. 
This shows that $\mathcal{L}^{hrep}$ is a sub-tribe.
Let us show that the inclusion functor $J:\mathcal{L} \subseteq \mathcal{L}^{hrep}$
is a weak equivalence. 
The functor $Ho(J)$ is fully faithful by Proposition \ref{fullyfaithfulhomo},
since $J$ is fully faithful. 
Let us show that the functor $Ho(J)$ is essentially surjective.
The inclusion functor $K:\mathcal{L}^{hrep}\to  \mathcal{E}$
is a morphism of tribes. The
functor $Ho(K)$ is fully faithful by Proposition \ref{fullyfaithfulhomo},
since $K$ is fully faithful.
Thus, $Ho(K)$ is conservative, since a fully faithful functor
is conservative.
For every object $X\in \mathcal{L}^{hrep}$
there exists an object $Y\in \mathcal{L}$ 
together with a homotopy equivalence $u:X\to Y$.
But the map $u:X\to Y$ is a homotopy equivalence in $ \mathcal{L}^{hrep}$,
since the functor $Ho(K)$ is conservative.
It follows that the functor $Ho(J)$ is essentially surjective,
since $u:X\to J(Y)$ is an equivalence in $ \mathcal{L}^{hrep}$.
 \end{proof}

\begin{cor}\label{propcontratibleobjecttribe} 
If $\mathcal{E}$ is a tribe, then the full subcategory $\mathcal{E}^c$ of contractible objects of $\mathcal{E}$
is a sub-tribe. It is the largest contractible sub-tribe of $\mathcal{E}$.
\end{cor}

 \begin{proof} Left to the reader.
 \end{proof}

\bigskip

If $p:E\to B$ is a fibration  in a tribe $ \mathcal{E}$, we shall denote by
$\Gamma(E,p)$ the set of sections of $p$.
A morphism of tribes $F:\mathcal{E} \to \mathcal{E}'$
induces a map $\Gamma(E,p)\to \Gamma(FE,F(p))$.
We say that $F$ is {\bf full  on the sections} of $p:E\to B$
if the map $\Gamma(E,p)\to \Gamma(FE,F(p))$ induced by $F$ is surjective.

\medskip

 \begin{defi} \label{generoushomomorphismdefinition} We say
 that a morphism of tribes $F: \mathcal{E}\to \mathcal{E}'$
 is {\bf generous} if it satisfies the following two conditions:
 \begin{itemize}
 \item{} The functor $F$ is full on sections of every fibration $p:E\to B$ in $ \mathcal{E}$.
  \item{} For every object $Y\in  \mathcal{E}'$ there exists
  an object $X\in  \mathcal{E}$ together with an anodyne map
 $Y\to FX$.
  \end{itemize}
  \end{defi}

\medskip

If $\mathcal{E}$ is a tribe and $A, B\in  \mathcal{E}$, then every section of the projection $p_1:A\times B\to A$
is of the form $\langle 1_A,f\rangle:A\to A\times B$ for a map $f:A\to B$. 
This defines a canonical bijection $\Gamma(A\times B,p_1)\simeq \mathcal{E}(A,B)$.

 \begin{lemma} \label{genereousisfull} 
 A generous morphism of tribes is full.
    \end{lemma}

\begin{proof} 
If $F: \mathcal{E}\to \mathcal{E}'$ is a morphism of tribes and $A, B\in  \mathcal{E}$,
then the following square commutes
\begin{equation}\label{squareforgeneroushomo}
 \xymatrix{
\Gamma(A\times B,p_1) \ar@{=}[d]_{}\ar[rr]^{}  && \Gamma(FA\times FB,F(p_1))   \ar@{=}[d]_{} \\
 \mathcal{E}(A,B) \ar[rr]^{ }&&  \mathcal{E}'(FA,FB)
 }
 \end{equation}
where the horizontal maps are induced by $F$.
The top map of the square is surjective,
since $F$ is generous by hypothesis.
Hence the bottom map is surjective, since the vertical maps
are bijective. This proves that $F$ is full.
\end{proof}

 \begin{prop} \label{anodynearegenerous} If $u:A\to B$ is an anodyne map in a tribe $\mathcal{E}$,
 then the base change functor $u^\star: \mathcal{E}(B)\to \mathcal{E}(A)$ is generous.
  \end{prop}

\begin{proof} 
Let us show that the map $\Gamma(E,g) \to \Gamma(u^\star E,u^\star(g))$
induced by $u^\star$ is surjective for any fibration 
$g:(E,p)\to (C,q)$ in $ \mathcal{E}(B)$.
The top square of the following diagram is cartesian by Lemma \ref{lemmacartesiansq}.
$$
 \xymatrix{
u^\star(E) \ar[rr]^{u_E} \ar[d]_{u^\star(g)
}  && E \ar[d]^{g} \\
u^\star(C)\ar[rr]^{u_C}\ar[d]&& C\ar[d]^q \\
A\ar[rr]^{u}&& B
}$$
The map $u_C$ is anodyne by base change, since $u$ is anodyne
and $q$ is a fibration. If $s:u^\star(C)\to u^\star(E)$
is a section of the fibration $u^\star(g)$, then
the following square commutes,  since $g\circ u_E\circ s=u_C\circ  u^\star(g)\circ s=u_C$.
$$
 \xymatrix{
u^\star(C) \ar[rr]^{u_E\circ s} \ar[d]_{u_C}  && E \ar[d]^{g} \\
C \ar@{=}[rr]&& C
}$$
Hence the square has a diagonal filer $s':C\to E$.
We have $s'\in \Gamma(E,g)$ since $g\circ s'=1_C$.
Moreover, the following square commutes, since $u_C\circ s'=s\circ u_E$.
 $$
 \xymatrix{
u^\star(C) \ar[rr]^{u_E}   && E \\
u^\star(C) \ar[rr]^{u_C} \ar[u]^s&& C\ar[u]_{s'} 
}$$
 Thus,  $u^\star(s')=s$. This shows that the map $\Gamma(E,g) \to \Gamma(u^\star E,u^\star(g))$
induced by $u^\star$ is surjective.
Let us show that for every object $(Y,q)\in  \mathcal{E}(A)$ there exists
  an object $(X,p)\in  \mathcal{E}(B)$ together with an anodyne map
 $(Y,q)\to u^\star(X,p)$. For this, let us factor the map $uq:Y\to B$
as an anodyne map $v:Y\to X$ followed by a fibration $p:X\to B$.
$$
 \xymatrix{
Y  \ar[d]_q \ar[rr]^{v}   && X  \ar[d]^{p}  \\
A \ar[rr]^{u}&& B
}$$
By pulling back the fibration $p:X\to B$ along $u$, we obtain 
the following diagram
$$
 \xymatrix{
Y \ar[dr]_q \ar[r]^{(q,v)}& u^\star(X)\ar[d]^{p_1}\ar[rr]^{p_2}   && X  \ar[d]^{p}  \\
&A \ar[rr]^{u}&& B
}$$
The projection $p_2$ is anodyne by base change, since $u$ is anodyne
and $p$ is a fibration. Hence the map $(q,v)$ is anodyne by Lemma \ref{anodyne closure},
since $p_2(q,v)=v$ is anodyne. This show that
the morphism of tribes $u^\star: \mathcal{E}(B)\to \mathcal{E}(A)$ is generous.
 \end{proof}

\medskip

We say that two section $s,t\in \Gamma(E,p)$ of a fibration $p:E\to A$ are {\it fibrewise homotopic} if the morphism
$s,t:(A,1_A)\to (E,p)$ are homotopic in the tribe $ \mathcal{E}(A)$.
We shall denote by $\pi_0\Gamma(E,p)$ the quotient of $\Gamma(E,p)$ by the fibrewise homotopy 
relation.

 \begin{lemma} \label{sectionvsmaps} If $A$ and $B$ are two objects of a tribe $\mathcal{E}$, then the
canonical bijection $\Gamma(A\times B, p_1)\simeq \mathcal{E}(A,B)$ respects the homotopy relation. It thus
 induces a bijection $\pi_0\Gamma(A\times B, p_1)\simeq \pi_0\mathcal{E}(A,B).$
 \end{lemma}

\begin{proof} The functor $\Sigma_A: \mathcal{E}(A)\to  \mathcal{E}$
is left adjoint to the functor $e_A: \mathcal{E} \to  \mathcal{E}(A)$
and we have $\Sigma_A(1_A)=A$, where $1_{\! A}=(A,1_A)$ is the terminal object of $ \mathcal{E}(A)$.
If $B\in \mathcal{E}$, then the bijection $Hom_A( 1_{\! A}, e_A(B)) \simeq Hom(\Sigma_A( 1_{\! A}),B)$
defined by the adjunction $\Sigma_A\dashv e_A$
coincide with the canonical bijection $\Gamma(A\times B,p_1)\simeq \mathcal{E}(A,B)$.
But the adjunction $\Sigma_A\dashv e_A$ induces an adjunction $Ho(\Sigma_A)\dashv Ho(e_A)$
by Proposition \ref{inverthoequiv4}. Hence the bijection $Hom_A( 1_{\! A}, e_A(B)) \simeq Hom(\Sigma_A( 1_{\! A}),B)$
 induces a bijection $\pi_0Hom_A( 1_{\! A}, e_A(B)) \simeq \pi_0Hom(\Sigma_A( 1_{\! A}),B)$.
 This shows that the bijection $\Gamma(A\times B,p_1)\simeq \mathcal{E}(A,B)$
 induces a bijection $\pi_0\Gamma(A\times B,p_1)\simeq \pi_0\mathcal{E}(A,B).$
\end{proof}

\medskip

 \begin{lemma} \label{veryfulbasechange} Let $F: \mathcal{E}\to \mathcal{E}'$
 be a morphism of tribes. If  the map
 $$\pi_0\Gamma(E,p) \to \pi_0\Gamma(FE,F(p))$$
  induced by $F$ is surjective for every fibration $p:E\to A$ in $ \mathcal{E}$,
  then the functor $$Ho(F):Ho(\mathcal{E})\to Ho(\mathcal{E}')$$
  is fully faithful.
 \end{lemma}

\begin{proof} The commutative square (\ref{squareforgeneroushomo})
induces a commutative square
 $$ \xymatrix{
\pi_0\Gamma(A\times B,p_1) \ar@{=}[d]_{}\ar[rr]^{}  && \pi_0\Gamma(FA\times FB,F(p_1))   \ar@{=}[d]_{} \\
\pi_0 \mathcal{E}(A,B) \ar[rr]^{ }&& \pi_0 \mathcal{E}'(FA,FB)
 }$$
 The top map of the square is surjective by the hypothesis on $F$.
 Hence the bottom map is surjective, since the vertical maps are bijective
 by Lemma \ref{sectionvsmaps}. 
 This shows that the functor $Ho(F)$ is full. Let us show that it is faithful.
 Let $f,g:A\to B$ be two maps in  $\mathcal{E}$
and suppose that the maps $F(f),F(g):FA\to FB$ are homotopic.
If $(PB,\partial_0,\partial_1,\sigma)$ is a path object of $B$,
then the maps $f$ and $g$ are homotopic if and only if 
the fibration $p_1:E \to A$ defined by the following 
 pullback square has a section
 \begin{equation}\label{hsurjectiontriangle}
 \xymatrix{
E \ar[rr]^{p_2} \ar[d]_{p_1}  && PB \ar[d]^{\langle \partial_0,\partial_1\rangle } \\
A\ar[rr]^{\langle f,g\rangle }&& B\times B 
}
 \end{equation}
The image by the functor $F$ of the square (\ref{hsurjectiontriangle})
 is a pullback square of the same nature, since 
 $(F(PB), F(\partial_0), F(\partial_1), F(\sigma))$ is a path object for $FB$
by Lemma \ref{homotakepathobject}.
 \begin{equation}\label{hsurjectiontriangle2}
 \xymatrix{
FE \ar[rr]^{F(p_2)} \ar[d]_{F(p_1)}  && F(PB) \ar[d]^{\langle F(\partial_0),F(\partial_1)\rangle } \\
FA\ar[rr]^{\langle F(f),F(g)\rangle}&& FB\times FB 
}
\end{equation}
Hence the fibration $F(p_1):FE \to FA$ has a section, since the maps $F(f),F(g):FA\to FB$ are homotopic
by hypothesis.
But the map $\pi_0 \Gamma(E,p_1)\to  \pi_0\Gamma(FE,F(p_1))$
induced by $F$ is surjective by the hypothesis on $F$.
Thus, $\Gamma(E,p_1)\neq \emptyset$ and
this shows that $f$ is homotopic to $g$.
We have proved that the functor $Ho(F)$ is faithful.
  \end{proof} 

\medskip

 \begin{prop} \label{quasi-surjectiveisequivalence} 
 A generous morphism of tribes is a weak equivalence.
  \end{prop} 

\begin{proof} Let $F: \mathcal{E}\to \mathcal{E}'$ be
a generous morphism of tribes. The functor $Ho(F):Ho(\mathcal{E})\to Ho(\mathcal{E}')$
is fully faithful by Proposition \ref{veryfulbasechange}.
Let us show that it is essentially surjective.
For every object $Y\in  \mathcal{E}'$ there exists
  an object $X\in  \mathcal{E}$ together with an anodyne map $u:Y\to FX$,
  since $F$ is generous.
  But $u$ is invertible in the homotopy category $Ho(\mathcal{E}')$ by \ref{anohomotopyequiv}.
 Hence the functor $Ho(F)$ is essentially surjective.
 \end{proof}

 \begin{cor} \label{anodynebasechangeweakequiv} 
 If  $u:A\to B$ is an anodyne map in a tribe $\mathcal{E}$, 
then the base change functor $u^\star: \mathcal{E}(B)\to \mathcal{E}(A)$ is 
a weak equivalence of tribes.
  \end{cor}

\begin{proof} 
This follows from Proposition \ref{anodynearegenerous}
and Proposition \ref{quasi-surjectiveisequivalence}.
 \end{proof}

 \begin{prop} \label{hobasechangehasaleftadjointallmap} 
 If $f:A\to B$ is a map in a tribe $ \mathcal{E}$,
then the functor  $Ho(f^\star):Ho( \mathcal{E}(B))\to  Ho(\mathcal{E}(A))$
 has a left adjoint 
 $$\tilde{\Sigma}_f:Ho(\mathcal{E}(A))\to Ho( \mathcal{E}(B))$$
   \end{prop} 
   
\begin{proof} Let us choose an $AF$-factorisation $f=pu:A\to C\to B$.
Then the functor $Ho(f^\star)$ is isomorphic to the functor $Ho(u^\star)Ho(p^\star)$,
since the functor $f^\star$ is isomorphic to the functor $u^\star p^\star$
by Proposition \ref{compositebasechangetribe}. 
The functor $Ho(p^\star)$ has 
a left adjoint $Ho(\Sigma_p)$ by Proposition \ref{inverthoequiv4}, since $p$ is a fibration.
The functor $Ho(u^\star)$ has 
has a pseudo-inverse $W$ by Corollary \ref{anodynebasechangeweakequiv},
since $u$ is anodyne. 
It follows that the functor $Ho(f^\star)=Ho(u^\star)Ho(p^\star)$
has a left adjoint $\tilde{\Sigma}_f=Ho(\Sigma_p)W$.
\end{proof}

When $f$ is a fibration, we have $\tilde{\Sigma}_f=Ho(\Sigma_f)$
 by Proposition \ref{inverthoequiv4}.  
 In general, if $(E,p)\in \mathcal{E}(A)$
 let us choose a factorization of the map $fp:E\to B$
 as an anodyne map $v:E\to E'$ followed
 by a fibration $p':E'\to B$,
 $$ \xymatrix{
E\ar[d]_p  \ar[rr]^v && E'\ar[d]^{p'} \\
A \ar[rr]^f && B
}$$
 Then $\tilde{\Sigma}_f(E,p)=(E',p')$.

 \begin{prop} \label{hobasechangehasaleftadjointcomp} 
 If $f:A\to B$ and $g:B\to C$ are two maps in a tribe,
 then $\tilde{\Sigma}_{gf}\simeq  \tilde{\Sigma}_g \tilde{\Sigma}_f $.
  \end{prop} 
  
 \begin{proof} We have $Ho((gf)^\star)\simeq Ho(f^\star) Ho(g^\star)$,
 since we have $(gf)^\star\simeq f^\star g^\star$,
 The result follows from the adjunctions 
 $$\tilde{\Sigma}_f\vdash Ho(f^\star), \quad  \quad
\tilde{\Sigma}_g\vdash Ho(g^\star) \quad {\rm and} \quad \tilde{\Sigma}_{gf}\vdash Ho((gf)^\star)$$
  \end{proof}

\medskip

\begin{prop} \label{homotopicimplieshomtoyisomorphic} If two maps $f,g: A\to B$ in a tribe $ \mathcal{E}$ are homotopic, then the functors 
$$Ho(f^\star), Ho(g^\star):Ho(\mathcal{E}(B))\to Ho(\mathcal{E}(A))$$
are isomorphic.
\end{prop}

\begin{proof} If $PB=(PB, \partial_0,\partial_1,\sigma)$ be a path object for $B$, then
 the following diagram of maps commutes:
$$ \xymatrix{
&&  \ar@/_1pc/[lld]_-{1_B} B\ar[d]_\sigma   \ar@/^1pc/[rrd]^-{1_B} &&  \\
B &&\ar[ll]_{\partial_0}  PB \ar[rr]^{\partial_1} && B.
 }$$
Hence the following diagram of functors commutes up to natural isomorphisms by Proposition \ref{compositebasechangetribe},
$$ \xymatrix{
&& \mathcal{E}(B) & &  \\
 \ar@/^1pc/[rru]^-{Id}  \mathcal{E}(B) \ar[rr]^{\partial_0^\star}   &&\ar[u]^{\sigma^\star}  \mathcal{E}(PB)  && 
 \ar[ll]_{\partial_1^\star} \mathcal{E}(B)  \ar@/_1pc/[llu]_-{Id}  
 }$$
This shows that $\sigma^\star \partial_0^\star \simeq  Id \simeq \sigma^\star \partial_1^\star $. 
A fortiori we have  $Ho(\sigma^\star) Ho( \partial_0^\star)  \simeq Ho(\sigma^\star) Ho(\partial_1^\star)$. 
But the functor $Ho(\sigma^\star)$
is an equivalence of categories by \ref{quasi-surjectiveisequivalence}, since $\sigma$ is anodyne.
It follows that $Ho( \partial_0^\star)  \simeq Ho(\partial_1^\star)$.
If $h:A\to PB$ is a homotopy $f\leadsto g$, then $\partial_0h=f$ and $\partial_1h=g$.
Hence we have  $f^\star\simeq h^\star \partial_0^\star $ and $g^\star\simeq h^\star \partial_1^\star $ 
by Proposition \ref{compositebasechangetribe}.
It follows that 
$$Ho(f^\star)\simeq Ho(h^\star)Ho(\partial_0^\star)\simeq Ho(h^\star)Ho(\partial_1^\star)\simeq Ho(g^\star).$$
\end{proof}

 \begin{thm} \label{homotopyequivalenceareweakequiv} 
A map $f:A\to B$ is a homotopy equivalence 
if and only if the  functor $f^\star: \mathcal{E}(B)\to \mathcal{E}(A)$
 is a weak equivalence of tribes.
  \end{thm}

 \begin{proof} 
  ($\Rightarrow$) The map $f:A\to B$
 has a homotopy inverse $g:B\to A$ by hypothesis.
 By definition, we have $gf\sim 1_A$ and $fg\sim 1_B$.
 It then follows from Proposition \ref{homotopicimplieshomtoyisomorphic}
 that we have $Ho((gf)^\star)\simeq Id$ and $Ho((fg)^\star)\simeq Id$.
 But we have 
 $$Ho((gf)^\star)\simeq Ho(f^\star g^\star)=Ho(f^\star) Ho(g^\star) \quad {\rm and} \quad Ho((fg)^\star)\simeq Ho(g^\star f^\star)=Ho(g^\star) Ho(f^\star)$$
by Proposition \ref{compositebasechangetribe}.
It follows that $Ho(f^\star) Ho(g^\star) \simeq Id$ and $Ho(g^\star) Ho(f^\star) \simeq Id$.
This shows that the functor $Ho(f^\star)$ is an equivalence
 of categories. ($\Leftarrow$)  Let us choose an $AF$-factorisation $f=pu:A\to C\to B$.
Then the functor $Ho(f^\star)$ is isomorphic to the functor $Ho(u^\star)Ho(p^\star)$.
But the functor $Ho(u^\star)$
is an equivalence of categories by Corollary \ref{anodynebasechangeweakequiv},
since $u$ is anodyne. Hence the functor $Ho(p^\star)$ is an equivalence
of categories by 3-for-2, since the functor 
$Ho(u^\star)Ho(p^\star)\simeq Ho(f^\star)$ is an equivalence of categories by hypothesis.
The functor $Ho(p^\star): Ho(\mathcal{E}(B))\to (\mathcal{E}(C))$
 has a left adjoint $Ho(\Sigma_p): Ho(\mathcal{E}(C))\to Ho(\mathcal{E}(B))$
 by Theorem \ref{inverthoequiv4}, since $p$ is a fibration.
The functor $Ho(\Sigma_p)$ is an equivalence of categories,
 since the functor $Ho(p^\star)$ is an equivalence of categories.
 Hence the object $Ho(\Sigma_p) (C,1_C)$ is terminal in $Ho(\mathcal{E}(B))$,
 since the object $(C,1_C)$ is terminal object of $Ho(\mathcal{E}(C))$ by Proposition \ref{hocatiscartesian}.
 This shows that the object $(C,p)=\Sigma_p (C,1_C)$ is contractible.
  It then follows by Proposition \ref{homotopyequiv} that the map $p:C\to B$ is a homotopy equivalence.
 But the map $u:A\to C$ is a homotopy equivalence by Proposition
 \ref{anohomotopyequiv}, since $u$ is anodyne.
It follows by 3-for-2 that $f=pu$ is a homotopy equivalence.
 \end{proof}

\bigskip

\begin{defi} \label{def:h-conservative}  We say that a morphism of tribes
$F:\mathcal{E} \to \mathcal{E}'$ is  {\bf homotopoically conservative} (in short,  {\bf h-surjective})  
if the functor $Ho(F): Ho(\mathcal{E}) \to Ho(\mathcal{E}')$ is conservative.
Similarly we say that $F$ is  {\bf homotopically surjective} (in short,  {\bf h-surjective}) 
if the functor  $Ho(F)$ is essentially surjective.
\end{defi}

By definition, a morphism of tribes $F:\mathcal{E} \to \mathcal{E}'$ i is h-conservative if and only if
it reflects homotopy equivalences. 
The morphism of tribes $F$ is h-surjective if and only if
for every object $Y\in  \mathcal{E}'$ there exists an object $X\in \mathcal{E}$
together with a homotopy equivalence $Y\to FX$.

\bigskip

Recall that if $F:\mathcal{E} \to \mathcal{E}'$ is a morphism of tribes,
then so is the functor $ F_{(A)}:\mathcal{E}(A)\to \mathcal{E}'(FA)$ induced by $F$
for each object $A\in \mathcal{E}$
by Proposition \ref{inducedhfunctor}.

\begin{lemma}\label{h-conservativelemma}  If a morphism of tribes $F:\mathcal{E} \to \mathcal{E}'$ is h-conservative, then so is 
the morphism $F_{(A)}:\mathcal{E}(A)\to \mathcal{E}'(F\!A)$ for every object $A\in \mathcal{E}$.
\end{lemma}  

\begin{proof}  We shall use the following commutative square of functors.
$$ \xymatrix{
\mathcal{E}(A) \ar[rr]^{F_{(A)}} \ar[d]_{\Sigma_A}  && \mathcal{E}'(F\!A) \ar[d]^{\Sigma_{F\!A}} \\
\mathcal{E} \ar[rr]^{F}  &&\mathcal{E}'
}$$
The horizontal functors preserve homotopy equivalences by 
 Corollary \ref{inverthoequiv3}. 
The vertical functors preserve and reflect homotopy equivalences by 
Proposition \ref{homotopyequirefl}.
Hence the functor $F_{(A)}$ reflects  homotopy equivalences,
since the functor $F$  reflects  homotopy equivalences by hypothesis.
\end{proof}

 \begin{prop} \label{genereousisfullandstable} 
  If a morphism of tribes $F: \mathcal{E}\to \mathcal{E}'$ is generous, then so is the morphism of tribes
 $F_{(A)}: \mathcal{E}(A)\to \mathcal{E}'(FA)$ for every object $A\in  \mathcal{E}$.  
\end{prop} 
   
\begin{proof} 
Let us show that  the morphism of tribes $F_{(A)}$ is full on sections of every fibration in $ \mathcal{E}(A)$.
Every fibration in $ \mathcal{E}(A)$ is of the form $p:(E,fp)\to (B,f)$,
where  $p: E\to B$ and  $f: B\to A$ are fibrations in $ \mathcal{E}$.
Every section $s:B\to E$ of $p:E\to B$ defines a section
$s:(B,f)\to (E,fp)$ of $p:(E,fp)\to (B,f)$  and conversely.
Moreover, the following square commutes by functorriality
 $$ \xymatrix{
\Gamma(E,p) \ar@{=}[d]_{}\ar[rr]^{}  &&\Gamma(FE,F(p))   \ar@{=}[d]_{} \\
\Gamma((E,fp),p)\ar[rr]^{ }&&\Gamma((FE,F(fp)),F(p))
 }$$
The top map of the square is surjective, since $F$ is full on sections of every fibration.
Hence the bottom map of the square is surjective.
This shows that the morphism of tribes $F_{(A)}$ is full on sections of every fibration.
It remains to show that for every object $(Y,q)\in  \mathcal{E}'(FA)$,
there exists an object $(X,p)\in \mathcal{E}(A)$ together with an anodyne map $u:(Y,q)\to (FX,F(p))$.
In other words, it remains to show that for
every fibration $q:Y\to FA$ in  $\mathcal{E}'$ there exists
a fibration $p:X\to A$ in $ \mathcal{E}$ together with an anodyne map $u:Y\to FX$
such that $F(p)u=q$. But there exists an object $E\in  \mathcal{E}$ together
with an anodyne map $v:Y\to FE$, since $F$ is generous.
The map $q:Y\to FA$ can be extended along $v$ as a map $q':FE\to FA$,
since $v$ is anodyne.
There is then a map $g:E\to A$ such that $F(g)=q'$, since $F$ is full.
Let us choose an $AF$-factorisation $g=pw:E\to X\to A$.
The map $F(w):FE\to FX$ is anodyne, since $F$
is a morphism of tribes.
Hence map $u\defeq F(w)v:Y\to FX$ is a anodyne, since $v$ is anodyne.
Moreover, we have $F(p)u=F(p)F(w)v=F(pw)v=F(g)v=q'v =q$.
\end{proof}  

\begin{cor}\label{genereousisfullandstable2}
 If $F: \mathcal{E}\to \mathcal{E}'$ is generous, then for every fibration $q:Y\to FA$ in $ \mathcal{E}'$, there exists
a fibration $p:X\to A$ in $ \mathcal{E}$ together with an anodyne map $u:Y\to FX$
such that $F(p)u=q$. 
$$ \xymatrix{
Y \ar[d]_{q}\ar[r]^{u}  &  \ar[dl]^{F(p)} FX \\
 FA& 
 }$$
\end{cor}

\begin{proof}  The functor $F_{(A)}:\mathcal{E}(A)\to \mathcal{E}'(FA)$ is generous
by Lemma \ref{genereousisfullandstable}. Thus, for every object $(Y,q)\in  \mathcal{E}'(FA)$,
there exists an object $(X,p)\in \mathcal{E}(A)$ together with an anodyne map $u:(Y,q)\to (FX,F(p))$.
\end{proof}

\bigskip

\begin{thm}\label{h-equivalencetheorem}  A morphism of tribes
$F:\mathcal{E} \to \mathcal{E}'$ is a weak equivalence
if and only if the following two conditions are satisfied:
\begin{itemize}
\item{} $F$ is h-conservative;
\item{} the morphism $F_{(A)}:\mathcal{E}(A)\to \mathcal{E}'(F\!A)$ induced by $F$ is h-surjective
for every object $A\in \mathcal{E} $.
\end{itemize}
\end{thm}

\begin{proof} ($\Rightarrow$) The functor $Ho(F)$ is conservative, since it is an equivalence of categories.
Thus, $F$ h-conservative.
Let us show that the functor $Ho(F_{(A)})$
is essentially surjective for every object $A\in \mathcal{E}$.
If $q:Y\to FA$ is a fibration in $ \mathcal{E}'$, then
there exists an object $E\in \mathcal{E}$ together
with a homotopy equivalence $v:Y\to FE$, since  the functor $Ho(F)$ is essentially surjective. 
If $v':FE \to Y$ is homotopy inverse to $v$, 
then there is a map $g:E\to A$, such that $F(g)$ is homotopic to $qv':FE\to FA$,
since the functor $Ho(F)$ is full. 
The map $F(g)v:Y\to FA$ is then homotopic to $q$, since $F(g)$ is homotopic to $qv'$.
Hence the following triangle commutes up to homotopy,
$$ \xymatrix{
&FE\ar[dd]^{F(g)}\\
Y \ar[ru]^{v} \ar[dr]_{q}  &  \\
& FA  
}$$
Let us choose a factorization
$g=pw:E\to X\to A$ with $w$ a homotopy equivalence and $p$ a fibration.
$$ \xymatrix{
E  \ar[dd]_{g} \ar[rd]^{w}  & \\
 & X \ar[dl]^{p} \\
 A & 
}\quad \quad \quad
 \xymatrix{
&FE\ar[dd]^{F(g)} \ar[dr]^{F(w)} & \\
Y \ar[ru]^{v} \ar[dr]_{q}&  & FX   \ar[dl]^{F(p)} \\
& FA  &
}$$
The map $F(p)$ is a fibration, since $p$ is a fibration.
The following triangle commutes up to homotopy, since $F(p)F(w)v=F(pw)v=F(g)v\sim q$.
$$ \xymatrix{
Y \ar[rr]^{F(w)v} \ar[dr]_{q}  && FX \ar[dl]^{F(p)} \\
& FA & 
}$$
It then follows from Lemma \ref{homotopyliffting} that there exists a map $u:Y\to FX$ homotopic to 
$F(w)v$ such that $F(p)u=q$. 
$$ \xymatrix{
Y \ar[rr]^{u} \ar[dr]_{q}  && FX \ar[dl]^{F(p)} \\
& FA & 
}$$
The map $F(w)$ is a homotopy
equivalence, since $w$ is a homotopy equivalence.
Thus, $F(w)v$ is a homotopy equivalence, since $v$ is a homotopy
equivalence. It follows that $u$ is a homotopy equivalence,
since $u$ is homotopic to $F(w)v$. 
Moreover, the morphism $u:(Y,q)\to (FX,F(p))$ is a fibrewise homotopy
equivalence in $ \mathcal{E}'(FA)$ by \ref{homotopyequirefl},
since $u:Y\to FX$
is a homotopy equivalence in $ \mathcal{E}'$.
This shows that the functor $Ho(F_{(A)})$
is essentially surjective.
($\Leftarrow$) 
Let us show that the functor $Ho(F)$ is is an equivalence of categories.
 For this, let us first verify that the morphism of tribes $F_{(A)}:\mathcal{E}(A)\to \mathcal{E}'(FA)$ satisfies the same hypothesis as the functor $F$
 for every object $A\in \mathcal{E}$.
 The morphism of tribes $F_{(A)}$ is
 h-conservative by Lemma \ref{h-conservativelemma},
 since $F$ is h-conservative.
 For every object $(E,p)\in \mathcal{E}(A)$, we have $(F_{(A)})_{(E,p)}=F_{(E)}$.
 Hence  then the morphism of tribes $(F_{(A)})_{(E,p)}$ is h-surjective,
 since the morphism of tribes $F_{(E)}$  is h-surjective by the hypothesis on $F$.
 This shows the morphism of tribes $F_{(A)}$ satisfies the same hypothesis as the functor $F$.
The functor $Ho(F)$ is essentially surjective. since $F$ is h-surjective by hypothesis.
 Let us show that the $Ho(F)$ is full. We shall first prove that for every object $A\in \mathcal{E}$
and every element $b:FA$, there exists an element
$a:A$ such that $F(a)\sim b$.
For this,  let us choose a factorization $b=qv:1 \to Y\to FA$
 with $v$ a homotopy equivalence and $q:Y\to FA$ a fibration.
 Then there exists a fibration $p:X\to A$ in $ \mathcal{E}$ together with a homotopy equivalence $u:Y\to FX$
such that $F(p)u =q$, since the functor $F_{(A)}$ is h-surjective,
$$ \xymatrix{
&Y \ar[rr]^{u} \ar[dr]^{q}  && FX \ar[dl]^{F(p)} \\
 1\ar[ru]^v \ar[rr]^b && FA & 
}$$
The map $uv:1\to  FX$ is a homotopy equivalence, since
$u$ and $v$ are equivalence. Hence the object $FX$
is contractible. It follows that $X$ is contractible, since $F$
reflects homotopy equivalences. Hence there exists an element $x_0:X$.
Let us put $a=px_0$.
The map $F(x_0):1\to  FX$ is homotopic to the map $uv:1 \to FX$,
since $FX$ is contractible. Hence the element $F(a)=F(px_0)=F(p)F(x_0)$
is homotopic to the element $F(p)uv=b$. 
This shows that there exists an element $a:A$ such that $F(a)\sim b$.
Let us now show that the functor $Ho(F)$ is full.
We saw above that the functor $F_{(A)}:\mathcal{E}(A)\to \mathcal{E}'(FA)$
satisfies the same hypothesis as the functor $F$.
If apply the result just proved to the functor $F_{(A)}$ instead of $F$,
we obtain that  if $p:E\to A$ is a fibration in $\mathcal{E}$
and $b:FA\to FE$ is a section of $F(p):FE\to FA$, then there exist a section $a:A\to E$
of the fibration $p:E\to A$ 
such that $F(a)$ is fibrewise homotopic to $b$.
It then follows from Proposition \ref{veryfulbasechange}
that the functor $Ho(F)$ is fully faithful.
 \end{proof}

\begin{cor}\label{h-equivalencestable}  If 
$F:\mathcal{E} \to \mathcal{E}'$ is a weak equivalence of tribes,
then so is the morphism of clans 
$F_{(A)}:\mathcal{E}(A)\to \mathcal{E}'(FA)$ 
for every object $A\in \mathcal{E} $.
\end{cor}

\begin{proof}  
The morphism of clans $F_{(A)}:\mathcal{E}(A)\to \mathcal{E}'(FA)$ is
 h-conservative by Lemma \ref{h-conservativelemma}.
 Moreover, the morphism of clans $(F_{(A)})_{(E,p)}=F_{(E)}$ is h-surjective for every object $(E,p)\in \mathcal{E}(A)$ by Theorem \ref{h-equivalencetheorem}.
 It then follows by the same theorem that  $F_{(A)}$ is a weak equivalence.
 \end{proof}

\subsection{Truncated types}

There is a notion of homotopy cartesian square in a tribe,
since there is a notion of homotopy cartesian square in any Brown fibration category.
See Definition \ref{def:hcsquare}.

\begin{defi} \label{def:hmonic} A map $u:A\to B$
in a tribe is said to be a {\bf homotopy monomorphism}, 
if the following square is homotopy cartesian
$$
\xymatrix{
A \ar@{=}[r] \ar@{=}[d] & A \ar[d]^u\\
 A \ar[r]^u & B
}
$$
\end{defi} 

A fibration $u:A\to B$
is homotopy monic if and only if its diagonal $A\to A\times_B A$
is a homotopy equivalence.

\begin{defi} \label{def:-1obj} An object $A$  in a tribe
is say to be a {\it mere proposition} if the map
$A\to 1$ is homotopy monic.
\end{defi}

\begin{prop} \label{prop:caractprop} 
If $(PA,\partial_0,\partial_1,\sigma)$ is a path object for $A$, then the following conditions are equivalent:
\begin{enumerate}
\item{} The object $A$ is a mere proposition;
\item{} the diagonal $\delta_A:=(1_A,1_A):A\to A\times A$ is a homotopy equivalence;
\item{} The map $A\to 1$ is monic in the homotopy category;
\item{} the fibration $(\partial_0,\partial_1):PA\to A\times A$ is trivial;
\item{} the fibration $(\partial_0,\partial_1):PA\to A\times A$ has a section.
\end{enumerate}
\end{prop}

\begin{proof}
($1 \Leftrightarrow 2\Leftrightarrow 3$) The map $A\to 1$ is homotopy monic if 
and only if the diagonal $\delta_A:A\to A\times A$ is a homotopy equivalence 
if and only if the map $A\to 1$ is monic in the homotopy category.
($2\Leftrightarrow 4$) The map $\sigma:A\to PA$ is a homotopy equivalence by \ref{anohomotopyequiv}, since it is anodyne.
But we have $(\partial_0,\partial_1)\sigma=\delta_A$,
  $$ \xymatrix{
  & PA  \ar[d]^{(\partial_0,\partial_1)} \\
A \ar[r]_(0.5){\delta_A} \ar[ur]^\sigma & A\times A 
 }$$
 It follows by 3-for-2 that $\delta_A$ is a homotopy equivalence
if and only if the map $(\partial_0,\partial_1)$ is a homotopy equivalence
 ($4\Rightarrow 5$) A  trivial fibration has a section by \ref{trivialfibrationhavesection}.
  ($5\Rightarrow 3$) A section $H: A\times A\to PA$ of the map $(\partial_0,\partial_1)$
   is a homotopy
  between the projections $p_1,p_2 :A\times A\to A$.
  Thus, $[p_1]=[p_2]$ in the homotopy category.
  It follows that the map
$A\to 1$ is monic in the homotopy category.
\end{proof}

\begin{defi} \label{def:rightcancellation} We shall say that
a class  $\mathcal{M}$ of maps in a category $\mathcal{E}$
is {\it closed under right cancellation}
for every commutative triangle
$$
\xymatrix{
X    \ar[dr]_{h}   \ar[rr]^f && Y \ar[dl]^g  \\
 & Z &
}$$
if $h\in \mathcal{M}$ and $g\in  \mathcal{M}$, then $ f\in \mathcal{M}.$
 \end{defi}

\begin{prop} \label{monoclosedcompcancel} 
The class of homotopy monomorphisms is closed 
under composition and right cancellation. 
\end{prop}

\begin{proof} Left to the reader.
\end{proof}

\begin{prop} \label{prop:morphismpreserves-1truncated} A morphism of tribes
$F:\mathcal{E}\to \mathcal{E}'$
takes homotopy monomorphisms to  homotopy monomorphisms.
 \end{prop}

\begin{proof} Left to the reader.
\end{proof}

Obviously, every contractible object $A$ is a $(-1)$-truncated, since the map $A\to 1$ is invertible in the homotopy category.

\begin{lemma} \label{propcontratible} 
A mere proposition $A$ is contractible if and only if there exists a map $a:1 \to A$.
\end{lemma} 

\begin{proof} The necessity is clear. Conversely, let us show
that the maps $a:1 \to A$ and  $t_A:A\to 1$ are mutually inverses in the homotopy category.
We have $t_Aa=1_1$, since the object $1$ is terminal.
Hence we have $t_Aat_A=1_1t_A=t_A=t_A1_A$.
It follows that $at_A=1_A$ in the homotopy category, since
 the map $t_A$ is monic in the homotopy category by 
Proposition \ref{prop:caractprop}. 
\end{proof}

\smallskip

Let $e_A:\mathcal{E}\to \mathcal{E}(A)$
be the base change functor defined 
by putting $e_A(X)=(X\times A,p_2)$.

\begin{lemma} \label{merepropcaract}
An object $A$ in a tribe $\mathcal{E}$ is a mere proposition
if and only if the object $e_A(A)=(A\times A, p_2)$
 is contractible.
\end{lemma}

\begin{proof} ($\Rightarrow$) If $A$ is a mere proposition, then so is the object $e_A(A)$
by Proposition \ref{prop:morphismpreserves-1truncated}, since the  functor $e_A:\mathcal{E}\to \mathcal{E}(A)$
is a morphism of tribes by Proposition \ref{basechangehomotopical}.
But the diagonal $\Delta_A:A\to A\times A$
is a map $\delta_A:\top_A\to e_A(A)$.  Thus, $e_A(S)$
is contractible. ($\Leftarrow$) If $e_A(A)$ is contractible,
then the map $\delta_A:\top_A\to e_A(A)$ is a homotopy
equivalence. It then follows by Theorem \ref{homotopyequirefl}
that the map $\Delta_A=\Sigma_A(\delta_A)$ is a homotopy
equivalence.
\end{proof}

\medskip

\begin{defi} \label{def:closedhomobasechange} We say that a class $\mathcal{M}$ of maps
in a tribe $\mathcal{E}$ is closed under {\it homotopy base changes}
if for any  homotopy cartesian square
$$
\xymatrix{
X    \ar[d]_f   \ar[r] & Y \ar[d]^g  \\
A \ar[r]  & B
}$$
with $g\in \mathcal{M}$, we have $f\in \mathcal{M}$.
\end{defi}

\begin{defi} \label{def:docileclass} We say that a class $\mathcal{M}$ of maps
in a tribe $\mathcal{E}$ is {\it docile} if the following two conditions hold:
\begin{enumerate}
\item{} Every homotopy equivalence belongs to $\mathcal{M}$;
\item{}  $\mathcal{M}$  is closed under homotopy base changes;
\item{} If $w$ is a homotopy equivalence and $gw \in \mathcal{M}$, then $g\in \mathcal{M}$
\end{enumerate}
 \end{defi} 

For example, the class of homotopy equivalences
is docile. We shall see below in \ref{cor:docileclassmono} that the class
of homotopy monomorphisms is docile. 

\begin{lemma} Let $\mathcal{M}$ 
be a docile class of maps in a tribe $\mathcal{E}$.
Suppose that the horizontal maps of the following
commutative square are homotopy equivalences:
\begin{equation} \label{squarehequiv}
\xymatrix{
A' \ar[r]^v \ar[d]_{f'} & A \ar[d]^{f}\\
 B' \ar[r]^w & B
}
\end{equation}
Then $f'\in \mathcal{M}\Leftrightarrow f\in \mathcal{M}$
\end{lemma}

\begin{proof} ($\Rightarrow$)
The square is homotopy cartesian by Lemma \ref{homotopycartesiancriterion1},
since the horizontal maps are homotopy equivalences.
Thus, $f\in \mathcal{M}\Rightarrow f'\in \mathcal{M}$, since $ \mathcal{M}$
is closed under homotopy base changes. Conversely,
 if $ f'\in \mathcal{M}$ let us show that $f\in \mathcal{M}$.
We shall first consider the case where $f$ is a fibration and
the square (\ref{squarehequiv}) is cartesian. 
The map $f'$ is then a fibration, since $f$ is a fibration.
The map $w$ has a homotopy inverse $h:B\to B'$,
since it is a homotopy equivalence. 
Consider the following diagram of cartesian squares
$$
\xymatrix{
B\times_{B'} A' \ar[r] \ar[d]_{p_1} & A' \ar[d]_{f'} \ar[r]& A\ar[d]^f \\
 B \ar[r]^h & B'  \ar[r]^w &  B
}
$$
We have $(wh)^\star(A,f)=(B\times_{B'} A',p_1)$,
since the composite of the two squares is cartesian.
But $p_1\in \mathcal{M}$, since $ f'\in \mathcal{M}$
and the left hand square is (homotopy) cartesian.
The map $wh:B\to B$ is homotopic to the identity,
since $w$ and $h$ are homotopy inverses.
Hence the functor $(wh)^\star :Ho(\mathcal{E}(B))\to Ho(\mathcal{E}(B))$
is isomorphic to the identity functor by Proposition \ref{homotopicimplieshomtoyisomorphic}.
It follows that the object $(B\times_{B'} A',p_1)$ of $\mathcal{E}(B)$
is homotopically equivalent to the object $(A,f)$.
Hence there exists a homotopy equivalence $e:A\to B\times_{B'} A'$
such that $p_1e=f$. Thus, $f\in \mathcal{M}$ by the part ($\Rightarrow$) of the
proof, since $p_1\in \mathcal{M}$.
We can now prove the implication $f'\in \mathcal{M}\Rightarrow f\in \mathcal{M}$
in the general case by reducing it to the special case.
For this, let us choose a factorisation $f=g u:A\to E\to B$,
with $u$ a homotopy equivalence and $g$ a fibration.
Consider the following diagram with a pullback square:
$$
\xymatrix{
A' \ar[rr]^v \ar[d]_{(f',uv)} && A \ar[d]^{u}\\
B'\times_B E \ar[rr]^{p_2} \ar[d]_{p_1} && E \ar[d]^{g}\\
 B' \ar[rr]^w && B
}
$$
The map $p_2$ is a homotopy equivalence by Corollary \ref{basechangehomotopyequiv},
since $w$ is a homotopy equivalence and $g$ is a fibration. 
Hence the map $(f',uv)$ is a homotopy equivalence by 3-for-2, 
since the top square of the diagram commutes and
the maps $v,u$ and $p_2$ are homotopy equivalences.
Thus,  $p_1\in  \mathcal{M}$ since since $p_1(f',uv)=f'\in  \mathcal{M}$
and the class $ \mathcal{M}$ is docile.
 It then follows by the first part of the 
proof that $g\in  \mathcal{M}$, since $g$ is a fibration and the 
bottom square of the diagram is a pullback.
Thus, $f\in  \mathcal{M}$ by the part ($\Rightarrow$) of the proof, since $f=gu$ and $u$ is a homotopy equivalence.
\end{proof}

Recall that a {\it fibrant replacement} of a map $f:A\to B$
is the fibration $p:E\to B$ in a factorisation $f=pu$
with $u$ a homotopy equivalence and $p$ a fibration.
The {\it homotopy diagonal} ${}^h\Delta(f)$ of $f:A\to B$
is defined to be diagonal $E\to E\times_B E$
of the fibration $p:E\to B$.

\begin{defi} \label{def:deriveddocileclass} Let $\mathcal{M}$
be a docile class of maps in a tribe $\mathcal{E}$.
We shall denote by $D(\mathcal{M})$ the class 
of maps $f:A\to B$ such that ${}^h\Delta(f)\in \mathcal{M}$.
We shall say that $D(\mathcal{M})$ is the {\it derived class}
of  $\mathcal{M}$.
 \end{defi}

\begin{prop} \label{prop:deriveddocileclass} The derived class $D(\mathcal{M})$
of a docile class $\mathcal{M}$ is docile.
Moreover, $D(\mathcal{M})$
is closed under composition (resp. right cancellation) if $\mathcal{M}$
is closed under composition (resp. right cancellation).
 \end{prop}

\begin{proof} Left to the reader.
\end{proof}

\begin{cor} \label{cor:docileclassmono} 
The class of homotopy monomorphisms is docile
and it is closed under composition and right cancellation.
 \end{cor} 
 
 \begin{proof} A map $f:A\to B$
 is a homotopy monomorphism if and only if
 its homotopy diagonal ${}^h\Delta(f)$ 
 is  a homotopy equivalence.
 But the class $\mathcal{J}$ of homotopy equivalences
 is docile and closed under composition (resp. right cancellation). It follows by Proposition \ref{prop:deriveddocileclass}
 that the class of homotopy monomorphisms
 is docile and closed under composition (resp. right cancellation). 
 \end{proof}

 The notion of $n$-truncated map is defined by induction on $n\geq -2$.

\begin{defi} \label{def:n-truncated} A map $f:A\to B$ in a tribe
is said to be {\bf $(-2)$-truncated} if it is a homotopy equivalence.
If $n\geq -1$, a map $f:A\to B$ is said to be {\bf $n$-truncated} if its homotopy diagonal
${}^h\Delta(f)$ is $(n-1)$-truncated.
We say that an object $A$ is {\bf $n$-truncated} 
if the map $A\to 1$ is $n$-truncated.
 \end{defi}

In particular, an object is $(-2)$-truncated if and only if it is contractible.
An object is $(-1)$-truncated if and only if it is a mere proposition.
A {\it $0$-truncated} object is said to be {\bf discrete}.
A map $f:A\to B$ is $(-1)$-truncated if it is a homotopy monomorphism.

\begin{cor} \label{cor:docileclassntruncated} 
The class of $n$-truncated maps is docile and it is
closed under composition and right cancellation.
 \end{cor} 
 
 \begin{proof} This follows from Proposition \ref{prop:deriveddocileclass}.
  \end{proof}

\begin{prop} \label{prop:morphismpreservesntruncated} A morphism of tribes
$F:\mathcal{E}\to \mathcal{E}'$
takes $n$-truncated maps to $n$-truncated maps.
 \end{prop}

\begin{proof} Left to the reader.
\end{proof}

\smallskip

Let $e_A:\mathcal{E}\to \mathcal{E}(A)$
be the base change functor defined 
by putting $e_A(X)=(X\times A,p_2)$.

\begin{lemma} 
An object $A$ in a tribe $\mathcal{E}$ is $n$-truncated
if and only if the map $\delta_A: \top_A\to e_A(A)$
 is $(n-1)$-truncated.
\end{lemma}

\begin{proof} Left to the reader.
\end{proof}

\subsection{$\pi$-tribes}

\begin{defi} We say that a $\pi$-clan  $\mathcal{E}$ is a $\pi$-{\bf tribe} if it is a tribe
and the internal product functor 
$$\Pi_f:\mathcal{E}(A)\to \mathcal{E}(B)$$
along a fibration $f:A\to B$
takes anodyne maps to anodyne maps.
\end{defi}

\begin{lemma} \label{localpitribe} If $\mathcal{E}$ is a $\pi$-tribe, then so is the
local tribe $\mathcal{E}(A)$ for every object $A\in \mathcal{E}$.
\end{lemma}

\begin{proof} Left to the reader.
\end{proof}

Examples of $\pi$-tribes:

\begin{itemize}
\item{} The category of groupoids $\mathbf{Grpd}$ (Hofmann and Streicher); 
\item{} The category of Kan complexes $\mathbf{Kan}$ (Awodey-Warren, Voevodsky);
\item{}  The syntaxic category of type theory with the axiom of 
extensionality (Gambino and Garner).
\end{itemize}

\begin{lemma} \label{lem:productisamorphismoftribes} If $\mathcal{E}$ is a $\pi$-tribe,
then the internal product functor 
$$\Pi_f:\mathcal{E}(A)\to \mathcal{E}(B)$$
along a fibration $f:A\to B$ is a morphism of tribes. Moreover, the
functor $[A,-]: \mathcal{E}\to \mathcal{E}$
is a morphism of tribes for every object $A\in \mathcal{E}$.
\end{lemma}

\begin{proof} The  functor $\Pi_f$
is a morphism of clans 
by Proposition \ref{productmorphismofclan}.
It is thus a morphism of tribes, since it takes
anodyne maps to anodyne maps.
Let us shows that  the functor $[A,-]$
is a morphism of tribes.
By Remark \ref{rem:exponentialareprod}, the functor $[A,-]$
is the composite of the functors
$$\xymatrix{
\mathcal{E} \ar[rr]^{e_A} && \mathcal{E}(A)\ar[rr]^{\Pi_A} && \mathcal{E}
}$$
where $e_A$ is the base change functor along the map $A\to 1$.
The functor $e_A$ is a morphism of tribes 
by \ref{basechangehomotopical}.
The  functor $\Pi_A$  s a morphism of tribes 
by the first part of the proof.
The functor $e_A$ is a morphism of tribes 
by \ref{basechangehomotopical}, 
Hence the composite $[A,-]$
is a morphism of tribes.
\end{proof}

\medskip

\begin{prop} \label{prop:derivedrightardj1}  If  $f:A\to B$ is a fibration in a $\pi$-tribe $\mathcal{E}$, then the 
adjunction
 $f^\star:\mathcal{E}(B)\leftrightarrow \mathcal{E}(A):\Pi_f$
induces an adjunction
$$Ho(f^\star):Ho(\mathcal{E}(A))\leftrightarrow Ho(\mathcal{E}(B)):Ho(\Pi_f)$$
\end{prop}

\begin{proof} The functor $f^\star$ is a 
morphism of tribes by \ref{basechangehomotopical}, and
the functor $\Pi_f$ is a 
morphism of tribes by Lemma  \ref{lem:productisamorphismoftribes}.
Hence the adjunction $f^\star\dashv \Pi_f$ induces an adjuntion
 $Ho(f^\star)\dashv Ho(\Pi_f)$ by Proposition \ref{the2-functorHo}.
 \end{proof}

\begin{prop} \label{homotopycatclosed} The homotopy category $Ho(\mathcal{E})$
of a $\pi$-tribe $\mathcal{E}$ is cartesian closed.
Moreover, the canonical functor $\mathcal{E}\to Ho(\mathcal{E})$
is cartesian closed.
\end{prop}

\begin{proof} If $A\in \mathcal{E}$, then
the functor $E(A):=[A,-]: \mathcal{E}\to \mathcal{E}$
is right adjoint to the cartesian product functor
$(-)\times A: \mathcal{E}\to \mathcal{E}$.
The functor $E(A)$ induces a functor $Ho(E(A)): Ho(\mathcal{E})\to Ho(\mathcal{E})$
by Proposition \ref{h-homotohomotop},
since it is a morphism of tribes by Lemma \ref{lem:productisamorphismoftribes}.
The induced functor  $Ho(E(A))$ is then
right adjoint to the cartesian product functor
$(-)\times A: Ho(\mathcal{E})\to Ho(\mathcal{E})$.
\end{proof}

\begin{cor} If $\mathcal{E}$ is a $\pi$-tribe,
then the category $Ho(\mathcal{E}(A))$
is cartesian closed for every object $A\in \mathcal{E}$.
\end{cor}

\begin{proof} This follows from Proposition \ref{homotopycatclosed} 
and Lemma \ref{localpitribe}.
\end{proof}

 \begin{prop}  \label{prop:derivedrightardj2}  
$ \mathcal{E}$ is a $\pi$-tribe,
then the functor  $Ho(f^\star):Ho( \mathcal{E}(B))\to  Ho(\mathcal{E}(A))$
 has a right adjoint 
 $$\tilde{\Pi}_f:Ho(\mathcal{E}(A))\to Ho(\mathcal{E}(B))$$
 for any map $f:A\to B$.
   \end{prop} 
   
\begin{proof} Let us choose an $AF$-factorisation $f=pu:A\to C\to B$.
Then the functor $Ho(f^\star)$ is isomorphic to the functor $Ho(u^\star)Ho(p^\star)$,
since the functor $f^\star$ is isomorphic to the functor $u^\star p^\star$
by Proposition \ref{compositebasechangetribe}. The functor $Ho(p^\star)$ has 
a right adjoint $Ho(\Pi_p)$ by Proposition \ref{prop:derivedrightardj1}, since $p$ is a fibration.
The functor $Ho(u^\star)$ has has a pseudo-inverse $W$ by Corollary \ref{anodynebasechangeweakequiv},
since $u$ is anodyne. 
It follows that the functor $Ho(f^\star)=Ho(u^\star)Ho(p^\star)$ has a right
adjoint $\tilde{\Pi}_f:=Ho(\Pi_p)W$.
\end{proof}

When $f$ is a fibration, we have $\tilde{\Pi}_f=Ho(\Pi_f)$
 by Proposition \ref{prop:derivedrightardj1}.

 \begin{prop} \label{hobasechangehasarightadjointcomp} 
 If $f:A\to B$ and $g:B\to C$ are two maps in a tribe,
 then $\tilde{\Pi}_{gf}\simeq  \tilde{\Pi}_g \tilde{\Pi}_f $.
  \end{prop} 
  
 \begin{proof} We have $Ho((gf)^\star)\simeq Ho(f^\star) Ho(g^\star)$,
 since we have $(gf)^\star\simeq f^\star g^\star$,
 The result follows by using the adjunctions: 
 $$Ho(f^\star)  \vdash \tilde{\Pi}_f, \quad \quad  Ho(g^\star) \vdash  \tilde{\Pi}_g \quad {\rm and} \quad  Ho((gf)^\star) \vdash \tilde{\Pi}_{gf}.$$
  \end{proof}

 \begin{defi} \label{embeddingpi-tribe} We shall say that a morphism of $\pi$-tribes  $F:\mathcal{E}\to \mathcal{E}'$
 is an {\it embedding} if it is fully faithful and if it reflects fibrations and anodyne maps.
  \end{defi}

\begin{lemma} \label{embeddingpitribes} If a morphism of $\pi$-tribes  $F:\mathcal{E}\to \mathcal{E}'$ is
 fully faithful and reflects fibrations, then it is an embedding.
\end{lemma}

\begin{proof} This follows from Lemma  \ref{embeddingtribes}.
\end{proof}

 \begin{defi} \label{subtribe=embedding} If  $\mathcal{E}$ is a $\pi$-tribe,
 we say that a sub-tribe $\mathcal{L}\subseteq \mathcal{E}$
 is {\it $\pi$-closed} if $\mathcal{L}$
is a $\pi$-tribe and the inclusion functor 
$\mathcal{L}\to \mathcal{E}$ is $\pi$-closed.
\end{defi}

\medskip

Recall from Corollary \ref{propcontratibleobjecttribe} 
that if $\mathcal{E}$ is a tribe, then the full subcategory $\mathcal{E}^c$ of contractible objects of $\mathcal{E}$
is a sub-tribe.

\begin{prop}\label{propcontratibleobjecpittribe} 
If $\mathcal{E}$ is a $\pi$-tribe, then the sub-tribe $\mathcal{E}^c$ 
is $\pi$-closed.
\end{prop}

 \begin{proof} Left to the reader.
 \end{proof}

\subsection{Contractibilty in a $\pi$-tribe}

Recall that an object $A$ is a tribe $\mathcal{E}$ is said to be {\it contractible}
if and only if the map $t_A:A\to 1$
has a homotopy inverse.
But a map $a:1\to A$ is a homotopy inverse of
the map $t_A:A\to 1$ if and only if 
there exists a  homotopy 
$h:at_A\leadsto 1_A$;
the pair $(a,h)$ is called a {\it contraction}
of the object $A$.

\medskip

If $A$ is an object of a $\pi$-tribe, consider the functors
$$\xymatrix{
\mathcal{E}(A\times A) \ar[rr]^{\Pi_{p_1}} && \mathcal{E}(A) \ar[rr]^{\Sigma_{A}} && \mathcal{E}
}$$
If $PA=(PA, \partial_0,\partial_1,\sigma)$ is a path object for $A$,
then the map $(\partial_0,\partial_1):PA\to A\times A$ is a fibration.
The  functor  $ \Pi_{p_1}$
takes the object $PA=(PA, (\partial_0,\partial_1))$
to an object $(\Pi_{p_1} (PA),p)$ and the functor 
$\Sigma_A$
takes the object $(\Pi_{p_1} (PA),p)$
to the object  $\Pi_{p_1} (PA)$.
Let us put
$$isCont(A):=\Sigma_A \Pi_{p_1} (PA).$$
In type theoretic notation, we have
$$isCont(A):=\sum_{x_1:A}\prod_{x_2:A} x_1=_A x_2,$$
where $x_1=_A x_2$ denotes the fiber of the map $(\partial_0,\partial_1):PA\to A\times A$
at $(x_1,x_2):A\times A$.

\begin{prop} {\rm (Voevodsky)} \label{iscontract3} The object $isCont(A)$
is a mere proposition for every object $A$.
An object $A$ is contractible if and only if 
the object $isCont(A)$ is inhabited.
\end{prop}

\begin{proof} Let us first show that 
the set of elements of 
$isCont(A)$ is in bijection with the set of contractions
of $A$. If $p:\Pi_{p_1} (PA)\to A$ is the structure map,
then to each element $c:1\to \Pi_{p_1} (PA)$
corresponds an element $a=p(c):A$
and the following triangle
commutes
$$\xymatrix{
1\ar[dr]_{a}\ar[rr]^(0.4){c}&&  \Pi_{p_1} (PA)\ar[dl]^{p}\\
&A&
}$$ 
Thus, $c$ is a map $(1,a)\to  \Pi_{p_1} (PA)$
in the category $ \mathcal{E}(A)$.
But $c$ corresponds by the adjunction $p_1^\star \dashv \Pi_{p_1}$ to a map $h:p_1^{\star}(1,a)\to PA$
in $\mathcal{E}(A\times A)$.
But the following square is cartesian,
where $t_A$ is the map $A\to 1$, 
$$\xymatrix{
A\ar[rr]^{(at_A,1_A)} \ar[d] && A\times A  \ar[d]^{p_1}\\
1 \ar[rr]^{a} &&  A
}$$ 
Thus, $h:A\to PA$ and $(\partial_0,\partial_1 )h=(a t_A,1_A)$.
This shows that $h:at_A\leadsto 1_A$.
Hence the pair $(a,h)$ is a contraction of $A$.
The map $c\mapsto (a,h)$ so defined is obviously bijective.
It follows that object $A$ is contractible if and only if 
the object $isCont(A)$ is inhabited.
Observe also that if $A$  is contractible, then so is the object $isCont(A)=\Sigma_A\Pi_{p_1}A$
by Proposition \ref{propcontratibleobjecpittribe}.
Thus, the object  $isCont(A)$ is inhabited if and only if it is contractible.
It remains to show that the object $isCont(A)$ is a mere proposition for any object $A$.
By Lemma \ref{merepropcaract}, it suffices
to show that the image of the object $K:=isCont(A)$
by the base change functor $e_K:\mathcal{E}\to \mathcal{E}(K)$
is contractible. By \ref{genericelement}, object $e_K(K)$ is inhabited by the element $\delta_K:\bot_K: e_K(K)$.
We have $e_K(K)=e_K(isCont(A))=isCont(e_K(A))$
since the functor $e_K$ is a morphism of $\pi$-tribes.
Hence the object $isCont(e_K(A))$ is inhabited. 
It follows by the argument above that the object $isCont(e_K(A))$ is contractible.
Thus, $e_K(K)$ is contractible. 
This proves that the object $K=isCont(A)$
is a mere proposition.
\end{proof}

\newpage

\section{Simplicial tribes}

\subsection{Simplicial clans}

We shall denote the category of simplicial sets by ${\sSet}$
and the full sub-category of finite simplicial sets by ${\sSet}_f$.
For the notion of simplicial category see Appendix \ref{hcsquareinaBrownfibcat}.
We shall denote the underlying category of a simplicial category $\mathcal{C}$  by $\mathcal{C}_0$.
Recall that a simplicial category $\mathcal{C}$ is said to be {\it finitely cotensored}
if every object $X\in \mathcal{C}$ admits a cotensor product $X^K$ by every
finite simplicial set $K$.
For simplicity, we shall write $X^{[n]}$ instead of $X^{\Delta[n]}$ 
and write $X^{\partial [n]}$ instead of $X^{\partial \Delta[n]}$.
In particular, $X^{[0]}=X$ and $X^{\partial [0]}=1$.
Moreover, we shall write
$$\partial_i:=X^{d_i}:X^{[n]}\to X^{[n-1]}\quad {\rm and} \quad  \sigma_j:=X^{s_j}:X^{[n-1]}\to X^{[n]}$$
for every $i\in [n]$ and $j\in [n-1]$.

\medskip

\begin{defi}\label{def:simpclan} Let $\mathcal{E}$ be a simplicial category.
If $\mathcal{E}$ is finitely cotensored, then 
a {\it simplicial clan structure} on $\mathcal{E}$ is a clan structure on $\mathcal{E}_0$
such that the following conditions hold:
\begin{itemize}
\item{} the functor $(-)^{[n]}:\mathcal{E}_0\to \mathcal{E}_0$
is a morphism of clans for every $n\geq 0$.
\item{} if $f:X\to Y$ is a fibration, then so is the gap map 
$$X^{[n]}\to Y^{[n]} \times_{Y^{\partial [n]} } X^{\partial [n]}$$
of the following square
$$
\xymatrix{
X^{[n]} \ar[d]_{f^{[n]}} \ar[rr]^{X^{i_n}} &&\ar[d]  X^{\partial [n]}\ar[d]^{f^{\partial[n]}}  \\
Y^{[n]} \ar[rr]^{Y^{i_n}}  && Y^{\partial [n]}
}
$$ 
for every $n\geq 0$.
\end{itemize}
\end{defi}

For example, the category ${\sSet}^{op}$ has the structure
of a simplicial clan, where a map $u^{o}:K^o\to L^o$
is a fibration if the map $u:L\to K$ is monic.
Similarly, the category $({\sSet}_f)^{op}$ 
has the structure of a simplicial clan.

\begin{rem} If $\mathcal{E}$ is a simplicial clan, then the terminal object $1\in \mathcal{E}_0$
is $s$-teminal. In other words, we have  $ \mathcal{E}(Q,1)=1$ for every object $Q\in \mathcal{E}$.
To see this, it suffices to show that we have $ \mathcal{E}(Q,1)_n=1$ for every $n\geq 0$.
But we have $ \mathcal{E}(Q,1)_n= \mathcal{E}_0(Q,1^{[n]})$ 
by definition of the cotensor product. Moreover, 
we have $1^{[n]}=1$ for every $n\geq 0$, since the functor $(-)^{[n]}:\mathcal{E}_0\to \mathcal{E}_0$
is a morphism of clans. 
Thus,
$$ \mathcal{E}(Q,1)_n= \mathcal{E}_0(Q,1^{[n]})_0=\mathcal{E}_0(Q,1)_0=1$$
since the object $1$ is terminal in $\mathcal{E}_0$.
\end{rem}

\begin{rem} In a simplicial clan $\mathcal{E}$ every fibration $f:X\to B$ is $s$-carrable.
To see this, we have to show that the square 
\begin{equation} \label{s-cartesiansquare}
\xymatrix{
A\times_B X \ar[d]_{p_1} \ar[rr]^{p_2}  && \ar[d]^{f} X  \\
A \ar[rr]^{g}  && B
}\end{equation}
is $s$-cartesian  for every map $g:A\to B$.
In other words, we have to show
that the square
 $$
\xymatrix{
\mathcal{E}(Q,A\times_B X)_n \ar[d] \ar[rr] && \ar[d]^{\mathcal{E}(Q,f)_n} \mathcal{E}(Q,X)_n  \\
\mathcal{E}(Q, A)_n \ar[rr]^{\mathcal{E}(Q,g)_n}  && \mathcal{E}(Q, B)_n
}$$
is cartesian for every object $Q\in \mathcal{E}$ and every $n\geq 0$.
But we have  $\mathcal{E}(Q,P)_n=\mathcal{E}_0(Q,P^{[n]})$ for every object
$P\in \mathcal{E}_0$ by definition of the cotensor product. 
Moreover, the following square is cartesian 
$$
\xymatrix{
\mathcal{E}_0(Q,(A\times_B X)^{[n]}) \ar[d] \ar[rr] && \ar[d]^{\mathcal{E}_0(Q,f^{[n]})} \mathcal{E}_0(Q,X^{[n]})  \\
\mathcal{E}_0(Q, A^{[n]}) \ar[rr]^{\mathcal{E}_0(Q,g^{[n]})}  && \mathcal{E}_0(Q, B^{[n]})
}$$
since  the functor $(-)^{[n]}:\mathcal{E}_0\to \mathcal{E}_0$
is a morphism of clans.
\end{rem}

Let us say that a monomorphism of simplicial sets $u:A\to B$
is {\it elementary} if it is the cobase change of an inclusion $i_n:\partial\Delta[n] \to \Delta[n] $. In which case we have 
a pushout square
\begin{equation} \label{elementmono}
\xymatrix{
\partial\Delta[n] \ar[d]_{i_n} \ar[rr]^{a} && \ar[d]^{u}  A  \\
\Delta[n] \ar[rr]^{b}  && B
}
\end{equation}
Every monorphism in ${\sSet}_f$ is the composite of  a finite 
sequence of elementary monomorphisms.

\begin{lemma}\label{expmono} If $X$ is an object in a simplicial clan $\mathcal{E}$
and  $u:A\to B$ is a monomorphism in $\mathcal{S}_f$,
then the map $X^u:X^B\to X^A$
is a fibration.
\end{lemma}

\begin{proof} If the proposition is true for two monomorphisms $u:A\to B$ and
 $v:B\to C$ in $\mathcal{S}_f$, then it is also true for the composite $vu:A\to C$.
 Hence it suffices to prove the proposition in the case where $u:A\to B$
 is an elementary monomorphism. From the pushout square (\ref{elementmono})
 we obtain a pullback square
$$
\xymatrix{
X^B \ar[d]_{X^u} \ar[rr]^{X^{b}} &&\ar[d]  X^{[n]} \ar[d]^{X^{i_n}}  \\
X^A \ar[rr]^{X^{a}}  &&X^{\partial [n]} 
}
$$ 
Thus, it suffices to show that the map $X^{i_n}$ is a fibration.
But  $X^{i_n}$ is the gap map of the following square.
$$
\xymatrix{
X^{[n]} \ar[d]_{t^{[n]}} \ar[rr]^{X^{i_n}} &&\ar[d]  X^{\partial [n]}\ar[d]^{t^{\partial[n]}}  \\
1^{[n]} \ar[rr]^{1^{i_n}}  && 1^{\partial [n]}
}
$$
where $t$ is the map $X\to 1$.
Thus,  $X^{i_n}$ is a fibration, since $t$ is a fibration and $\mathcal{E}$ is a simplicial clan.
 \end{proof}

\medskip

\begin{prop}\label{strongdefsimpclan} 
If $f:X\to Y$ is a fibration in a simplicial clan $\mathcal{E}$, 
then the following square 
\begin{equation} \label{squareforsimpclan}
\xymatrix{
X^B \ar[d]_{f^B} \ar[rr]^{X^u} &&\ar[d]  X^A\ar[d]^{f^A}  \\
Y^B \ar[rr]^{Y^u}  && Y^A
}
\end{equation}
is Reedy fibrant for every monomorphism $u:A\to B$ in $\mathcal{S}_f$.
\end{prop}

\begin{proof} The composite of two Reedy fibrant squares is Reedy fibrant
by Lemma \ref{compostionfibsquare}. Thus, if the proposition is true for two monomorphisms $u:A\to B$ and
 $v:B\to C$ in $\mathcal{S}_f$, then it is also true for their composite $vu:A\to C$.
 Hence it suffices to prove the proposition in the case where $u:A\to B$
 is an elementary monomorphism. 
 We may suppose that $f^A$ is a fibration by induction on the number of (non-degenerate) cells of $A$
 (the map $f^\emptyset:1\to 1$ is a fibration, since it is invertible).
 We shall use the following commutative cube:
\begin{equation}\label{cubeformapofsquaresimplicialcaln}
 \xymatrix{
 X^B \ar[dr] \ar[rrr]^{X^{u}} \ar[ddd]_{f^B} & &  &  X^A \ar '[d] [ddd]^(0.4){f^A}   \ar[dr]   &  \\
  &  X^{[n]} \ar[ddd]_(0.4){}\ar[rrr]^(0.4){X^{i_n}} &  & &  X^{\partial[n]} \ar[ddd]^(0.5){} \\ 
  &&&&&\\
 Y^B \ar '[r] [rrr]^(0.4){Y^{u}} \ar[dr] & &  &  Y^A\ar[dr] &  \\
  &  Y^{[n]} \ar[rrr]^{Y^{i_n}} &  & &   Y^{\partial[n]}   }
\end{equation}
The top and bottom faces of the cube are pullback squares,
since the square (\ref{elementmono}) is a pushout.
The maps $X^u$, $X^{i_n}$, $Y^u$ and  $Y^{i_n}$
are fibrations by Lemma \ref{expmono}.
The cartesian gap map of the front face of the cube is a fibration by the hypothesis on $\mathcal{E}$.
Let us show that the cartesian gap map of the back face is a fibration.
Let $S$ be the (diagonal) square obtained by composing the back face of the cube
with the bottom face. The cartesian gap map of 
$S$ is equal to the cartesian gap map of the back face of the cube,
since the bottom face is cartesian. 
Hence it suffices to show that the cartesian gap map of 
$S$ is a fibration. Observe that 
$S$ is also the composite of the top and front faces, since the cube commutes.
Hence the square $S$ admits the following decomposition,
where $\gamma'$ is the cartesian gap map of $S$
and $\gamma$ is the cartesian gap map of the front face of the cube,
$$
\xymatrix{
X^B \ar[r]^(0.3){\gamma'} \ar[d]&Y^{[n]}\times_{Y^{\partial [n]}} X^{A} \ar[d] \ar[rr]^{p_2}  & & X^A \ar[d] \\
X^{[n]}\ar[r]^(0.3)\gamma \ar[dr]_{f^{[n]}}&Y^{[n]}\times_{Y^{\partial [n]}} X^{\partial[n]} \ar[d]_{p_1} \ar[rr]^{p_2} && \ar[d]^{f^{\partial[n]}}   X^{\partial[n]}   \\
& Y^{[n]} \ar[rr]^{Y^{i_n}}  && Y^{\partial [n]}
}
$$ 
The composite of the top squares of this diagram is the top face of the cube.
Hence the top square on the left is a pullback by Lemma \ref{lemmacartesiansq}, since
the top face of the cube is a pullback. It follows that 
$\gamma'$ is a fibration, since $\gamma$ is a fibration.
This shows that the  cartesian gap map of $S$ is a fibration.
It follows that the cartesian gap map $\gamma'':X^B\to  Y^{B}\times_{Y^A} X^{A} $
of the back face of the cube is a fibration.
It remains to show that the map $f^B$ is a fibration.
But the map $f^A$ is a fibration by induction hypothesis.
Thus, the projection $p_1 $ in the following diagram is a fibration by base change,
$$
\xymatrix{
X^{B}\ar[r]^(0.3){\gamma''} \ar[dr]_{f^{B}}&Y^{B}\times_{Y^{A}} X^{A} \ar[d]_{p_1} \ar[rr]^{p_2} && \ar[d]^{f^{A}}   X^{A}   \\
& Y^{B} \ar[rr]^{Y^{u}}  && Y^{A}
}
$$ 
This shows that $f^B=p_1 \gamma''$ is a fibration.

\end{proof}

Recall that if $A$ is an object of a simplicial category  $\mathcal{C}$, then 
the slice category $\mathcal{C}/A$ is simplicial.
By construction, an object of $\mathcal{C}/A$ is an object $X\in \mathcal{C}$
equipped with a map $p:X\to A$.
If $X=(X,p)$ and $Y=(Y,q)$ are two objects of $\mathcal{C}/A$,
then the simplicial set $(\mathcal{C}/A)(X,Y)$
is defined by 
the following pullback square of simplicial sets:
\begin{equation} \label{squareforsimpslicehom}
\xymatrix{
(\mathcal{C}/A)(X,Y) \ar[d] \ar[r] &\mathcal{C}(X,Y) \ar[d]^{ \mathcal{C}(X,q) }  \\
1 \ar[r]^{p}  &  \mathcal{C}(X,A) 
}
\end{equation}

If $\mathcal{E}$ is a simplicial clan, we shall denote by $\mathcal{E}(A)$
the full simplical subcategory of $\mathcal{E}/A$ whose
objects are fibrations $p:X\to A$.

If $\mathcal{E}$ is a simplicial clan, then the category 
 $\mathcal{E}(A)_0=:\mathcal{E}_0(A)$ has the structure
 of a clan for every object $A\in \mathcal{E}$.

\begin{prop}\label{localsimpclan} If $\mathcal{E}$ is a simplicial clan, then so is the
simplicial category $\mathcal{E}(A)$ for every object $A\in \mathcal{E}$.
A map $f:(X,p)\to (Y,q)$ in  $\mathcal{E}(A)$ is a fibration if $f$ is a fibration in $\mathcal{E}$.
\end{prop}

\begin{proof} Let us show that the functor $(-)^{[n]}:\mathcal{E}_0(A)\to \mathcal{E}_0(A)$
is a morphism of clans for every $n\geq 0$.
 If $(X,p)\in \mathcal{E}(A)_0$, then
the object $(X,p)^{[n]}\in \mathcal{E}(A)_0$ is constructed
by the following pullback squares 
$$
\xymatrix{
(X,p)^{[n]}\ar[d]  \ar[r] & X^{[n]}\ar[d]^{p^{[n]}} \\
  A  \ar[r]^{\delta} & A^{[n]}
 }$$
where $\delta$ is the diagonal map.
A map $(X,p)\to (Y,q)$ in  $\mathcal{E}_0(A)$ 
is a map $f:X\to Y$ such that $qf=p$;
let us denote $f/A$ the map $(X,p)\to (Y,q)$ defined by $f$.
Then the map $(f/A)^{[n]}:(X,p)^{[n]}\to (X,p)^{[n]}$ in $\mathcal{E}_0(A)$ 
is constructed by  the following diagram with two pullback squares
$$
\xymatrix{
(X,p)^{[n]}\ar[d]_{(f/A)^{[n]}} \ar[r] & X^{[n]}\ar[d]^{f^{[n]}} \\
 (Y,q)^{[n]} \ar[r] \ar[d] & Y^{[n]} \ar[d]^{q^{[n]}} \\
 A  \ar[r]^{\delta} & A^{[n]}
 }$$
If $f:X\to Y$ is a fibration in $\mathcal{E}$, then so is the map $f^{[n]}:X^{[n]}\to Y^{[n]}$,
since $\mathcal{E}$ is a simplicial clan.
Thus, $(f/A)^{[n]}:(X,p)^{[n]}\to (X,p)^{[n]}$
is a fibration in $\mathcal{E}_0(A)$.
We leave to the reader the verification that
the functor $(-)^{[n]}:\mathcal{E}_0(A)\to \mathcal{E}_0(A)$
preserves base changes of fibrations and terminal objects.
It remains to show that if $f/A:(X,p)\to (Y,q)$ is a fibration, then so is the gap map 
$$(X,p)^{[n]}\to (Y,q)^{[n]} \times_{(Y,q)^{\partial [n]} } (X,p)^{\partial [n]}$$
of the following square for every $n\geq 0$.
\begin{equation}\label{gapmapslice}
\xymatrix{
(X,p)^{[n]} \ar[d]_{(f/A)^{[n]}} \ar[rr]^{(X,p)^{i_n}} &&\ar[d]  (X,p)^{\partial [n]}\ar[d]^{(f/A)^{\partial[n]}}  \\
(Y,q)^{[n]} \ar[rr]^{(Y,q)^{i_n}}  && (Y,q)^{\partial [n]}
}
\end{equation}
The following commutative diagram 
\begin{equation} \label{1943}
\xymatrix{
 X^{[n]}\ar[d]_{f^{[n]}} \ar[rr]^{X^{i_n}}\ar@{}[rrd]|{(a)} &&  X^{\partial [n]}  \ar[d]^{f^{\partial [n]}}     \\
 Y^{[n]}\ar[d]_{q^{[n]}} \ar[rr]^{Y^{i_n}} \ar@{}[rrd]|{(b)}&&  Y^{\partial[n]}  \ar[d]^{q^{[n]}} \\
  A^{[n]}  \ar[rr]^{A^{i_n}}&& A^{\partial [n]}
 }\end{equation}
admits the following decomposition with three pullback squares:
$$
\xymatrix{
 X^{[n]}\ar[drr]_{f^{[n]}} \ar[rr]^{\gamma_0} && C \ar[rr]^{\gamma_1}\ar[d] \ar@{}[rrd]|{(c)} &&D \ar[rr]\ar[d] \ar@{}[rrd]|{(d)} &&  X^{\partial [n]}  \ar[d]^{f^{\partial [n]}}     \\
&& Y^{[n]}\ar[drr]_{q^{[n]}} \ar[rr]^{\gamma_2} && \ar[rr]\ar[d] E \ar@{}[rrd]|{(e)}  && Y^{\partial[n]}  \ar[d]^{q^{[n]}} \\
&& &&   A^{[n]}  \ar[rr]^{A^{i_n}}&& A^{\partial [n]}
 }$$
 The base change of the maps $X^{\partial [n]} \to  Y^{\partial[n]} \to A^{\partial [n]}$ along the diagonal $A\to A^{\partial [n]}$
 is equal to the base change of the maps $D\to E\to A^{[n]}$ along the diagonal $A\to A^{[n]}$, since the
 squares $(d)$ and $(e)$ are cartesian.
It follows that the square (\ref{gapmapslice}) is a base change along the map $A\to A^{[n]}$ of the square $(f)$ in the following diagram 
 \begin{equation}\label{gapmapsliceprogress}
\xymatrix{
 X^{[n]}\ar[d]_{f^{[n]}} \ar[rr]^{\gamma_1 \gamma_0} \ar@{}[rrd]|{(f)} &&D  \ar[d]  \\
 Y^{[n]}\ar[drr]_{q^{[n]}} \ar[rr]^{\gamma_2} &&\ar[d] E\\
&&   A^{[n]} 
 }\end{equation}
 Hence the gap map of the the square (\ref{gapmapslice}) is a base change
 of the gap map of the square $(f)$. Hence it suffices to show that the
 gap map of the square $(f)$ is a fibration.
 But $\gamma_0$ is the gap map of the square  $(f)$, since the square $(c)$
 of diagram (\ref{1943}) is cartesian.
The map $\gamma_0$ is also the gap map of the square  $(a)$ in the diagram (\ref{1943}),
since the square $(d)$ is cartesian. Thus $\gamma_0$ is a fibration, since 
$\mathcal{E}$ is a simplicial clan. We have proved that the
 gap map of square (\ref{gapmapslice}) is a fibration.
\end{proof}

\medskip

\begin{defi}\label{morphsimpclan} 
If $\mathcal{E}$ and $ \mathcal{E}'$ are simplicial clans,
we say that a simplicial functor  $F:\mathcal{E}\to \mathcal{E}'$
is a {\it morphism of simplicial clans} if it
preserves finite cotensors and it is a morphism of of clans.
\end{defi}

We shall denote by $\mathbf{sClan}$
the category of small simplicial clans and morphisms of simplicial clans.
The category $\mathbf{sClan}$ has the structure of a 2-category
where a 2-cell is a strong natural transformation.

\begin{prop}\label{cotensorsimpenriched} If $\mathcal{E}$ is a simplicial clan,
then the simplicial functor $(-)^K:\mathcal{E}\to \mathcal{E}$
is a morphism of simplicial clans for every finite simplicial set $K$.
\end{prop}

 If $\mathcal{E}$ is a simplicial clan,
then the base change functor $f^\star:\mathcal{E}_0(B)\to \mathcal{E}_0(A)$
is simplicially enriched for every map  $f:A\to B$ in $\mathcal{E}$.
We shall denote
the resulting simplicial functor by $f^\star:\mathcal{E}(B)\to \mathcal{E}(A)$.

\begin{prop}\label{basechangesimpclan} 
If $\mathcal{E}$ is a simplicial clan,
then the base change functor $f^\star:\mathcal{E}(B)\to \mathcal{E}(A)$
is a morphism of simplicial clans for any map $f:A\to B$ in $\mathcal{E}$.
\end{prop}

\subsection{Simplicial $\pi$-clans}

Let $f:A\to B$ be a fibration in a clan $\mathcal{E}$.
Recall from Definition \ref{definitiontribe} that the {\it internal product} of an object $E=(E,p)\in \mathcal{E}(A)$
along the map $f$  is defined to be an object $P=(P,q)\in \mathcal{E}(B)$
equipped with a map $\epsilon:f^\star(P)\to E$  in   $\mathcal{E}(A)$
which is cofree with respect to the functor $f^\star:\mathcal{E}/B \to \mathcal{E}/A$.

\begin{defi}\label{def:simpproduct}
Let $f:A\to B$ be a fibration in a simplicial clan $\mathcal{E}$
and let $P\in \mathcal{E}(B)$ be the internal
product of an object $E=(E,p)\in \mathcal{E}(A)$ along the map $f$
and let $\epsilon:f^\star(P)\to E$ be the evaluation.
We shall say that the pair $(P,\epsilon)$ is a {\it strong internal product} 
if the map $\epsilon:f^\star(P)\to E$ is cofree
with respect to the {\it simplicial} functor $f^\star:\mathcal{E}/B \to \mathcal{E}/A$.
\end{defi}

By definition, the pair $(P,\epsilon)$ is a strong internal product if the map
$$
\xymatrix{
(\mathcal{E}/B)(Q,P)\ar[rr]^{f^\star} && (\mathcal{E}/A)(f^\star Q,f^\star P) \ar[rr]^{\epsilon\circ (-)} &&  (\mathcal{E}/A)(f^\star Q,E)  
}
$$
is an {\it isomorphism of simplicial sets} for every object $C=(C,g)\in \mathcal{E}/B$.

\bigskip

Recall that a fibration in a simplicial clan is $s$-carrable.

\begin{defi}\label{definitionsimppiclan}
We say that a simplicial clan $\mathcal{E}$ is a {\bf simplicial $\pi$-clan} if it is a $\pi$-clan
and  the internal product $\Pi_f(E)=\Pi_f(E,p)$ of every object $E=(E,p)\in \mathcal{E}(A)$
is strong for every fibration $f:A\to B$.
\end{defi}

\begin{prop}\label{localsimppiclan} If $\mathcal{E}$ is a simplicial $\pi$-clan, then so is the
simplicial clan $\mathcal{E}(A)$ for every object $A\in \mathcal{E}$.
\end{prop}

\begin{lemma}\label{basechangesimpenriched} If $f:A\to B$
is a fibration in a simplicial $\pi$-clan $\mathcal{E}$,
then the simplicial functor $$\Pi_f:\mathcal{E}(A)\to \mathcal{E}(B)$$
is right adjoint to the simplicial functor $f^\star:\mathcal{E}(B)\to \mathcal{E}(A)$
and it is a morphism of simplicial clans. 
Moreover, the
functor $[A,-]: \mathcal{E}\to \mathcal{E}$
is a morphism of simplicial clans for every object $A\in \mathcal{E}$.
\end{lemma}

\begin{defi}\label{morphsimppiclan} 
If $\mathcal{E}$ and $ \mathcal{E}'$ are simplicial $\pi$-clans,
we say that a morphism of simplicial clans $F:\mathcal{E}\to \mathcal{E}'$
is {\it $\pi$-closed}, or that it is a {\it morphism of simplicial $\pi$-clans}, if it preserves internal products.
\end{defi}

\begin{prop}\label{cotensorsimpenrichedpiclan} If $\mathcal{E}$ is a simplicial $\pi$-clan,
then the simplicial functor $(-)^K:\mathcal{E}\to \mathcal{E}$
is a morphism of simplicial $\pi$-clans for every finite simplicial set $K$.
\end{prop}

\begin{prop}\label{basechangesimppiclan} If $\mathcal{E}$ is a simplicial $\pi$-clan,
then the base change functor $f^\star:\mathcal{E}(B)\to \mathcal{E}(A)$
is a morphism of simplicial $\pi$-clans for any map $f:A\to B$ in $\mathcal{E}$.
\end{prop}

\subsection{Simplicial tribes}

\begin{defi}\label{defsimptribe} We say that a simplicial clan $\mathcal{E}$ is a 
 {\it simplicial tribe} if it is a tribe and the map $\sigma_0:X\to X^{[1]}$
 is anodyne for every object $X$.
 \end{defi}

\begin{lemma}\label{pathsimpclan} If $X$ is an object in a simpicial tribe $\mathcal{E}$,
then the quadruple $PX=(X^{[1]}, \partial_0,\partial_1, \sigma_0)$ is a path object for $X$.
\end{lemma}

\begin{proof} We have $(\partial_0,\partial_1)\sigma_0=(1_X,1_X)$ by the functoriality of the 
cotensor product $K\mapsto X^K$.
The map $(\partial_0,\partial_1):X^{[1]}\to X\times X$
is a fibration by Corollary \ref{expmono}, since $\mathcal{E}$ is a simpliclal clan.
Moreover, the map $\sigma_0:X\to X^{\Delta[1]}$ is anodyne by Definition \ref{defsimptribe}.
Thus, the quadruple $PX=(X^{[1]}, \partial_0,\partial_1, \sigma_0)$ is a path object for $X$
by Definition \ref{pathobjecetdef}.
\end{proof}

\medskip

We saw in Proposition \ref{localsimpclan} that if $\mathcal{E}$ is a simplicial clan, then so is the
simplicial category $\mathcal{E}(A)$ for every object $A\in \mathcal{E}$.
A map $f:(X,p)\to (Y,q)$ in  $\mathcal{E}(A)$ is a fibration if $f$ is a fibration in $\mathcal{E}$.

\begin{prop}\label{localsimptribe} If $\mathcal{E}$ is a simplicial tribe, then so is the
simplicial clan $\mathcal{E}(A)$ for every object $A\in \mathcal{E}$.
\end{prop}

\begin{proof} The clan $\mathcal{E}(A)_0=\mathcal{E}_0(A)$
is a tribe, since the clan $\mathcal{E}_0$ is a tribe.
 If $(X,p)\in \mathcal{E}(A)$, let us show that 
 the diagonal $\sigma:(X,p)\to (X,p)^{[1]}$ is anodyne.
 The map $p^{[1]}: X^{[1]}\to A^{[1]}$ is a fibration in $\mathcal{E}_0$, since $p:X\to A$ is a fibration  in $\mathcal{E}_0$.
The object $(X,p)^{[1]}\in \mathcal{E}(A)$ and the map $\sigma$
are constructed the following commutative diagram with a 
pullback square
$$
\xymatrix{
X   \ar@/_1.5pc/[ddr]_-{p}   \ar@/^1.5pc/[drr]^-{\sigma_0} \ar[dr]^\sigma &  &  \\
& (X,p)^{[1]} \ar[d]_{p_1}  \ar[r]^{p_2}  &  X^{[1]}  \ar[d]^{p^{[1]}}    \\
& A \ar[r]^{\sigma_0}  &  A^{[1]} 
}
$$ 
The map $\sigma_0: A\to   A^{[1]} $
and $\sigma_0: X\to   X^{[1]} $ are anodyne,
since $\mathcal{E}$ is a simplicial tribe.
The map $p_2: (X,p)^{[1]}\to   X^{[1]} $ is also anodyne
by base change, since $\sigma_0: A\to   A^{[1]} $ is anodyne.
It then follows by Lemma \ref{anodyne closure}  that $\sigma$ is anodyne,
since $p_2\sigma =\sigma_0: X\to   X^{[1]} $ is anodyne.
\end{proof}

\begin{lemma}\label{simpcaractanodyne} A map $u:A\to B$ in a simplicial tribe
 is anodyne if and only if  there exists a pair of maps $r:B\to A$
 and $h:B\to B^{[1]}$ such that $ru=1_A$  and
the following diagram commutes 
\begin{equation} \label{diagforstrongdefretracts}
\xymatrix{
A \ar[d]_u  \ar[r]^{\sigma_0}  &  A^{[1]}  \ar[r]^{  u^{[1]}  }   & B^{[1]}  \ar[d]^{(\partial_0,\partial_1)} \\
B \ar[rur]_{h}   \ar[rr]_-{(ur,1_B)} && B\times B
}
\end{equation}
\end{lemma}

\begin{proof} By proposition \ref{strongdefoormation},
a map $u:A\to B$ in a tribe is anodyne if and only if it is a strong deformation retract.
This means that 
there exists a map $r:B \to A$ such that $ru=1_A$ together with
a homotopy $h:ur\leadsto 1_B$ 
such that the homotopy $hu:u\leadsto u$ is the identity homotopy $\sigma_0 u:u\leadsto u$.
The homotopy $h:ur\leadsto 1_B$ is a map $B\to  B^{[1]}$
such that $(\partial_0,\partial_1)h=(ur,1_B)$.
The condition $hu=\sigma_0 u$ is equivalent to the condition 
$hu=u^{[1]} \sigma_0$, since the following square commutes,
$$
\xymatrix{
A \ar[d]_u  \ar[r]^{\sigma_0}  &  A^{[1]}  \ar[d]^{  u^{[1]}  }   \\
B \ar[r]^{\sigma_0}  &  B^{[1]} 
}
$$ 
\end{proof}

\begin{lemma}\label{morphsimptribe} If $\mathcal{E}$ and $ \mathcal{E}'$ are simplicial tribes,
then any morphism of simplicial clans  $F:\mathcal{E}\to \mathcal{E}'$
preserves anodyne maps; it is therefore a morphism of tribes.
\end{lemma}

\begin{proof} The functor $F$ preserves cotensor products, 
since it is a morphism of simplicial clans.
If $u:A\to B$
is anodyne, then there exists a pair of maps $r:B\to A$
 and $h:B\to B^{[1]}$ satisfying the
 conditions of Lemma \ref{simpcaractanodyne}.
 The pair of maps $F(r):F(B)\to A$
 and $F(h):F(B)\to F(B)^{[1]}$ 
 then satisfies the same conditions.
Thus, $F(u)$ is anodyne by Lemma \ref{simpcaractanodyne}. 
\end{proof}

\begin{defi} If $\mathcal{E}$ and $ \mathcal{E}'$ are simplicial tribes,
we say that a simplicial functor  $F:\mathcal{E}\to \mathcal{E}'$
is a {\it morphism of simplicial tribes} if it 
is a morphism of simplicial clans.
\end{defi}

We shall denote by $\mathbf{sTrib}$
the category of small simplicial tribes and morphisms between them.
The category $\mathbf{sTrib}$ has the structure of a 2-category
where a 2-cell is a strong natural transformation.

\begin{prop}\label{cotensorsimpenrichedtribe} If $\mathcal{E}$ is a simplicial tribe,
then the simplicial functor $(-)^K:\mathcal{E}\to \mathcal{E}$
is a morphism of simplicial tribes for every finite simplicial set $K$.
\end{prop}

\begin{proof} This follows from Proposition \ref{cotensorsimpenriched}
since a morphism of simplicial clans between simplicial
tribes is a morphism of simplicial tribes.
\end{proof}

\begin{prop}\label{basechangesimptribe} If $\mathcal{E}$ is a simplicial tribe,
then the base change functor $f^\star:\mathcal{E}(B)\to \mathcal{E}(A)$
is a morphism of simplicial tribes for any map  $f:A\to B$ in $\mathcal{E}$.
\end{prop}

\begin{proof} This follows from Proposition \ref{basechangesimpclan},
since a morphism of simplicial clans between simplicial
tribes is a morphism of simplicial tribes.
\end{proof}

\subsection{Simplicial $\pi$-tribes}

\begin{lemma}\label{simppitribe} If a simplicial tribe $\mathcal{E}$
is a $\pi$-clan, then it is a $\pi$-tribe.
  \end{lemma}

\begin{proof} It suffices to show that if $f:A\to B$ is a fibration,
then the functor  $\Pi_f:\mathcal{E}(A)\to \mathcal{E}(B)$
takes anodyne maps to anodyne maps.
But the simplicial functor $\Pi_f$ preserves cotensor products,
since it is a right adjoint. It thus preserves
anodyne maps by Lemma \ref{simpcaractanodyne}.
\end{proof}

\begin{defi}\label{definitionpitribe} 
We say that a simplicial tribe $\mathcal{E}$
is a {\bf simplicial $\pi$-tribe} if it is a $\pi$-clan.
\end{defi}

\begin{prop}\label{localsimppitribe} If $\mathcal{E}$ is a simplicial $\pi$-tribe, then so is the
simplicial tribe $\mathcal{E}(A)$ for every object $A\in \mathcal{E}$.
\end{prop}

\begin{lemma}\label{lem:productisamorphismofsimptribes}  If $f:A\to B$
is a fibration in a simplicial $\pi$-tribe $\mathcal{E}$,
then the simplicial functor $$\Pi_f:\mathcal{E}(A)\to \mathcal{E}(B)$$
is right adjoint to the simplicial functor $f^\star:\mathcal{E}(B)\to \mathcal{E}(A)$
and it is a morphism of simplicial tribes. 
Moreover, the
functor $[A,-]: \mathcal{E}\to \mathcal{E}$
is a morphism of simplicial tribes for every object $A\in \mathcal{E}$.
\end{lemma}

\begin{lemma}\label{morphsimppiclanaretribe} If $\mathcal{E}$ and $ \mathcal{E}'$ are simplicial $\pi$-tribes,
then any morphism of simplicial $\pi$-clans  $F:\mathcal{E}\to \mathcal{E}'$
preserves anodyne maps; it is therefore a morphism of $\pi$-tribes.
\end{lemma}

\begin{proof} Similar to the proof of Lemma \ref{morphsimptribe}.
\end{proof}

\begin{defi} If $\mathcal{E}$ and $ \mathcal{E}'$ are simplicial $\pi$-tribes,
we say that a simplicial functor  $F:\mathcal{E}\to \mathcal{E}'$
is a {\it morphism of simplicial $\pi$-tribes} if it 
is a morphism of simplicial $\pi$-clans.
\end{defi}

\begin{prop}\label{cotensorsimpenrichedpitribe} If $\mathcal{E}$ is a simplicial $\pi$-tribe,
then the simplicial functor $(-)^K:\mathcal{E}\to \mathcal{E}$
is a morphism of simplicial $\pi$-tribes for every finite simplicial set $K$.
\end{prop}

\begin{prop}\label{basechangesimppitribe} If $\mathcal{E}$ is a simplicial $\pi$-tribe, 
then the base change functor $f^\star:\mathcal{E}(B)\to \mathcal{E}(A)$
is a morphism of simplicial $\pi$-tribes for any map $f:A\to B$ in $\mathcal{E}$.
\end{prop}

\newpage

\newpage
\section{Appendix on category theory}

\subsection{Grothendieck fibrations}

We first recall the notion of Grothendieck fibrations.
If $P:\mathcal{E}\to  \mathcal{B}$ is a functor,
and $f:S \to T$ is a map in $\mathcal{E}$, then
 the following square commutes, where 
 the vertical sides are the maps induced by $P$.
\begin{equation}\label{defcartmorphism}
\xymatrix{
\mathcal{E}(C,S)\ar[rr]^{ \mathcal{E}(C,f)}  \ar[d]&&\mathcal{E}(C,T)  \ar[d]   \\
\mathcal{B}(PC,PS) \ar[rr]^{ \mathcal{E}(PC,P(f))}  &&\mathcal{B}(PC,PT)
}
\end{equation}
This means that for any map $u:PC\to PS$
and any map $g:C\to T$ such that $P(g)=P(f)u$, there exists a unique
map $v: C\to S$ such that $P(v)=u$ and $g=fu$.
$$\xymatrix{
C  \ar@{--}[dd]   \ar@{..>}[dr]_(0.5)v  \ar@/^1.5pc/[drrr]^g&&&\\
&S  \ar@{--}[dd]  \ar[rr]_{f}& & T \ar@{--}[dd] \\
PC  \ar@/^1.5pc/[drrr]^{F(g)} \ar[dr]_u &&&\\
&PS \ar[rr]_{P(f)} && PT 
}$$ 

The functor $P$ is said to be a {\it Grothendieck fibration}
if for every object $T\in \mathcal{E}$ 
and every arrow $u:A\to PT$ with target $PT$
there exist a cartesian arrow $f\in \mathcal{E}$ 
with target $T$ such that $P(f)=g$.
$$\xymatrix{
S  \ar@{--}[d]  \ar[rr]^{f}& & T \ar@{--}[d] \\
A \ar[rr]^{u} && PT 
}$$ 
We shall say that the arrow $f$ is a {\it cartesian lift} of the arrow $u$.

\medskip

If $F:\mathcal{E} \to \mathcal{B}$ is a Grothendieck fibration,
we shall denote by $F^{-1}A$ the {\it fiber} of $F$ at an object $A\in \mathcal{B}$.
A map $f:X\to Y$ in $\mathcal{E}$ is {\it over} 
a map $u:A\to B$ in $\mathcal{B}$ if $F(f)=u$.
$$\xymatrix{ 
X  \ar[rr]^{f}   \ar@{--}[d]   && Y \ar@{--}[d]  \\
A \ar[rr]^{u} && B\\ 
}
$$
A map $f:X\to Y$ in $\mathcal{E}$ is said to be a $F$-{\it unit}, 
if $F(f)$ is a unit. 
Recall that for every map $u:A\to B$  in  $\mathcal{B}$ and for every object $X\in F^{-1}(B)$,
there exists a cartesian morphism $v:X'\to X$ over $u$.
The pair $(X',v)$ is unique up to a unique isomorphism in $F^{-1}(A)$,
and we shall denote it by $(u^{\star}(X), u_X)$.
$$\xymatrix{ 
u^\star(X)  \ar[rr]^{u_X}   \ar@{--}[d]   && X \ar@{--}[d]  \\
A \ar[rr]^{u} && B\\ 
}
$$
The {\it base change functor} $u^\star:F^{-1}(B)\to F^{-1}(A)$
is obtained by choosing a cartesian morphism $u_X:u^\star(X)\to X$
 for each object $X\in F^{-1}(B)$. By construction,
if $g:X\to Y$ is a map in $F^{-1}(B)$, then $u^\star(g)$
is the unique map $u^\star(X)\to u^\star(Y)$
 in $F^{-1}(A)$ such that the following square commutes
\begin{equation} \label{cartsquarebasechange2}
\xymatrix{
u^\star(X)  \ar[d]_{u^\star(g)} \ar[rr]^{u_X}&& X  \ar[d]^g  \\
u^\star(Y)    \ar[rr]^{u_Y}& & Y 
}
\end{equation}
The square is actually cartesian in the category $\mathcal{E}$.
Every map $f:X\to Y$ in $\mathcal{E}$ admits a factorization
$f=f^\sharp f_\sharp$, with $f_\sharp$ a $P$-unit
and $f^\sharp$ a cartesian morphism,
$$\xymatrix{ 
X   \ar[d]_-{f_\sharp}    \ar[drr]^-{f}  && \\
\ar@{--}[d] \ar[rr]_{f^\sharp}  F(f)^\star(Y)  &&Y \ar@{--}[d] \\
 FX \ar[rr]^{F(f)} && FY\\ 
}
$$

\subsection{Free and cofree objects}

Let $F:\mathcal{C}\to \mathcal{D}$ be a functor between two categories.
If $S\in \mathcal{C}$  and $X\in  \mathcal{D}$ we shall say that 
a map $\epsilon:F(S)\to X$ is  {\it cofree with respect to the functor $F$}, or {\it $F$-cofree},
if for very object $T\in   \mathcal{C}$ and every map
$v:F(T)\to X$ in $ \mathcal{D}$ there exists a unique map
$u:T\to S$ in $ \mathcal{C}$ such $\epsilon F(u) =v$.
$$\xymatrix{
F(S) \ar[rr]^{\epsilon} && X \\
F(T) \ar[u]^{F(u)} \ar[rru]_v &&
}$$
Suppose that for every object $X \in  \mathcal{D}$, there exists an object $U(X) \in \mathcal{C} $ 
equipped with a $F$-cofree map  $\epsilon_X:F(U(X))\to X$.
Then for every map $g:X\to Y$ in $ \mathcal{D}$
 there is a unique map $U(g):U(X)\to U(Y)$  in $ \mathcal{C}$
such that the following square commutes,
$$\xymatrix{
F(U(X))  \ar[d]_{F(U(g))} \ar[rr]^{\epsilon_X} && X  \ar[d]^g    \\
F(U(Y))   \ar[rr]^{\epsilon_Y} && Y\ . 
}$$
This defines a functor $U: \mathcal{D}\to  \mathcal{C}$ right adjoint to the functor $F$.
The counit of the adjunction $F\dashv U$ is
the natural transformation $\epsilon:FU \to Id$
defined by the maps $\epsilon_X:F(U(X))\to X$.
Hence the map 
$$\theta: Hom_\mathcal{C}(S,U(X)) \to  Hom_\mathcal{D}(F(S),X)$$
defined by putting 
$\theta(u)=\epsilon_X F(u)$ is bijective.

\medskip

In type theory, the right adjoint $U$ to a functor $F:\mathcal{C}\to \mathcal{D}$  is usually described by the following set of rules:
\begin{itemize}
\item{} an operation $X\mapsto U(X)$
called the {\it $U$-formation rule};
\item{} an operation $\lambda:Hom_\mathcal{D}(F(S),X) \to  Hom_\mathcal{C}(S,U(X))$
called the {\it $U$-introduction rule};
\item{} an operation which associates an object $X$ a morphism $\epsilon_X: FU(X)\to X$
called the {\it $U$-elimination rule};
\item{} the condition $\theta \lambda (u)=u$
called the {\it $U$-computation rule};
\item{} the condition $\lambda\theta (v)=v$
called the {\it $U$-uniqueness principle}.
\end{itemize}

\medskip

\begin{rem} A word about the terminology.  The operation $\lambda$ is said to be an {\it introduction} rule because 
it takes a morphism $u:F(S)\to X$ to a morphism $\lambda(u):S\to U(X)$ with target
$U(X)$ ($\lambda(u)$ is {\it introducing} something in the object $U(X)$). Similarly, 
the operation $X\mapsto \epsilon_X$
is said to be an {\it elimination} rule because the map $\epsilon_X: FU(X)\to X$ is $FU(X)$ is {\it eliminating} something from the object $U(X)$). 
\end{rem}

\medskip

For example, in a category $\mathcal{E}$ with binary cartesian product,
the product functor $\sqcap:\mathcal{E}\times \mathcal{E} \to \mathcal{E}$
is right adjoint to the diagonal functor $\mathcal{E}\to \mathcal{E}\times \mathcal{E}$,
\begin{itemize}
\item{} the $\sqcap$-fomation rule associates to a pair of objects $(A,B)\in \mathcal{E}\times \mathcal{E} $,
an object $A\sqcap B$;
\item{} the $\sqcap$-introduction rule 
associates to a pair
of morphisms $f:C\to A$ and $g:C\to B$, a morphism $(f,g):C\to A\times B$;
\item{} the $\sqcap$-elimination rule associates to a pair of objects $(A,B)$,
a pair of morphisms $p_1:A\sqcap B\to A$ and $p_2:A\sqcap B\to B$;
\item{} the $\sqcap$-computation rule asserts that $p_1( f,g)=f$
and $p_2( f,g)=g$;
\item{} $\sqcap$-uniqueness principle asserts that 
$(p_1h, p_2h)=h$ for every morphism $h:C\to A\times B$.
\end{itemize}

\medskip

A carrable object $A$ in a category $\mathcal{C}$ is exponentiable if and only if the
functor $(-)\times A:\mathcal{C} \to \mathcal{C}$
has a right adjoint $[A,-]:\mathcal{C}\to \mathcal{C}.$
In type theory, the right adjoint $[A,-]$ is usually described by the following set of rules.
\begin{itemize}
\item{} the formation rule associates  to each object $B\in \mathcal{E}$ an object $[A,B]$;
\item{} the introduction rule associates 
 to each map $f:C\times A\to B$ a new map $\lambda^A(f):C\to [A,B]$;
\item{} the elimination rule associates
 to each object $B$ a map $\epsilon_B:[A,B]\times A\to B$;
\item{} the computation rule asserts that 
$\epsilon_B(\lambda^A(u)\times A)=u$;
\item{} the uniqueness principle asserts that 
$ v=\lambda^A(\epsilon_B(v\times A))$.
\end{itemize}

\bigskip

Dually, let $U:\mathcal{D}\to \mathcal{C}$ be a functor between two categories.
If $S\in \mathcal{C}$ and $X\in \mathcal{D}$ we shall say
that a map $\eta:S\to U(X)$ is {\it $U$-free} 
if for very object $Y\in   \mathcal{D}$ 
and every map
$u:S\to U(Y)$ in $ \mathcal{C}$ there exists a unique map
$v:X\to Y$ in $ \mathcal{D}$ such $U(v)\eta =u$,
$$\xymatrix{
S\ar[drr]_u \ar[rr]^{\eta} && U(X) \ar[d]^{U(v)} \\
 && U(Y) \ .
}$$
Suppose that for every object $S \in  \mathcal{C}$, there exists an object $F(S)\in \mathcal{D} $ 
equipped with a $U$-free map $\eta_S:S\to U(F(S))$.
Then for every map $f:S\to T$ in $ \mathcal{C}$
 there is a unique map $F(f):F(S)\to F(T)$  in $ \mathcal{D}$
such that the following square commutes,
$$\xymatrix{
S\ar[d]_f \ar[rr]^{\eta_S} && U(F(S)) \ar[d]^{U(F(f))} \\
T \ar[rr]^{\eta_T} && U(F(T)) 
}$$
This defines a functor $F: \mathcal{C}\to  \mathcal{D}$ left adjoint to
the functor $U$. The unit of the adjunction $F\dashv U$
is the natural transformation $\eta:Id \to UF$
defined by the maps $\eta_S: S \to UF(S)$.
Hence the map 
$$\theta: Hom_\mathcal{D}(F(S),X)\to  Hom_\mathcal{C}(S,U(X))$$
defined by putting 
$\theta(v)=U(v)\eta_S$ is bijective.

\medskip

In type theory the left adjoint $F$ to a functor $U:\mathcal{D}\to \mathcal{C}$ is normally described by a set of rules:
\begin{itemize}
\item{} a map $S\mapsto F(S)$ called the {\it $F$-formation rule};
\item{} an operation $\sigma : Hom_\mathcal{C}(S,U(X)) \to Hom_\mathcal{D}(F(S),X)$
called the {\it $F$-elimination rule};
\item{} a family of morphism $\eta_S: S\to UF(S)$ called the {\it $F$-introduction rule};
\item{} the condition $\theta \sigma(u)=u$ called the {\it $F$-computation rule};
\item{} the condition $\sigma\theta (v)=v$
is called the {\it $F$-uniqueness principle}.
\end{itemize}

\medskip

\begin{rem} About the terminology.  The operation $\sigma$  is said to be an {\it elimination} rule because 
it takes a morphism $u:S\to U(X)$ to
a morphism $\sigma(u):F(S) \to X$ (it is {\it eliminating} something from the object $F(S)$). 
Similarly, the operation $S\mapsto \eta_S$
is said to be an {\it introduction} rule because the map $\eta_S: S\to UF(S)$
is {\it introducing} something in the object $F(S)$).
\end{rem}

\medskip

For example, a category $\mathcal{E}$ has binary coproducts iff the diagonal
functor $\Delta:\mathcal{E}\to \mathcal{E}\times \mathcal{E}$
has a left adjoint $\sqcup:\mathcal{E}\times \mathcal{E} \to \mathcal{E}$.
\begin{itemize}
\item{} the $\sqcup$-fomation rule associates to a pair of objects $(A,B)\in \mathcal{E}\times \mathcal{E} $,
an object $A\sqcup B$;
\item{} $\sqcup$-introduction rule associates to a pair of objects $(A,B)$,
a  pair of morphisms $i_1:A\to A\sqcup B$ and $i_2:A\to A\sqcup B$;
\item{} the $\sqcup$-elimination rule associates to a pair
of morphisms $f:A\to C$ and $g:B\to C$, a morphism $\langle f,g\rangle:A\sqcup B\to C$;
\item{} the $\sqcup$-computation rule asserts that $\langle f,g\rangle i_1=f$
and $\langle f,g\rangle i_2=g$;
\item{} $\sqcup$-uniqueness principle asserts that 
$\langle hi_1, hi_2 \rangle =h$ for every morphism $h:A\sqcup B\to C$.
\end{itemize}

\subsection{Cartesian squares}

Recall that a diagram 
$$
\xymatrix{
A & C  \ar[l]_{p}   \ar[r]^{q}& B  
}
$$
in a category $\mathcal{C}$, is said to {\it exhibit} the object $C$ as the {\it cartesian product}
of $A$ with $B$ if for every object $K\in \mathcal{C}$
and every pair of maps $a:K\to A$ and $b:K\to B$, there exists a unique map $c:K\to C$
such that $pc=a$ and $qc=b$,
$$\xymatrix{
 & K \ar@{..>}[d]^c  \ar[dl]_a \ar[dr]^b &\\
A &C  \ar[l]^{p} \ar[r]_{q}&  B 
}$$ 
In which case we often put $(A\times B,p_1,p_2)\defeq (C,p,q)$,
the map $p_1:A\times B\to A$ is called
the  {\it first projection} and the map $p_2:A\times B\to B$
 the  {\it second projection}. Moreover, the unique map $c:K\to A\times B$
 such that $p_1c=a$ and $p_2c=b$ is written as a pair $c=(a,b):K\to A\times B$.

\medskip

Recall that a commutative square in a category $\mathcal{C}$ 
\begin{equation}\label{pullbacksquaredef}
\xymatrix{
C  \ar[d]_{p} \ar[rr]^{q}& & B \ar[d]^{v} \\
A \ar[rr]^{u} && E
}
\end{equation} 
can be viewed as a diagram 
\begin{equation}\label{productdiagram}
\xymatrix{
(A,u)  & (C,w)  \ar[l]_{p} \ar[r]^{q}& (B,v)  
}
\end{equation}
in the category $\mathcal{C}/E$, where  $w\defeq up=vq$.
The square (\ref{pullbacksquaredef}) is said to be {\it cartesian}, or to be a {\it pullback},
if (\ref{productdiagram})
is a product diagram  in $\mathcal{C}/E$.
This means that
for every object $K\in \mathcal{C}$ and every 
pair of maps $a:K\to A$ and $b:K\to B$
such that $ua=vb$, there exists a unique map $c:K\to C$
such that $pc=a$ and $qc=b$. 
$$\xymatrix{
K \ar@{..>}[dr]^(0.65)c  \ar[ddr]_a \ar[drrr]^b &&&\\
&C  \ar[d]^{p} \ar[rr]_{q}& & B \ar[d]^{v} \\
&A \ar[rr]^{u} && E 
}$$ 
The object $C$ in a cartesian square (\ref{pullbacksquaredef}) is said to be the {\it fiber product} of the maps $u:A\to E$ and $v:B\to E$
and we my write that $(C, p, q)=(A\times_E B, p_1, p_2)$. 
$$\xymatrix{
A\times_E B  \ar[d]_{p_1} \ar[rr]^{p_2}& & B \ar[d]^{v} \\
A \ar[rr]^{u} && E
}$$ 
The map $p_1:A\times_E B \to A$ is said to be
the {\it base change} of the map $v:B\to E$
{\it along} the map $u:A\to E$.
Dually, the map $p_2:A\times_E B \to B$
is said to be the {\it base change} of $u:A\to E$
along $v:B\to E$.

\begin{lemma} \label{lemmacartesiansq} Suppose that we have a commutative diagram
$$\xymatrix{
A'  \ar[d] \ar[rr]^{u'}& & B' \ar[d] \ar[rr]^{v'} && C' \ar[d]\\
A \ar[rr]^{u} && B \ar[rr]^{v}  && C
}  $$ 
in which right hand square is cartesian.
Then the left hand square is cartesian 
if and only if the composite square is cartesian,
$$\xymatrix{
A'  \ar[d] \ar[rr]^{v'u'}  && C' \ar[d]\\
A  \ar[rr]^{vu}  && C 
}$$ 
\end{lemma}

\begin{proof}
Left as an exercise to the reader.
\end{proof}

\begin{lemma} \label{cubelemmacartesiansq} 
Suppose that we have commutative cube
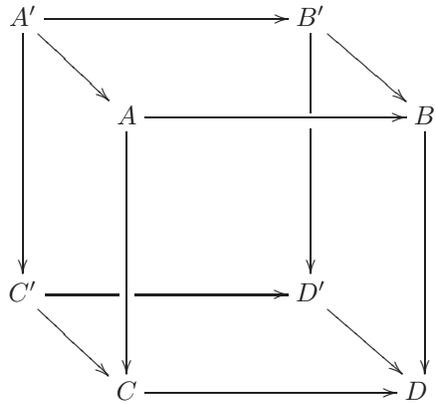
\[
\xymatrix{
A' \ar[dr] \ar[rrr] \ar[ddd] & &  & B' \ar '[d] [ddd]^(0,3){}   \ar[dr]   &  \\
  & A \ar[ddd] \ar[rrr] &  & &  B\ar[ddd] \\ 
  &&&&\\
C' \ar '[r] [rrr]^(0.3){} \ar[dr] & &  &D' \ar[dr] &   \\
  & C  \ar[rrr]^{} &  & & D  \ .}
\]
in which the left and right vertical faces are cartesian.
If the front face is cartesian, then the back face is cartesian.
\end{lemma}

\begin{proof}
The diagonal square $A'BC'D$ is cartesian by lemma \ref{lemmacartesiansq}, since the left vertical face and the front face are cartesian,
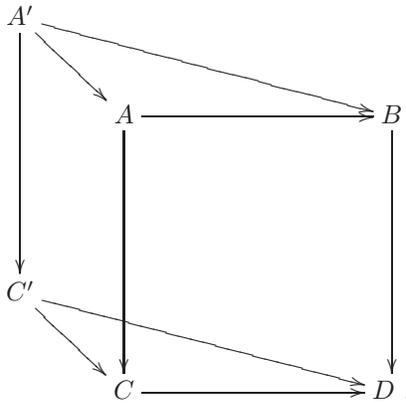
\[
\xymatrix{
A' \ar[dr] \ar[rrrrd] \ar[ddd] & &  &   &  \\
  & A \ar[ddd] \ar[rrr] &  & & B \ar[ddd] \\ 
  &&&&\\
C' \ar[drrrr]  \ar[dr] & &  & &   \\
  & C  \ar[rrr]^{} &  & & D   \ .}
\]
But the diagonal square is also the composite 
of the back face with the right vertical face:
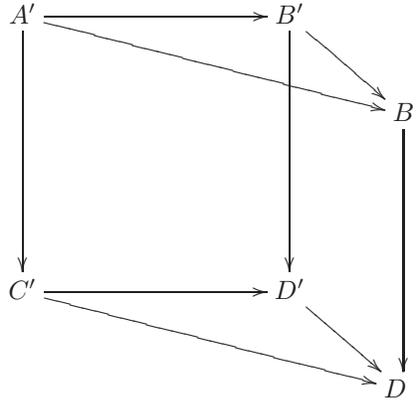
\[
\xymatrix{
A' \ar[rrr] \ar[ddd] \ar[rrrrd] & &  & B' \ar[ddd]^(0,3){}   \ar[dr]   &  \\
  &  &  & &  B \ar[ddd] \\ 
  &&&&\\
C' \ar[rrr]^(0.3){} \ar[rrrrd] & &  &D' \ar[dr] &   \\
  &&  & & D   \ .}
\]
It then follows from lemma  \ref{lemmacartesiansq} that the 
back face is cartesian, since the right vertical face is cartesian.
\end{proof}

\subsection{Carrable objects and maps}

\begin{defi} \label{defcarrableobject} We say that an object $B$  in a category $\mathcal{C}$ is {\bf carrable}
if the cartesian product $A\times B$ exists for any object $A\in \mathcal{C}$, 
 $$\xymatrix{
A & \ar[l]_{p_1} A\times B    \ar[r]^(0.6){p_2}& B.
}$$ 
\end{defi}

It is easy to see that an object $B$  is carrable
if and only if the forgetful functor $\mathcal{C}/B\to \mathcal{C}$
has a right adjoint $e_B:  \mathcal{C} \to \mathcal{C}/B$; by construction,
$e_B(X)=(B\times X,p_1)$ for every $X\in \mathcal{C}$.

\medskip

\begin{defi}  \label{defcarrablemap} We say that a map $p:X\to B$  in a category $\mathcal{C}$
 is {\bf carrable} if the object $(X,p)$ of the category $ \mathcal{E}/B$
 is carrable.  
\end{defi}

A map $p:X\to B$ is carrable if and only if the fiber product
$$\xymatrix{
A\times_B X  \ar[d]_{p_1}\ar[rr]^{p_2}& & X \ar[d]^{p} \\
A \ar[rr]^{f}&& B
}$$ 
exists for every map $f:A\to B$.

\medskip

Recall that a map $f:A\to B$ in a category $\mathcal{C}$ induces a {\bf push-forward functor} 
$f_!:\mathcal{C}/A\to \mathcal{C}/B$ defined by putting $f_!(X,p)=(X,fp)$ for a map $p:X\to A$.
$$\xymatrix{
X  \ar[d]_{p}\ar@{=}[rr] & & X \ar[d]^{fp} \\
A \ar[rr]^{f}&& B
}$$

\begin{prop} \label{carrablerightadjoint} A map $f:A\to B$  in a category $\mathcal{C}$
 is carrable if and only if the push-forward functor $f_!:\mathcal{C}/A\to \mathcal{C}/B$
 admits a right adjoint $f^\star:\mathcal{C}/B\to \mathcal{C}/A$.
\end{prop}

If $f:A\to B$ is cartable, then $f^\star:\mathcal{C}/B\to \mathcal{C}/A$ is the
 {\bf base change functor} along $f$. By construction 
 $f^\star(Y)=(A\times_B Y,p_1)$ for $Y=(Y,q)\in \mathcal{C}/B$.
We shall put $f^\star(q)\defeq p_1$ and $f_Y\defeq p_2$,
$$\xymatrix{
f^\star(Y)  \ar[d]_{f^\star(q)}\ar[rr]^{f_Y}& & Y \ar[d]^{q} \\
A \ar[rr]^{f}&& B.
}$$ 
The counit of the adjunction $f_!\dashv f^\star$ is given by the map $f_Y:f_!f^\star(Y)\to Y$
for $Y=(Y,q)\in  \mathcal{C}/B$.

\begin{prop}\label{basechangetribe0} 
If $u:(X,p)\to (Y,q)$ in $\mathcal{E}/B$ is a map in $\mathcal{E}/B$,
then the following square is cartesian,
\begin{equation}\label{basechangemap}
\xymatrix{
f^\star(X)  \ar[d]_{f^\star(u)} \ar[rr]^{f_X}&& X  \ar[d]^u  \\
f^\star(Y) \ar[rr]^{f_Y}& & Y 
}
\end{equation}
\end{prop}

\begin{proof}
We have a commutative diagram
 $$\xymatrix{
  \ar@/_2pc/[dd]_-{f^\star(p)} f^\star(X)  \ar[d]^{f^\star(u)} \ar[rr]^{f_X}& & X \ar[d]^{u}   \ar@/^2pc/[dd]^-{p}  \\
 f^\star(Y)  \ar[d]^{f^\star(q)}\ar[rr]^{f_Y}& & Y \ar[d]^{q}   \\
A \ar[rr]^{f}& & B  
}$$ 
where $f^\star(u)\defeq A\times_Bu:A\times_B X  \to  A\times_B Y$.
The top square of the diagram is cartesian by lemma \ref{lemmacartesiansq},
since the bottom square and the composite square are cartesian.
\end{proof}

\begin{prop}\label{propofquadrable} Every isomorphism is carrable.
The composite of two carrable maps is carrable.
The base change of a carrable map along any map
is carrable.
\end{prop}

\begin{proof} Let $f:A\to B$
and $g:B\to C$ be two carrable maps  in a category $\mathcal{C}$. 
The functor $f_!:\mathcal{C}/A\to \mathcal{C}/B$
admits a right adjoint $f^\star$ by proposition \ref{carrablerightadjoint}, since $f$ is carrable. Similarly, the functor $g_!:\mathcal{C}/B\to \mathcal{C}/C$
admits a right adjoint $g^\star$. It follows that the functor $(gf)_!=g_! f_!:\mathcal{C}/A\to \mathcal{C}/C$
admits a right adjoint $ f^\star g^\star$.  Hence the map $gf$ is carrable by proposition \ref{carrablerightadjoint}.
Let us show that the base change of a carrable map $p:E\to C$
 along any map $g:B\to C$ is carrable.  For this we have to show that 
 the projection $B\times_C E\to B$ in the following pullback square is quarrable,    
$$\xymatrix{
B\times_C E\ar[d]_{} \ar[r]^(0.6){}    & E \ar[d]^p\\
B  \ar[r]^{g}  & C.
}$$ 
If $f:A\to B$, then the fiber product $A\times_CE$ exists, since $p$ is carrable.
Moreover,  the left hand square of the following diagram is cartesian
by lemma \ref{lemmacartesiansq},
$$\xymatrix{
A\times_C E   \ar[d]^{} \ar[rr]^{f \times_C E}& & B\times_C E  \ar[d]_{} \ar[rr]^{} && E \ar[d]\\
A \ar[rr]^{f} && B \ar[rr]^{g}  && C
}$$ 
This proves that the map  $ B\times_C E \to C$
is carrable.
 \end{proof}

\begin{defi}\label{genuineprojectiondef}
We shall say that a map $p:E\to B$ in a category 
is a {\bf cartesian projection}, if there exists a map $q:E\to F$ such that the
pair $(p,q)$ exhibits $E$ as the cartesian product of $B$ by $F$.
\end{defi}

Recall that a category with finite cartesian products is
said to be {\it cartesian}.

\begin{prop}\label{genuineprojection}
Cartesian projections are closed under composition. 
 In a cartesian category, every cartesian projection is carrable
and the base change of a cartesian projection is a cartesian projection.
\end{prop}

 \begin{proof}
 Left to the reader.
   \end{proof}

 \begin{prop} {\rm (Frobenius reciprocity law, first form)}\label{Frobrecip1}
  Let  $f:A\to B$ and $g:B\to C$ be two maps in a category $\mathcal{E}$.
 If a map $p:X\to C$ is carrable, then the map $B\times_C X\to B$ is carrable and we have
   \begin{equation}\label{frobeniusidentity}
  A\times_B (B\times_C X) =  A\times_C X 
     \end{equation}
   \end{prop}

 \begin{proof}
 The right hand square of the following diagram is cartesian by construction,
  $$\xymatrix{
A\times_C X \ar[rr]^{f \times_C X} \ar[d]&& B\times_C X  \ar[rr] \ar[d] &&X\ar[d]^p \\
 A\ar[rr]^f && B \ar[rr]^g && C} $$
 The composite square is also cartesian for the same reason.
  It then follows from lemma \ref{lemmacartesiansq} that the left
  hand square is cartesian.
  \end{proof}

 Let $p:X\to A$ and  $f:A\to B$ be two carrable maps in a category $\mathcal{E}$.
 Then the map $fp:X\to B$ is carrable and  $f_!(X)=(X,fp)$.
 If $Y=(Y,q)\in  \mathcal{E}/B$ then the fiber products
$f_!(X)\times_BY$ and $X\times_A f^\star(Y)$ exists.

 \begin{prop} {\rm (Frobenius reciprocity law, second form)}\label{Frobrecip2}
 With the hypothesis above, we have 
   \begin{equation}\label{frobeniusidentity2}
 f_!(X\times_A f^\star(Y) )=f_!(X)\times_BY
     \end{equation}
   \end{prop}

 \begin{proof}
 Consider
 the following diagram of cartesian squares,
 $$\xymatrix{
X\times_A f^\star(Y) \ar[r] \ar[d]& f^\star(Y)  \ar[rr] \ar[d] &&Y\ar[d]^q\\
 X\ar[r]^p & A \ar[rr]^f && B 
} $$
The composite square is cartesian by lemma \ref{lemmacartesiansq}.
\end{proof}

\subsection{Lifting properties}

Recall that a map $u:A\to B$ in a category $\mathcal{E}$  is said to have the
{\it left lifting property} with respect to a map $f:X\to Y$,
and that $f$ is said to have the {\it right lifting property} 
with respect to $u$, if every commutative square
$$
\xymatrix{ A\ar[d]_u \ar[r]^a& X\ar[d]^f \\
B \ar[r]^b &Y
}
$$
has a diagonal filler $d:B\to X$ ($fd=b$ and $du =a$),
$$
\xymatrix{ A\ar[d]_u \ar[r]^a& X\ar[d]^f \\
B \ar[r]^b \ar@{-->}[ru]^{d} \ar@{..>}[ur]^{d} &Y \ .
}
$$
We shall denote this relation by $u\pitchfork f$.
A map $f \in \mathcal{E}$ is invertible if and only if we have $f\pitchfork f$.

\medskip

If $\mathcal{K}$ is a class of maps in a category $ \mathcal{E}$,
let us put 
$$\mathcal{K}^\pitchfork =\{f\in \mathcal{E}: \forall u\in \mathcal{K}\  \  u\pitchfork f \}$$
$${}^\pitchfork \mathcal{K}=\{u\in \mathcal{E}: \forall f\in \mathcal{K}\  \  u\pitchfork f \}$$

\medskip

Recall a map $r:B\to A$ in a category $\mathcal{E}$ is said be a {\it retraction} of a map $i:A\to B$ is $ri=1_A$.
$$\xymatrix{
A \ar[rr]^i && B   \ar@/_2pc/[ll]_-{r}
}$$ 
A map $i:A\to B$ is said to be a {\it split monomorphism} it it admits a retraction.
An object $A$ is said to be
a {\it retract} of an object $B$ if there exists a split monomorphism $i:A\to B$.
A map $f$  is said to be 
a {\it retract} of a map $g$ if the object $f$  of the category of arrows $\mathcal{E}^I$
is a retract of the object $g$.
A map $f:A\to B$  is said to be a {\it codomain retract}
of a map $g:A\to B$ if the object $(B,f)$ of the category $A\backslash \mathcal{E}$
is a retract of the object $(B,g)$
Dually, a map $f:A\to B$  is said to be a {\it domain retract}
of a map $g:C\to B$ if the object $(A,f)$ of the category $\mathcal{E}/B$
is a retract of the object $(C,g)$.

\medskip

Recall that a class $\mathcal{C}$ of maps in a category $\mathcal{E}$ is said to be {\it closed under retracts}
if every retract of a map in $\mathcal{C}$ belongs to $\mathcal{C}$.

\medskip

\begin{prop}\label{triviftingprop}
The class $\mathcal{K}^\pitchfork $ contains the isomorphisms, 
and is closed under composition, retracts and base changes.
Dually, the class ${}^\pitchfork \mathcal{K} $ contains the isomorphisms, 
and is closed under composition, retracts and cobase changes .
\end{prop}

\begin{proof} Left to the reader.
\end{proof}

\subsection{Simplicial categories} \label{simplicialcategories}

We shall denote by $\Delta$ the category whose objects are finite non-empty ordinals $[n]=\{0,\ldots, n\}$ ($n\geq 0$)
and whose morphisms are order preserving maps. Recall that a {\it simplicial set} is a pre-sheaf on $\Delta$.
We shall denote the category of simplicial sets by $\sSet$.
The category $\sSet$ is cartesian closed. We shall denote by $[X,Y]$
the simplicial set of maps $X\to Y$ between two simplicial sets $X$ and $Y$.
A {\it simplicial category} $\mathcal{C}$ is a category enriched over $\sSet$. 
We shall denote the ordinary category underlying a simplicial category $\mathcal{C}$ by $\mathcal{C}_0$.
We shall say that a commutative square
\begin{equation} \label{squareforsimppullback}
\xymatrix{
C \ar[d]_{u} \ar[rr]^{g} &&\ar[d]  D\ar[d]^{v}  \\
A \ar[rr]^{f}  && B
}
\end{equation}
is 
 in $\mathcal{C}_0$ is {\it $s$-cartesian} if the following square of simplicial sets
 is cartesian for every object $K\in \mathcal{C}$.
 \begin{equation} \label{squareforsimppullback2}
\xymatrix{
\mathcal{C}(K,C) \ar[d]_{\mathcal{C}(K,u)} \ar[rr]^{\mathcal{C}(K,g)} &&\ar[d]  \mathcal{C}(K,D) \ar[d]^{\mathcal{C}(K,v)}  \\
\mathcal{C}(K,A) \ar[rr]^{\mathcal{C}(K,f)}  && \mathcal{C}(K,B)
}
\end{equation}

\begin{defi}  \label{defstrongterminal} We say that an object $\top$
in a simplicial category $\mathcal{C}$
 is {\it $s$-terminal} if $\mathcal{C}(K,\top)=1$
 for every object $K\in \mathcal{C}$.
  \end{defi}

\begin{defi}  \label{defstronglycarrablemap} We say that a map $f:A\to B$  in a simplicial category $\mathcal{C}$
 is {\bf $s$-carrable} if $f$ is carrable and  
 for every map $p:X\to B$ the cartesian square
\begin{equation} \label{stronglycaarble}
\xymatrix{
A\times_B X \ar[d]_{} \ar[rr]^{} &&\ar[d]  X\ar[d]^{p}  \\
A \ar[rr]^{f}  && B
}
\end{equation}
is $s$-cartesian 
 \end{defi}

If a map $f:A\to B$ in a simplicial category $\mathcal{E}$ is a strongly carrable, 
then the base change functor $f^\star:\mathcal{E}/B\to \mathcal{E}/A$
is simplicial.

If $A$ is an object in a simplicial category  $\mathcal{C}$, then the category $\mathcal{C}_0/A$
is simplicially enriched for every object $A\in \mathcal{C}$.
If $p:X\to A$ and $q:Y\to A$, then the simplicial set $Hom((X,p),(Y,q))$
is defined by the following pullback square of simplicial sets:
\begin{equation} \label{squareforsimphomset}
\xymatrix{
Hom((X,p),(Y,q)) \ar[d] \ar[r] &\mathcal{C}(X,Y) \ar[d]^{ \mathcal{C}(X,q) }  \\
1 \ar[r]^{p}  &  \mathcal{C}(X,A) 
}
\end{equation}
This defines a simplicial category $\mathcal{C}/A$ if we put 
$(\mathcal{C}/A)(X,Y):=Hom((X,p),(Y,q))$.

\medskip

We shall denote by $\sCat$ the category of simplicial categories and
simplicial functors. Recall that if $\mathcal{C}$ and $ \mathcal{D}$
are simplicial categories, then a natural transformation $\alpha:F\to G$
between simplicial functors $F,G:\mathcal{C}\to \mathcal{D}$
is said to be {\it strong} if the following square 
commutes for every pair of objects $A,B\in\mathcal{C}$.
$$\xymatrix{
\mathcal{C}(A,B) \ar[rr]^{F_{AB}} \ar[d]_{G_{AB}} && \mathcal{D}(FA,FB) \ar[d]^{\mathcal{D}(FA,\alpha_B)}\\
\mathcal{D}(GA,GB) \ar[rr]^{\mathcal{D}(\alpha_A,GB)} && \mathcal{D}(FA,GB)
}$$

The category  $\sCat$ has the structure 
of a 2-category in which a 2-cell is a strong natural transformation.
An adjunction $F\vdash G$ between two simplicial functors $F:\mathcal{C}\leftrightarrow \mathcal{D}:G$ is said to be {\it strong}
if its unit and counit are strong natural transformations.
A {\it simplicial presheaf} on a simplicial category $\mathcal{C}$ is a simplicial functor $\mathcal{C}^{op}\to \mathcal{S}$. 
The category of simplicial pre-sheaves on $\mathcal{C}$ is a cartesian closed simplicial category
$[\mathcal{C}^{op}, \mathcal{S}]$.
 If $F,G\in [\mathcal{C}^{op}, \mathcal{S}]$,
then $[F,G]_n$ is the set of strong natural transformations $\Delta[n]\times F\to G$.
The {\it Yoneda functor} $y:\mathcal{C} \to [\mathcal{C}^{op}, \mathcal{S}]$
takes an object  $A\in \mathcal{C}$ to the simplicial functor $y(A)=\mathcal{C}(-,A)$.
If $F:\mathcal{C}^{op}\to \mathcal{S}$ is a simplicial 
pre-sheaf and $A\in \mathcal{C}$, then the map $\eta:[y(A),F]_0\to F(A)_0$
defined by putting $\eta(\alpha)=\alpha(1_A)\in F(A)_0$
is a bijection between the set of strong natural transformations $\alpha:y(A)\to F$
and the set $F(A)_0$ (Yoneda lemma).  
When $\alpha:y(A)\to F$ is invertible, we say that the presheaf
$F$ is represented by the object $A$ equipped with  $\alpha(1_A)\in F(A)_0$.
If $X$ is an object of a simplicial category $\mathcal{C}$
and $K$ is a simplicial set, we say that an object $X^K\in \mathcal{C}$
equipped with a map $e:K\to \mathcal{C}(X^K,X)$
is the {\it cotensor} of $X$ by $K$
if the simplicial presheaf  $[K,\mathcal{C}(-,X)]: \mathcal{C}^{op}\to \mathcal{S}$
is represented by $X^K$ equipped with the map $e:K\to \mathcal{C}(X^K,X)$.
In which case the map $\epsilon: K\to \mathcal{C}(X^K, X)$ is universal
in the following sense: for every object $Y\in \mathcal{C}$
and every map $f:K\to  \mathcal{C}(Y, X)$ there is a unique map
$g:Y\to X^K$ such that the following square diagram commutes
$$
\xymatrix{
K \ar[dr]_f \ar[r]^(0.4){e} &\ar[d]  \mathcal{C}(X^K, X) \ar[d]^{\mathcal{C}(g, X)}  \\
&  \mathcal{C}(Y, X) 
}
$$
We say that the cotensor $X^K$ is {\it finite} when $K$ is finite.
We say that a simplicial category  $\mathcal{C}$ is 
{\it finitely cotensored} if the cotensor $X^K$ exists for any object
$X\in \mathcal{C}$ and any finite simpllcial set $K$.
In which case the cotensor product  $(K,X)\mapsto X^K$ is a
simplicial functor $(\mathcal{S}_f)^{op}\times \mathcal{C}\to \mathcal{C}$
where $\mathcal{S}_f$ denotes the category of finite simplicial sets.

\newpage

\section{Appendix on homotopy theory}\label{Appendix on homotopy theory}

\subsection{Brown fibration categories}\label{Appendix on fibration categories}

The following notion  is originally due to Ken Brown  \cite{Br} \cite{RB}.

\begin{defi} \label{definitionfibrationcat} A {\bf  fibration category} is clan $\mathcal{E}$ 
equipped with a class of  {\bf acyclic maps} such 
that:
\begin{itemize}
\item{} Every isomorphism is acyclic;
\item{} The class of acyclic maps has the 3-for-2 property; 
\item{} Every morphism can be factored as an acyclic map followed by a fibration;
\item{} The class of acyclic fibrations is stable under base change.
\end{itemize}
\end{defi}

\medskip

We shall say that a fibration category is {\it degenerated} if every map is acyclic.
Every tribe can be given the structure of a degenerated fibration category.

\medskip

Examples of fibration categories:
\begin{itemize}
\item{}  The category of fibrant objects of a model category has the structure of a fibration category in which the acyclic maps are the weak equivalences; 
\item{} We shall see in \ref{htribeisfibcat} that a $h$-tribe has the structure of a fibration category in which the acyclic maps are the homotopy equivalences.

\end{itemize}

If $A$ is an object of a fibration category $\mathcal{E}$, we shall denote by $\mathcal{E}(A)$
the full subcategory of $\mathcal{E}/A$ whose objects are the fibrations $E\to A$.

\begin{prop}\label{localfibcat} If $\mathcal{E}$ is a fibration category,
then so is the tribe $\mathcal{E}(A)$ for every object $A\in \mathcal{E}$,
where a map $f:(X,p)\to (Y,q)$ in $\mathcal{E}(A)$ is acyclic iff the map $f:X\to Y$  
is acyclic in $\mathcal{E}$.
\end{prop}

\begin{proof} Left to the reader
\end{proof}

\begin{defi} \label{defexactfunctorfibcat} An {\it exact functor} $F:\mathcal{E}\to \mathcal{E}'$ between fibration categories
is a morphism of tribes which preserves acyclic maps.
\end{defi}

 The fibration categories are the objects of a 2-category in which a 1-cell is an exact functor and
 a 2-cell is a natural transformation.
 We shall denote the 2-category of fibration categories by $\mathbf{FCat}$.

\medskip

The {\bf homotopy category} of a fibration category $\mathcal{E}$ is the category of fractions
$$Ho (\mathcal{E})=\mathcal{W}^{-1}\mathcal{E}$$
where $\mathcal{W}$ is the class of acyclic maps.
If $F:\mathcal{E} \to \mathcal{E}'$ is an exact functor,
then the following square of functors 
commutes by definition, where the vertical functors
are canonical,
\begin{equation} \label{compathomo}
\xymatrix{
\mathcal{E} \ar[d]_{[-]} \ar[rr]^{F} && \mathcal{E}'\ar[d]^{[-]}\\
Ho(\mathcal{E}) \ar[rr]^{Ho(F)} && Ho(\mathcal{E}') 
}\end{equation}

 \begin{defi} \label{defweakequivalenceofftribes} 
 We shall say that an exact
 functor $F: \mathcal{E}\to \mathcal{E}'$ between fibration categories is a {\it weak equivalence}
 if the induced functor $Ho(F): Ho(\mathcal{E})\to Ho(\mathcal{E}')$
 is an equivalence of categories.
 \end{defi}

\begin{prop} \label{htribeisfibcat} A $h$-tribe has the structure of a fibration category
in which the acyclic maps are the homotopy equivalences.
A sharp functor between $h$-tribes is exact.
\end{prop} 

\begin{proof} 
 The class of acyclic maps in a $h$-tribe  has the 3-for-2 property by \ref{trivialfibrationbasechange}.
Moreover, every morphism can be  factored as an anodyne map followed by a fibration, by definition of a $h$-tribe.
Thus, every morphism can be  factored as an acyclic map followed by a fibration,
since an anodyne map
is a homotopy equivalence by \ref{anohomotopyequiv}. 
Finally,
the class of acyclic fibrations is stable under base change by \ref{trivialfibrationbasechange}.
The first statement is proved.
Let us prove the second statement.
A sharp functor  between $h$-tribes 
 respects the homotopy relation by \ref{inverthoequiv3}.
 It thus preserves homotopy equivalences.
  \end{proof}

\begin{defi}
A  {\bf weak path object} of an object $A$ in a fibration category is a
commutative diagram
$$ 
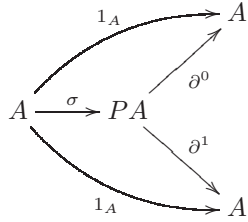
\xymatrix{
&& A \\
A  \ar@/_1pc/[rrd]_-{1_A}  \ar@/^1pc/[rru]^-{1_A}  \ar[r]^\sigma  & PA \ar[ur]_{\partial^0}\ar[dr]^{\partial^1}& \\
 && A
 }$$
obtained by factoring the diagonal $A\to A\times A$ 
as an acyclic map $\sigma:A\to PA$ followed by a fibration $(\partial^0,\partial^1):PA \to A\times A$.
\end{defi}

The maps $\partial^0$ and $\partial^1$ 
are fibrations, since  $\partial^0=p_1( \partial^0,\partial^1)$,
$\partial^1=p_2(\partial^0,\partial^1)$ and 
the projections $p_1,p_2:A\times A\to A$
are fibrations. The maps $\delta^0$ and $\delta^0$ are acyclic by 3-for-2, since $\sigma$ is acyclic
and we have $\delta^0\sigma=1_A$ and  $\delta^1\sigma=1_A$.

\medskip

 \begin{defi}
 A  {\bf weak mapping path object} of a map $f:A\to B$ in a fibration category is a
diagram
$$ 
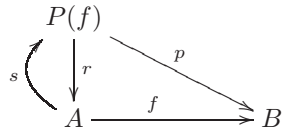
\xymatrix{
P(f)   \ar[d]^-{r} \ar[drr]^{p}  & &  \\
 \ar@/^1.5pc/[u]^-{s}     A  \ar[rr]^(0.4)f && B
 }$$
 obtained by factoring the map $(1_A,f) :A\to A\times B$ as 
an acyclic map  $s:A \to P(f)$
followed by a fibration $(r,p) :P(f)\to A\times B$.
 \end{defi}
 
By construction, we have $f=ps$ and $rs=1_A$.
The maps $r$ and $p$
are fibrations, since the projections $p_1:A\times B\to A$ and $p_2:A\times B\to B$
are fibrations and we have $r=p_1(r,p)$ and $p=p_2(r,p)$.
The map $r$ is acyclic by 3-for-2, since $rs=1_A$
and $s$ is acyclic.

\begin{lemma} \label{KBrownlemma0} 
Every map $f:A\to B$ in a fibration category admits a factorization
$f=ps:A\to P\to B$ with $p$ a fibration and $s$ a section of an acyclic fibration $r:P\to A$.
$$ \xymatrix{
P   \ar[d]^-{r} \ar[drr]^p  & &  \\
 \ar@/^1.5pc/[u]^-{s}     A  \ar[rr]^(0.4)f&& B
 }$$
\end{lemma}

\begin{lemma} \label{KBrownlemma} {\rm (Ken Brown's lemma)} 
 Let $ \mathcal{E}$ be a fibration category and
$\mathcal{C} $ be a category equipped with a class of weak equivalences
 having the 3-for-2 property and containing the isomorphisms.
If a functor $F:\mathcal{E} \to \mathcal{C} $ 
takes acyclic fibrations to weak equivalences,
then it takes acyclic maps  to weak equivalences.
\end{lemma}

\begin{proof} If a map $f:A\to B$  in $\mathcal{E} $ is acyclic, let us show that $F(f)$ is a weak equivalence.
By lemma  \ref{KBrownlemma0}, the map $f:A\to B$ 
admits a factorization
$f=ps:A\to P\to B$, with $p$ a fibration an $s$ a section of an acyclic fibration $r:P\to A$.
The map $F(r)$ is a weak equivalence by the hypothesis on $F$.
It follows that $F(s)$ is a weak equivalence by 3-for-2, since $F(r)F(s)=F(rs)=F(1_A)=1_{FA}$
and $1_{FA}$ is a  weak equivalence.
The fibration $p$ is acyclic by 3-for-2, since the maps $f$ and $s$ are acyclic
and we have $f=ps$. Thus, $F(p)$ is a weak equivalence by the hypothesis on $F$.
It follows that $F(f)$ is a  weak equivalence by 3-for-2, since $F(f)=F(ps)=F(p)F(s)$
and $F(s)$ is a weak equivalence.
 \end{proof}

\begin{prop} \label{exactfunctorfibcat}
A morphism of tribes $F:\mathcal{E}\to \mathcal{E}'$ between fibration categories is exact if and only if
it takes acyclic fibrations to acyclic fibrations.
\end{prop}

\begin{proof} This follows from Lemma \ref{KBrownlemma}. 
 \end{proof}

\begin{prop} \label{basechangeexact} If $ \mathcal{E}$ is a fibration category, 
then the base change functor $f^\star:\mathcal{E}(B) \to \mathcal{E}(A)$
is exact for every map $f:A\to B$. 
\end{prop}

\begin{proof} Let us first verify that the base change functor $f^\star:\mathcal{E}(B) \to \mathcal{E}(A)$
preserves acyclic fibrations. If $p:X\to Y$ is a fibration in $\mathcal{E}(B)$,
then $f^\star(p)$ is a base change of $p$, since
 the following square is cartesian by \ref{basechangetribe0}, 
$$ \xymatrix{
f^\star X \ar[d]_{f^\star(p)}\ar[r]^(0.6){f_X}   & X \ar[d]^{p}      \\
f^\star Y \ar[r]^(0.6){f_Y }  & Y
 }$$
Thus, $f^\star(p)$ is an acyclic fibration  when $p$ is an acyclic fibration.
It then follows from \ref{exactfunctorfibcat} that the functor $f^\star$
preserves acyclic maps.
 \end{proof}

We say that an object $A$ in a fibration category is {\it contractible}
if the map $A\to 1$ is acyclic.

\begin{lemma}\label{basechangeofacylicmap} Let $B$ a contractible object of a fibration category.
If $f^{-1}(b)$ is the fiber of a fibration $f:A\to B$ at a point $b:B$,
then the inclusion  $i:f^{-1}(b)\to A$ is acyclic,
    $$ \xymatrix{
f^{-1}(b)  \ar[d] \ar[r]^(0.6){i} & \ar[d]  A \ar[d]^-{f}  \\
1 \ar[r]^{b}  & B \ .
}$$
\end{lemma}

 \begin{proof}
 The projection $p_1:A\times B\to A$ is an acyclic fibration by base change, since  
 the map $B\to 1 $ is an acyclic fibration.
Hence the map $(1_A,f):A\to A\times B$ is acyclic by 3-for-2, since $p_1(1_A,f)=1_A$.
Observe that $(1_A,f)$ is a map between two objects of $\mathcal{E}(B)$, since
$p_2(1_A,f)=f$ and $f$ is a fibration.
The top square of the following diagram is cartesian by lemma \ref{lemmacartesiansq}, 
since the composite square is cartesian,
 $$ \xymatrix{
 f^{-1}(b) \ar[d]_-{i}   \ar[r]^(0.6){i} & \ar[d] A  \ar[d]^-{(1_A,f)}  \\
A\ar[r]^{A\times b}\ar[d]  & A\times B   \ar[d]^-{p_2}\\
1 \ar[r]^{b}  & B 
}$$
Hence we have $b^\star(1_A,f)=i$, where $b^\star:\mathcal{E}(B)\to \mathcal{E}$ is the base change
functor along the map $b:1 \to B$.
It then follows from proposition \ref{basechangeexact} that $i$
is acyclic, since $(1_A,f)$ is acyclic.
\end{proof}

\begin{prop} \label{basechangeofacylicmap2} In a fibration category,  the base change of an acyclic map along a fibration 
is acyclic. 
\end{prop}

 \begin{proof} Let $w:C\to B$ be an acyclic map and
$f:A\to B$ be a fibration.
 We wish to show that the map $w_A$ in the following cartesian square is acyclic,
$$ \xymatrix{
  \ar[d]_-{w^\star(f)} w^\star(A)  \ar[r]^(0.6){w_A} & \ar[d] A  \ar[d]^-{f}  \\
C\ar[r]^{w}  & B \ .
}$$
Let us first consider the case where $w$ is the section $s:C\to B$ of an acyclic 
fibration $r:B\to C$. 
    $$ \xymatrix{
  \ar[d]_-{s^\star(f)} s^\star(A)  \ar[r]^(0.6){s_A} & \ar[d] A  \ar[d]^-{f}  \\
C \ar[r]^{s}  & B  \ar@/^1pc/[l]^-{r}   \ 
}  $$
In this case, the square can be lifted to the category $\mathcal{E}(C)$, 
$$ \xymatrix{
  \ar[d]_-{s^\star(f)}  s^\star(A) \ar[r]^(0.6){s_A} & \ar[d] A  \ar[d]_-{f} \ar[ddr]^{rf}&  \\
C \ar@{=}[drr] \ar[r]^{s}  & B  \ar[dr]^(0.3){r} &\\
&& C
}  $$
The object $(B,r)$ of $\mathcal{E}(C)$ is contractible, since the fibration $r:B\to C$
is acyclic. Moreover, the object $s^\star(A)$ is the fiber
of the map $f:(A,rf) \to (B,r)$ at the point $s:(C,1_C)\to (B,r)$. It then follows from lemma \ref{basechangeofacylicmap}
that the map $s_A:s^\star(A) \to A$ is acyclic. Let us now consider the general case of an acyclic map $w:C\to B$.
By lemma \ref{KBrownlemma0}, the map $w:C\to B$
admits a factorization $w=ps:C\to P\to B$ with  $p$ a fibration and $s$ a section of an acyclic fibration $r:P\to A$.
The map $u$ is acyclic by 3-for-2, since $r$ is acyclic and $ru=1_A$. 
 The map $p$ is also acyclic by 3-for-2, since $w$ is acyclic and we have
   $ps=w$. 
The left hand square of following diagram  is cartesian by lemma \ref{lemmacartesiansq}, since the right hand square
   and the composite square are cartesian by construction,
     $$ \xymatrix{
 \ar[d]_-{f'}  C\times_B A  \ar[r]_{s'}    \ar@/^2pc/[rr]^-{w_A}  &  \ar[d]_-{p_1}  P\times_BA  \ar[r]_(0.6){p_2} &  \ar[d]^-{f}  A\\
 C \ar[r]^s  \ar@/_2pc/[rr]_-{w}  & P \ar[r]^{p} &   B
}$$
The map $p_2$ is an acyclic fibration by base change, since $p$ is an acyclic fibration.
Hence it suffices to show that the map $s'$ is acyclic, since $w_A=p_2's'$.
But $s'$ is the base change of $s$ along $p_1$, since the left hand square is cartesian.
The projection $p_1$ is a fibration by base change, since $f$ is a fibration.
The map $s$ is the section of the acyclic fibration $p: P\to C$ by construction.
It then follows by the first part of the proof that $s'$ is acyclic.
 \end{proof}

\subsection{Homotopy cartesian squares} \label{hcsquareinaBrownfibcat}

The notion of homotopy cartesian square can be defined in any Brown fibration category

\begin{defi} \label{def:hcsquare}
A commutative square in Brown fibration category 
\begin{equation} \label{hocartsquarefirst}
\xymatrix{
C    \ar[d]_u   \ar[r]^g  & D \ar[d]^v  \\
A \ar[r]^{f}  & B
}
\end{equation}
is said to be {\it homotopy cartesian} if the map 
$(u,v_0g):C\to A\times_B D'$ in the following diagram 
is acyclic
\begin{equation} \label{hocartsquare2}
\xymatrix{
C    \ar[d]_{(u,v_0g)}   \ar[r]^g  & D \ar[d]^{v_0}  \\
A\times_B D'    \ar[d]   \ar[r]  & D' \ar[d]^{v_1}  \\
A \ar[r]^{f}  & B
}
\end{equation}
for a choice of factorisation $v=v_1v_0:D\to D'\to B$ with $v_0$ an acyclic map
and $v_1$ a fibration.
\end{defi}

\begin{lemma} If the square (\ref{hocartsquarefirst}) is homotopy cartesian,
then the map $(u,v_0g):C\to A\times_B D'$ is acyclic
 for {\rm every} factorisation  $v=v_1v_0:D\to D'\to B$ with $v_0$ an acyclic map
and $v_1$ a fibration, Moreover, the transposed square 
$$
\xymatrix{
C    \ar[d]_g   \ar[r]^u  & A \ar[d]^f  \\
D \ar[r]^{v}  & B
}
$$
is homotopy 
cartesian:
\end{lemma}

\begin{proof} By hypothesis, the map $(u,v_0g):C\to A\times_B D'$ in the diagram (\ref{hocartsquare2})
is acyclic for a choice of factorisation $v=v_1v_0:D\to D'\to B$ with $v_0$ an acyclic map
and $v_1$ a fibration.
Consider a factorisation $f=f_1f_0:D\to D'\to B$ with $f_0$ an acyclic map
and $f_1$ a fibration. Let us show that 
the map 
$(f_0u,g):C\to A'\times_B D$ in the following diagram 
\begin{equation} \label{hocartsquarefirst2}
\xymatrix{
C    \ar[d]_u  \ar[r]^(0.4){(f_0u,g)} & A'\times_B D \ar[d]  \ar[r]  & D \ar[d]^v  \\
A\ar[r]^{f_0} &  A' \ar[r]^{f_1}  & B
}
\end{equation}
is acyclic. Consider the following 
commutative diagram with three pullback squares $(b), (c)$ and $(d)$.
\begin{equation} \label{hocartsquarethird}
\xymatrix{
C    \ar[d]_{(u,v_0g)}    \ar[r]^{(f_0u,g)} \ar@{}[dr]|{(a)} & A'\times_B D \ar[d]_{v'_0}   \ar[r]^{}  \ar@{}[dr]|{(b)} & D \ar[d]^{v_0}  \\
 A\times_B D'    \ar[d]  \ar[r]^{f'_0}  \ar@{}[dr]|{(c)}  & A'\times_B D' \ar[d]  \ar[r]  \ar@{}[dr]|{(d)}  & D' \ar[d]^{v_1}  \\
A\ar[r]^{f_0} &  A' \ar[r]^{f_1}  & B
}
\end{equation}
The map $v'_0:=A'\times_B v_0:A'\times_B D \to A'\times_B D'$
is a base change of the map $v_0$ along the fibration $f_1$.
Hence the map $v'_0$ is acyclic by Proposition \ref{basechangeofacylicmap2},
since $v_0$ is acyclic.
Similarly, the map $f'_0:=f_0\times_B D' : A'\times_B D'   \to A'\times_B D'$
is acyclic since $f_0$ is acyclic.
It follows by 3-for-2 that the map $(f_0u,g)$
is acyclic, since the square $(a)$ commutes and the maps
$(u,v_0g)$, $f'_0$ and $v'_0$ are acyclic.
\end{proof}

The following three lemmas are classical.

\begin{lemma}\label{homotopycartesiancriterion1} Suppose that the map $f$
in the following a commutative square is acyclic.
\begin{equation} \label{hocartsquare7}
\xymatrix{
C    \ar[d]   \ar[r]^g  & D \ar[d]  \\
A \ar[r]^{f}  & B
}
\end{equation}
Then the square is homotopy cartesian if and only if the map $g$ is acyclic.
\end{lemma}

\begin{proof}
Left as an exercise to the reader.
\end{proof}

\begin{lemma} \label{lemmahocartesiansqcriterion2} Suppose that 
the right hand square of the following commutative diagram is homotopy cartesian:
$$\xymatrix{
X  \ar[d] \ar[r]^{}&  Y \ar[d] \ar[r]^{} & Z \ar[d]\\
A \ar[r]^{} & B \ar[r]^{}  & C
}  $$ 
Then the left hand square is homotopy cartesian 
if and only if the composite square is homotopy cartesian.
\end{lemma}

\begin{proof}
Left as an exercise to the reader.
\end{proof}

\begin{lemma} \label{cubelemmahocartesiansqcriterion3} 
Suppose that we have commutative cube
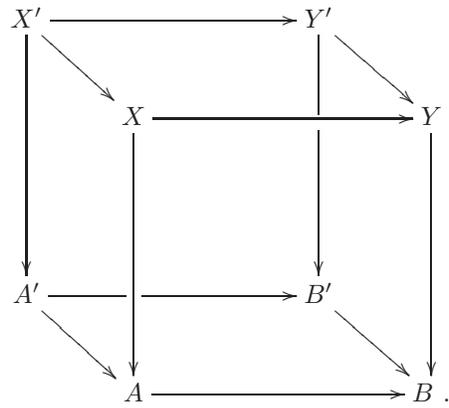
\[
\xymatrix{
X' \ar[dr] \ar[rrr] \ar[ddd] & &  & Y' \ar '[d] [ddd]^(0,3){}   \ar[dr]   &  \\
  & X \ar[ddd] \ar[rrr] &  & &  Y\ar[ddd] \\ 
  &&&&\\
A' \ar '[r] [rrr]^(0.3){} \ar[dr] & &  &B' \ar[dr] &   \\
  & A  \ar[rrr]^{} &  & & B   \ .}
\]
in which both the left and the right hand faces are homotopy cartesian.
If the front face is homotopy cartesian, then the back face is homotopy cartesian.
\end{lemma}

\begin{proof}
Left as an exercise to the reader.
\end{proof}


\newpage


\begin{thebibliography}{Article(3)}



\bibitem[A1]{A1} S. Awodey, \emph{A cubical model of homotopy type theory} (arXiv:1607.06413)



\bibitem[A2]{A2} S. Awodey, \emph{Natural models of homotopy type theory} (arXiv:1406.3219)



\bibitem[ABFJ1]{ABFJ1} M. Anel, G. Biedermann, E. Finster \& A. Joyal, \emph{A Generalised Blakers-Massey Theorem} (arXiv:1703.09050)



\bibitem[ABFJ2]{ABFJ2} M. Anel, G. Biedermann, E. Finster \& A. Joyal, \emph{Goodwillie's Calculus of Functors and Higher Topos Theory} (arXiv:1703.09632).





\bibitem[ALV]{ALV} B. Ahrens, P. LeFanu Lumsdaine \& V. Voevodsky, \emph{Categorical structures for type theory in univalent foundation}
(arXiv:1705.04310) 



\bibitem[AW]{AW} S. Awodey \& M. Warren, \emph{Homotopy theoretic models of identity types} {arXiv:0709.0248} Math, Proc. Camb. Phil. Soc. (2009), Vol. 146, pp. 45-55.



\bibitem[Ci1]{Ci1} D. C. Cisinki,  \emph{Les pr\'efaisceaux comme mod\`eles des types d'homotopie}, 
Ast\'erisque 308, S.M.F. (2006).




\bibitem[Ci2]{Ci2} D. C. Cisinki,  \emph{Cat\'egories d\'erivables}, bull. Soc. Math. France 138 (2010), no.3, 317-393.


\bibitem[C3i]{Ci3} D. C. Cisinki,  \emph{Invariance de la $K$-th\'eorie par \'equivalences d\'eriv\'ees}, J. K-Theory 6 (2010), no.3, 505-546.


\bibitem[Ci4]{Ci4} D. C. Cisinki,  \emph{Higher category theory and homotopical algebra}, In preparation.




\bibitem[CCHM]{CCHM} C. Cohen, T. Coquand, S. Huber \& A, M\"{o}rtberg, \emph{Cubical Type Theory:a constructive interpretation of the univalence axiom}, 
(arXiv:1611.02108).




\bibitem[FFLL]{FFLL} Kuen-Bang Hou (Favonia), E. Finster, D. Licata \& P. LeFanu Lumsdaine, \emph{A mechanisation of the Blakers-Massey connectivity theorem in Homotopy Type Theory} (arXiv:1605.03227),


\bibitem[GG]{GG} N. Gambino \& R. Garner, \emph{The identity type weak factorisation system} (arXiv:0803.4349).




\bibitem[HoTTb]{HoTb} \emph{Homotopy Type Theory Book}, {https://homotopytypetheory.org/book/}.




\bibitem[Jac]{Jac}  B. Jacob, \emph{Categorical Logic and Type Theory} Elsevier, 1999.





\bibitem[Jo]{Jo}  A. Joyal, \emph{The theory of quasi-categories and its applications}, Quadern 45, Vol. II, Centre de Recerca Matem`atica
Barcelona, 2008.






\bibitem[JT]{JT} A. Joyal and M. Tierney, \emph{Elements of Simplicial Homotopy Theory},
to appear.





\bibitem[KL]{KLV} C. Kapulkin \& P. LeFanu Lumsdaine, \emph{The homotopy theory of type theories} (arXiv:1610.00037), 



\bibitem[KLV]{KLV} C. Kapulkin, P. LeFanu Lumsdaine \& V. Voevodsky, \emph{Univalence in Simplicial Sets} (arXiv:1203.2553), 





\bibitem[KS]{KS} C. Kapulkin \& K. Szumilo, \emph{Internal Language of finitely complete $(\infty,1)$-categories} (arXiv:1709.09519), 






\bibitem[LS]{LS} P. LeFanu Lumsdaine \& M. Shulman  \emph{Semantics of higher inductive types} (arXiv:1705.07088)




\bibitem[Mo]{Mo} S. Moss, \emph{The dialectica model of type theory}, Cambridge Phd, submitted 2017.






\bibitem[No]{No} P. North, \emph{Toward a topological models of homotopy type theory} (https://www.dpmms.cam.ac.uk/~prn27/eca-paige-north.pdf)






\bibitem[RSS]{Rez2} E. Rijke, M. Shulman \& B. Spitters, \emph{Modalities in homotopy type theory} (arXiv:1706.07526).



\bibitem[Sh]{Shi} M. Schulman,  \emph{Univalence for inverse diagrams and homotopy canonicity}. Math. Structures Comput. Sci. 25 (2015), no.5, 1203-12.

 \bibitem[Sz1]{Sz1} K. Szumilo, \emph{Homotopy theory of cofibration categories} Homology Homotopy Appl. 18
(2016), no. 2, 345-357.

 \bibitem[Sz2]{Sz2} K. Szumilo, \emph{Homotopy theory of cocomplete quasicategories}, Alg. Geom. Top. 17 (2017), no. 2, 765-791.

\bibitem[VdB]{VdB} B. Van den Berg, \emph{Path categories and propositional identity types} (arXiv:1604.06001).



\bibitem[VG]{VG} T. von Glehn, \emph{Polynomials and Models of Type Theory}. Cambridege Phd Thesis, 2015. 


\bibitem[Vo]{Vo} V. Voevodsky, \emph{The equivalence axiom and univalence models of type theory} (arXiv:1402.5556).










\end{thebibliography}
\end{document}